\documentclass[11pt]{article}
\usepackage[utf8]{inputenc} 
\usepackage{amstext, amsmath,latexsym,amsbsy,amssymb}
\usepackage{hyperref}
\usepackage{amsfonts}
\newcommand{\hk}{\hspace*{.26in}}

\usepackage{graphicx}
\usepackage{amsfonts,graphicx}
\usepackage{amsfonts,amsmath,amssymb,amsthm}
\usepackage{epsf}
\usepackage{graphics}
\usepackage{psfrag}
\usepackage{times}
\usepackage{epsfig}
\usepackage{subfig}
\usepackage[titletoc]{appendix}
\usepackage{multirow}
\usepackage{hhline}

\RequirePackage{natbib}
\newcommand{\trace}{\ensuremath{tr}}

\long\def\comment#1{}
\renewcommand\vec[1]{\ensuremath\boldsymbol{#1}}

\renewcommand{\trace}{\ensuremath{\operatorname{tr}}}
\newcommand{\Ocal}{\ensuremath{\mathcal{O}}}
\newcommand{\Gcal}{\ensuremath{\mathcal{G}}}
\newcommand{\Xcal}{\ensuremath{\mathcal{X}}}

\comment{
\setlength{\topmargin}{0 in}
\setlength{\textwidth}{5.5 in}
\setlength{\textheight}{8 in}
\setlength{\oddsidemargin}{0.5 in}
}

\theoremstyle{plain}

\newtheorem{theorem}{Theorem}

\numberwithin{theorem}{section}

\newtheorem{proposition}{Proposition}

\numberwithin{proposition}{section}

\newtheorem{lemma}{Lemma}

\numberwithin{lemma}{section}

\newtheorem{definition}{Definition}

\numberwithin{definition}{section}

\numberwithin{condition}{section}

\numberwithin{problem}{section}

\newtheorem{corollary}{Corollary}

\numberwithin{corollary}{section}

\numberwithin{assumption}{section}

\newtheorem{example}{Example}

\numberwithin{example}{section}

\numberwithin{conjecture}{section}

\theoremstyle{definition}

\numberwithin{remark}{section}

\renewenvironment{abstract}
 {\small
  \begin{center}
  \bfseries \abstractname\vspace{-.5em}\vspace{0pt}
  \end{center}
  \list{}{%
    \setlength{\leftmargin}{15mm}
    \setlength{\rightmargin}{\leftmargin}%
  }%
  \item\relax}
 {\endlist}

\begin{document}
\small \normalsize

\comment{\begin{frontmatter}

\author{\fnms{Nhat} \snm{Ho}\ead[label=e1]{minhnhat@umich.edu} and
\fnms{XuanLong} \snm{Nguyen} \thanksref{t1} \ead[label=e2]{xuanlong@umich.edu}}
\affiliation{University of Michigan}
\thankstext{t1}{This research is supported in part by grants
NSF CCF-1115769, NSF OCI-1047871, NSF CAREER DMS-1351362,
and NSF CNS-1409303. The authors would like to thank
Elisabeth Gassiat, Xuming He, Judith Rousseau, Naisying 
Wang, Shuheng Zhou and several others for valuable discussions related to 
this work.}

\address{Department of Statistics \\
University of Michigan \\
456 West Hall \\
Ann Arbor, Michigan 48109-1107\\
USA \\
\printead{e1,e2}}

\runauthor{N. Ho, X. Nguyen}
}

\begin{center}

\textbf{\Large  Identifiability and optimal rates of convergence\\
\vspace{3 mm}for parameters of multiple types in finite mixtures}

\vspace{3 mm}

\large{Nhat Ho and XuanLong Nguyen}

\vspace{3 mm}

\large{Technical report 536 \\
Department of Statistics \\
University of Michigan}

\vspace{3 mm}
\large{January 9, 2015}

\end{center}

\begin{abstract}
This paper studies identifiability and convergence behaviors for
parameters of multiple types in finite mixtures, and the effects 
of model fitting with extra mixing components. First, we present
 a general theory for strong
identifiability, which extends from the previous work of~\cite{Nguyen-13}
and~\cite{Chen-95}
to address a broad range of mixture models and to handle
matrix-variate parameters. These models are shown to share the same Wasserstein distance
based optimal rates of convergence for the space of mixing distributions  
--- $n^{-1/2}$ under $W_1$ for the exact-fitted and 
$n^{-1/4}$ under $W_2$ for the over-fitted setting,
where $n$ is the sample size. 
This theory, however, is not applicable to several important model classes, 
including location-scale multivariate Gaussian mixtures, 
shape-scale Gamma mixtures and location-scale-shape skew-normal mixtures.
The second part of this work is devoted to demonstrating that for
these "weakly identifiable" classes, algebraic structures of 
the density family play a fundamental role in determining 
convergence rates of the model parameters, which
display a very rich spectrum of behaviors. For instance, the optimal rate 
of parameter estimation in an over-fitted location-covariance 
Gaussian mixture is precisely determined by the 
order of a solvable system of polynomial equations --- 
these rates deteriorate rapidly as more extra components are added 
to the model. The established rates for a variety
of settings are illustrated by a simulation study.
\footnote{This research is supported in part by grants
NSF CCF-1115769, NSF OCI-1047871, NSF CAREER DMS-1351362,
and NSF CNS-1409303. The authors would like to acknowledge
Elisabeth Gassiat, Xuming He, Judith Rousseau, Naisying 
Wang, Shuheng Zhou and several others for valuable discussions related to 
this work.


}
\end{abstract}

\section{Introduction}

Mixture models are popular modeling tools for making inference about 
heterogeneous data~\citep{Lindsay-1995,Mclachlan-1988}. Under the mixture 
modeling, data are viewed as samples from a collection of 
unobserved or latent subpopulations, each
posits its own distribution and associated parameters. 
Learning about subpopulation-specific parameters is essential to 
understanding of the underlying heterogeneity.
Theoretical issues related to parameter estimation in mixture models,
however, remain poorly understood --- as noted in a recent textbook~\citep{DasGupta-08} (pg. 571), 
``mixture models are riddled with difficulties such as 
nonidentifiability''. 

Research about parameter identifiability for mixture models goes back
to the early work of~\cite{Teicher-1961,Teicher-1963,Yakowitzspragins-1968} and
others, and continues to attract much interest~\citep{PeterHall-2003, PeterHall-2005, HallNeeman-2005, Allman-2009}. To address parameter estimation rates, a natural approach is to study
the behavior of mixing distributions that arise in the mixture model. 
This approach is well-developed in the context of nonparametric 
deconvolution~\citep{Carroll-Hall-88, Zhang-90, Fan-91}, but these results
are confined to only a specific type of model -- the location mixtures. 
Beyond location mixtures there have been far fewer results.
In particular, for finite mixture models, a notable contribution was made by Chen, who 
proposed a notion of strong identifiability and established the convergence
of the mixing distribution for a class of over-fitted finite mixtures~\citep{Chen-95}. Over-fitted
finite mixtures, as opposed to exact-fitted ones, are mixtures that
allow extra mixing components in their model specification, when
the actual number of mixing components is bounded by a known constant.
Chen's work, however, was restricted to models that have only a single scalar parameter.
This restriction was effectively removed by Nguyen, who showed that 
Wasserstein distances (cf.~\citep{Villani-2009}) provide a natural source of
metrics for deriving rates of convergence of mixing distributions~\citep{Nguyen-13}.
He established rates of convergence of mixing distributions for
a number of finite and infinite mixture models with multi-dimensional parameters.
Rousseau and Mengersen studied
over-fitted mixtures in a Bayesian estimation setting~\citep{Rousseau-Mengersen-11}. 
Although they did not focus on mixing distributions per se, they
showed that the mixing probabilities associated with extra mixing 
components vanish at a standard $n^{-1/2}$ rate, subject to 
a strong identifiability condition on the density class.
Finally, we mention a related literature in computer 
science, which focuses almost exclusively on the analysis
of computationally efficient procedures for clustering with
exact-fitted Gaussian mixtures
(e.g.,~\citep{Dasgupta-99, Belkin-Sinha-10, Kalai-etal-12}).

Due to requirements of strong identifiability, the existing theories
described above are applicable to only certain classes of mixture models, 
typically those that carry a single parameter type. Finite mixture models with multiple 
varying parameters 
(location, scale, shape, covariance matrix) are considerably more 
complex and many do not satisfy such strong identifiability assumptions.
They include location-scale mixtures of Gaussians, 
shape-scale mixtures of Gammas, location-scale-shape mixtures of of skew-normals
(also known as skew-Gaussians). A theory for such models remains open.

\paragraph{Setting}
The goal of this paper is to establish rates of convergence
for parameters of multiple types, including matrix-variate parameters, that
arise in a variety of finite mixture models.
Assume that each subpopulation is distributed by a density function (with respect
to Lebesgue measure on an Euclidean space $\mathcal{X}$) that belongs 
to a known density class  
$\left\{f(x|\theta,\Sigma),\theta \in \Theta \subset \mathbb{R}^{d_{1}},
\Sigma \in \Omega \subset S_{d_{2}}^{++}, x\in \mathcal{X} \right\}$.
Here, $d_{1} \geq 1,d_{2} \geq 0$, $S_{d_{2}}^{++}$ is the set of all $d_2\times d_2$
symmetric positive definite matrices. A finite mixture
density with $k$ mixing components can be defined in terms of $f$ 
and a discrete mixing measure $G = \sum_{i=1}^{k} p_i \delta_{(\theta_i,\Sigma_i)}$
with $k$ support points as follows
\[p_G(x) = \int f(x|\theta,\Sigma) dG(\theta,\Sigma) = \sum_{i=1}^{k}p_i f(x|\theta_i,
\Sigma_i).\]
Examples for $f$ studied in this paper include the location-covariance 
family (when $d_1 = d_2 \geq 1$) under Gaussian or some 
elliptical families of distributions, the location-covariance-shape 
family (when $d_1 > d_2$) under the generalized multivariate Gaussian,
skew-Gaussian or the exponentially modified Student's t-distribution, and the 
location-rate-shape family (when $d_1 = 3, d_2 = 0$) under Gamma or other distributions.
The combination of location parameter with
covariance matrix, shape and rate parameters in mixture modeling 
enables rich and more accurate description of 
heterogeneity, but the interaction among varying parameter types can be complex, 
resulting in varied identifiability and convergence behaviors. 
In addition, we shall treat the settings of exact-fitted mixtures
and over-fitted mixtures separately, as the later typically carries
more complex behavior than the former.

\comment{It is conspicuous that in clustering problems, the inclusion of covariance matrix in the model is of paramount importance since we would like to know how the correlations among dimensions actually influence the distribution of data in each group.  Moreover, the constraint that $d_{1}$ and $d_{2}$ are not necessarily equal allows us to consider classes of density functions other than location-covariance(or elliptical distribution) ones. For instance, in the paper, we consider the location-covariance-shape family, which includes generalized multivariate Gaussian distribution and exponentially t-distribution, or location-covariance-shape-rate family, which includes modified multivariate normal-gamma distribution. The appearance of shape or rate in mixture models is rather popular in practice, especially clustering problems,  because apart from the true location and true covariance matrix, we also would like to know about the true shape or rate of density functions to determine the exact nature of each group. 
}

As shown by Nguyen, the convergence of mixture model parameters can be measured
in terms of a Wassertein distance on the space of mixing measures $G$~\citep{Nguyen-13}.
Let $G=\sum_{i=1}^{k}p_i \delta_{(\theta_i,\Sigma_i)}$ and
$G_0 = \sum_{i=1}^{k_0}p_i^0 \delta_{(\theta_i^0,\Sigma_i^0)}$ be two discrete probability
measures on $\Theta \times \Omega$, which is equipped with metric $\rho$.
Recall the Wasserstein distance of order $r$, for a given
$r \geq 1$:
\begin{eqnarray}
W_{r}(G,G_0)=\left(\mathop {\inf }\limits_{\vec{q}}{\mathop {\sum }\limits_{i,j}{q_{ij}\rho^{r}((\theta_{i},\Sigma_{i}),(\theta_{j}^0,\Sigma_{j}^0))}}\right)^{1/r} \nonumber,
\end{eqnarray}
where the infimum is taken over all joint probability distributions $\vec{q}$ on 
$[1,\ldots,k] \times [1,\ldots,k_0]$ such that, when expressing $\vec{q}$ as a $k\times k_0$ matrix, the
marginal constraints hold:
$\mathop {\sum }\limits_{j}{q_{ij}}=p_{i}$ and $\mathop {\sum }\limits_{i}{q_{ij}}=p_{j}'$. 
Suppose that a sequence of mixing measures $G_n \rightarrow G_0$ under $W_r$ metric
at a rate $\omega_n = o(1)$. If all $G_n$ have the same number of atoms $k=k_0$ as that of $G_0$, 
then the set of atoms of $G_n$ converge to the 
$k_0$ atoms of $G_0$ at the same rate $\omega_n$ under $\rho$ metric. 
If $G_n$ have varying $k_n \in [k_0,k]$ number of atoms, where $k$ is a fixed upper bound,
then a subsequence of $G_n$ can be constructed so that each atom of $G_0$ is a 
limit point of a certain subset of atoms of $G_n$ --- the convergence to each such limit
also happens at rate $\omega_n$.  Some atoms of $G_n$ may have limit points that 
are not among $G_0$'s atoms --- the mass associated with those atoms of $G_n$ 
must vanish at the generally faster rate $\omega_n^r$.

In order to establish the rates of convergence for the mixing measure $G$, our strategy is
to derive sharp bounds which relate the Wasserstein distance of mixing measures 
$G,G'$ and a distance between corresponding mixture densities $p_G,p_{G'}$,
such as the variational distance $V(p_G, p_{G'})$.
It is relatively simple to obtain upper bounds for the variational distance of mixing
densities ($V$ for short) in terms of Wasserstein distances $W_r(G,G')$
(shorthanded by $W_r$). 
Establishing (sharp) lower bounds for $V$ in terms of $W_r$ is the main 
challenge. Such a bound may not hold, due to a possible lack of
identifiability of the mixing measures: one may have $p_{G} = p_{G'}$,
so clearly $V = 0$ but $G\neq G'$, so that $W_r \neq 0$.

\paragraph{General theory of strong identifiability}
\comment{
The plan for this paper is as follows. In section 2, the paper focuses on establishing the tight upper bound of hellinger distance between two  $p_{G}$ and $p_{G'}$ based on $W_{1}(G,G')$ or $W_{2}^{2}(G,G')$ for classes of density functions $\left\{f(x|\theta,\Sigma),\theta \in \mathbb{R}^{d_{1}},\Sigma \in S_{d_{2}}^{++}\right\}$. It is an extension of the results from section 2 of \cite{Nguyen-13}, where the paper considered families of density functions $\left\{g(x|\theta),\theta \in \mathbb{R}^{d}\right\}$, $d \geq 1$. Examples \ref{example-Gaussian}, \ref{example-Student}, \ref{example-modstudent}, and \ref{example-gamma} attempt to give speficic upper bound of $h(p_{G},p_{G_{'}})$ in terms of $W_{1}(G,G')$ or $W_{2}(G,G')$ up to some constant powers under various classes of density functions with matrix-variate parameter spaces.}
The classical identifiability condition requires that $p_{G} = p_{G'}$ entails $G = G'$.
This amounts to the linear independence of elements $f$ in the density class~\citep{Teicher-1963}.
In order to establish quantitative lower bounds on a distance of mixture densities, 
we introduce several notions of strong identifiability,
extending from the definition of~\cite{Chen-95} to handle multiple parameter types, including
matrix-variate parameters. There are two kinds of strong identifiability.
One such notion involves taking the first-order derivatives of the function $f$ 
with respect to all parameters in the model, and insisting that these quantities be linearly
independent in sense to be precisely defined.
This criterion will be called ``strong identifiability in the first order'', or
simply first-order identifiability. When the
second-order derivatives are also involved, we obtain the second-order identifiability
criterion. It is worth noting that prior studies on parameter estimation
rates tend to center primarily the second-order identifiability condition
or something even stronger~\citep{Chen-95,Liu-Shao-04,Rousseau-Mengersen-11,Nguyen-13}.
We show that for exact-fitted mixtures, the first-order identifiability
condition (along with some additional regularity conditions) suffices
for obtaining that
\begin{equation}
\label{Eqn-b1}
V(p_{G},p_{G_0}) \gtrsim W_1(G,G_0),
\end{equation}
when $W_1(G,G_0)$ is sufficiently small.  Moreover, for a broad
range of density classes, we also have $V \lesssim W_1$,  for which
we actually obtain $V(p_{G},p_{G_{0}}) \asymp W_{1}(G,G_{0})$. A consequence of this fact is
that for any estimation procedure that admits the $n^{-1/2}$ convergence
rate for the mixture density under $V$ distance, the mixture model parameters
also converge at the same rate under Euclidean metric.

Turning to the over-fitted setting, second-order identifiability along with mild
regularity conditions would be sufficient for establishing that
for any $G$ that has \emph{at most} $k$ support points where $k \geq k_{0}+1$ and $k$ is fixed, 
\begin{equation}
\label{Eqn-b2}
V(p_{G},p_{G_0}) \gtrsim W_2^2(G,G_0).
\end{equation}
when $W_{2}(G,G_{0})$ is sufficiently small. 
The lower bound $W_{2}^{2}(G,G_{0})$ is sharp,
 i.e we can not improve the lower bound to $W_{1}^{r}$ for any $r<2$
(notably, $W_2 \geq W_1$). A consequence
of this result is, take any standard estimation method (such that the MLE) 
which yields $n^{-1/2}$ convergence rate for $p_G$, the induced
rate of convergence for the mixing measure $G$ is the minimax optimal $n^{-1/4}$
under $W_2$. It also follows that the mixing probability mass converge at $n^{-1/2}$ rate
(which recovers the result of ~\cite{Rousseau-Mengersen-11}),
in addition to showing that the component parameters converge 
at $n^{-1/4}$ rate.  

We also show that there is a range 
of mixture models with varying parameters of multiple types that satisfies the 
developed strong identifiability criteria. All such models exhibit 
the same kind of rate for parameter estimation. In particular,
the second-order identifiability criterion (thus the first-order identifiability) 
is satisfied by many density families $f$
including the multivariate Student's t-distribution, the
exponentially modified multivariate Student's t-distribution. 
Second-order identifiability also holds 
for several mixture models with multiple types of (scalar) parameters. These results
are presented in Section~\ref{Sec:Characterization}. The
proofs of these characterization theorems are rather technical, but one useful insight one
can draw from them is that
the strong identifiability condition (in either the first or the second order) 
is essentially determined by the smoothness
of the kernel density in question (which can be expressed in terms of how fast
the corresponding characteristic function vanishes toward infinity).


\begin{center}
\begin{table} 
\captionsetup{justification=centering}
\footnotesize{
\begin{center}

\begin{tabular}{| p{1.6cm} | p{1.8cm} | p{2.3cm} | p{2.3cm} | p{2.5cm} | p{2.3cm}|}
    \hline
    & \footnotesize{Density classes} & \footnotesize{Exact-fitted mixtures} & 
\footnotesize{Over-fitted mixtures} & 
\footnotesize{MLE rate for $G$ \newline for $n$-iid sample} & \footnotesize{Minimax lower bound for $G$} \\ 
\hhline{|=|=|=|=|=|=|}
\footnotesize{(I) \newline First-order \newline identifiable} 
& \footnotesize{Generalized Gaussian, Student's t, ...} 
& \footnotesize{$V \gtrsim W_{1}$} &  
&\footnotesize{Exact-fit: \newline $W_1 \lesssim n^{-1/2}$} 
& \footnotesize{Exact-fit: \newline $W_1 \gtrsim n^{-1/2}$} \\ 
\hline
(II)\newline
Second-order \newline identifiable & 
Student's t, exponentially modified Student's t, ...& 
same as (I) & 
$V \gtrsim W_{2}^{2}$ & 
Exact-fit: \newline same as (I) &
Exact-fit: \newline same as (I) \\ 
\cline{5-6} & & & & 
Over-fit: \newline $W_2 \lesssim n^{-1/4}$ & 
Over-fit: \newline $W_1 \gtrsim n^{-1/4}$ \\ 
\hhline{|=|=|=|=|=|=|}
Not \newline second-order \newline identifiable & 
location-scale multivariate Gaussian & 
same as (I) & 
$V \gtrsim W_{\overline{r}}^{\overline{r}}$,  \newline
$\overline{r}$ depending on $k-k_0$&
Exact-fit: \newline same as (I) &
Exact-fit: \newline same as (I) \\ \cline{5-6} & & & 
If $k-k_{0}=1$, $\overline{r}=4$ \newline 
If $k-k_{0}=2$, $\overline{r}=6$ & 
Over-fit: \newline $W_{\overline{r}} \lesssim n^{-1/2\overline{r}}$ & 
Over-fit: \newline $W_1 \gtrsim n^{-1/2\overline{r}}$ \\
    \hline
    & 
Gamma distribution & 
Generic case: \newline
$V \gtrsim W_{1}$ & 
Generic case: \newline
$V \gtrsim W_{2}^{2}$ & 
Generic: $W_1 \lesssim n^{-1/2}$ or 
$W_2 \lesssim n^{-1/4}$ & 
Generic: \newline
$W_1 \gtrsim n^{-1/2}$ \newline
$W_2 \gtrsim n^{-1/4}$ \\ \cline{3-6} & & 
Patho. case: \newline
$V \not \gtrsim W_{r}^{r}$ for any $r \geq 1$ & 
Patho. case: \newline $V \not \gtrsim W_{r}^{r}$ 
for any $r \geq 1$ & 
Patho. case: \newline unknown& 
Patho. case: logarithmic, i.e $W_r \gtrsim n^{-1/r}$
\newline $\forall r \geq 1$ \\ \cline{2-6} 
Not \newline first-order \newline identifiability & 
Location-exponential distribution & 
$V \not \gtrsim W_{1}^{r}$ \newline $\forall r \geq 1$ & 
$V \not \gtrsim W_{1}^{r}$ \newline $\forall r \geq 1$ & 
Unknown  & 
logarithmic \newline
$W_1 \gtrsim n^{-1/r}$ \newline $\forall r\geq 1$
\\ 
\cline{2-6} & & & & Exact fit: & Exact-fit: \\ \cline{5-6} & & 
Generic case: \newline
$V \gtrsim W_{1}$ & Generic case: \newline $V \gtrsim W_{\overline{m}}^{\overline{m}}$, where $\overline{m}=\overline{r}$ or $\overline{r}+1$ & Generic: \newline $W_{1} \lesssim n^{-1/2}$ & Generic: \newline $W_{1} \gtrsim n^{-1/2}$ \\ \cline{3-6} 
& & Patho. conformant: \newline
$V \gtrsim W_{2}^{2}$ & Patho. conformant: \newline unknown & Patho. conformant: \newline $W_{2} \lesssim n^{-1/4}$ & Patho. conformant: \newline $W_{2} \gtrsim n^{-1/4}$ \\ \cline{3-6} & Skew-Gaussian distribution & Patho. non-conformant: \newline $V \gtrsim W_{\overline{s}}^{\overline{s}}$ for some $\overline{s}$ & Patho. non-conformant: \newline unknown & Patho. non-conformant:  \newline $W_{\overline{s}} \lesssim n^{-1/2\overline{s}}$ & Patho. non-conformant: \newline $W_{3} \gtrsim n^{-1/6}$, or $W_{4} \gtrsim n^{-1/8}$, or $W_{5} \gtrsim n^{-1/10}$, or $\ldots$ \\ \cline{3-6} & & Otherwise: \newline $V \not \gtrsim W_{1}^{r}$ for any $r \geq 1$ & Otherwise: \newline unknown &  Otherwise: \newline unknown & Otherwise: \newline logarithmic \\ \cline{5-6} & & & & 
Over-fit: \newline
$n^{-1/2\overline{m}}$ \newline
or unknown \newline &  
Over-fit: \newline
unknown \\ \hline
\end{tabular}
\caption{
Summary of results established in this paper. 
To be precise, all upper bounds for
MLE rates are of the form $(\log n/n)^{-\gamma}$, but
the logarithmic term is removed in the table to avoid cluttering.}
\label{tablesummary}
\end{center}}
\end{table}
\end{center}

\paragraph{Theory for weakly identifiable classes}
We hurry up to point out that many common density classes do not satisfy either or both 
strong identifiability criteria. The Gamma family of distributions (with both
shape and scale parameters vary) is not identifiable in the first order. 
Neither is the family of skew-Gaussian distributions~\citep{Azzalini-1999, Azzalini-1996}.
Convergence behavior for the  mixture parameters of these two families are unknown,
in both exact and over-fitted settings.
The ubiquitous Gaussian family, when both location and scale/covariance parameters
vary, is identifiable in the first order, but \emph{not} in the second order.
So, the general theory described above can be applied to analyze exact-fitted Gaussian
mixtures, but not for over-fitted Gaussian mixtures. It turns out that these
classes of mixture models 
require a separate and novel treatment. Throughout this work, we shall call such
density families ``weakly identifiable classes'', i.e., those that are identifiable 
in the classical sense, but
not in the sense of strong identifiability taken in either the first or second order.

Weak identifiability leads to an extremely rich (and previously unreported)
spectrum of convergence behavior. 
It is no longer possible to establish inequalities~\eqref{Eqn-b1} 
and~\eqref{Eqn-b2}, because they do not hold in general. 
Instead, we shall be able to establish sharp bounds of the types $V\gtrsim W_r^r$ for
some precise value of $r$, which depends on the specific class of density in consideration.
This entails minimax 
optimal but non-standard rates of convergence for mixture model parameters.
In our theory for these weakly identifiable classes, the algebraic structure of the
density $f$, not merely its smoothness,  will now play the fundamental role in determining
the rates.

\textit{Gaussian mixtures:} We will first discuss the Gaussian family of densities of 
the standard form
$f(x|\theta,\Sigma)$, where $\theta \in \mathbb{R}^d$ and $\Sigma \in S_{d}^{++}$
are mean and covariance parameters, respectively.
The lack of strong identifiability in the second order is due to the following identity:
\[\frac{\partial^2 f}{\partial \theta^2}(x|\theta,\Sigma) = 
2 \frac{\partial f}{\partial \Sigma}(x|\theta,\Sigma),\]
which entails that the derivatives of $f$ taken with respect to the parameters up to the 
second order are not linearly independent. 
Moreover, this algebraic structure plays the fundamental role in our proof
for the following inequality:
\begin{eqnarray}
V(p_{G},p_{G_{0}}) \gtrsim W_{\overline{r}}^{\overline{r}}(G,G_{0}), \label{eqn:introductionone}
\end{eqnarray}
where $\overline{r} \geq 1$ is defined as the minimum value of 
$r \geq 1$ such that the following system of polynomial equations \begin{eqnarray}
\mathop {\sum }\limits_{j=1}^{k-k_{0}+1}{\mathop {\sum }\limits_{n_{1}+2n_{2} = \alpha }{\dfrac{c_{j}^{2}a_{j}^{n_{1}}b_{j}^{n_{2}}}{n_{1}!n_{2}!}}}=0 \ \ \text{for all } \ 1 \leq \alpha \leq r \nonumber
\end{eqnarray}
does not have any non-trivial 
real solution $\left\{(c_{j},a_{j},b_{j})\right\}_{j=1}^{k-k_{0}+1}$. We
emphasize that the lower bound in Eq.~\eqref{eqn:introductionone} is sharp, in that 
it cannot be replaced by $W_1^r$ (or $W_r^r$) for any $r < \overline{r}$. A consequence
of this fact, by invoking standard results from asymptotic statistics, 
is that the minimax optimal rate of convergence for estimating $G$ is 
$n^{-1/2\overline{r}}$ under $W_{\overline{r}}$ distance metric. 
The authors find this correspondence quite striking -- one which links precisely 
the minimax optimal estimation rate of 
mixing measures arising from an over-fitted Gaussian mixture to
the solvability of an explicit system of polynomial equations.

Determining the solvability of a system of polynomial equations is a basic question
in (computational) algebraic geometry. For the system described above, there does not
seem to be an obvious answer to the general value of $\overline{r}$. Since the number of variables 
in this system is $3(k-k_{0}+1)$, one expects that 
$\overline{r}$ keeps increasing as $k-k_{0}$ increases. In fact, using a standard method
of Groebner bases~\citep{Bruno-Thesis}, 
we can show that for $k-k_0 = 1$ and $2$, $\overline{r} = 4$ and $6$, respectively.
In addition if $k-k_0\geq 3$, then $\overline{r} \geq 7$. Thus, the 
convergence rate of
the mixing measure for 
over-fitted Gaussian mixture deteriorates very quickly as more
extra components are included in the model.

\textit{Gamma mixtures:} We shall now briefly describe several other model classes studied in this paper.
Gamma densities represent one such class: the Gamma density $f(x|a,b)$ has
two positive parameters, $a$ for shape and $b$ for rate. This family
is not identifiable in the first order. 
The lack of identifiability boils down to 
the fundamental identity~\eqref{key-gamma}. By exploiting this identity, we
can show that there are particular combinations of the true parameter
values which prevent the Gamma class from enjoying strong convergence properties. By
excluding the measure-zero set of pathological cases of true mixing measures, 
the Gamma density class in fact can be shown to be strongly identifiable
in both orders. Thus, this class is \emph{almost} strongly identifiable, using
the terminology of~\cite{Allman-2009}.
The generic/pathological dichotomy in the convergence behavior within the Gamma class
is quite interesting: in the measure-one generic set of true mixing measures, 
the mixing measure can be estimated at the standard rate 
(i.e., $n^{-1/2}$ under $W_1$ for exact-fitted
and $n^{-1/4}$ under $W_2$ for over-fitted mixtures). The pathological cases
are not so forgiving: even for exact-fitted mixtures, one can do no better
than a logarithmic rate of convergence.

\textit{Location-exponential mixtures:}
Lest some wonder whether this unusually slow rate for the
exact-fitted mixture setting can happen only in the measurably
negligible (pathological) cases, we also introduce a location-extension
of the Gamma family, the location-exponential class: $f(x|\theta,\sigma)
:= \frac{1}{\sigma}\exp -\frac{x-\theta}{\sigma} 1(x>\theta)$. 
We show that the minimax lower bound for estimating the mixing measure in an
exact-fitted mixture of location-exponentials is no faster than a logarithmic
rate.

\textit{Skew-Gaussian mixtures:} The most fascinating example among 
those studied is perhaps skew-Gaussian distributions. This density class generalizes the Gaussian 
distributions, by having an extra parameter, shape, which 
controls density skewness. The skew-Gaussian family
exhibits an extremely broad spectrum of behavior, some of which shared with the
Gamma family, some with the Gaussian, but this family is really a league of its own.
It is not identifiable in the first order, for a reason that is 
somewhat similar to that of the Gamma family described above. As a consequence, one can construct
a full measure set of generic cases for the true mixing measures according to
which, the exact-fitted mixture model admits strong identifiablity and
convergence rate (as in the general theory). 

Within the seemingly benign setting of exact-fitted mixtures,
the pathological cases for the skew-Gaussian carry a very rich 
structure, resulting in a variety of behaviors:
for some subset of true mixing measures, the convergence rate is tied to solvability of
a certain system of polynomial equations;
for some other subset, the
convergence is poor -- the rate can be logarithmic at best.

Turning to over-fitted mixtures of skew-Gaussian
distributions, unfortunately our theory remains incomplete. 
The culprit lies in the fundamental identity~\eqref{key-skewnormal},
which shows that the first and second order derivatives of the skew-Gaussian
densities are dependent on a \emph{nonlinear} manner. This is in contrast to
the linear dependence that characterizes Gaussian and Gamma densities.
Thus, the method of proof that works well for the 
previous examples is no longer adequate -- the rates obtained are probably
not optimal. 

\paragraph{Key proof ideas} We now provide a brief description of our
method of proofs for the results  obtained in this paper, a summary of which
given in Table~\ref{tablesummary}. 
There are two different theories: a general
theory for the strongly identifiable classes and specialized theory for
weakly identifiable classes. Within each model classes, the
key technical objective is the same:
to derive sharp inequalities of the form $V(p_G,p_{G_0}) \gtrsim W_r^r(G,G_0)$,
where sharpness is expressed in the choice of $r$. 

For strongly identifiable classes, either in the first or the second order, the
starting point of our proof is an application of Taylor expansion on the
mixture density difference $p_{G_n}-p_{G_0}$, where $G_n$ represents 
a sequence of mixing measures that tend to $G_0$ in Waserstein distance
$W_r$, where $r=1$ or 2, the assumed order of strong identifiablity. The main
part of the proof involves trying to force all the Taylor coefficients in
the Taylor expansion to vanish according to the converging sequence of $G_n$.
If that is proved to be impossible, then one can arrive at the bound of the 
form $V\gtrsim W_r^r$. Thus, our proof technique is similar to that 
of~\cite{Nguyen-13}. To show
that the derived inequalities are sharp, we resort to careful constructions of a
``worst-case''sequence of $G_n$. 

For weakly identifiable classes, the Taylor expansion technique continues to
provide the proof's backbone, but the key issue now is determining
the ``correct''order up to which the Taylor expansion is exercised. Since high-order
derivatives of the density $f$ are no longer independent, the dependence
has to be taken into account before one can fall back to a similar technique
afforded by the general theory described above. If the high-order derivatives
are linearly dependent, as is the case of Gaussian densities,
it is possible to reduce the original Taylor expansion in terms of only a subset of
such derivative quantities that are linearly independent. This reduction
process paves the way for a system of polynomial equations to emerge.
It follows then that the right exponent $r$ in the desired bound described
above can be linked to the order of such a system which admits a non-trivial
solution. 

\paragraph{Practical implications} 
Problematic convergence behaviors exhibited by
widely utilized models such as Gaussian mixtures
may have long been observed in practice,
but to our knowledge, most of the obtained convergence rates are established
for the first time in this paper,
particularly those of weakly identifiable classes.
The results established for the popular Gaussian class
present a formal reminder about the limitation of Gaussian mixtures when 
it comes to assessing the quality of parameter estimation, but
only when the number of mixing components is unknown.
Since a tendency in practice is to ``over-fit'' the mixture 
generously with many more extra 
mixing components, our theory warns against this practice,
because the convergence rate for
subpopulation-specific parameters deteriorates rapidly with the number of 
redundant components. In particular, we expect that the value 
$\overline{r}$ in the rate $n^{-1/2\overline{r}}$ tends to infinity as 
the number of redundant Gaussian components increases to infinity.
To complete the spectrum of rates, we note the
logarithmic rate $(\log n)^{-1/2}$ of convergence of the mixing measure 
in infinite Gaussian location mixtures, via a Bayes estimate~\citep{Nguyen-13} 
or kernel-based deconvolution~\citep{Caillerie-etal-11}.

For Gamma and skew-Gaussian mixtures,
(for applications, see, e.g.~\citep{Ghosal-13,Lee-13, Wiper-01})
our theory paints a wide spectrum of convergence behaviors within each model class. 
We hope that the theoretical results obtained here may 
hint at practically useful ways for determining benign scenarios when the mixture
models enjoy strong identifiability properties and favorable convergence rates,
and for identifying pathological scenarios where the practioners would do
well by avoiding them.

\paragraph{Paper organization} The rest of the paper is organized as follows.
Section \ref{Sec-upperbound} provides some preliminary backgrounds and facts.
Section \ref{Sec-identifiability} presents a general theory of strong identifiability,
by addressing the exact-fitted and over-fitted settings separately 
before
providing a characteration of density classes for which the general theory
is applicable. 
Section~\ref{Sec-weakidentifiability} is devoted to a theory for weakly identifiable
classes, by treating each of the described three density classes separately.
Section \ref{Sec-convergence} contains easy consequences of the theory
developed earlier -- this includes minimax bounds and the convergence
rates of the maximum likelihood estimation, which are optimal in many cases.
The theoretical bounds are illustrated via simulations in Section~\ref{Sec:simulation}.
%
Self-contained proofs of representative theorems are given in Section \ref{Sec-proofs},
 while proofs of remaining results are presented in the Appendix.

\paragraph{Notation} 
Divergence distances studied in this paper include the total variational 
distance $V(p_{G},p_{G'})={\displaystyle \dfrac{1}{2}\int {|p_{G}(x)-p_{G'}(x)|}d\mu(x)}$ and the Hellinger distance
$h^{2}(p_{G},p_{G'})=\dfrac{1}{2} {\displaystyle \int {(\sqrt{p_{G}(x)}-\sqrt{p_{G'}(x)})^{2}}d\mu(x)}$.
As $K,L \in \mathbb{N}$, the first derivative of real function $g:\mathbb{R}^{K \times L} \to \mathbb{R}$ of matrix $\Sigma$ is defined as a $K\times L$ matrix whose 
$(i,j)$ element is $\partial g/\partial \Sigma_{ij}$. The second derivative of $g$,
denoted by $\dfrac{\partial^2 g}{\partial \Sigma^{2}}$ is a 
$K^2 \times L^2$ matrix made of $KL$ blocks of $K\times L$ matrix, whose $(i,j)$-block is
given by $\dfrac{\partial}{\partial \Sigma} \left(\dfrac{\partial g}{\partial \Sigma_{ij}} \right)$.
Additionally, as $N \in \mathbb{N}$, for function 
$g_{2}: \mathbb{R}^{N} \times \mathbb{R}^{K \times L} \to \mathbb{R}$
defined on $(\theta,\Sigma)$, 
the joint derivative between the vector component and matrix component 
$\dfrac{\partial^{2}{g_{2}}}{\partial{\theta}\partial{\Sigma}}=\dfrac{\partial^{2}{g_{2}}}{\partial{\Sigma}\partial{\theta}}$ is a $(KN) \times L$ matrix of 
$KL$ blocks for $N$-columns, whose $(i,j)$-block is given by
$\dfrac{\partial}{\partial\theta} \left(\dfrac{\partial g_2}{\partial \Sigma_{ij}} \right)$. Finally, for any symmetric matrix $\Sigma \in \mathbb{R}^{d \times d}$, 
$\lambda_{1}(\Sigma)$ and $\lambda_{d}(\Sigma)$ respectively denote its smallest and largest eigenvalue.

\section{Preliminaries}
\label{Sec-upperbound}

\comment{
Given $d_{1} \geq 1, d_{2} \geq 0$ and $\Theta \subset \mathbb{R}^{d_{1}}$, $\Omega \subset S_{d_{2}}^{++}$, where $S_{d_{2}}^{++}$ is the set of all positive definite matrices in $\mathbb{R}^{d_{2} \times d_{2}}$. We equip the space $\Theta \times \Omega$ with a distance $\rho$ to form a metric space $(\Theta \times \Omega,\rho)$. A finite discrete probability measure $G$ on a measure space equipped with the Borel sigma algebra takes the form $G=\mathop {\sum }\limits_{i=1}^{k}{p_{i}\delta_{(\theta_{i},\Sigma_{i})}}$ for some $k \in \mathbb{N}$, where $\textbf{p}=(p_{1},p_{2},\ldots,p_{k})$ denotes the proportion vector and $(\vec{\theta},\vec{\Sigma})=((\theta_{1},\Sigma_{1}),\ldots,(\theta_{k},\Sigma_{k}))$ denotes the associated atoms in $\Theta \times \Omega$.}

First of all, we need to define our notion of distances on the 
space of mixing measures $G$. In this paper, we restrict ourself to the 
space of discrete mixing measures with exactly $k_{0}$ 
distinct support points on $\Theta \times \Omega$, 
which is denoted by $\mathcal{E}_{k_{0}}(\Theta \times \Omega)$, and the 
space of discrete mixing measures with at most $k$ distinct support 
points on $\Theta \times \Omega$, which is denoted by 
$\mathcal{O}_{k}(\Theta \times \Omega)$. In addition, let $\mathcal{G}(\Theta \times \Omega)=\mathop {\cup }\limits_{k \in \mathbb{N}}{\mathcal{E}_{k}(\Theta \times \Omega)}$ be the set of all discrete measures with finite support points. 
Consider mixing measure $G=\mathop {\sum }\limits_{i=1}^{k}{p_{i}\delta_{(\theta_{i},\Sigma_{i})}}$, where $\textbf{p}=(p_{1},p_{2},\ldots,p_{k})$ denotes the proportion vector and $(\vec{\theta},\vec{\Sigma})=((\theta_{1},\Sigma_{1}),\ldots,(\theta_{k},\Sigma_{k}))$ denotes the supporting atoms in $\Theta \times \Omega$. Likewise, let
$G' = \sum_{i=1}^{k'}p'_i \delta_{(\theta'_i, \Sigma_{i}')}$. A coupling between $\vec{p}$ and
$\vec{p'}$ is a joint distribution $\vec{q}$ on $[1\ldots,k]\times [1,\ldots, k']$, which
is expressed as a matrix
$\vec{q}=(q_{ij})_{1 \leq i \leq k,1\ \leq j \leq k} \in [0,1]^{k \times k}$
and admits marginal constraints 
$\mathop {\sum }\limits_{i=1}^{k}{q_{ij}}=p_{j}'$ and $\mathop {\sum  }\limits_{j=1}^{k'}{q_{ij}}=p_{i}$ for any $i=1,2,\ldots,k$ and $j=1,2,\ldots,k'$. 
We call $\vec{q}$ a coupling of $\vec{p}$ and $\vec{p'}$,
and use $\mathcal{Q}(\vec{p},\vec{p'})$ to denote the space of all such couplings.

As in~\cite{Nguyen-13}, our tool for analyzing the identifiability and convergence
of parameters in a mixture model is by adopting Wasserstein distances, which
can be defined as the optimal cost of moving mass from one probability
measure to another~\citep{Villani-2009}. For any $r \geq 1$, the $r$-th order Wasserstein distance between 
$G$ and $G'$ is given by
\begin{eqnarray}
W_{r}(G,G') & = &
\biggr ( \inf_{\vec{q} \in \mathcal{Q}(\vec{p},\vec{p'})}{\mathop {\sum }\limits_{i,j}{q_{ij}(\|\theta_{i}-\theta_{j}'\|+\|\Sigma_{i}-\Sigma_{j}'\|)^{r}}}
\biggr )^{1/r}. \nonumber
\end{eqnarray}
In both equations in the above display, $\|\cdot\|$ denotes either the $l_{2}$ norm for 
elements in $\mathbb{R}^d$ or the entrywise $l_{2}$ norm for matrices.
A central theme of the paper is the relationship between the Wasserstein distances of mixing
measures $G,G'$ and distances of corresponding mixture densities $p_G,p_{G'}$.
Recall that mixture density $p_G$ is obtained by combining a mixing measure 
$G \in \mathcal{G}(\Theta \times \Omega)$ with a family of density functions 
$\left\{f(x|\theta,\Sigma),\theta \in \Theta,\Sigma \in\Omega\right\}$:
\begin{eqnarray}
p_{G}(x)=\int{f(x|\theta,\Sigma)}dG(\theta,\Sigma)=\mathop {\sum }\limits_{i=1}^{k}{p_{i}f(x|\theta_{i},\Sigma_{i})}. \nonumber
\end{eqnarray}

Clearly if $G = G'$ then $p_{G} = p_{G'}$. Intuively, if $W_1(G,G')$ or $W_2(G,G')$
is small, so is a distance between $p_G$ and $p_{G'}$. This can be quantified
by establishing an upper bound for the distance of $p_{G}$ and $p_{G'}$ in terms of 
$W_1(G,G')$ or $W_2(G,G')$. A general notion of distance between probability densities defined
on a common space is $f$-divergence (or Ali-Silvey distance)~\cite{Ali-1966}:
an $f$-divergence between two probability density functions $f$ and $g$ is 
defined as ${\displaystyle \rho_{\phi}\left(f,g\right)=\int {\phi\left(\dfrac{g}{f}\right)f}d\mu}$, where $\phi:\mathbb{R} \to \mathbb{R}$ is a convex function. Similarly, 
the $f$-divergence between $p_{G}$ and $p_{G'}$ is 
$\rho_{\phi}(p_{G},p_{G'})=\int {\phi\left(\dfrac{p_{G'}}{p_{G}}\right)p_{G}}d\mu$.
As $\phi(x)=\dfrac{1}{2}(\sqrt{x}-1)^{2}$, we obtain the squared Hellinger distance ($\rho_{h}^{2} \equiv h^{2})$. As $\phi(x)=\dfrac{1}{2}|x-1|$, we obtain the variational distance ($\rho_{V} \equiv V$).

A simple way of establishing an upper bound for
an $f$-divergence between $p_G$ and $p_{G'}$ is via
the ``composite transportation distance'' between mixing measures $G,G'$:
\begin{eqnarray}
d_{\rho_{\phi}}(G,G')=\mathop {\inf }\limits_{\vec{q} \in \mathcal{Q}(\vec{p},\vec{p'})}{\mathop {\sum }\limits_{i,j}{q_{ij}\rho_{\phi}(f_{i},f_{j}')}} \nonumber
\end{eqnarray}
where $f_{i}=f(x|\theta_{i},\Sigma_{i})$ and 
$f_{j}'=f(x|\theta_{j}',\Sigma_{j}')$ for any $i,j$. The following inequality
regarding the relationship between $\rho_{\phi}(p_{G},p_{G'})$ and $d_{\rho_{\phi}}(G,G')$
is a simple consequence of Jensen's inequality~\citep{Nguyen-13}:
\[\rho_{\phi}(p_{G},p_{G'}) \leq d_{\rho_{\phi}}(G,G').\]
It is straightforward to derive upper bounds for
$d_{\rho_{\phi}}(G,G')$ in terms of Wasserstein distances 
$W_r$, by taking into
account specific structures of the density family $f$, and then combine
with the inequality in the previous display
to arrive at upper bounds for $\rho_{\phi}(p_{G},p_{G'})$ in terms of
Wasserstein distances. Here are a few examples.

\begin{example}\label{example-Gaussian}
(Multivariate generalized Gaussian distribution \citep{Weisel-13})\\ The density family $f$ takes the form
$f(x|\theta,m,\Sigma)=\dfrac{m\Gamma(d/2)}{\pi^{d/2}\Gamma(d/(2m))|\Sigma|^{1/2}}\exp(-((x-\theta)^{T}\Sigma^{-1}(x-\theta))^{m})$, where $\theta \in \mathbb{R}^{d}, m >0$, and $\Sigma  \in S_{d}^{++}$. If $\Theta_{1}$ is bounded subset of $\mathbb{R}^{d}$, $\Theta_{2}=\left\{m \in \mathbb{R}^{+}: 1 \leq \underline{m} \right. \\ \left. \leq m \leq \overline{m}\right\}$, and $\Omega=\left\{\Sigma \in S_{d}^{++}: \underline{\lambda} \leq \sqrt{\lambda_{1}(\Sigma)} \leq \sqrt{\lambda_{d}(\Sigma)} \leq \overline{\lambda}\right\}$, where $\underline{\lambda},\overline{\lambda}>0$, then for any $G_{1},G_{2} \in \mathcal{G}(\Theta_{1} \times \Theta_{2}  \times \Omega)$, we obtain $h^{2}(p_{G_{1}},p_{G_{2}}) \lesssim W_{2}^{2}(G_{1},G_{2})$ and $V(p_{G_{1}},p_{G_{2}}) \lesssim W_{1}(G_{1},G_{2})$. 
\end{example}
\begin{example}\label{example-Student}(Multivariate Student's t-distribution) \\
The density family $f$ takes the form
$f(x|\theta,\Sigma)=C_{\nu}(\nu+(x-\theta)^{T}\Sigma^{-1}(x-\theta))^{-(\nu+d)/2}$, where $\nu$ is a fixed positive degree of freedom and $C_{\nu}=\dfrac{\Gamma((\nu+d)/2)\nu^{\nu/2}}{\Gamma(\nu/2)\pi^{d/2}}$. If $\Theta$ is bounded subset of $\mathbb{R}^{d}$ and $\Omega=\left\{\Sigma \in S_{d}^{++}: \underline{\lambda} \leq \sqrt{\lambda_{1}(\Sigma)} \leq \sqrt{\lambda_{d}(\Sigma)} \leq \overline{\lambda}\right\}$, then for any $G_{1},G_{2} \in \mathcal{G}(\Theta  \times \Omega)$, we obtain $h^{2}(p_{G_{1}},p_{G_{2}}) \lesssim W_{2}^{2}(G_{1},G_{2})$ and $V(p_{G_{1}},p_{G_{2}}) \lesssim W_{1}(G_{1},G_{2})$.
\end{example}

\begin{example}\label{example-modstudent}(Exponentially modified multivariate Student's t-distribution)\\
Let $f(x|\theta,\lambda,\Sigma)$ to be density function of $X=Y+Z$, where $Y$ follows multivariate t-distribution with location $\theta$, covariance matrix $\Sigma$, fixed positive degree of freedom $\nu$, and $Z$ is distributed by the product of $d$ independent exponential distributions with combined shape $\lambda=(\lambda_{1},\ldots,\lambda_{d})$.  If $\Theta$ is bounded subset of $\mathbb{R}^{d} \times \mathbb{R}^{d}_{+}$, where $\mathbb{R}^{d}_{+}=\left\{x \in \mathbb{R}^{d}:x_{i}>0 \ \forall i \right\}$, and $\Omega=\left\{\Sigma \in S_{d}^{++}: \underline{\lambda} \leq \sqrt{\lambda_{1}(\Sigma)} \leq \right. \\ \left.\sqrt{\lambda_{d}(\Sigma)} \leq \overline{\lambda}\right\}$, then for any $G_{1},G_{2} \in \mathcal{G}(\Theta \times \Omega)$, $h^{2}(p_{G_{1}},p_{G_{2}}) \lesssim W_{2}^{2}(G_{1},G_{2})$ and $V(p_{G_{1}},p_{G_{2}}) \lesssim W_{1}(G_{1},G_{2})$.
\end{example}

\begin{example}\label{example-gamma}(Modified Gaussian-Gamma distribution) \\
Let $f(x|\theta,\lambda,\beta,\Sigma)$ to be density function of $X=Y+Z$, where $Y$ is 
distributed by multivariate Gaussian distribution with mean $\theta$, covariance matrix 
$\Sigma$, and $Z$ is distributed by the product of independent Gamma distributions 
with combined shape vector $\alpha=(\alpha_{1},\ldots,\alpha_{d})$ and combined 
rate vector $\beta=(\beta_{1},...,\beta_{d})$. 
If $\Theta$ is bounded subset of 
$\mathbb{R}^{d} \times \mathbb{R}^{d}_{+} \times \mathbb{R}^{d}_{+}$ and 
$\Omega=\left\{\Sigma \in S_{d}^{++}: \underline{\lambda} \leq \sqrt{\lambda_{1}(\Sigma)} \leq \sqrt{\lambda_{d}(\Sigma)} \right. \\ \left. \leq \overline{\lambda}\right\}$,  
then for any $G_{1},G_{2} \in \mathcal{G}(\Theta \times \Omega)$, 
$h^{2}(p_{G_{1}},p_{G_{2}}) \lesssim V(p_{G_{1}},p_{G_{2}}) \lesssim W_{1}(G_{1},G_{2})$.
\end{example}

\section{General theory of strong identifiability}
\label{Sec-identifiability}

The objective of this section is to develop a general theory according to which
a small distance between mixture densities $p_G$ and $p_{G'}$ entails
a small Wasserstein distance between mixing measures $G$ and $G'$. 
The classical identifiability criteria requires that 
$p_{G} = p_{G'}$ entail $G = G'$, which essentially equivalent to
a linear independence requirement for the class of density family 
$\{f(x|\theta,\Sigma)| \theta \in \Theta, \Sigma \in \Omega\}$.
To obtain quantitative bounds, we need stronger notions of identifiability,
ones which
involve higher order derivatives of density function $f$, taken with respect to 
the multivariate and matrix-variate parameters present in the mixture model. 
The advantage of this theory, which extends from the work of~\cite{Nguyen-13}
and~\cite{Chen-95}, is that it is holds generally for a broad
range of mixture models, which allow for the same bounds on the Wasserstein
distances of mixing measures to hold. This in turn leads to ``standard'' rates of convergence
for the mixing measure. On the other hand, many popular mixture
models such as the location-covariance Gaussian mixture, mixture of Gamma, and mixture of skew-Gaussian distributions do not submit to the general
theory. Instead they require separate and fundamentally distinct
treatments; moreover, such models also exhibit non-standard rates of 
convergence for the mixing measure. Readers interested in results for such
models may skip directly to Section~\ref{Sec-weakidentifiability}. 

\comment{
Given $d_{1} \geq 1,d_{2} \geq 0$, we consider family of matrix-variate density functions $\left\{f(x|\theta,\Sigma),\theta \in \Theta,\Sigma \in \Omega\right\}$, where $\Theta \subset \mathbb{R}^{d_{1}}$ and $\Omega \subset S_{d_{2}}^{++}$, which is the set of all positive definite matrices in $\mathbb{R}^{d_{2} \times d_{2}}$. We will adapt the notion of strong identifiability in \cite{Chen-95, Nguyen-13}  to this family of density functions.
}

\subsection{Definitions and general bounds} \label{Sec:Definitionandgeneralbound}
\begin{definition}
\label{definition-firstorder}
The family $\left\{f(x|\theta,\Sigma),\theta \in \Theta, \Sigma \in \Omega \right\}$ is \textbf{identifiable in the first-order} if  $f(x|\theta,\Sigma)$ is differentiable in $(\theta,\Sigma)$ and the following assumption holds
\begin{itemize}
\item[A1.] For any finite $k$ different pairs 
$(\theta_{1},\Sigma_{1}),...,(\theta_{k},\Sigma_{k}) \in \Theta \times \Omega$, 
if we have $\alpha_{i} \in \mathbb{R},\beta_{i} \in \mathbb{R}^{d_{1}}$ and \textbf{symmetric matrices} $\gamma_i \in \mathbb{R}^{d_{2} \times d_{2}}$ (for all $i=1,\ldots,k$) 
such that
\begin{eqnarray}
\mathop {\sum }\limits_{i=1}^{k}{\alpha_{i}f(x|\theta_{i},\Sigma_{i})+\beta_{i}^{T}\dfrac{\partial{f}}{\partial{\theta}}(x|\theta_{i},\Sigma_{i})+\trace\left(\dfrac{\partial{f}}{\partial{\Sigma}}(x|\theta_{i},\Sigma_{i})^{T}\gamma_{i}\right)}=0 \ \ \text{for almost all} \ x \nonumber
\end{eqnarray}
then this will entail that 
$\alpha_{i}=0,\beta_{i}=\vec{0} \in \mathbb{R}^{d_{1}},\gamma_{i}=\vec{0} \in \mathbb{R}^{d_{2} \times d_{2}}$ for $i=1,\ldots, k$.
\end{itemize}
\end{definition}

\paragraph{Remark.} 
The condition that $\gamma_{i}$ is symmetric in Definition \ref{definition-firstorder} is crucial,
without which the identifiability condition would fail for many classes of density. For instance,
assume that $\dfrac{\partial{f}}{\partial{\Sigma}}(x|\theta_{i},\Sigma_{i})$ are symmetric matrices for all $i$ (this clearly holds for any elliptical distributions, such as multivariate Gaussian, Student's t-distribution, and logistics distribution). If we choose $\gamma_{i}$ to be anti-symmetric matrices,
then by choosing $\alpha_{i}=0$, $\beta_{i}=\vec{0}$, $(\gamma_{i})_{uu}=0$ for all $1 \leq u \leq d_{2}$ (i.e. all diagonal elements are 0) , the equation in condition A.1 holds while $\gamma_{i}$ 
can be different from $\vec{0}$ for all $i$.


Additionally, we say the family of densities $f$ is uniformly Lipschitz up to
the first order if the following holds: there are positive
constants $\delta_1,\delta_2$ such that for any $R_{1},R_{2},R_{3}>0$, 
$\gamma_{1} \in \mathbb{R}^{d_{1}}$, $\gamma_{2} \in \mathbb{R}^{d_{2} \times d_{2}}$,
$R_{1} \leq \sqrt{\lambda_{1}(\Sigma)} \leq \sqrt{\lambda_{d_{2}}(\Sigma)} \leq R_{2}$, $||\theta|| \leq R_{3}$, $\theta_{1},\theta_{2} \in \Theta$, $\Sigma_{1},\Sigma_{2} \in \Omega$,  there are positive constants $C(R_1,R_{2})$
and $C(R_3)$ such that for all $x \in \mathcal{X}$

\begin{equation}
\label{Eqn:Lipschitz1}
\left|\gamma_{1}^{T} \biggr ( \dfrac{\partial{f}}{\partial{\theta}}(x|\theta_{1},\Sigma)-\dfrac{\partial{f}}{\partial{\theta}}(x|\theta_{2},\Sigma) \biggr ) \right| \leq C(R_{1},R_{2})||\theta_{1}-\theta_{2}||^{\delta_{1}}||\gamma_{1}|| 
\end{equation}
and 
\begin{equation}
\label{Eqn:Lipschitz2}
\left|\trace\left(\left(\dfrac{\partial{f}}{\partial{\Sigma}}(x|\theta,\Sigma_{1})-\dfrac{\partial{f}}{\partial{\Sigma}}(x|\theta,\Sigma_{2})\right)^{T}\gamma_{2}\right)\right| \leq C(R_{3})||\Sigma_{1}-\Sigma_{2}||^{\delta_{2}}||\gamma_{2}||. 
\end{equation}

First-order identifiability is sufficient for deriving a lower bound of $V(p_{G},p_{G_0})$ in terms
of $W_1(G,G_0)$, under the \emph{exact-fitted} setting: This is the setting where $G_0$ has
exactly $k_0$ support points, $k_0$ known:
\begin{theorem}{\bf(Exact-fitted setting)}
\label{theorem-firstorder}
Suppose that the density family $f$ is identifiable in the first order and 
admits uniform Lipschitz property up to the first order. Then there are positive constants
$\epsilon_{0}$ and $C_0$, both depending on $G_0$, such that as long as 
$G \in \mathcal{E}_{k_0}(\Theta \times \Omega)$ and $W_{1}(G,G_{0}) \leq \epsilon_{0}$,
we have
\[V(p_{G},p_{G_{0}}) \geq C_0 W_{1}(G,G_{0}).\]

\end{theorem}

Note that we do not impose any boundedness on $\Theta$ or $\Omega$. 
Nonetheless, the bound is of local nature, in the sense
that it holds only for those $G$ sufficiently close
to $G_0$ by a Wassertein distance at most $\epsilon_0$, which again
varies with $G_0$. It is possible to extend this type of bound to hold 
globally over a compact subset of the space of mixing measures, under
a mild regularity condition,
as the following corollary asserts:

\begin{corollary}\label{corollary-lowerbound} 
Suppose that the density family $f$ is identifiable in the first order,
and admits uniform Lipschitz property up the first order. 
Further, there is a positive constant $\alpha > 0$ such that
for any $G_{1},G_{2} \in \mathcal{E}_{k_{0}}(\Theta \times \Omega)$, we have 
$V(p_{G_{1}},p_{G_{2}}) \lesssim W_{1}^{\alpha}(G_{1},
G_{2})$.
Then, for a fixed compact subset $\mathcal{G}$ of 
$\mathcal{E}_{k_{0}}(\Theta \times \Omega)$, 
there is a positive constant $C_0 = C_0(G_0)$ such that
\begin{eqnarray}
V(p_{G},p_{G_{0}}) \geq C_0
W_{1}(G,G_{0}) \quad \text{for \ all} \ G \in \mathcal{G}. \nonumber
\end{eqnarray}
\end{corollary}
We shall verify in the sequel that
the classes of densities $f$ described in Examples
\ref{example-Gaussian}, \ref{example-Student}, \ref{example-modstudent}, 
and \ref{example-gamma} are all identifiable in the first order. Thus, a remarkable
consequence of the result above is that for such classes of densities, the variational distance $V$
on mixture densities and the Wasserstein distance $W_1$ on the corresponding mixing
measures are in fact equivalent in the exact-fitted setting. That is, when 
$G$ share the same number of support points as that of $G_0$, we have
\[V(p_{G},p_{G_{0}}) \asymp W_{1}(G,G_{0})\]

Moving to the \emph{over-fitted} setting, where $G_0$ has exactly
$k_0$ support points lying in the interior of $\Theta \times \Omega$, but $k_0$ is unknown and only an upper bound for $k_0$ is given,
a stronger identifiability condition is required. This condition involves the second-order
derivatives of the density class $f$ that extends from the notion of
strong identifiability considered by~\cite{Chen-95,Nguyen-13}:

\begin{definition}\label{definition-secondorder}
The family $\left\{f(x|\theta,\Sigma),\theta \in \Theta, \Sigma \in \Omega \right\}$ is \textbf{identifiable in the second-order} if $f(x|\theta,\Sigma)$ is twice differentiable in $(\theta,\Sigma)$ and the following assumption holds
\begin{itemize}
\item[A2.] For any finite $k$ different pairs $(\theta_{1},\Sigma_{1}),...,(\theta_{k},\Sigma_{k}) \in \Theta \times \Omega$, if we have $\alpha_{i} \in \mathbb{R},\beta_{i},\nu_{i} \in \mathbb{R}^{d_{1}}$, $ \gamma_{i}, \eta_{i}$ \textbf{symmetric matrices} in $\mathbb{R}^{d_{2} \times d_{2}}$ as $i=1,\ldots, k$ such that
\begin{eqnarray}
\sum_{i=1}^{k} \biggr \{ \alpha_{i}f(x|\theta_{i},\Sigma_{i}) +
\beta_{i}^{T} \dfrac{\partial{f}}{\partial{\theta}}(x|\theta_{i},\Sigma_{i}) +
\nu_{i}^{T} \dfrac{\partial^{2}{f}}
{\partial{\theta}^{2}} (x|\theta_{i},\Sigma_{i}) \nu_{i} & + &  \notag \\
\trace \left(\dfrac{\partial{f}}{\partial{\Sigma}}(x|\theta_{i},\Sigma_{i})^{T}
\gamma_{i}\right ) 
+ 2\nu_{i}^{T}\left[\dfrac{\partial}{\partial{\theta}}\left(\trace \left(\dfrac{\partial{f}}{\partial{\Sigma}}(x|\theta_{i},\Sigma_{i})^{T}\eta_{i}\right)\right)\right]  & + & \notag \\
\trace \left(\dfrac{\partial{}}{\partial{\Sigma}}\left(\trace\left(\dfrac{\partial{f}}{\partial{\Sigma}}(x|\theta_{i},\Sigma_{i})^{T}\eta_{i}\right)\right)^{T}\eta_{i}\right)
\biggr \}
& = & 0 \quad \text{for almost all} \ x, \nonumber
\end{eqnarray}
then this will entail that 
$\alpha_{i}=0,\beta_{i}=\nu_{i} = \vec{0} \in \mathbb{R}^{d_{1}},\gamma_{i}=\eta_{i} = \vec{0} \in \mathbb{R}^{d_{2} \times d_{2}}$ for $i=1,\ldots,k$.\\\hk
\end{itemize}
\end{definition}

In addition, we say the family of densities $f$ is uniformly Lipschitz up to the second order
if the following holds: there are positive constants $\delta_3,\delta_4$ such that for any
$R_{4},R_{5},R_{6}>0$, 
$\gamma_{1} \in \mathbb{R}^{d_{1}}$, $\gamma_{2} \in \mathbb{R}^{d_{2} \times d_{2}}$,
$R_{4} \leq \sqrt{\lambda_{1}(\Sigma)} \leq \sqrt{\lambda_{d_{2}}(\Sigma)} \leq R_{5}$, $||\theta|| \leq R_{6}$, $\theta_{1},\theta_{2} \in \Theta$, $\Sigma_{1},\Sigma_{2} \in \Omega$,  there are positive constants $C_1$ depending on $(R_4,R_{5})$
and $C_2$ depending on $R_6$ such that for all $x \in \mathcal{X}$
\begin{equation}
\label{Eqn:Lipschitz3}
|\gamma_{1}^{T}(\dfrac{\partial^{2}{f}}{\partial{\theta}\partial{\theta^{T}}}(x|\theta_{1},\Sigma)-\dfrac{\partial^{2}{f}}{\partial{\theta}\partial{\theta^{T}}}(x|\theta_{2},\Sigma))\gamma_{1}| \leq C_1 \|\theta_{1}-\theta_{2}\|_{1}^{\delta_{3}}\|\gamma_{1}\|_{2}^{2} \notag 
\end{equation}
and 
\begin{multline}
\label{Eqn:Lipschitz4}
\left|\text{tr}\left(\left[\dfrac{\partial{}}{\partial{\Sigma}}\left(\text{tr}\left(\dfrac{\partial{f}}{\partial{\Sigma}}(x|\theta,\Sigma_{1})^{T}\gamma_{2}\right)\right)-\dfrac{\partial{}}{\partial{\Sigma}}\left(\text{tr}\left(\dfrac{\partial{f}}{\partial{\Sigma}}(x|\theta,\Sigma_{2})^{T}\gamma_{2}\right)\right)\right]^{T}\gamma_{2}\right)\right| \leq \notag \\
C_2\|\Sigma_{1}-\Sigma_{2}\|_{2}^{\delta_{4}}\|\gamma_{2}\|_{2}^{2}. 
\end{multline}

Let $k \geq 2$ and $k_{0} \geq 1$ be fixed positive integers where $k \geq k_{0}+1$. 
$G_0 \in \mathcal{E}_{k_0}$ while $G$ varies in $\mathcal{O}_{k}$. Then,
we can establish the following result

\begin{theorem} {\bf (Over-fitted setting)} \label{theorem-secondorder} 
\begin{itemize}
\item [(a)] Suppose that the density family $f$ is identifiable in the second order and admits 
uniform Lipschitz property up to the second order. 
Moreover, $\Theta$ is bounded subset of $\mathbb{R}^{d_{1}}$ and 
$\Omega$ is subset of $S_{d_{2}}^{++}$ such that the largest eigenvalues of 
elements of $\Omega$ are bounded above. 
In addition, suppose that
$\mathop {\lim }\limits_{\lambda_{1}(\Sigma) \to 0}{f(x|\theta,\Sigma)}=0$ for all $x \in \mathcal{X}$ and $\theta \in \Omega$.
Then there are positive constants $\epsilon_{0}$ and $C_{0}$ depending on $G_0$
such that as long as  $W_{2}(G,G_{0}) \leq \epsilon_{0}$,
\begin{eqnarray}
V(p_{G},p_{G_{0}}) \geq C_{0}W_{2}^{2}(G,G_{0}). \nonumber
\end{eqnarray}
\item [(b)] (Optimality of bound for variation distance) Assume that $f$ is second-order differentiable with respect to $\theta, \Sigma$ and ${\displaystyle \mathop {\sup }\limits_{\theta \in \Theta, \Sigma \in \Omega}{\int \limits_{x \in \mathcal{X}}{\left|\dfrac{\partial^{2}{f}}{\partial{\theta^{\alpha_{1}}}\partial{\Sigma^{\alpha_{2}}}}(x|\theta,\Sigma)\right|}dx}} <\infty$ for all $\alpha_{1}=(\alpha_{i}^{1})_{i=1}^{d_{1}} \in \mathbb{N}^{d_{1}}$, $\alpha_{2}=(\alpha_{uv}^{2})_{1 \leq u,v \leq d_{2}} \in \mathbb{N}^{d_{2} \times d_{2}}$ such that $\mathop {\sum }\limits_{i=1}^{d_{1}}{\alpha_{i}^{1}}+\mathop {\sum }\limits_{1 \leq u,v \leq d_{2}}{\alpha_{uv}^{2}} =2$. Then, for any $1 \leq r <2$:
\begin{eqnarray}
\lim_{\epsilon \rightarrow 0} \inf_{G \in \mathcal{O}_k(\Theta\times\Omega)}
\biggr \{ V(p_{G},p_{G_0})/W_1^r(G,G_0): W_1(G,G_0) \leq \epsilon \biggr \} = 0. \nonumber
\end{eqnarray}
\item[(c)] (Optimality of bound for Hellinger distance) Assume that $f$ is second-order differentiable with respect to $\theta$, $\Sigma$ and we can find $c_{0}$ sufficiently small such that 
\begin{eqnarray}
{\displaystyle \mathop {\sup }\limits_{||\theta-\theta^{'}||+||\Sigma-\Sigma^{'}|| \leq c_{0}}{\int \limits_{x \in \mathcal{X}}{\left(\dfrac{\partial^{2}{f}}{\partial{\theta^{\alpha_{1}}}\partial{\Sigma^{\alpha_{2}}}}(x|\theta,\Sigma)\right)^{2}/f(x|\theta^{'},\Sigma^{'})}dx}} <\infty, \nonumber
\end{eqnarray} 
where $\alpha_{1}, \alpha_{2}$ are defined as that of part (b). Then, for any $1 \leq r <2$:
\begin{eqnarray}
\label{lb-strong}
\lim_{\epsilon \rightarrow 0} \inf_{G \in \mathcal{O}_k(\Theta\times\Omega)}
\biggr \{ h(p_{G},p_{G_0})/W_1^r(G,G_0): W_1(G,G_0) \leq \epsilon \biggr \} = 0. 
\end{eqnarray}
\end{itemize}
\end{theorem}
Here and elsewhere, the ratio $V/W_r$ is set to be $\infty$ if $W_r(G,G_0) = 0$.
We make a few remarks.
\begin{itemize}
\item [(i)] A counterpart of part (a) for finite mixtures
with multivariate parameters was given in~\cite{Nguyen-13} (Proposition 1). 
The proof in that paper has a problem: it relies on Nguyen's Theorem 1, 
which holds only for the exact-fitted setting, but not for the over-fitted setting.
This was pointed out to the second 
author by Elisabeth Gassiat who attributed it to Jonas Kahn.
Fortunately, this error
can be simply corrected by replacing Nguyen's Theorem 1 with a weaker version, which
holds for the over-fitted setting and suffices for our purpose,
for which his method of proof continues
to apply. For part (a), 
it suffices to prove only the following weaker version:
\[\lim_{\epsilon \rightarrow 0} \inf_{G \in \mathcal{O}_k(\Theta\times\Omega)}
\biggr \{ V(p_G,p_{G_0})/W_2^2(G,G_0): W_2(G,G_0) \leq \epsilon \biggr \} > 0.\]
%
\item [(ii)] The 
mild condition $\mathop {\lim }\limits_{\lambda_{1}(\Sigma) \to 0}{f(x|\theta,\Sigma)}=0$ 
is important for the matrix-variate parameter $\Sigma$. In particular,
it is useful for addressing the scenario when the smallest eigenvalue of matrix parameter
$\Sigma$ is not bounded away from $0$. This condition, however,
can be removed if we impose that $\Sigma$ is a positive definite matrix whose
eigenvalues are bounded away from 0.
%

\item [(iii)] Part (b) demonstrates the sharpness of the bound in part (a). In particular,
we cannot improve the lower bound in part (a) to any quantity $W_1^r(G,G_0)$ for any $r < 2$.
For any estimation method that yields $n^{-1/2}$ convergence rate under
the Hellinger distance for $p_G$, part (a) induces $n^{-1/4}$ convergence rate under $W_2$
for $G$. Part (c) implies that $n^{-1/4}$ is minimax optimal.

\item [(iv)] The boundedness of $\Theta$, as well as the boundedness from above of the eigenvalues
of elements of $\Omega$ are both necessary conditions. Indeed, it is possible to show that
if one of these two conditions is not met, it is not possible to obtain the lower
bound of $V(p_{G},p_{G_{0}})$ as established, because distance $h \geq V$ can vanish
much faster than $W_r(G,G_0)$, as can be seen by:
\end{itemize}

\begin{proposition} \label{proposition:unboundenesssupport} 
Let $\Theta$ be a subset of $\mathbb{R}^{d_{1}}$ and $\Omega=S_{d_{2}}^{++}$. Then for any $r \geq 1$ and $\beta >0$ we have 
\begin{eqnarray}
\mathop {\lim }\limits_{\epsilon \to 0}{\mathop {\inf }\limits_{G \in \mathcal{O}_{k}(\Theta \times \Omega)}{\left\{\exp\left(\dfrac{1}{W_{r}^{\beta}(G,G_{0})}\right)h(p_{G},p_{G_{0}}):W_{r}(G,G_{0}) \leq \epsilon \right\}}} =0. \nonumber
\end{eqnarray}
\end{proposition}

As in the exact-fitted setting, in order to establish the bound 
$V \gtrsim W_2^2$ globally, we simply add a compactness condition
on the subset within which $G$ varies:

\begin{corollary} \label{corollary-overfitted} Assume that $\Theta$ and $\Omega$ are two compact subsets of $\mathbb{R}^{d_{1}}$ and $S_{d_{2}}^{++}$ respectively. Suppose that the density family $f$ is identifiable in the second order and admits uniform Lipschitz property up to the second order. Further, there is a positive constant $\alpha \leq 2$ such that for any $G_{1}, G_{2} \in \mathcal{O}_{k}(\Theta \times \Omega)$, we have $V(p_{G_{1}},p_{G_{2}}) \lesssim W_{2}^{\alpha}(G_{1},G_{2})$. Then for a fixed compact subset $\mathcal{O}$ of $\mathcal{O}_{k}(\Theta \times \Omega)$ there is a positive constant $C_{0}=C_{0}(G_{0})$ such that
\begin{eqnarray}
V(p_{G},p_{G_{0}}) \geq C_{0}W_{2}^{2}(G,G_{0}) \ \ \text{for all} \ G \in \mathcal{O}. \nonumber
\end{eqnarray}
\end{corollary}

\comment{A simple setting in which
the compactness condition specified in both Proposition \ref{proposition-lowerbound} and Proposition \ref{proposition-overfitted} hold is when 
the mixing measures have compact support. However, under this setting
the gain in the lower bound obtained by Proposition~\ref{proposition:exactfittedcase} is lost
(due to the fact that under compact support, we have
$W_2^2(G,G_0) = O(W_1(G,G_0))$).
}

\subsection{ Characterization of strong identifiability} \label{Sec:Characterization}

In this subsection we identify a broad range of density classes for which the strong identifiability
conditions developed previously hold either in the first or the second order. 
Then we also present a general result 
which shows how strong identifiablity conditions continue to be preserved under 
certain transformations with respect to the parameter space. 

First, we consider univariate density functions with parameters of multiple types:
\begin{theorem}\label{identifiability-univariatecharacterization} (Densities with multiple varying parameters)
\begin{itemize}
\item[(a)] Generalized univariate logistic density function: 
Let 
$f(x|\theta,\sigma) :=
\dfrac{1}{\sigma}f((x-\theta)/\sigma)$, 
where $f(x)=\dfrac{\Gamma(p+q)}{\Gamma(p)\Gamma(q)}\dfrac{\exp(px)}{(1+\exp(x))^{p+q}}$, 
and $p,q$ are fixed positive integers. Then the family 
$\left\{f(x|\theta,\sigma),\theta \in \mathbb{R}, \right. \\ \left. \sigma \in \mathbb{R}_{+}\right\}$ is 
identifiable in the second order.
\item[(b)] Generalized Gumbel density function: 
Let $f(x|\theta,\sigma,\lambda) :=\dfrac{1}{\sigma}f((x-\theta)/\sigma,\lambda)$, 
where $f(x,\lambda) =\dfrac{\lambda^{\lambda}}{\Gamma(\lambda)}\exp(-\lambda(x+\exp(-x)))$ 
as $\lambda>0$. Then the family 
$\left\{f(x|\theta,\sigma,\lambda),\theta \in \mathbb{R},\sigma \in \mathbb{R}_{+},\lambda \in \mathbb{R}_{+}\right\}$ is identifiable in the second order.
\item[(c)] Univariate Weibull distribution: Let $f_{X}(x|\nu,\lambda)=\dfrac{\nu}{\lambda}\left(\dfrac{x}{\lambda}\right)^{\nu-1}\exp\left(-\left(\dfrac{x}{\lambda}\right)^{\nu}\right)$,
for $x \geq 0$, where $\nu,\lambda>0$ are shape and scale parameters, respectively. 
Then the family $\left\{f_{X}(x|\nu,\lambda),\nu \in \mathbb{R}_{+},\lambda \in \mathbb{R}_{+}\right\}$ is identifiable in the second order.
\item[(d)] Von Mises distributions ~\citep{Mardia-1975, Fraser-1981,Kent-1983}:
Denote $f(x|\mu,\kappa)=\dfrac{1}{2\pi I_{0}(\kappa)}\exp(\kappa\cos(x-\mu)).1_{\left\{x \in [0,2\pi)\right\}}$, where $\mu \in [0,2\pi), \kappa >0$, and $I_{0}(\kappa)$ is the modified Bessel function of order 0. Then the family $\left\{f(x|\mu,\kappa),\mu \in [0,2\pi),\kappa \in \mathbb{R}_{+}\right\}$ is identifiable in the second order.
\end{itemize}
\end{theorem}

Next, we turn to density function classes with matrix-variate parameter spaces,
as introduced in Section~\ref{Sec-upperbound}:

\begin{theorem}\label{identifiability-multivariatecharacterization} (Densities with matrix-variate parameters)
\begin{itemize}
\item[(a)] The family $\left\{f(x|\theta,\Sigma,m),\theta \in \mathbb{R}^{d}, \Sigma \in S_{d}^{++}, m \geq 1\right\}$ of multivariate generalized Gaussian distribution is identifiable in the first order.
\item[(b)] The family $\left\{f(x|\theta,\Sigma),\theta \in \mathbb{R}^{d},\Sigma \in S_{d}^{++}\right\}$ of multivariate t-distribution with fixed odd degree of freedom is  identifiable in the second order.
\item[(c)] The family $\left\{f(x|\theta,\Sigma,\lambda), \theta \in \mathbb{R}^{d}, \Sigma \in S_{d}^{++},\lambda \in \mathbb{R}^{d}_{+}\right\}$ of exponentially modified multivariate t-distribution with fixed odd degree of freedom is identifiable in the second order.
 \item[(d)] The family $\left\{f(x|\theta,\Sigma,a,b),\theta \in \mathbb{R}^{d},\Sigma \in S_{d}^{++}, a \in \mathbb{R}^{d}_{+},b \in \mathbb{R}^{d}_{+}, d \geq 2 \right\}$ of modified multivariate Gaussian-Gamma distribution is  identifiable in the first order.
\end{itemize}
\end{theorem}

We note that these theorems are quite similar to Chen's analysis on classes of 
density with single parameter spaces (cf.~\cite{Chen-95}). 
The proofs of these results, however, are technically nontrivial even if conceptually somewhat
straightforward. For the transparency of our idea, we only demonstrate the results in Theorem \ref{identifiability-univariatecharacterization} and Theorem \ref{identifiability-multivariatecharacterization} up to the first-order identifiability. The proof technique for the second-order identifiability is similar. They are given in the Appendices.
As can be seen in these proofs, the strong identifiability of these density
classes are established by exploiting how the corresponding
characteristics functions (i.e., Fourier transform of the density) vanish at 
infinity. Thus it can be concluded that the common feature in establishing strong
identifiability hinges on the smoothness of the density $f$ in question. 
(It is interesting to contrast this with the story in the next section,
where  we shall meet weakly identifiable
density classes whose algebraic structures play a more significant role in 
our theory).

We also add several technical remarks: 
Regarding part (a), we demonstrate in Proposition \ref{proposition:generaloverfittedGaussian} 
later that the class of multivariate Gaussian or generalized Gaussian distribution is not 
identifiable in the second order. 
The condition \textit{odd degree of freedom} in part (b) and (c) of 
Theorem \ref{identifiability-multivariatecharacterization} is mainly due to our proof 
technique. We believe both (b) and (c) hold for any fixed positive degree of freedom, 
but do not have a proof for such setting.

Before ending this section, we state a general result which is a response to
a question posed by Xuming He on the identifiability in
transformed parameter spaces..
The following theorem states that 
the first-order identifiability with respect to a transformed parameter space is preserved
under some regularity conditions of the transformation operator.
Let $T$ be a bijective mapping from $\Theta^{*} \times \Omega^{*}$ to $\Theta \times \Omega$ such that 
\[T(\eta,\Lambda)=(T_{1}(\eta,\Lambda),T_{2}(\eta,\Lambda))=(\theta,\Sigma)\] 
for all $(\eta,\Lambda) \in \Theta^{*} \times \Omega^{*}$, where $\Theta^{*} \subset \mathbb{R}^{d_{1}}$, $\Omega^{*} \subset S_{d_{2}}^{++}$. Define the class of density functions $\left\{g(x|\eta,\Lambda),\eta \in \Theta^{*},\Lambda \in \Omega^{*}\right\}$ by
\[g(x|\eta,\Lambda) := f(x|T(\eta,\Lambda)).\] 
Additionally, for any $(\eta,\Lambda) \in \Theta^{*} \times \Omega^{*}$, let
$J(\eta,\Lambda) \in \mathbb{R}^{(d_{1}+d_{2}^{2}) \times (d_{1}+d_{2}^{2})}$ be 
the modified Jacobian matrix of $T(\eta,\Lambda)$, i.e. the usual Jacobian matrix when 
$(\eta,\Lambda)$ is taken as a $d_{1}+d_{2}^{2}$ vector. 

\begin{theorem} \label{theorem:identifiabilitytransformation}
Assume that $\left\{f(x|\theta,\Sigma),\theta \in \Theta, \Sigma \in \Omega \right\}$ 
is identifiable in the first order. Then the class of density functions 
$\left\{g(x|\eta,\Lambda),\eta \in \Theta^{*},
\Lambda \in \Omega^{*} \right\}$ is identifiable in the first order if and only if 
the modified Jacobian matrix $J(\eta,\Lambda)$ is non-singular for all 
$(\eta,\Lambda) \in \Theta^{*} \times \Omega^{*}$.
\end{theorem}

The conclusion of Theorem 
\ref{theorem:identifiabilitytransformation} still holds if we replace the first-order 
identifiability by the second-order identifiability.
As we have seen previously, strong identifiablity
(either in the first or second order) yields sharp lower bounds of $V(p_G,p_{G_0})$
in terms of Wasserstein distances $W_r(G,G_0)$. It is useful to know that in the
transformed parameter space, one may still enjoy the same inequality. Specifically,
for any discrete probability measure $Q = \sum_{i=1}^{k} p_i \delta_{(\eta_i,\Lambda_i)}
\in \mathcal{E}_{k}(\Theta^{*} \times \Omega^{*})$, 
denote
\begin{eqnarray}
p_{Q}'(x)=\int {g(x|\eta,\Lambda)}dQ(\eta,\Lambda)=\mathop {\sum }\limits_{i=1}^{k}{p_{i}g(x|\eta_{i},\Lambda_{i})}. \nonumber
\end{eqnarray}
Let $Q_{0}$ to be a fixed discrete probability measure on $\mathcal{E}_{k_{0}}(\Theta^{*} \times \Omega^{*})$, while probability measure $Q$ varies in $\mathcal{E}_{k_{0}}(\Theta^{*} \times \Omega^{*})$.
\comment{
 as the first-order identifiability is preserved under the non-singularity of the modified Jacobian matrix of bijective mapping $T$, the lower bound of the total variation distance between $p_{Q}'$ and $p_{Q_{0}}'$ in terms of $W_{1}(Q,Q_{0})$ continues to hold under certain conditions of class of density functions 
$\left\{f(x|\theta,\Sigma),\theta \in \Theta,\Sigma \in \Omega\right\}$ and the bijective mapping $T$.}

\begin{corollary} \label{corollary:lowerboundtransformation} Assume that the conditions of Theorem \ref{theorem:identifiabilitytransformation} hold. Further, suppose that 
the first derivative of $f$ in terms of $\theta$,$\Sigma$ and the first derivative of $T$ in terms of $\eta,\Lambda$ are $\alpha$-H\"older continuous and bounded where $\alpha>0$. Then there are positive constants $\epsilon_{0} := \epsilon_0(Q_{0})$ and
$C_0 := C_0(Q_{0})$ such that as long as $Q \in \mathcal{E}_{k}(\Theta^{*} \times \Omega^{*})$ and $W_{1}(Q,Q_{0}) \leq \epsilon_{0}$, we have
\begin{eqnarray}
V(p_{Q}',p_{Q_{0}}') \geq C_{0}W_{1}(Q,Q_{0}). \nonumber
\end{eqnarray}
\end{corollary}
\paragraph{Remark.} If $\Theta$ and $\Omega$ are bounded sets, the condition on the boundedness of the first derivative of $f$ in terms of $\theta,\Sigma$ and the first derivative of $g$ in terms of $\eta,\Lambda$ can be left out. Additionally, the restriction that these derivatives should be $\alpha$-H\"older continuous can be relaxed to only that the first derivative of $f$ and the first derivative of $g$ are $\alpha_{1}$-H\"older continuous and $\alpha_{2}$-H\"older continuous where $\alpha_{1},\alpha_{2}>0$ can be different.

\section{Theory for weakly identifiable classes} 
\label{Sec-weakidentifiability}
\newcommand{\rbar}{\ensuremath{{\overline{r}}}}
\newcommand{\Ecal}{\ensuremath{{\mathcal{E}}}}

The general theory of strong identifiability developed in the previous section
encompasses many classes of distributions, but they are not applicable to 
some important classes, those that we shall call \emph{weakly identifiable} classes
of distributions. These are the families of densities that 
are identifiable in the classical sense in a finite mixture setting, but they do 
not satisfy the strong identifiability conditions we have defined previously. 
Such classes of densities give rise to the ubiquitous location-covariance
Gaussian mixture, as well as mixture of Gamma distributions,
and mixture of skew-Gaussian distributions. We will see that these density
classes carry a quite varied and fascinating range of behaviors:
the specific algebraic structure of the density class in 
question now plays the fundamental role in determining identifiability and convergence
properties for model parameters and the mixing measure.

\subsection{Over-fitted mixture of location-covariance Gaussian distributions}

Location-covariance Gaussian distributions belong to the broader class of 
generalized Gaussians (cf. Example~\ref{example-Gaussian}),
which is identifiable in the first order according to
Theorem~\ref{identifiability-multivariatecharacterization}.
The class of location-covariance Gaussian distributions, however, is not identifiable
in the second order. This implies that in the over-fitted mixture
setting, Theorem~\ref{theorem-secondorder} is not applicable. 

In this section the multivariate Gaussian densities
$\left\{f(x|\theta,\Sigma),\theta \in \mathbb{R}^{d},\Sigma \in S_{d}^{++} \right\}$
is defined in the usual way,
i.e., 
$f(x|\theta,\Sigma)=\dfrac{1}{(2\pi)^{d/2}|\Sigma|^{1/2}}\exp(-(x-\theta)^{T}\Sigma^{-1}(x-\theta)/2)$. (Note that the scaling in the exponent 
slightly differs from the version 
given in Example~\ref{example-Gaussian}, where Gaussian distribution corresponds to setting $m=1$,
but this discrepancy is inconsequential). In fact, using
the same approach as the proof of Theorem \ref{identifiability-multivariatecharacterization}, 
we can verify that for any fixed positive number $m>1$, the class of generalized Gaussian
distributions is also identifiable in the second order. So within this broader family,
it is essentially only the class of Gaussian distributions with \emph{both} location and covariance parameters 
varying that is weakly identifiable.

\begin{proposition} \label{proposition:generaloverfittedGaussian} The family $\left\{f(x|\theta, \Sigma), \theta \in \mathbb{R}^{d}, \Sigma \in S_{d}^{++} \right\}$ of multivariate Gaussian distribution is not identifiable in the second order. 
\end{proposition}
\begin{proof}
\comment{Let $k \geq 1$ be a fixed positive integer and $k$ different pairs $(\theta_{1},\Sigma_{1}),...,(\theta_{k},\Sigma_{k}) \in \Theta \times \Omega$. Assume that we have $\alpha_{i} \in \mathbb{R},\beta_{i},\nu_{i} \in \mathbb{R}^{d_{1}}$, $\gamma_{i}, \eta_{i}$ symmetric matrices in $\mathbb{R}^{d_{2} \times d_{2}}$ as $i=1,\ldots, k$ such that
\begin{eqnarray}
\sum_{i=1}^{k} \biggr \{ \alpha_{i}f(x|\theta_{i},\Sigma_{i}) +
\beta_{i}^{T} \dfrac{\partial{f}}{\partial{\theta}}(x|\theta_{i},\Sigma_{i}) +
\nu_{i}^{T} \dfrac{\partial^{2}{f}}
{\partial{\theta}\partial{\theta^T}} (x|\theta_{i},\Sigma_{i}) \nu_{i} & + &  \notag \\
\trace \left(\dfrac{\partial{f}}{\partial{\Sigma}}(x|\theta_{i},\Sigma_{i})^{T}
\gamma_{i}\right ) 
+ 2\nu_{i}^{T}\left[\dfrac{\partial}{\partial{\theta}}\left(\trace \left(\dfrac{\partial{f}}{\partial{\Sigma}}(x|\theta_{i},\Sigma_{i})^{T}\eta_{i}\right)\right)\right]  & + & \notag \\
 \trace \left(\dfrac{\partial{}}{\partial{\Sigma}}\left(\trace\left(\dfrac{\partial{f}}{\partial{\Sigma}}(x|\theta_{i},\Sigma_{i})^{T}\eta_{i}\right)\right)^{T}\eta_{i}\right) 
\biggr \}
& = & 0 \quad \text{for almost all} \ x, \nonumber
\end{eqnarray}
}
The proof is immediate thanks to the following key identity,
which holds for all $\theta \in \mathbb{R}^{d}$ and $\Sigma \in S_{d}^{++}$:
\begin{equation}
\label{key-gaussian}
\dfrac{\partial^{2}{f}}{\partial{\theta}^{2}}(x|\theta,\Sigma)=2\dfrac{\partial{f}}{\partial{\Sigma}}(x|\theta,\Sigma).
\end{equation}
This identity is stated as Lemma~\ref{lemma:multivariatenormaldistribution} whose proof
is given in the Appendix.
Now, by choosing $\alpha_{i}=0 \in \mathbb{R}$, $\beta_{i}= \vec{0} \in \mathbb{R}^{d}$, $\eta_{i}= \vec{0} \in \mathbb{R}^{d \times d}$, and $2\nu_{i}\nu_{i}^{T}+\gamma_{i}=0$ for all $1 \leq i \leq k$, 
the equation given in [A2.] of Definition~\ref{definition-secondorder}
is clearly satisfied for all $x$. 
Since $\nu_{i}$ and $\gamma_{i}$ need not be $\vec{0}$, the second-order identifiability
does not hold.
\comment{ the above equation becomes
\begin{eqnarray}
\sum_{i=1}^{k} \biggr \{\alpha_{i}+ 4(x-\theta_{i})^{T}\Sigma_{i}^{-1}\nu_{i}\nu_{i}^{T}\Sigma_{i}^{-1}(x-\theta_{i}) - 2\nu_{i}^{T}\Sigma_{i}^{-1}\nu_{i}+ 2\beta_{i}^{T}\Sigma_{i}^{-1}(x-\theta_{i})-\dfrac{1}{2}\trace(\Sigma_{i}^{-1}\gamma_{i})+ \nonumber \\
(x-\theta_{i})^{T}\Sigma_{i}^{-1}\gamma_{i}\Sigma_{i}^{-1}(x-\theta_{i}) + \dfrac{\trace(\Sigma_{i}^{-1}\eta_{i})^{2}}{4} - \dfrac{\trace(\Sigma_{i}^{-1}\eta_{i})\trace(\Sigma_{i}^{-1}(x-\theta_{i})(x-\theta_{i})^{T}\Sigma_{i}^{-1}\eta_{i})}{2}+ \nonumber \\
 \dfrac{\trace(\Sigma_{i}^{-1}\eta_{i}\Sigma_{i}^{-1}\eta_{i})}{2} + 2(x-\theta_{i})^{T}\Sigma_{i}^{-1}\eta_{i}\Sigma_{i}^{-1}\eta_{i}\Sigma_{i}^{-1}(x-\theta_{i})+ \left((x-\theta_{i})^{T}\Sigma_{i}^{-1}\eta_{i}\Sigma_{i}^{-1}(x-\theta_{i})\right)^{2} - \nonumber \\
 \nu_{i}^{T}\Sigma_{i}^{-1}(x-\theta_{i})^{T}\trace(\Sigma_{i}^{-1}\eta_{i})-2\nu_{i}^{T}\Sigma_{i}^{-1}\eta_{i}\Sigma_{i}^{-1}(x-\theta_{i})+ \nonumber \\
2\nu_{i}^{T}\Sigma_{i}^{-1}(x-\theta_{i})(x-\theta_{i})^{T}\Sigma_{i}^{-1}\eta_{i}\Sigma_{i}^{-1}(x-\theta_{i}) \biggr \}\dfrac{\exp(-(x-\theta_{i})^{T}\Sigma_{i}^{-1}(x-\theta_{i}))}{\pi^{d/2}|\Sigma_{i}|^{1/2}}= 0. \nonumber
\end{eqnarray}}
\end{proof}
Identity~\eqref{key-gaussian} is the reason that strong identifiability fails for over-fitted
location-scale mixture of Gaussians. We shall see that it 
also provides the key for uncovering the precise convergence behavior of the 
mixing measure in the over-fitted Gaussian mixture model.

Let $G_{0}$ be a fixed probability measure with exactly $k_{0}$ support points, 
$\Theta$ is bounded subset of $\mathbb{R}^{d}$ and $\Omega$ is subset of $S_{d}^{++}$ 
where the largest eigenvalue of their elements are bounded above. Let $G$ vary
in the larger set $\mathcal{O}_k(\Theta \times \Omega)$, where $k \geq k_0 + 1$.
We shall no longer expect bounds of the kind $V \gtrsim W_2^2$
such as those established by Theorem~\ref{theorem-secondorder}. In fact, we can obtain  
sharp bounds of the type
$V(p_{G},p_{G_0}) \gtrsim W_{r}^{r}(G,G_0)$, where $r$ is determined by the (in)solvability
of a system of polynomial equations that we now describe.

For any fixed $k,k_{0} \geq 1$ where $k \geq k_{0}+1$, 
we define $\overline{r} \geq 1$ to be the \emph{minimum} 
value of $r \geq 1$ such that the following system of polynomial equations \begin{eqnarray}
\mathop {\sum }\limits_{j=1}^{k-k_{0}+1}{\mathop {\sum }\limits_{\substack{n_{1}+2n_{2} = \alpha \\ n_{1},n_{2} \geq 0}}{\dfrac{c_{j}^{2}a_{j}^{n_{1}}b_{j}^{n_{2}}}{n_{1}!n_{2}!}}}=0 \ \ \text{for each} \;
\alpha = 1,\ldots, r 
\label{eqn:generalovefittedGaussianzero}
\end{eqnarray}
does not have any \textbf{non-trivial} solution for the unknowns
$(c_{1},\ldots,c_{k-k_{0}+1},a_{1},\ldots,a_{k-k_{0}+1},b_{1},\ldots,b_{k-k_{0}+1})$. 
A solution is considered non-trivial if 
$c_{1},\ldots,c_{k-k_{0}+1}$ differ from 0 and at least one of 
$a_{1},\ldots,a_{k-k_{0}+1}$ differs from 0.

\paragraph{Remark.} This is a system of $r$ polynomial equations for $3(k-k_0+1)$ unknowns.
The condition $c_{1},\ldots,c_{k-k_{0}+1} \neq 0$ is very important. 
In fact, if $c_{1}=0$, then by choosing $a_{1} \neq 0$, $a_{i}=0$ for all $2 \leq i \leq k-k_{0}+1$ and $b_{j}=0$ for all $1 \leq j \leq k-k_{0}+1$, we can check that $\mathop {\sum }\limits_{j=1}^{k-k_{0}+1}{\mathop {\sum }\limits_{\substack{n_{1}+2n_{2} = \alpha \\ n_{1},n_{2} \geq 0}}{\dfrac{c_{j}^{2}a_{j}^{n_{1}}b_{j}^{n_{2}}}{n_{1}!n_{2}!}}}=0$ is satisfied for all $\alpha \geq 1$. Therefore, without this condition, $\overline{r}$ does not exist.

\paragraph{Example.}
To get a feel for the system of equations~\eqref{eqn:generalovefittedGaussianzero},
let us consider the case $k=k_0+1$, and let $r=3$. Then we obtain the equations:
\begin{eqnarray*}
& &c_{1}^{2}a_{1}+c_{2}^{2}a_{2}=0 \\ 
& &\dfrac{1}{2}(c_{1}^{2}a_{1}^{2}+c_{2}^{2}a_{2}^{2})+ c_{1}^{2}b_{1}+c_{2}^{2}b_{2} = 0 \\ 
& &\dfrac{1}{3!}(c_{1}^{2}a_{1}^{3}+c_{2}^{2}a_{2}^{3})+ c_{1}^{2}a_{1}b_{1}+c_{2}^{2}a_{2}b_{2} =0.
\end{eqnarray*}
It is simple to see that a non-trivial solution exists, by choosing $c_2 = c_1 \neq 0$, 
$a_1 = 1, a_2 = -1, b_1 = b_2 =-1/2$. 
Hence, $\overline{r} \geq 4$. For $r=4$, the system consists of the three equations
given above, plus
\begin{equation*}
\dfrac{1}{4!}(c_{1}^{2}a_{1}^{4}+c_{2}^{2}a_{2}^{4})+\dfrac{1}{2!}(c_{1}^{2}a_{1}^{2}b_{1}+c_{2}^{2}a_{2}^{2}b_{2})+\dfrac{1}{2!}(c_{1}^{2}b_{1}^{2}+c_{2}^{2}b_{2}^{2})=0 \label{S.1.4}
\end{equation*}
It can be shown in the sequel that this system has no non-trivial solution. Therefore for 
$k=k_0+1$, we have $\overline{r} = 4$.
Determining the exact value of $\overline{r}$ in the general case appears very difficult.
Even for the specific value of $k-k_0$, finding $\rbar$ is not easy.
There are well-developed methods in computational algebra for dealing with this type of
polynomial equations, such as Groebner bases~\citep{Bruno-Thesis} and 
resultants~\citep{Sturmfels}. Using the Groebner bases method, we can show that:
\begin{proposition}\label{proposition-specificvalueoverliner} 
{\bf (Values of $\rbar$)}
\begin{itemize}
\item[(i)] If $k-k_{0}=1$, $\overline{r}=4$.
\item[(ii)] If $k-k_{0}=2$, $\overline{r}=6$.
\item[(iii)] If $k-k_{0} \geq 3$, $\overline{r} \geq 7$.
\end{itemize}
\end{proposition}
\paragraph{Remark.} 
The results of this proposition appear to suggest that that $\overline{r} =
2(k-k_0+1)$. We leave this as a conjecture.
\comment{
Determining a general upper bound for $\overline{r}$ appears 
more manageable, as most over-determined systems of polynomial equations 
in general are inconsistent, i.e they admit no solution. 
We expect that $\overline{r} \leq 3(k-k_{0})+1$, which is the number of 
variables after rescaling in system \eqref{eqn:generalovefittedGaussianzero},
but we do not have a proof.}

The main result for this section is a precise relationship between
the identifiability and convergence behavior of mixing measures
in an over-fitted Gaussian mixture with the solvability of system
of equations~\eqref{eqn:generalovefittedGaussianzero}.
\begin{theorem}(Over-fitted Gaussian mixture) Let $\rbar$ be defined in the
preceeding paragraphs.
\label{theorem:generaloverfittedGaussian} 
\begin{itemize}
\item [(a)] For any $1 \leq r < \overline{r}$, there holds:
\begin{eqnarray}
\mathop {\lim }\limits_{\epsilon \to 0}{\mathop {\inf }\limits_{G \in \mathcal{O}_{k}(\Theta \times \Omega)}{\biggr \{ h(p_{G},p_{G_{0}})/
W_{1}^{r}(G,G_{0}): W_{1}(G,G_{0}) \leq \epsilon \biggr \}}} = 0. 
\label{lb-gauss}
\end{eqnarray}
\item [(b)] For any $c_{0}>0$, define
$\mathcal{O}_{k,c_{0}}(\Theta \times \Omega)=\left\{G=\mathop {\sum }\limits_{i=1}^{k^{*}}{p_{i}\delta_{(\theta_{i},\Sigma_{i})}} \in \mathcal{O}_{k}(\Theta \times \Omega):  p_{i} \geq c_{0} \ \forall \ 1 \leq i \leq k^{*}\right\}$. 
Then, for $G \in \mathcal{O}_{k,c_{0}}(\Theta \times \Omega)$ and 
$W_{\overline{r}}(G,G_{0})$ sufficiently small, there holds:
\begin{eqnarray}
V(p_{G},p_{G_{0}}) \gtrsim 
W_{\overline{r}}^{\overline{r}}(G,G_{0}) \geq W_{1}^{\overline{r}}(G,G_{0}). \nonumber
\end{eqnarray}
\end{itemize}
\end{theorem}
We make several remarks.
\begin{itemize} 
\item [(i)] 
Close investigation of the proof of part (a) and part (b) together shows that $W_\rbar^\rbar(G,G_0)$ is the sharp lower bound for 
the distance of mixture densities $h(p_G,p_{G_0}) \geq V(p_G,p_{G_0})$ when $c_{0}$ is sufficiently small. 
In particular, we cannot improve the the lower bound to $W_1^r$ for any $r < \rbar$.
\item [(ii)] This theorem yields an interesting
link between the convergence behavior of $G$ and the solvability of system of equation
\eqref{eqn:generalovefittedGaussianzero}. Part (b) is that, take any standard
estimation method such as the MLE, which yields $n^{-1/2}$ convergence rate 
under Hellinger distance for the mixture density under fairly general conditions, 
the convergence rate for $G$ under $W_{\rbar}$ is $n^{-1/(2\rbar)}$. 
Moreover, part (a) entails that $n^{-1/2\rbar}$ is also a minimax lower bound
for $G$ under $W_{\overline{r}}$ or $W_{1}$ distance.
\item [(iii)] The convergence behavior of $G$ depends only on the
number of extra mixing components $k-k_0$ assumed in the finite mixture model. 
The convergence rate deteriorates astonishingly fast as $k-k_0$ increases. 
For a practitioner this amounts to a sober caution against over-fitting
the mixture model with many more Gaussian components than actually needed.
\item [(iv)] As we have seen from part (b) of Theorem \ref{theorem:generaloverfittedGaussian}, 
under the general setting of $k-k_{0}$, $G$ is restricted to 
the set $\mathcal{O}_{k}(\Theta \times \Omega)$ to $\mathcal{O}_{k,c_{0}}(\Theta \times \Omega)$,
which places a constraint on the mixing probability mass.
However, this restriction seems to be an artifact of our proof technique. In fact,
it can be removed with extra hard work, 
at least for the case $k-k_0 \leq 2$, as the following proposition demonstrates:
\end{itemize}
\begin{proposition} 
\label{theorem:Gaussianoverfittedbytwo} Let $k-k_{0} = 1$ or $2$. 
For $G \in \mathcal{O}_{k}(\Theta \times \Omega)$ and
$W_{\overline{r}}(G,G_{0})$ sufficiently small, 
\begin{eqnarray}
V(p_{G},p_{G_{0}}) \gtrsim W_{\overline{r}}^{\overline{r}}(G,G_{0}). \nonumber
\end{eqnarray}
\end{proposition}

\subsection{Mixture of Gamma distributions and the location extension}
The Gamma family of univariate densities takes the form
$f(x|a,b) := \dfrac{b^{a}}{\Gamma(a)}x^{a-1}\exp(-bx)$ for $x > 0$, and $0$ otherwise, 
where $a,b$ are positive shape and rate parameters, respectively. 
\begin{proposition}\label{example-notGamma}
The Gamma family of distributions is \emph{not} identifiable in the first order.
\end{proposition}
\begin{proof}
The proof is immediate thanks to the following algebraic identity, which holds
for any $a,b > 0$:
\begin{equation}
\label{key-gamma}
\frac{\partial f}{\partial b} = \frac{a}{b}f(x|a,b) - \frac{a}{b}f(x|a+1,b).
\end{equation}
Now given $k=2$, $a_{2}=a_{1}-1,b_{1}=b_{2}$. 
By choosing $\beta_{1}=\beta_{2}=0$, $\gamma_{1}=0$, 
$\alpha_{1}b_{1}=\gamma_{2}a_{2}$, $\alpha_{2}b_{1}=-\gamma_{2}a_{2}$ and 
$\alpha_{1}= - \alpha_{2} \neq 0$, then we can verify that
\begin{eqnarray}
\mathop {\sum }\limits_{i=1}^{2}{\alpha_{i}f(x|a_{i},b_{i})+\beta_{i}\dfrac{\partial{f}}{\partial{a}}(x|a_{i},b_{i})+\gamma_{i}\dfrac{\partial{f}}{\partial{b}}(x|a_{i},b_{i})}=0. \nonumber
\end{eqnarray}
\end{proof}
The Gamma family is still strongly identifiable in the first order
if either shape or rate parameter is fixed. It is when both parameters 
are allowed to vary that strong identifiablity is violated. Thus, neither Theorem~\ref{theorem-firstorder} 
nor Theorem~\ref{theorem-secondorder} is applicable to shape-rate Gamma mixtures. 
Comparing the algebraic identity~\eqref{key-gaussian} for the Gaussian 
and~\eqref{key-gamma} for the Gamma reveals an interesting feature
for the latter. In particular, the linear dependence of the collection of Gamma
density functions and its derivatives are due to certain specific combinations
of the Gamma parameter values. This suggests that outside of these value combinations
the Gamma densities may well be identifiable in the first order and even the second order 
Indeed, this observation leads to the following results, which we shall state in
two separate mixture settings.

Fix the true mixing measure
$G_{0}=\mathop {\sum }\limits_{i=1}^{k_{0}}{p_{i}^{0}\delta_{(a_{i}^{0},b_{i}^{0})}}
\in \Ecal_{k_0}(\Theta)$ where $k_{0} \geq 2$ and $\Theta \subset \mathbb{R}^{2,+}$. 

\begin{theorem}{\bf (Exact-fitted Gamma mixtures)} 
\label{theorem:exactfittedGamma}
\begin{itemize}
\item [(a)] (Generic cases) Assume that $\left\{|a_{i}^{0}-a_{j}^{0}|,|b_{i}^{0}-b_{j}^{0}|\right\} \neq \left\{1,0\right\}$ for all $1 \leq i,j \leq k_{0}$, and $a_{i}^{0} \geq 1$ for all $1 \leq i \leq k_{0}$. 
Then for $G \in \mathcal{E}_{k_{0}}(\Theta)$ and $W_{1}(G,G_{0})$ sufficiently small, we have
\begin{eqnarray}
V(p_{G},p_{G_{0}}) \gtrsim W_{1}(G,G_{0}). \nonumber
\end{eqnarray}
\item [(b)] (Pathological cases)
If there exist $1 \leq i,j \leq k_{0}$ such that  $\left\{|a_{i}^{0}-a_{j}^{0}|,|b_{i}^{0}-b_{j}^{0}|\right\} = \left\{1,0\right\}$, 
then for any $r \geq 1$, 
\begin{eqnarray}
\mathop {\lim }\limits_{\epsilon \to 0}{\mathop {\inf }\limits_{G \in \mathcal{E}_{k_{0}}(\Theta)}{\biggr\{
V(p_{G},p_{G_{0}})/W_{r}^{r}(G,G_{0}): W_{r}(G,G_{0}) \leq \epsilon \biggr \}}} = 0. \nonumber
\end{eqnarray}
\end{itemize}
\end{theorem}

Turning to the over-fitted Gamma mixture setting, as before
let $G_0 \in \Ecal_{k_0}(\Theta)$, while $G$ varies in a larger subset of $\Ocal_{k}(\Theta)$
for some given $k \geq k_0+1$.
\begin{theorem} {\bf (Over-fitted Gamma mixture)}
\label{theorem:overfittedGamma}
\begin{itemize}
\item [(a)] (Generic cases) Assume that $\left\{|a_{i}^{0}-a_{j}^{0}|,|b_{i}^{0}-b_{j}^{0}|\right\} \not \in \biggr\{\left\{1,0\right\}, \left\{2,0\right\}\biggr\}$ for all $1 \leq i,j \leq k_{0}$, and $a_{i}^{0} \geq 1$ for all $1 \leq i \leq k_{0}$. For any $c_{0}>0$, define a subset of $\Ocal_k(\Theta)$:
\begin{eqnarray}
\mathcal{O}_{k,c_{0}}(\Theta)=
\biggr \{G=\mathop {\sum }\limits_{i=1}^{k^{'}}{p_{i}\delta_{(a_{i},b_{i})}}: k^{'} \leq k \ \text{and } \ |a_{i}-a_{j}^{0}| \not \in [1-c_{0},1+c_{0}] \cup [2-c_{0},2+c_{0}] \forall \ (i,j) 
\biggr \}. \nonumber
\end{eqnarray}
Then, for $G \in \mathcal{O}_{k,c_0}(\Theta)$ and $W_{2}(G,G_{0})$ sufficiently small, 
we have \begin{eqnarray}
V(p_{G},p_{G_{0}}) \gtrsim W_{2}^{2}(G,G_{0}). \nonumber
\end{eqnarray}
\item [(b)] (Necessity of restriction on $G$)
Under the same assumptions on $G_{0}$, for any $r \geq 1$,
\begin{eqnarray}
\mathop {\lim }\limits_{\epsilon \to 0}{\mathop {\inf }\limits_{G \in \mathcal{O}_{k}(\Theta)}{
\biggr\{V(p_{G},p_{G_{0}})/W_{r}^{r}(G,G_{0})}: W_{r}(G,G_{0}) \leq \epsilon \biggr\}} = 0. \nonumber
\end{eqnarray}
\item [(c)] (Pathological cases)
If there exist $1 \leq i,j \leq k_{0}$ such that $\left\{|a_{i}^{0}-a_{j}^{0}|,|b_{i}^{0}-b_{j}^{0}|\right\} \in \biggr\{\left\{1,0\right\}, \left\{2,0\right\}\biggr\}$, then for any $r \geq 1$ and any $c_{0}>0$,
\begin{eqnarray}
\mathop {\lim }\limits_{\epsilon \to 0}{\mathop {\inf }\limits_{G \in \mathcal{O}_{k,c_{0}}(\Theta)}
{ \biggr \{V(p_{G},p_{G_{0}})/W_{r}^{r}(G,G_{0})}: W_{r}(G,G_{0}) \leq \epsilon \biggr\}} = 0. \nonumber
\end{eqnarray}
\end{itemize}
\end{theorem}

Part (a) of both theorems asserts that outside of a measure zero set of the true mixing measure $G_0$,
we can still consider Gamma mixture as if it is strongly identifiable: the strong bounds
$V\gtrsim W_1$ and $V\gtrsim W_2^2$ continue to hold. In these so-called generic
cases, if we take any standard estimation method that yields $n^{-1/2}$ convergence rate
under Hellinger/variational distance for the mixture density $p_G$, the corresponding
convergence for $G$ will be $n^{-1/2}$ for exact-fitted and $n^{-1/4}$ for over-fitted mixtures.

The situation is not so forgiving for the so-called pathological cases in both settings:
it is not possible to obtain the bound of the form $V\gtrsim W_r^r$
for any $r \geq 1$. A consequence of this result is a minimax lower bound
$n^{-1/r}$ under $W_r$ for the estimation of $G$, for \emph{any} $r\geq 1$. 
This implies that, even for the exact-fitted mixture, the convergence of 
Gamma parameters $a_i$ and $b_i$ to the true values cannot be faster than $n^{-1/r}$
for any $r\geq 1$. In other words, the convergence of these parameters 
is mostly likely logarithmic.

\paragraph{Location extension.} Before ending this subsection, we introduce a location extension
of the Gamma family, for which the convergence behavior of its parameters is always slow.
Actually, this is the location extension of the exponential distribution (which is a special case of
Gamma by fixing the shape parameter $a=1$). 
The location-exponential distribution $\left\{f(x|\theta,\sigma),\theta \in \mathbb{R},\sigma \in \mathbb{R}_{+} \right\}$ is parameterized as  
$f(x|\theta,\sigma)=\dfrac{1}{\sigma}\exp(-\dfrac{x-\theta}{\sigma}).1_{\left\{x>\theta\right\}}$ for all $x \in \mathbb{R}$. Direct calculation yields that 
\begin{eqnarray}
\label{key-location}
\dfrac{\partial{f}}{\partial{\theta}}(x|\theta,\sigma)=\dfrac{1}{\sigma}f(x|\theta,\sigma) \ \text{when} \ x \neq \theta. \label{eqn:identitylocationexponential}
\end{eqnarray}
This algebraic identity is similar to that of location-scale multivariate Gaussian distribution, except for the non-constant coefficient $1/\sigma$. Since this identity holds in general, we would expect non-standard 
convergence behavior for $G$. This is indeed the case. 
We shall state a result for the exact-fitted setting only. Let $\Theta = \mathbb{R} \times \mathbb{R}_+$,
and $G_{0}=\mathop {\sum }\limits_{i=1}^{k_{0}}{p_{i}^{0}\delta_{(\theta_{i}^{0},\sigma_{i}^{0})}}
\in \Ecal_{k_0}(\Theta)$ where $k_{0} \geq 2$.  
\begin{theorem}{\bf (Exact-fitted location-exponential mixtures)}
\label{theorem:exactlocationgamma} 
For any $r \geq 1$,
\begin{eqnarray}
\mathop {\lim }\limits_{\epsilon \to 0}{\mathop {\inf }\limits_{G \in \mathcal{E}_{k_{0}}(\Theta)}
{\biggr \{V(p_{G},p_{G_{0}})/W_{1}^{r}(G,G_{0}): W_{1}(G,G_{0}) \leq \epsilon \biggr\}}} = 0. \nonumber
\end{eqnarray}
\end{theorem}
Unlike Gamma mixtures, there is no generic/pathological dichotomy for mixtures of
location-exponential distributions. The convergence behavior of the mixing measure $G$
is always extremely slow: even in the exact-fitted setting, the minimax lower bound for $G$ under $W_1$
is no smaller than $n^{-1/r}$ for any $r$. The convergence rate the model parameters is most
likely logarithmic.

\subsection{Mixture of skew-Gaussian distributions} \label{subsection:skewnormal}
The skew-normal density takes the form $f(x|\theta,\sigma,m) :=
\dfrac{2}{\sigma}f\left(\dfrac{x-\theta}{\sigma}\right) 
\Phi(m (x-\theta)/\sigma)$, where 
$f(x)=\dfrac{1}{\sqrt{2\pi}}\exp\left(-\dfrac{x^{2}}{2}\right)$, and ${\displaystyle \Phi(x)=\int \limits_{-\infty}^{x}{f(t)}dt}$. $m \in \mathbb{R}$ is the shape, $\theta$ the location
and $\sigma$ the scale parameter. This generalizes the Gaussian family, which corresponds to 
fixing $m=0$. In general, letting $m\neq 0$ makes the density asymmetric (skew), with the skewness
direction dictated by the sign of $m$. We will see that this 
density class enjoys an extremely rich range of behaviors.

We first focus on exact-fitted mixtures of skew-Gaussian distributions. Note that:
\begin{proposition}\label{proposition-notskewnormal}
The skew-Gaussian family  
$\left\{f(x|\theta,\sigma,m),\theta \in \mathbb{R},\sigma \in \mathbb{R}_{+}, m \in \mathbb{R}\right\}$ 
is not identifiable in the first order.
\end{proposition}
An examination of the proof of Proposition~\ref{proposition-notskewnormal} reveals that, like
the Gamma family, there are certain combinations of the skew-Gaussian distribution's 
parameter values that prevent the skew-Gaussian family from satisfying strong 
identifiability conditions. Outside of these ``pathological'' combinations, 
the skew-Gaussian mixtures continue to enjoy strong convergence properties. Unlike the 
Gamma family, however, the pathological cases have very rich structures, which result 
in a varied range of convergence behaviors we have seen in both Gamma and Gaussian
mixtures.

Throughout this section, 
$\left\{(f(x|\theta,\sigma, m), (\theta, m) \in \Theta, \sigma^{2} \in \Omega \right\}$ is a class of skew-Gaussian density function where $\Theta \subset \mathbb{R}^{2}$ and $\Omega \subset \mathbb{R}_{+}$.
Fix the true mixing measure
$G_{0}=\mathop {\sum }\limits_{i=1}^{k_{0}}{p_{i}^{0}\delta_{(\theta_{i}^{0},(\sigma_{i}^{0})^{2},m_{i}^{0})}}$.  
Assume that $\sigma_{i}^{0}$ are pairwise different and 
$\dfrac{(\sigma_{i}^{0})^{2}}{1+(m_{i}^{0})^{2}} \not \in \left\{(\sigma_{j}^{0})^{2}: 1 \leq j \neq i \leq k_{0} \right\}$ for all $1 \leq i \leq k_{0}$. For each $1 \leq j \leq k_{0}$, define the
{\bf cousin set} for $j$ to be
\begin{eqnarray}
I_{j}=\left\{i \neq j: (\dfrac{(\sigma_{i}^{0})^{2}}{1+(m_{i}^{0})^{2}},\theta_{i}^{0}) \equiv (\dfrac{(\sigma_{j}^{0})^{2}}{1+(m_{j}^{0})^{2}},\theta_{j}^{0})\right\}. \nonumber
\end{eqnarray}
The cousin set consists of the indices of skew-Gaussian components that share the same
location and a rescaled version of the scale parameter. 
We further say that a non-empty cousin set $I_{j}$ 
{\bf conformant} if for any $i \in I_{j}$, $m_{i}^{0}m_{j}^{0}>0$. 
To delineate the structure underlying parameter values of $G_0$, we define
a sequence of increasingly weaker conditions.
\begin{itemize}
\item[(S1)] $m_{i}^{0} \neq 0$ and $I_{i}$ is empty for all $i=1,\ldots,k_0$.
\item[(S2)] There exists at least one set $I_{i}$ to be non-empty. Moreover, for any $1 \leq i \leq k_{0}$, if $|I_{i}| \geq 1$, $I_{i}$ is conformant.
\item[(S3)] There exists at least one set $I_{i}$ to be non-empty. Additionally, there is $k^{*} \in [1,k_{0}-1]$ such that 
for any non-empty and non-conformant cousin set $I_{i}$, we have $|I_{i}| \leq k^{*}$.
\end{itemize}
We make several clarifying comments.
\begin{itemize}
\item [(i)] Condition (S1) corresponds to generic situations of true parameter values where
the exact-fitted mixture of skew-Gaussians will be shown to enjoy behaviors akin to strong identifiability.
They require that the true mixture corresponding to $G_0$ has
no Gaussian components and no cousins for all skew-Gaussian components.
\item [(ii)] Condition (S2) allows the presence of either Gaussian components and/or 
non-empty cousin sets, all of which have to be conformant.
\item [(iii)] (S3) is introduced to address the presence of non-conformant cousin sets.
\end{itemize}

\begin{theorem}{\bf (Exact-fitted conformant skew-Gaussian mixtures)}
\label{theorem:exactfittedskewnormal} 
\begin{itemize}
\item [(a)] (Generic cases) If (S1) is satisfied, then for 
any $G \in \mathcal{E}_{k_{0}}(\Theta \times \Omega)$ such that
$W_{1}(G,G_{0})$ is sufficiently small, there holds
\begin{eqnarray}
V(p_{G},p_{G_{0}}) \gtrsim W_{1}(G,G_{0}). \nonumber
\end{eqnarray}
\item [(b)] (Conformant cases) 
If (S2) is satisfied, then for any
$G \in \mathcal{E}_{k_{0}}(\Theta \times \Omega)$ and $W_{2}(G,G_{0})$ is sufficiently small,
there holds
\begin{eqnarray}
V(p_{G},p_{G_{0}}) \gtrsim W_{2}^{2}(G,G_{0}). \nonumber
\end{eqnarray}
Moreover, this lower bound is sharp.
\end{itemize}
\end{theorem}
When only condition (S3) holds, the convergence behavior of the exact-fitted skew-Gaussian
mixture is linked to the (in)solvability of a system of polynomial equations.
Specifically, define $\overline{s}$ to be the minimum value of $r \geq 1$ such that 
the following system of polynomial equations 
\begin{eqnarray}
\mathop {\sum }\limits_{i=1}^{k^{*}+1}{a_{i}b_{i}^{u}c_{i}^{v}}=0 \label{eqn:noncannonicalexactfittedskewnorma}
\end{eqnarray}
does not admit any \textbf{non-trivial} solution. By non-trivial, we require that
$a_{i}>0$ for all $i = 1,\ldots, k^{*}+1$, all $b_{i} \neq 0$ and pairwise different, $(a_{i},|b_{i}|) \neq (a_{j},|b_{j}|)$ for all $1 \leq i \neq j \leq k^{*}+1$, and at least one of $c_{i}$ differs from 0, where the indices $u,v$ in this system of polynomial equations satisfy $1 \leq v \leq r$, $u \leq v$ are all odd numbers when $v$ is even or $0 \leq  u \leq v$ are all even number when $v$ is odd. For example, if $r=3$, and $k^{*}=1$, the above system of polynomial equations is
\begin{eqnarray}
a_{1}c_{1}+a_{2}c_{2}=0, \nonumber \\
a_{1}b_{1}c_{1}^{2}+a_{2}b_{2}c_{2}^{2}=0, \nonumber \\
a_{1}c_{1}^{3}+a_{2}c_{2}^{3}=0, \nonumber \\
a_{1}b_{1}^{2}c_{1}^{3}+a_{2}b_{2}^{2}c_{2}^{3}=0. \nonumber
\end{eqnarray}
Similar to system of equations~\eqref{eqn:generalovefittedGaussianzero} that arises
in our theory for Gaussian mixtures,
the exact value of $\overline{s}$ is hard to determine in general. 
The following proposition gives  specific values for $\overline{s}$.
\begin{proposition}
\label{proposition:upperboundskewnormal}
{\bf (Values of $\overline{s}$)} 
\begin{itemize}
\item[(i)] If $k^{*}=1$, $\overline{s}=3$.
\item[(ii)] If $k^{*}=2$, $\overline{s}=5$. 
\end{itemize}
\end{proposition}
The following theorem describes the role of $\overline{s}$ in the non-conformant case
of skew-Gaussian mixtures:
\begin{theorem}
\label{theorem:nonconformantexactfittedskewnormal}
{\bf (Exact-fitted non-conformant skew-Gaussian mixtures)}
Suppose that $(S3)$ holds.
\begin{itemize}
\item [(a)] Assume further that for any non-conformant cousin set $I_{i}$ 
we have $(p_{i}^{0},|m_{i}^{0}|) \neq (p_{j}^{0},|m_{j}^{0}|)$ for any $j \in I_{i}$. 
Then, for any $G \in \mathcal{E}_{k_{0}}(\Theta \times \Omega)$ such that
$W_{\overline{s}}(G,G_{0})$ is sufficiently small,
\begin{eqnarray}
V(p_{G},p_{G_{0}}) \gtrsim W_{\overline{s}}^{\overline{s}}(G,G_{0}). \nonumber
\end{eqnarray}
\item [(b)] If the assumption of part (a) does not hold, then for any $r \geq 1$,
\begin{eqnarray}
\mathop {\lim }\limits_{\epsilon \to 0}{\mathop {\inf }\limits_{G \in \mathcal{E}_{k_{0}}(\Theta)}
{\biggr \{V(p_{G},p_{G_{0}})/W_{1}^{r}(G,G_{0}): W_{1}(G,G_{0}) \leq \epsilon \biggr\}}} = 0. \nonumber
\end{eqnarray}
\end{itemize}
\end{theorem}

We note that the lower bound established in part (a) may be not sharp. Nonetheless,
it can be used to derive an upper bound on the convergence of $G$ for any standard
estimation method: an $n^{-1/2}$ convergence rate for $p_G$ under the variational
distance entails $n^{-1/(2\overline{s})}$ convergence rate for $G$ under $W_{\overline{s}}$.
If the assumption of part (a) fails to hold, no polynomial rate (in terms of $n^{-1}$) is
possible as can be inferred from part (b).

\paragraph{Over-fitted skew-Gaussian mixtures.}
Like what we have done with Gaussian mixtures, the analysis
of over-fitted skew-Gaussian mixtures hinges upon the algebraic structure of the
density function and its derivatives taken up to the second order.
The fundamental identity for the skew-Gaussian density is
\begin{eqnarray}
\label{key-skewnormal}
\dfrac{\partial^{2}{f}}{\partial{\theta}^{2}}(x|\theta,\sigma,m)-2\dfrac{\partial{f}}{\partial{\sigma}^{2}}(x|\theta,\sigma,m)+\dfrac{m^{3}+m}{\sigma^{2}}\dfrac{\partial{f}}{\partial{m}}(x|\theta,\sigma,m)=0. \label{eqn:overfittedskewnormaldistributionzero}
\end{eqnarray}
The proof for this identity is in Lemma \ref{lemma:skewnormaldistribution}. This implies
that the skew-Gaussian class is without exception \emph{not} identifiable in the second order.
By no exception, we mean that there is no generic/pathological dichotomy due to certain
combinations of the parameter values as we have seen in the first-order analysis.
Note that if $m=0$ this is reduced to Eq.~\eqref{key-gaussian} in the univariate case. 
The presence of nonlinear coefficient $(m^{3}+m)/\sigma^{2}$, which depends on both $m$ and $\sigma$,
makes the analysis of the skew-Gaussians much more complex than that of the Gaussians.

\comment{
With this identity, similar to multivariate Gaussian distribution case, skew normal distribution is not strongly identifiable in the second order in general. Additionally, the analysis of the (sharp) lower bound under over-fitted case in skew normal distribution may be more intricate than that of location-scale multivariate Gaussian distribution as we have the element $\dfrac{m^{3}+m}{\sigma^{2}}\dfrac{\partial{f}}{\partial{m}}(x|\theta,\sigma, m)$  \eqref{eqn:overfittedskewnormaldistributionzero} and the coefficient $(m^{3}+m)/\sigma^{2}$ in the identity, which depend on both $m$ and $\sigma$. Therefore, when we go to higher order derivatives, it becomes harder to decode their linear dependent structures.   
}

The following theorem gives a bound of the type $V\gtrsim W_r^r$, under some conditions.
\begin{theorem}
\label{theorem:Overfittedskewnormal} 
{\bf (Over-fitted skew-Gaussian mixtures)} 
Assume that the support points of $G_{0}$ satisfy the condition $(S1)$. Let $k\geq k_0+1$ and $\overline{r} \geq 1$ to be defined as in \eqref{eqn:generalovefittedGaussianzero}. 
For a fixed positive constant $c_{0}>0$, we define a subset of $\Ocal_{k}(\Theta)$: 
\begin{eqnarray}
\mathcal{O}_{k,c_{0}}(\Theta \times \Omega)=\left\{G=\mathop {\sum }\limits_{i=1}^{k^{*}}{p_{i}\delta_{(\theta_{i},\sigma_{i}^{2},m_{i})}} \in \mathcal{O}_{k}(\Theta \times \Omega):  p_{i} \geq c_{0} \ \forall \ 1 \leq i \leq k^{*} \leq k\right\}. \nonumber
\end{eqnarray}
Then, for any $G \in \mathcal{O}_{k,c_{0}}(\Theta \times \Omega)$ and 
$W_{\overline{m}}(G,G_{0})$ sufficiently small, there holds
\begin{eqnarray}
V(p_{G},p_{G_{0}}) \gtrsim W_{\overline{m}}^{\overline{m}}(G,G_{0}), \nonumber
\end{eqnarray}
where $\overline{m}=\overline{r}$ if $\overline{r}$ is even,
and $\overline{m}=\overline{r}+1$ if $\overline{r}$ is odd. \\
\end{theorem}
\paragraph{Remarks.} 
\begin{itemize}
\item [(i)] If $k-k_{0}=1$, we can allow $G\in \Ocal_{k}(\Theta\times\Omega)$, and
the above bound holds for $\overline{m} = 4$. Moreover this bound is sharp.
\item [(ii)] Our proof exploits assumption (S1), which
entails the linear independent structure of high order derivatives of $f$ with respect to \emph{only} $\theta$ and $m$, and
the instrinsic dependence of 
$\dfrac{\partial^{2}{f}}{\partial{\theta}^{2}}$ on $\dfrac{\partial{f}}{\partial{\sigma^{2}}}$. 
Although we make use of Eq.~\eqref{key-skewnormal} in the proof
we do not fully account for the dependence of 
$\dfrac{\partial^{2}{f}}{\partial{\theta}^{2}}$ on $\dfrac{\partial{f}}{\partial{m}}$ 
as well as the nonlinear coefficient $(m^{3}+m)/\sigma^{2}$. For these reasons
the bound produced in this theorem may not be sharp in general.
\item [(iii)] If $k-k_0=2$, it seems that the best lower bound for
$V(p_{G},p_{G_{0}})$ is  $W_{4}^{4}(G,G_{0})$. 
(See the arguments following the proof of Theorem \ref{theorem:Overfittedskewnormal}
in the Appendix). 
\item [(iv)] 
The analysis of lower bound of $V(p_{G},p_{G_{0}})$ when $G_{0}$ satisfies 
either (S2) or (S3) 
is highly non-trivial since they contain complex dependence of 
high order derivatives of $f$. This is beyond the scope of this paper.
\end{itemize}
\section{Minimax lower bounds, MLE rates and illustrations}
\subsection{Convergence of MLE and minimax lower bounds}
\label{Sec-convergence}

Given $n$-iid sample $X_{1},X_{2},...,X_{n}$ distributed according to 
mixture density $p_{G_{0}}$, 
where $G_{0}$ is unknown true mixing distribution with exactly $k_{0}$ support points,
and class of densities 
$\left\{f(x|\theta,\Sigma),\theta \in \Theta,\Sigma \in \Omega\right\}$ 
is assumed known. 
Given $k \in \mathbb{N}$ such that $k \geq k_{0}+1$.
The support of $G_0$ is $\Theta\times\Omega$. In this section we shall 
assume that $\Theta$ is a compact subset of $\mathbb{R}^{d_{1}}$ and 
$\Omega=\left\{\Sigma \in S_{d_{2}}^{++}:\underline{\lambda} \leq \sqrt{\lambda_{1}(\Sigma)} \ \leq \ \sqrt{\lambda_{d}(\Sigma)} \leq \overline{\lambda}\right\}$,
where $0< \underline{\lambda},\overline{\lambda}$ are known and $d_{1} \geq 1, d_{2} \geq 0$. 
The maximum likelihood estimator
for $G_0$ in the over-fitted mixture setting is given by
\begin{eqnarray}
\widehat{G}_{n}=\mathop {\arg \max}\limits_{G \in \mathcal{O}_{k}(\Theta \times \Omega)}{\mathop {\sum }\limits_{i=1}^{n}{\log(p_{G}(X_{i}))}}. \nonumber
\end{eqnarray}
For the exact-fitted mixture setting, $\mathcal{O}_{k}$ is replaced by $\mathcal{E}_{k_{0}}$. 

\comment{
Herethereupon, when we said that the convergence rate of $W_{r}(\widehat{G}_{n},G_{0})$ is $\delta_{n}$ for some real number $r \geq 1$ and sequence $\delta_{n} \to 0$, it means that $P(W_{r}(\widehat{G}_{n},G_{0})>c\delta_{n}) \to 0$ as $n \to \infty$ where $c$ is some sufficiently large positive constant.
}

According to the standard asymptotic theory for the MLE (cf.,e.g.,~\cite{Vandegeer-1996}),
under the boundedness assumptions given above, along
with a sufficient regularity condition on the smoothness of density $f$, one 
can show that the MLE for the \emph{mixture density} 
yields $(\log n/n)^{1/2}$ rate under Hellinger distance.
That is, $h(p_{\widehat{G}_n},p_{G_0}) = O_P((\log n/n)^{1/2})$,
where $O_P$ denotes in $p_{G_0}$-probability bound. 
It is relatively simple to verify that
this bound is applicable to all density classes considered in this paper.
As a consequence, whenever an
identifiability bound of the form $V \gtrsim W_r^r$ holds,
we obtain that $W_r(\widehat{G}_n,G_0) \lesssim (\log n/n)^{1/2r}$ in probability.

Furthermore, if we can also show that $h \gtrsim W_r^r \geq W_1^r$ 
is the best bound possible in a precise sense -- for instance, in the
sense given by part (c) of Theorem \ref{theorem-secondorder} (for $r=2$) or part (a) of Theorem \ref{theorem:generaloverfittedGaussian} (for $r = \rbar$), 
then an immediate consequence, by invoking Le Cam's method (cf.~\cite{Yu-97}),
is the following minimax lower bound:
\[\inf_{\widehat{G}_{n}} \sup_{G_0} W_1(\widehat{G}_n,G_0) \gtrsim n^{-1/(2r')},\]
where $r'$ is any constant $r'\in [1,r)$, the supremum is taken over the 
given set of possible values for $G_0$,
and the infimum is taken over all possible estimators. Combining with
an upper bound of the form $(\log n/n)^{1/2r}$ guaranteed by the MLE method,
we conclude that $n^{-1/2r}$ is the optimal estimation rate, up to
a logarithmic term, under $W_r$ distance for the mixing measure.

For mixtures of Gamma, location-exponential and
skew-Gaussian distributions, we have seen pathological settings where $V$ cannot be
lower bounded by a multiple of $W_r^r$ for any $r \geq 1$. This entails that the minimax
estimation rate cannot be faster than $n^{-1/r}$ for any $r\geq 1$. It follows 
that the minimax rate for estimating $G_0$ in such settings  
cannot be faster than a logarithmic rate.

In summary, we obtain a number of convergence rates and minimax lower bounds
for the mixing measure under many density classes. They are collected
in Table ~\ref{tablesummary}.

\label{Simulations}
\subsection{Illustrations}
\label{Sec:simulation}
For the remainder of this section we shall illustrate via simulations the rich spectrum
of convergence behaviors of the mixing measure in a number of settings. This is reflected
by the identifiability bound $V\gtrsim W_r^r$ and its sharpness 
for varying values of $r$, as well as the convergence rate of the MLE.

\paragraph{Strong identifiability bounds.} We illustrate the bound $V \gtrsim W_1$
for exact-fitted mixtures, and $V\gtrsim W_2^2$ for over-fitted mixtures of
the class of Student's t-distributions. See
Figure~\ref{figure-exactfitted-overfitted-Student}. The upper bounds of $V$ and $h$
were also proved earlier in Section~\ref{Sec-upperbound}.
For details, we choose 
$\Theta=[-10,10]^{2}$ and $\Omega=\left\{\Sigma \in S_{2}^{++}: \sqrt{2} \leq \sqrt{\lambda_{1}(\Sigma)} \leq \sqrt{\lambda_{d}(\Sigma)} \leq 2 \right\}$. The true mixing probability measure $G_{0}$ has exactly $k_{0}=2$ support points with locations $\theta_{1}^{0}=(-2,2)$, $\theta_{2}^{0}=(-4,4)$, covariances $\Sigma_{1}^{0}= \begin{pmatrix} 9/4 & 1/5 \\ 1/5 & 13/6 \end{pmatrix}$, $\Sigma_{2}^{0}=\begin{pmatrix} 5/2 & 2/5 \\ 2/5 & 7/3 \end{pmatrix}$, and $p_{1}^{0}=1/3, p_{2}^{0}=2/3$.
5000 random samples of discrete mixing measures $G \in \Ecal_2$, 5000 samples of  $G\in \Ocal_3$ 
were generated to construct these plots.
%

\paragraph{Weak identifiability bounds.}
We experiment with two interesting classes of densities: Gaussian and skew-Gaussian densities. 
According to our theory, sharp bounds of the form $V\gtrsim W_r^r$ continue to hold, 
but with varying values of $r$ depending on the specific mixture setting. $r$ can
also vary dramatically  within the same density class.

The results for mixtures of location-covariance Gaussian distributions is given in 
Figure~\ref{figure-weak-ident-Gaussian}.  Simulation details are as follows.
The true mixing measure $G_{0}$ has exactly $k_{0}=2$ support points with locations $\theta_{1}^{0}=-2$, $\theta_{2}^{0}=4$, scales $\sigma_{1}^{0}=1$, $\sigma_{2}^{0}=2$, and $p_{1}^{0}=1/3, p_{2}^{0}=2/3$. 
5000 random samples of discrete mixing measures $G \in \Ecal_2$, 5000 samples of  $G\in \Ocal_3$
and another 5000 for $G \in \Ocal_4$, where the support points are uniformly generated in
$\Theta=[-10,10]$ and $\Omega=[0.5,5]$. 

The bounds for skew-Gaussian mixtures are illustrated by 
Figure \ref{figure-weak-ident-skewgauss}. Here are the simulation details.
The true parameters for mixing measure $G_{0}$ will be divided into three cases.
\begin{itemize}
\item[•] Generic case: $(\theta_{1}^{0}, m_{1}^{0}, \sigma_{1}^{0}) = (-2, 1, 1), (\theta_{2}^{0}, m_{2}^{0}, \sigma_{2}^{0}) = (4, 2, 2), (\theta_{3}^{0}, m_{3}^{0}, \sigma_{3}^{0}) = (-5, -3, 3)$, $p_{1}^{0}=p_{2}^{0}=p_{3}^{0}=1/3$. 
\item[•] Conformant case: $(\theta_{1}^{0}, m_{1}^{0}, \sigma_{1}^{0}) = (-2, 0, 1), (\theta_{2}^{0}, m_{2}^{0}, \sigma_{2}^{0}) = (4, \sqrt{3}, 2), (\theta_{3}^{0}, m_{3}^{0}, \sigma_{3}^{0}) = (4, \sqrt{8}, 3)$, $p_{1}^{0}=p_{2}^{0}=p_{3}^{0}=1/3$.
\item[•] Non-conformant case: $(\theta_{1}^{0}, m_{1}^{0}, \sigma_{1}^{0}) = (-2, 0, 1), (\theta_{2}^{0}, m_{2}^{0}, \sigma_{2}^{0}) = (4, \sqrt{3}, 2), (\theta_{3}^{0}, m_{3}^{0}, \sigma_{3}^{0}) = (4, -\sqrt{8}, 3)$, $p_{1}^{0}=p_{2}^{0}=p_{3}^{0}=1/3$.
\end{itemize}
As before, 5000 random samples of discrete mixing measures $G \in \Ecal_2$, 
5000 samples of  $G\in \Ocal_3$ and another 5000 for 
$G \in \Ocal_4$, where the support points are uniformly generated in
$\Theta=[-10,10]$ and $\Omega=[0.5,5]$. 

It can be observed that both lower bounds and upper bounds match exactly our theory
developed in the previous two sections.

\begin{figure*}[t]
\centering
\captionsetup{justification=centering}
\begin{minipage}[b]{.25\textwidth}
\includegraphics[width=50mm,height=40mm]{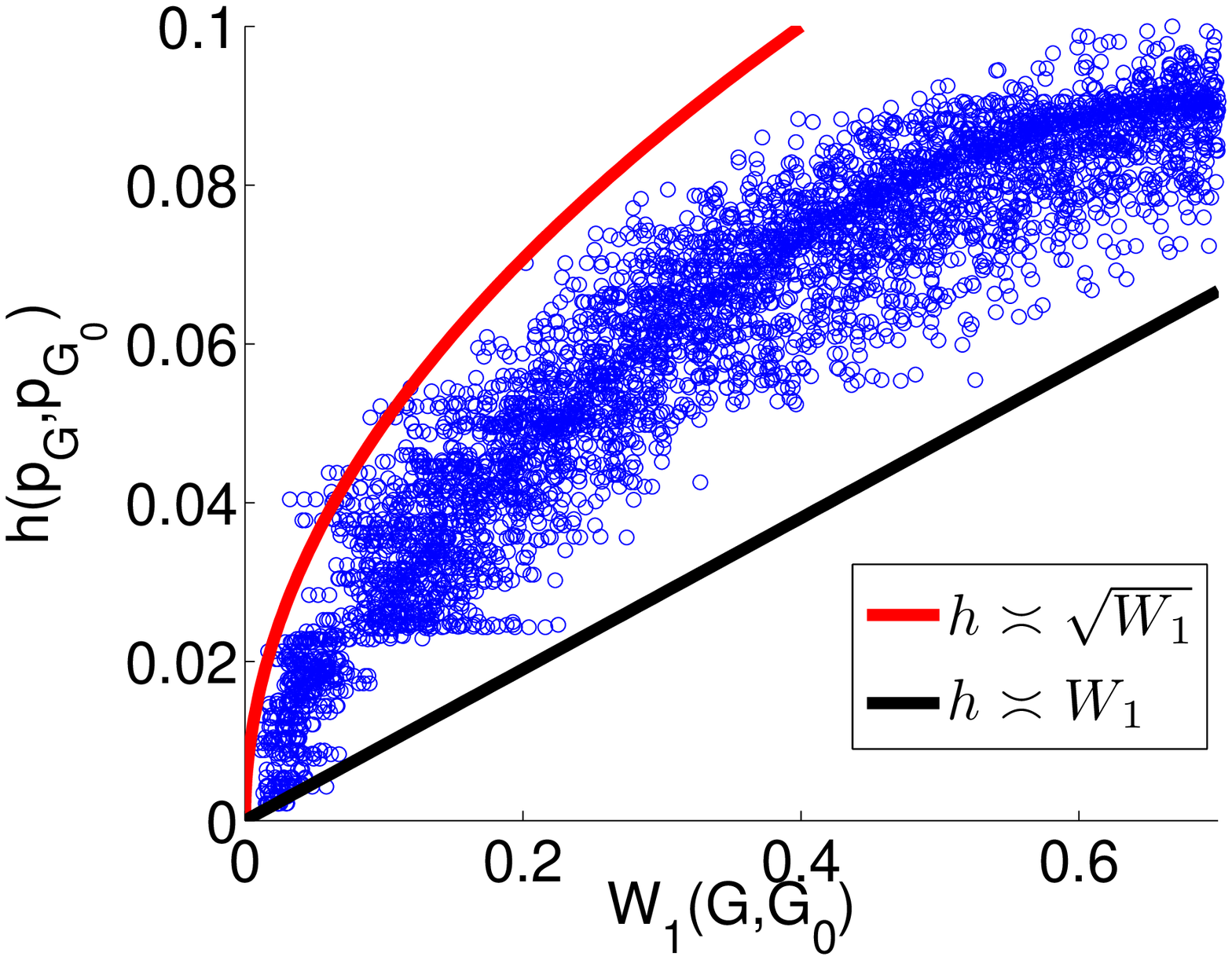}
\end{minipage}
\quad \quad
\begin{minipage}[b]{.25\textwidth}
\includegraphics[width=50mm,height=40mm]{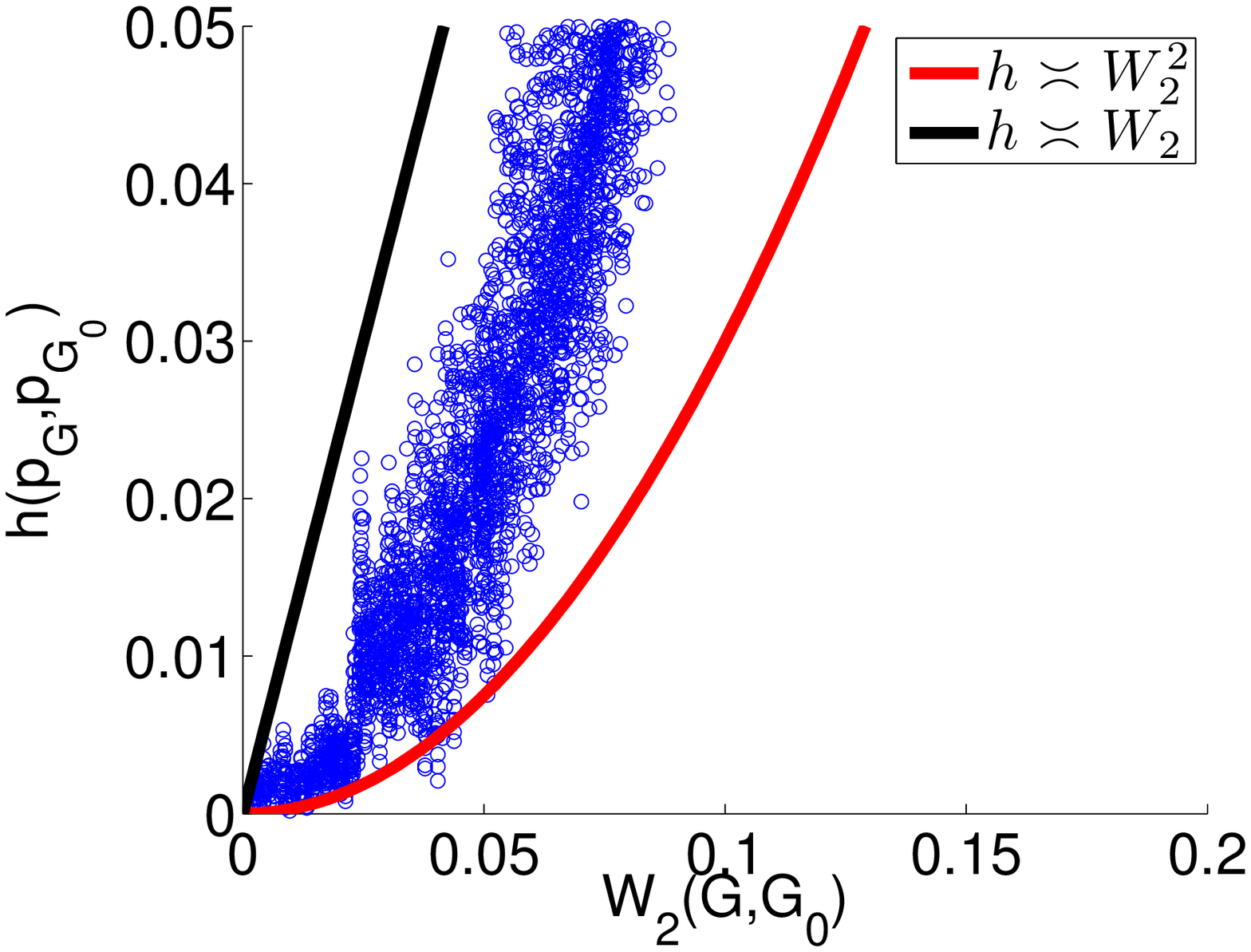} 
\end{minipage}
\caption{\footnotesize{Mixture of Student's t-distributions.
Left: Exact-fitted setting.
Right: Over-fitted setting.
}}
\label{figure-exactfitted-overfitted-Student}
\end{figure*}

\begin{figure*}[h]
\centering
\captionsetup{justification=centering}
\begin{minipage}[b]{.25\textwidth}
\includegraphics[width=50mm,height=50mm]{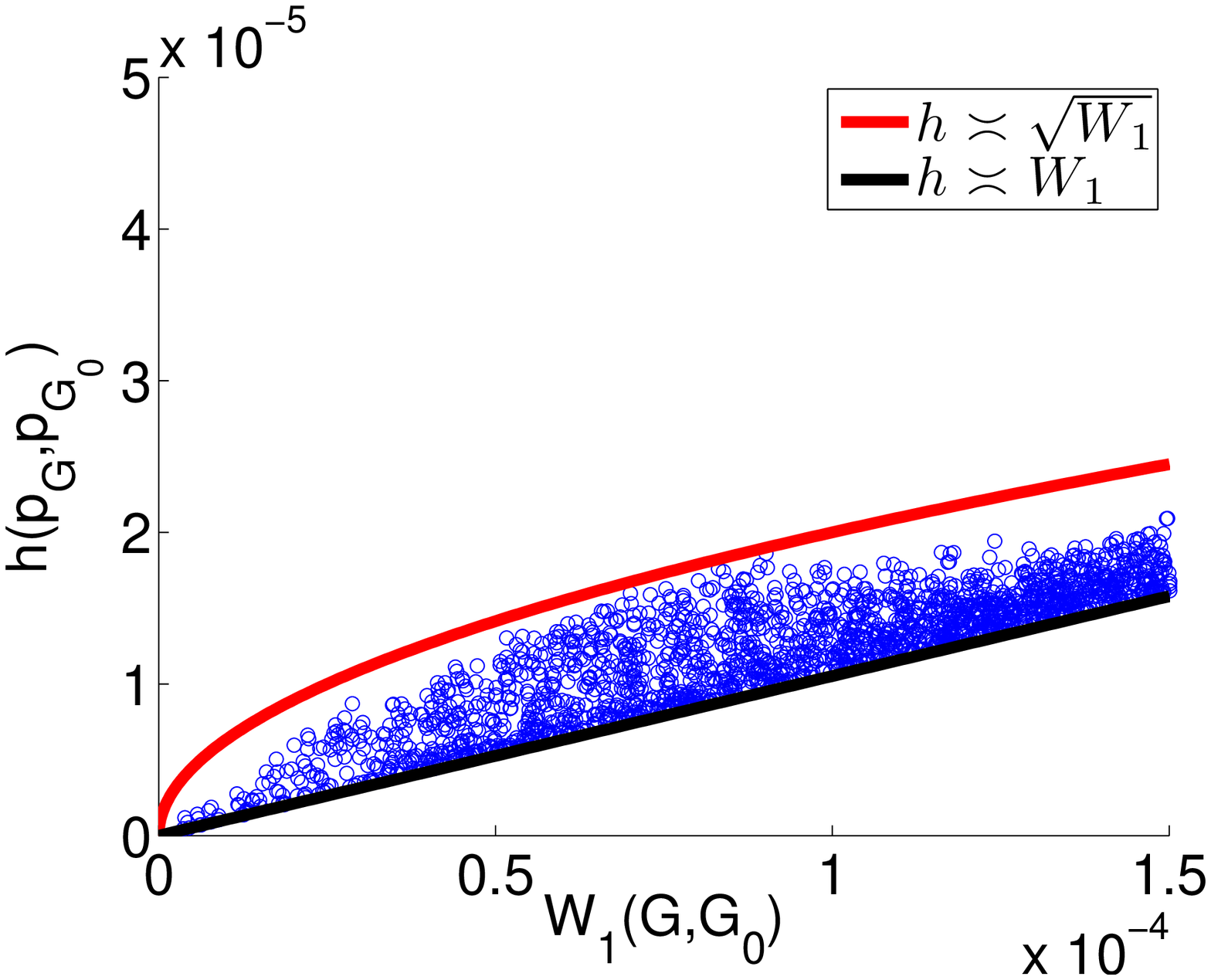}
\end{minipage}
\quad \quad
\begin{minipage}[b]{.25\textwidth}
\includegraphics[width=50mm,height=50mm]{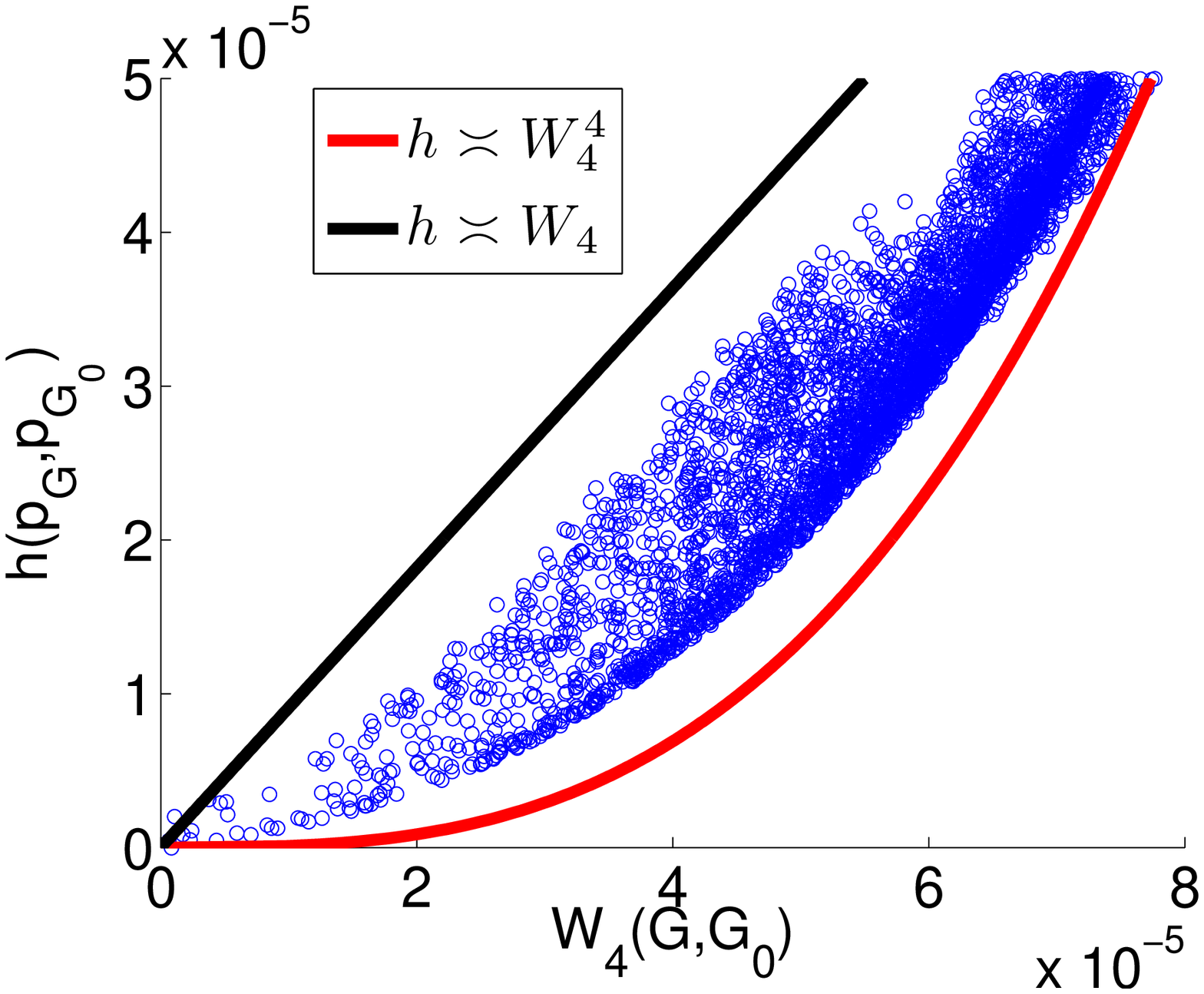} 
\end{minipage}
\quad \quad
\begin{minipage}[b]{.25\textwidth}
\includegraphics[width=50mm,height=50mm]{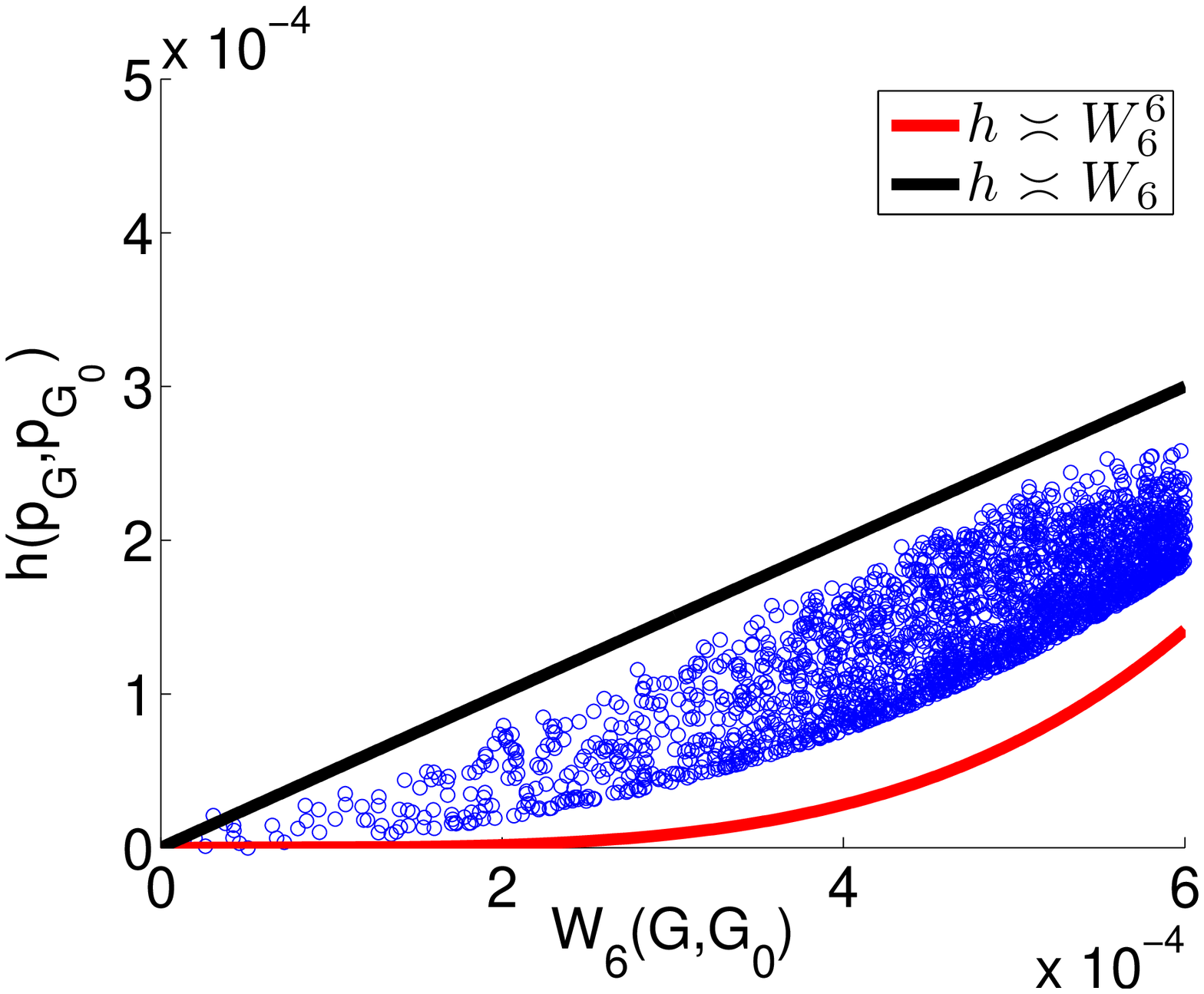} 
\end{minipage}
\caption{\footnotesize{Location-scale Gaussian mixtures. From left to right:
(1) Exact-fitted setting; (2) Over-fitted by one component; (3) Over-fitted by two components.
}}
\label{figure-weak-ident-Gaussian}
\end{figure*}

\begin{figure*}[h]
\centering
\captionsetup{justification=centering}
\begin{minipage}[b]{.2\textwidth}
\includegraphics[width=40mm,height=40mm]{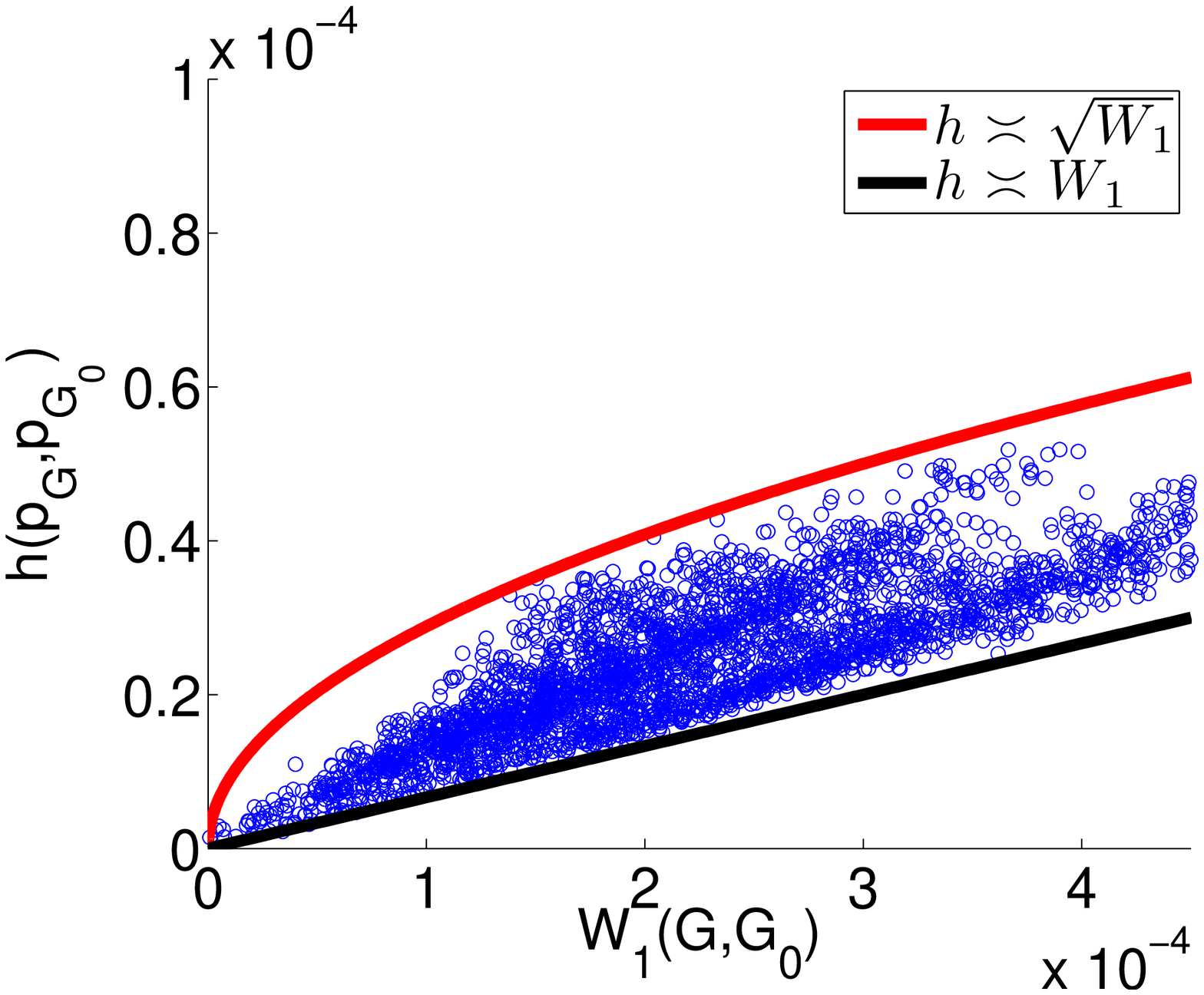}
\end{minipage}
\quad \quad
\begin{minipage}[b]{.2\textwidth}
\includegraphics[width=40mm,height=40mm]{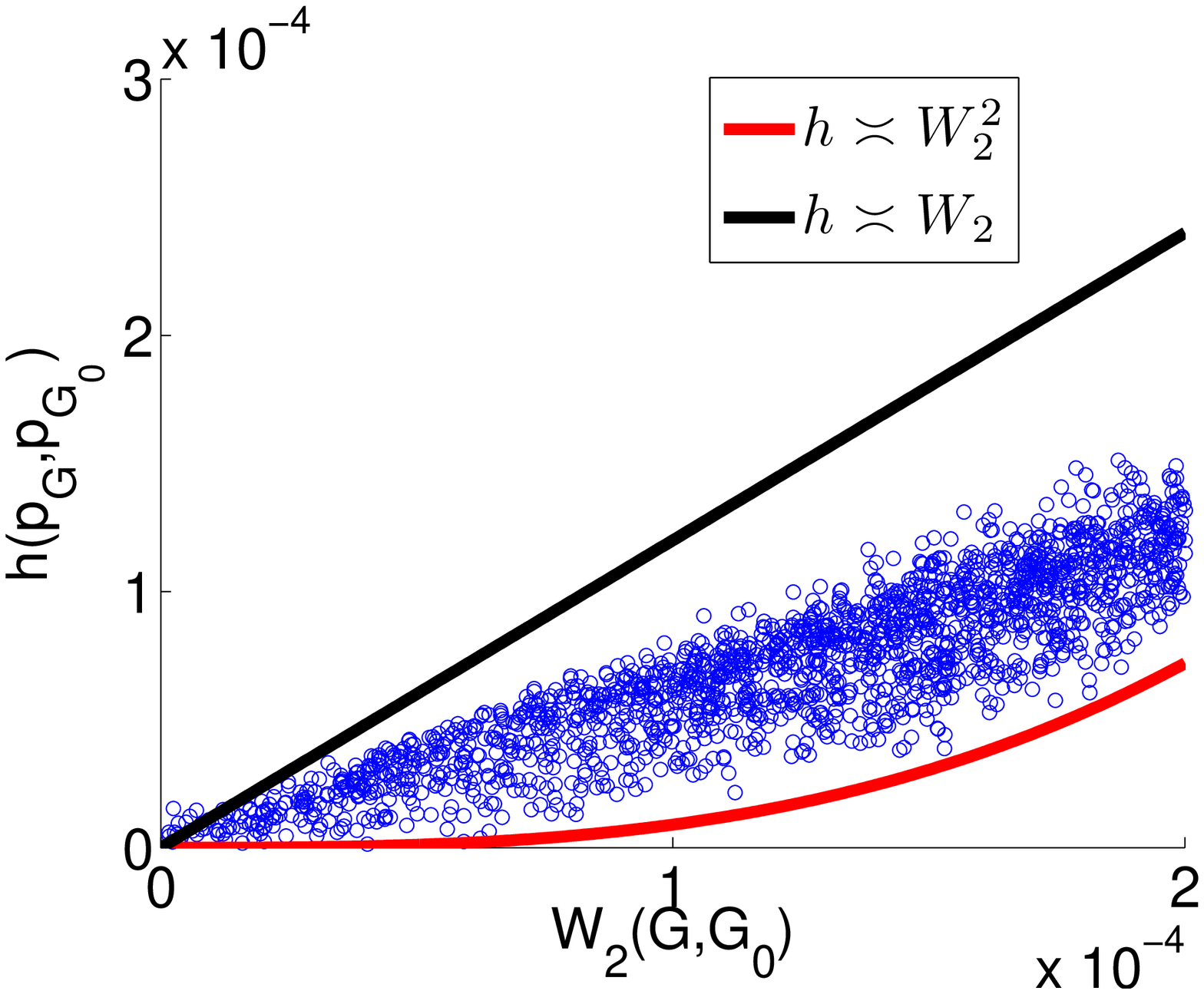} 
\end{minipage}
\quad \quad
\begin{minipage}[b]{.2\textwidth}
\includegraphics[width=40mm,height=40mm]{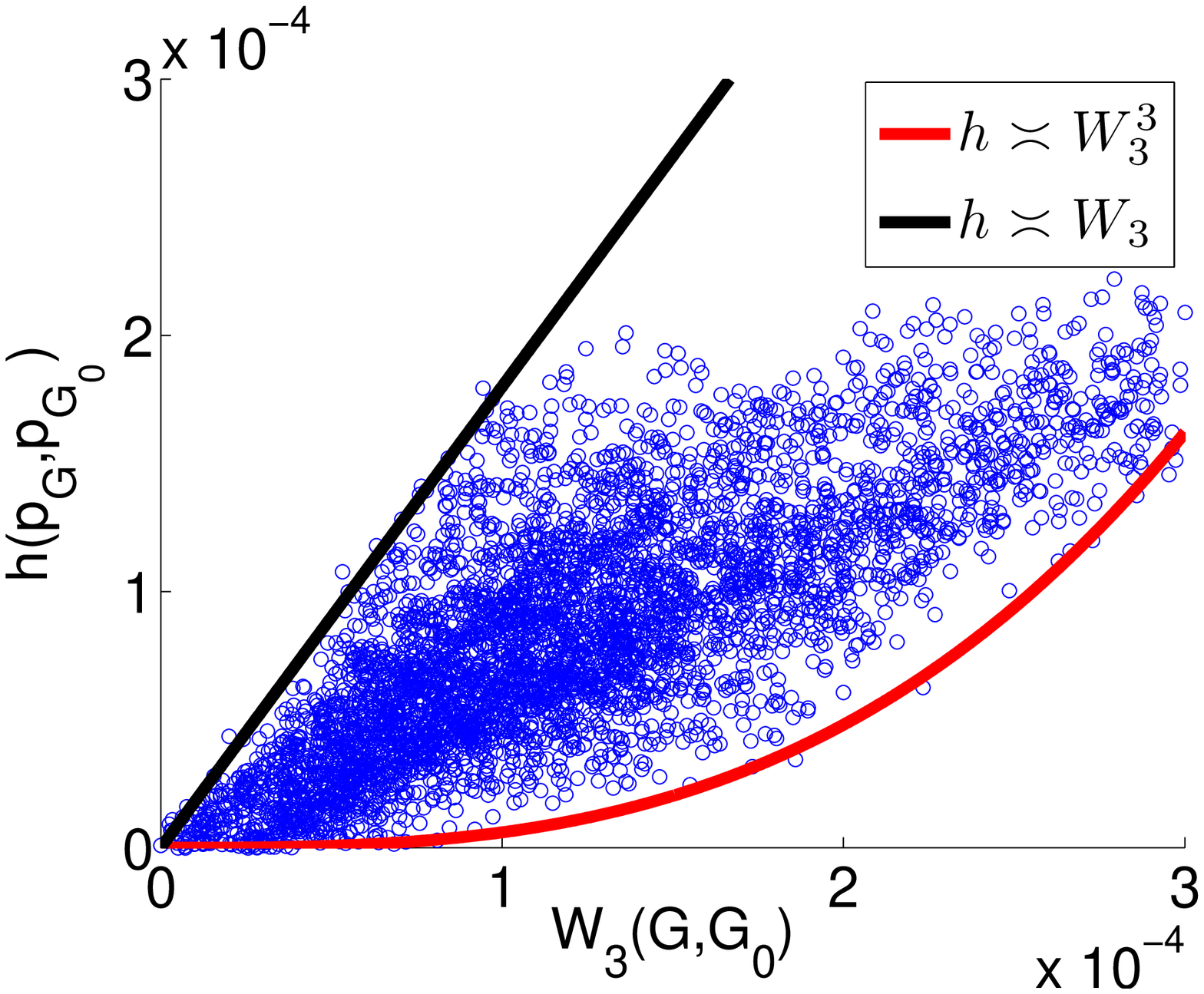} 
\end{minipage}
\quad \quad
\begin{minipage}[b]{.2\textwidth}
\includegraphics[width=40mm,height=40mm]{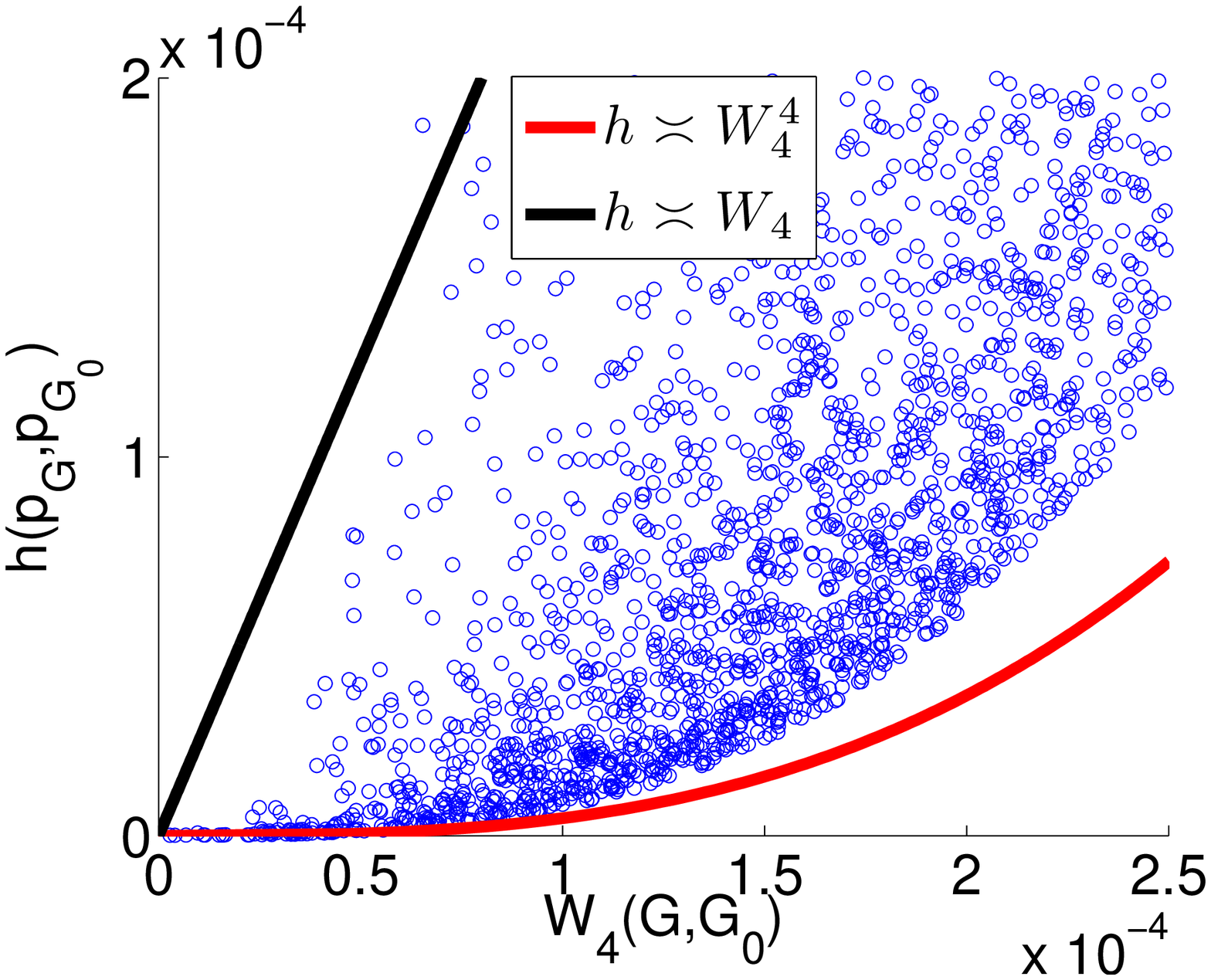} 
\end{minipage}
\caption{\footnotesize{Skew-Gaussian mixtures. From left to right:
(1) Exact-fitted generic case; (2) Exact-fitted conformant case; (3) Exact-fitted
non-conformant case; (4) Over-fitted by one component.}}
\label{figure-weak-ident-skewgauss}
\end{figure*}

\paragraph{Convergence rates of MLE.}
First, we generate $n$-iid samples from a mixture of location-scale multivariate Gaussian 
distributions which has exactly three components.
The true parameters for the mixing
measure $G_0$ are:
$\theta_{1}^{0}=(0,3), \theta_{2}^{0}=(1,-4), \theta_{3}^{0}=(5,2)$, $\Sigma_{1}^{0}=\begin{pmatrix} 4.2824 & 1.7324 \\ 1.7324 & 0.81759 \end{pmatrix}$, $\Sigma_{2}^{0}=\begin{pmatrix} 1.75 & -1.25 \\ -1.25 & 1.75 \end{pmatrix}$, $ \Sigma_{3}^{0}=\begin{pmatrix} 1 & 0 \\ 0 & 4 \end{pmatrix}$, and $\pi_{1}^{0}=0.3,\pi_{2}^{0}=0.4,\pi_{3}^{0}=0.3$. Maximum likelihood estimators are obtained by the EM algorithm as we assume 
that the data come from a mixture of $k$ Gaussians where $k \geq k_{0}=3$. 
See Figure \ref{figure-exactfitted-overfitted-Gaussian},
where the Wasserstein distance metrics are plotted against
varying sample size $n$. The error bards are obtained by running
the experiment 7 times for each $n$.

%
These simulations are in complete agreement with the established convergence theory 
and confirm that the convergence slows down rapidly as $k-k_0$ increases.

\begin{figure*}
\centering
\captionsetup{justification=centering}
\begin{minipage}[b]{.2\textwidth}
\includegraphics[width=40mm,height=40mm]{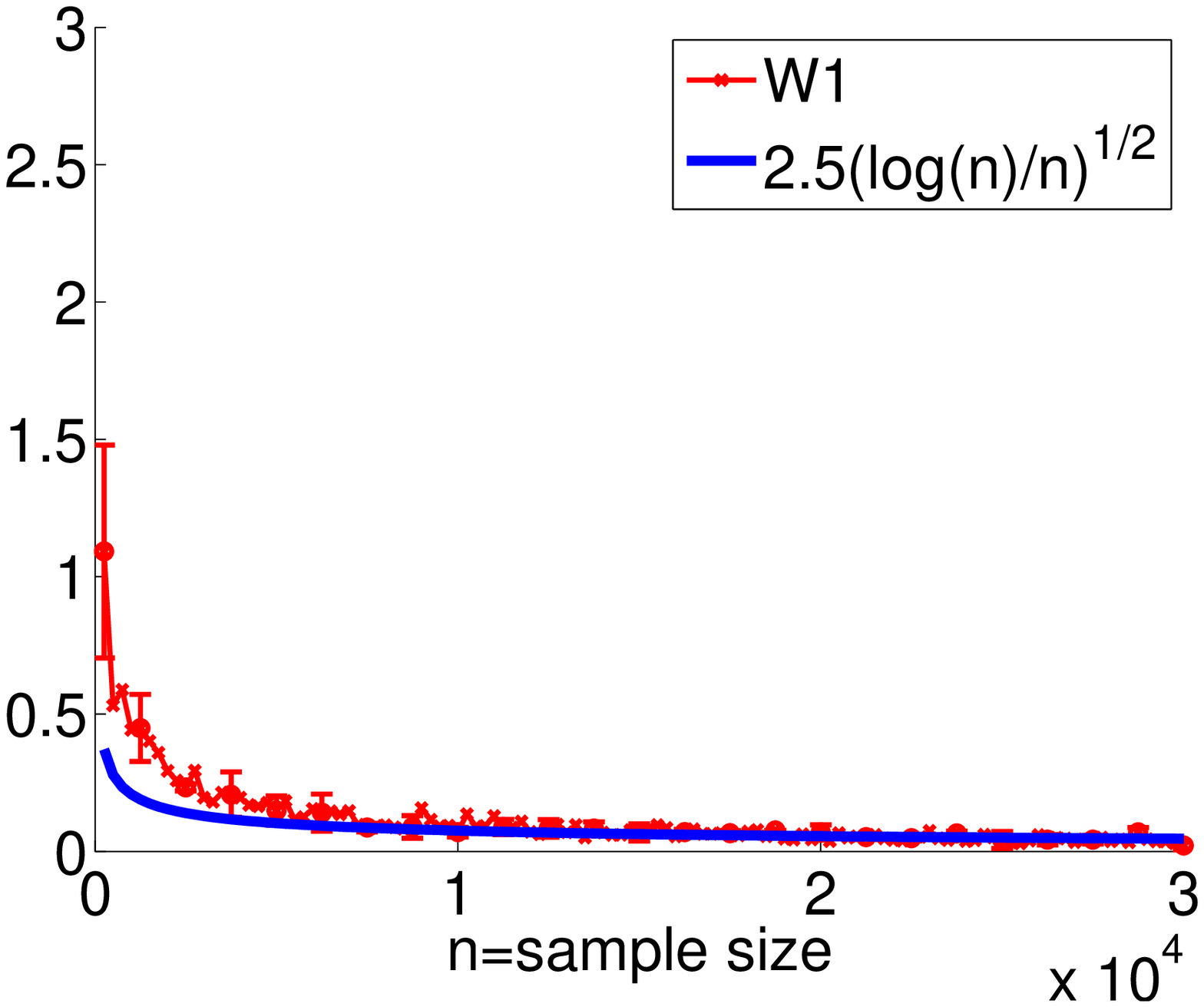}
\end{minipage}
\quad \quad
\begin{minipage}[b]{.2\textwidth}
\includegraphics[width=40mm,height=40mm]{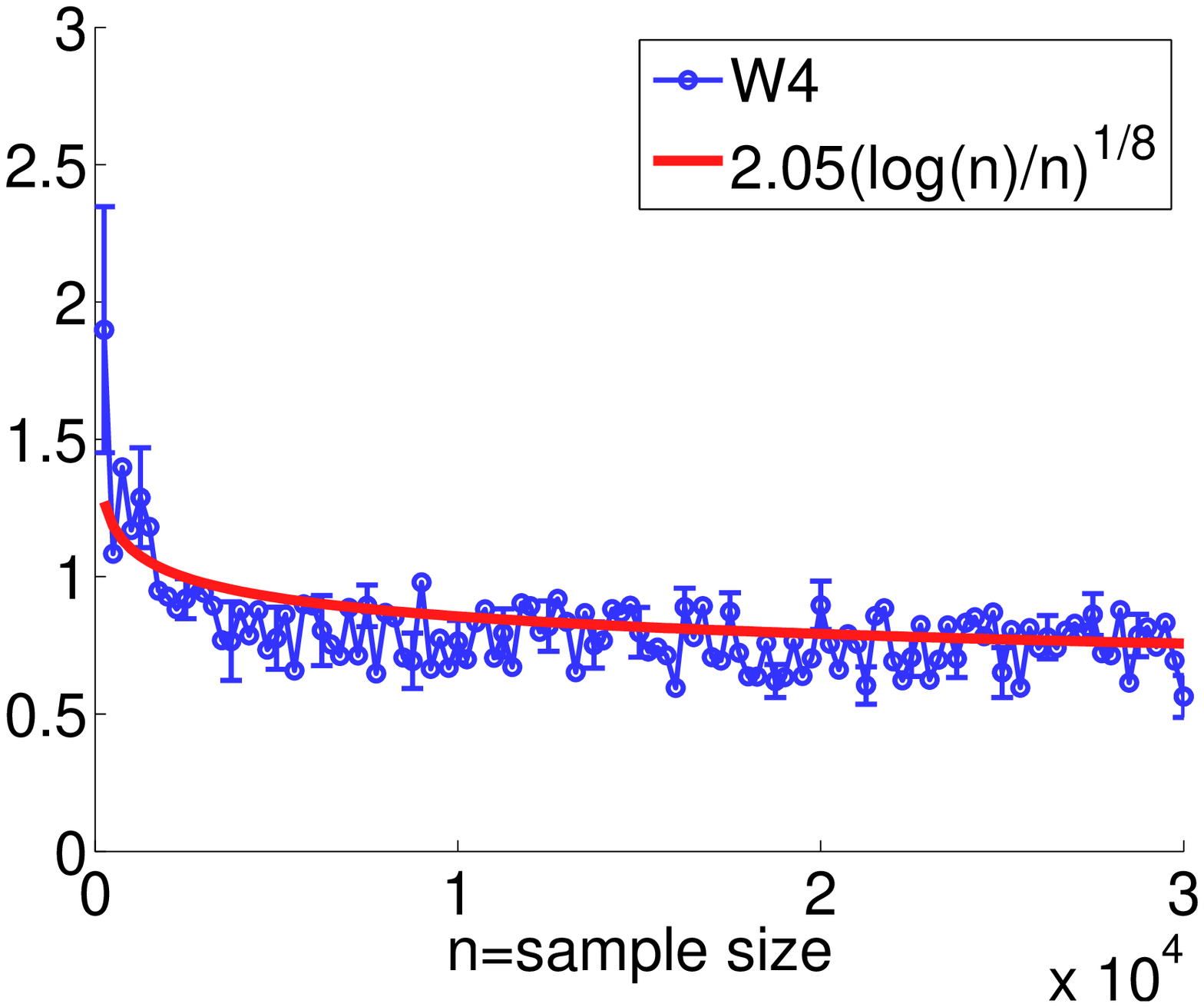} 
\end{minipage} 
\quad \quad
\begin{minipage}[b]{.2\textwidth}
\includegraphics[width=40mm,height=40mm]{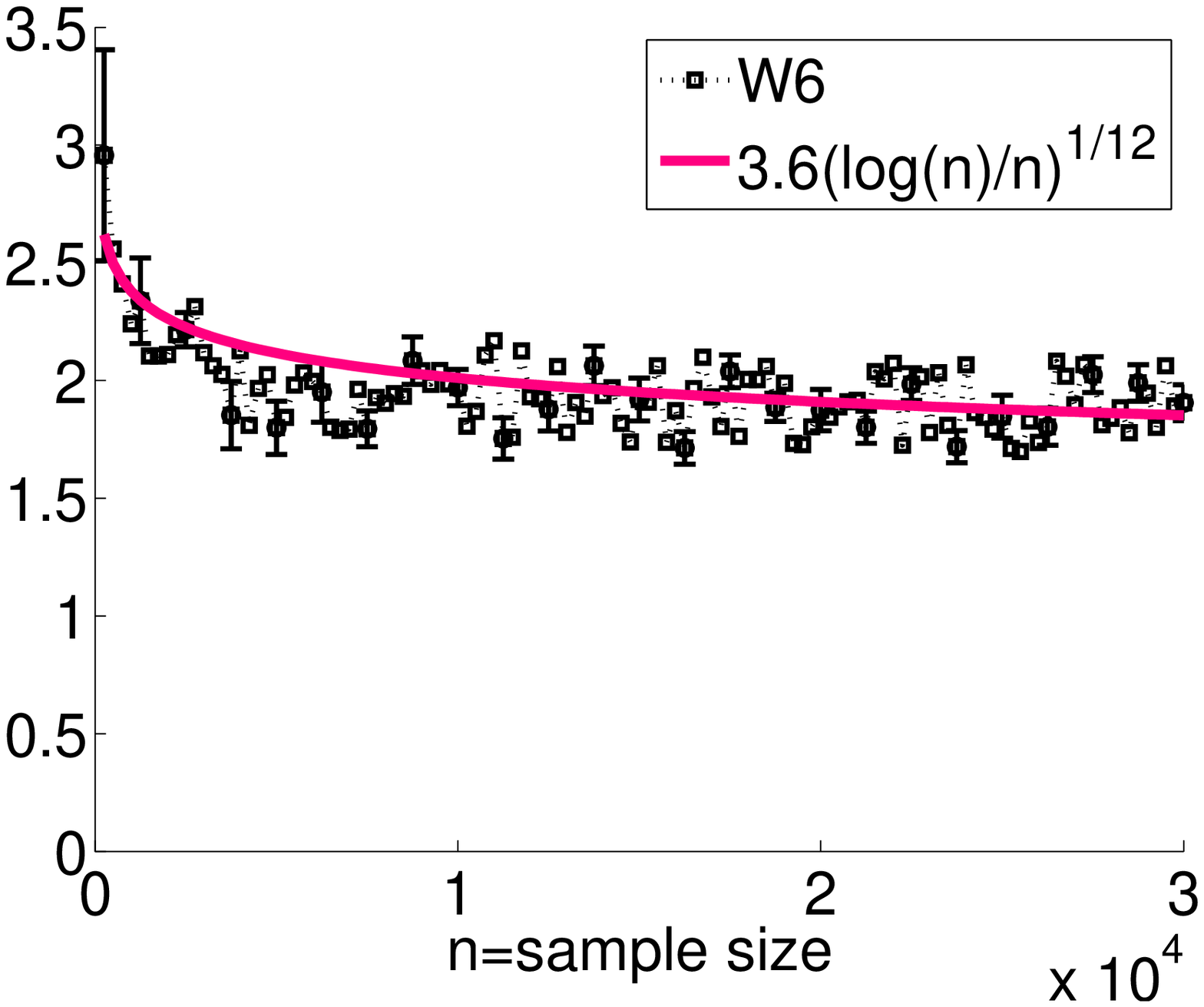}
\end{minipage}
\quad \quad
\begin{minipage}[b]{.2\textwidth}
\includegraphics[width=40mm,height=40mm]{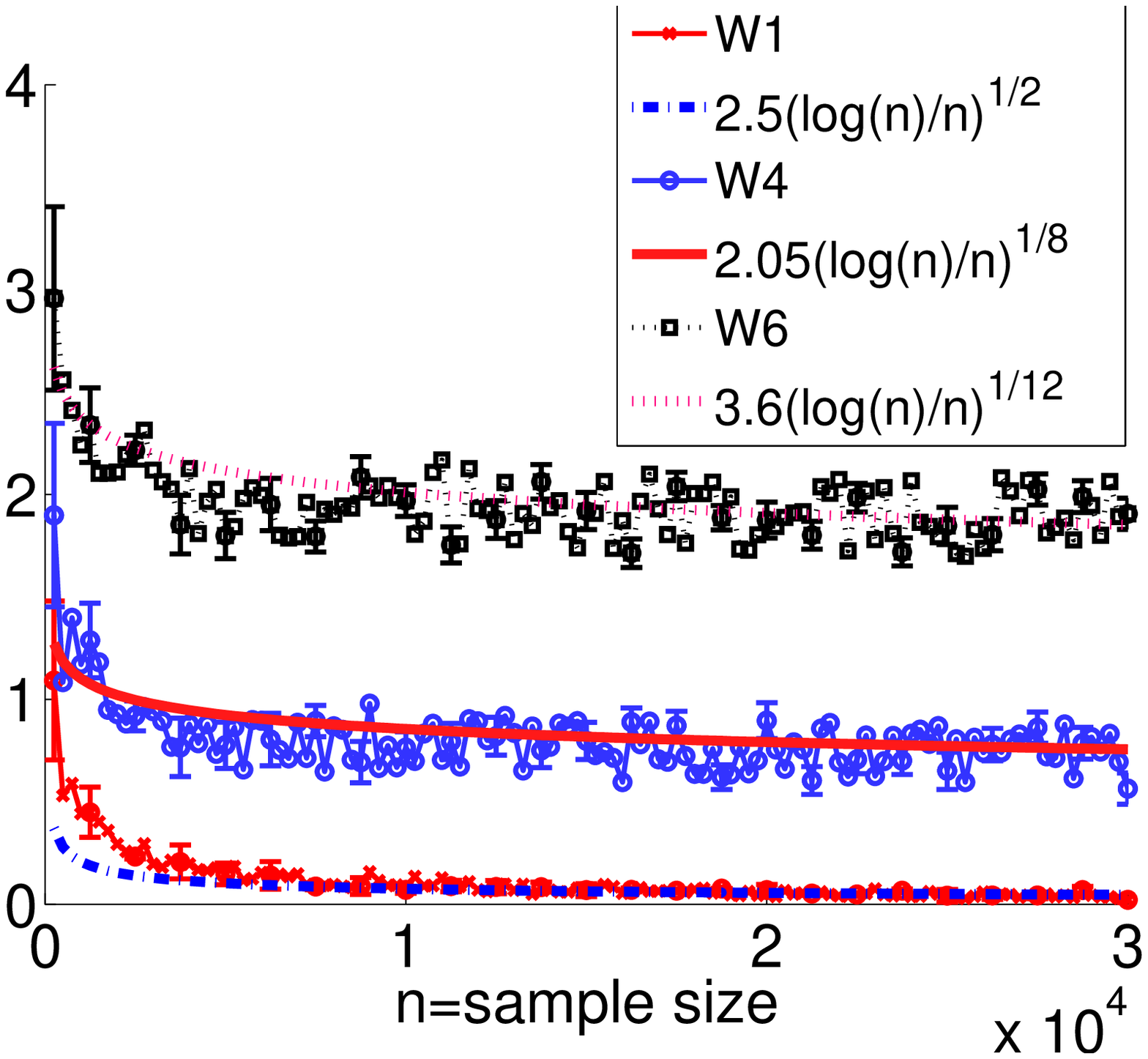}
\end{minipage}
\caption{\footnotesize{
MLE rates for location-scale mixtures of Gaussians. L to R:
(1) Exact-fitted: $W_1 \sim n^{-1/2}$.
(2) Over-fitted by one: $W_4 \sim n^{-1/8}$.
(3) Over-fitted by two: $W_6 \sim n^{-1/12}$.
}}
\label{figure-exactfitted-overfitted-Gaussian}
\end{figure*}

We turn to mixtures of Gamma distributions. There are two cases
\begin{itemize}
\item[•] \textit{Generic case:} We generate $n$-iid samples from Gamma mixture model that has exactly two mixing components. The true parameters for the mixing measure $G_{0}$ are: $a_{1}^{0}=8$, $a_{2}^{0}=2$, $b_{1}^{0}=3$, $b_{2}^{0}=4$, $\pi_{1}^{0}=1/3$, $\pi_{2}^{0}=2/3$. 
\item[•] \textit{Pathological case:} We carry out the same procedure as that of generic case with the only difference is about the true parameters of $G_{0}$. In fact, we choose $a_{1}^{0}=8$, $a_{2}^{0}=7$, $b_{1}^{0}=3$, $b_{2}^{0}=3$, $\pi_{1}^{0}=1/3$, $\pi_{2}^{0}=2/3$. 

\end{itemize}

\begin{figure*}
\centering
\captionsetup{justification=centering}
\begin{minipage}[b]{.20\textwidth}
\includegraphics[width=40mm,height=40mm]{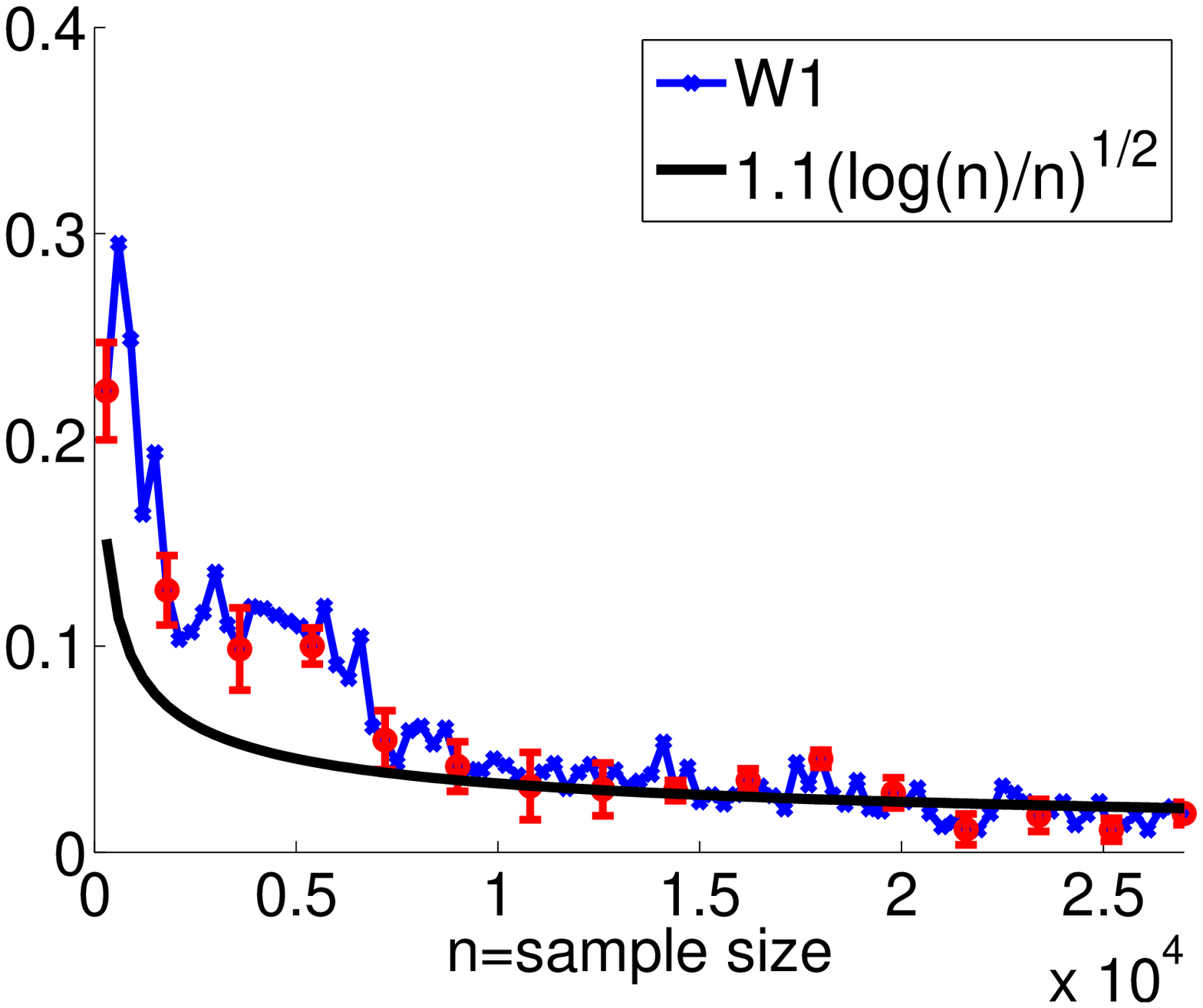}
\end{minipage}
\quad \quad
\begin{minipage}[b]{.20\textwidth}
\includegraphics[width=40mm,height=40mm]{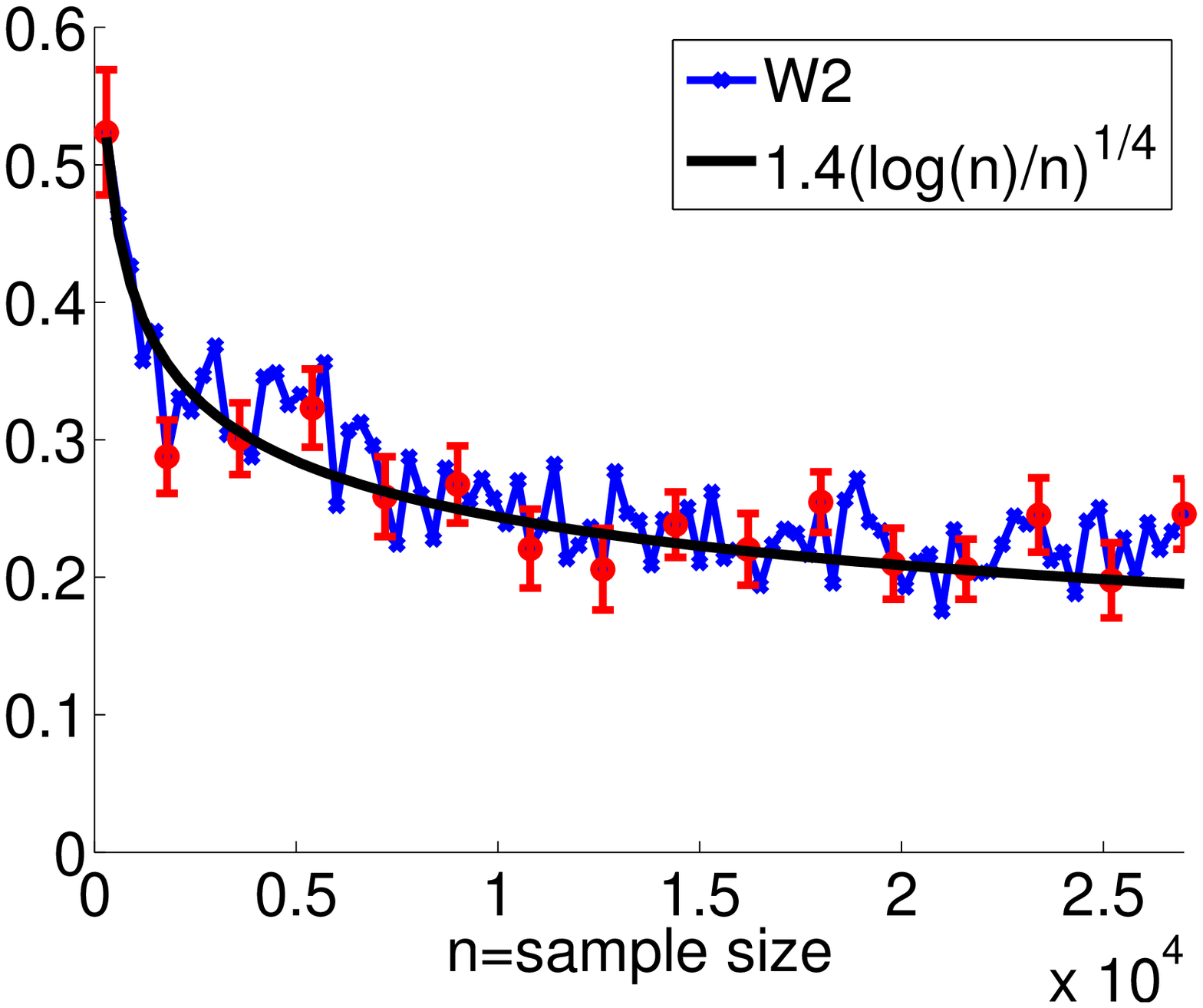} 
\end{minipage} 
\quad \quad 
\begin{minipage}[b]{.20\textwidth}
\includegraphics[width=40mm,height=40mm]{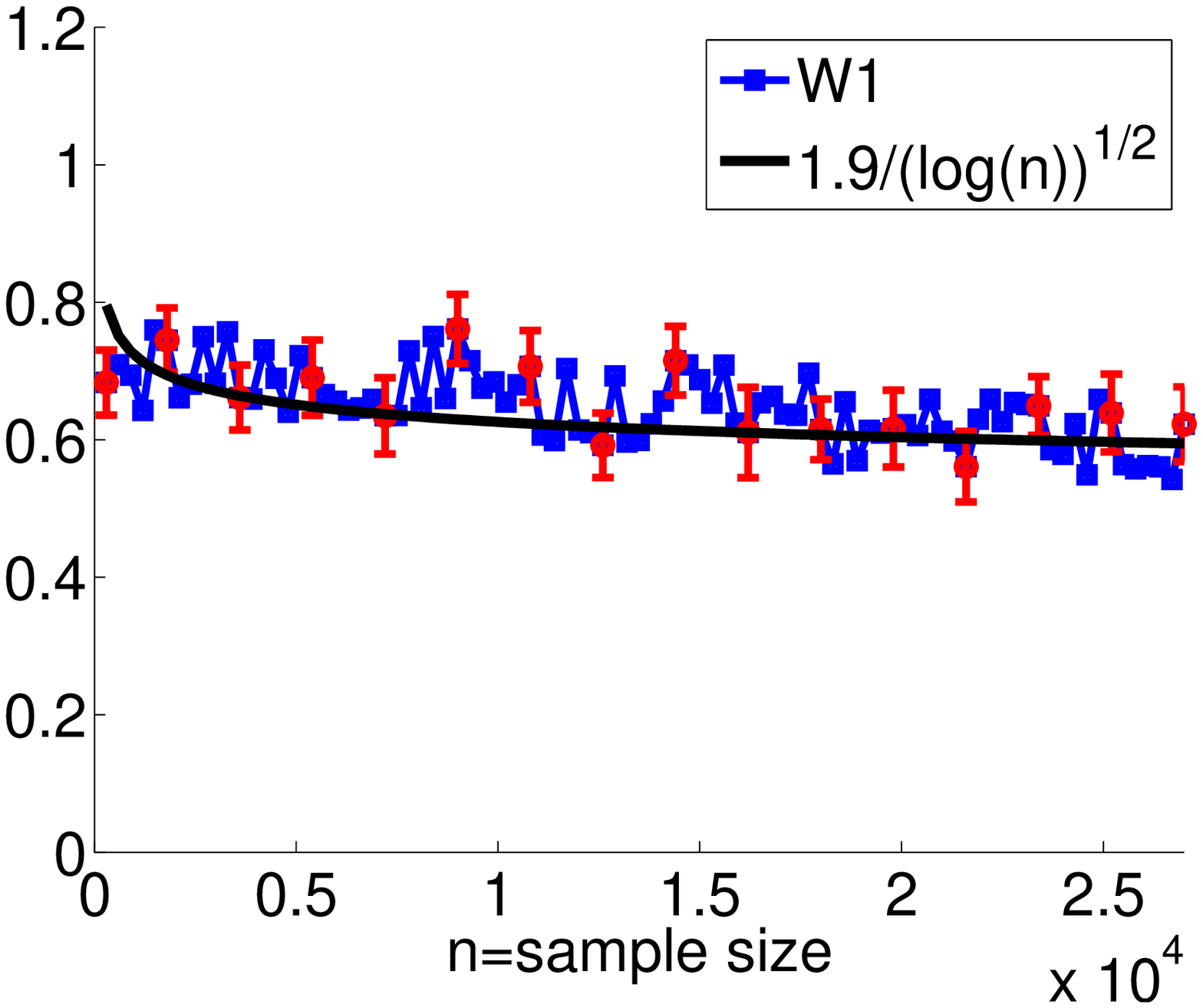}
\end{minipage}
\quad \quad
\begin{minipage}[b]{.20\textwidth}
\includegraphics[width=40mm,height=40mm]{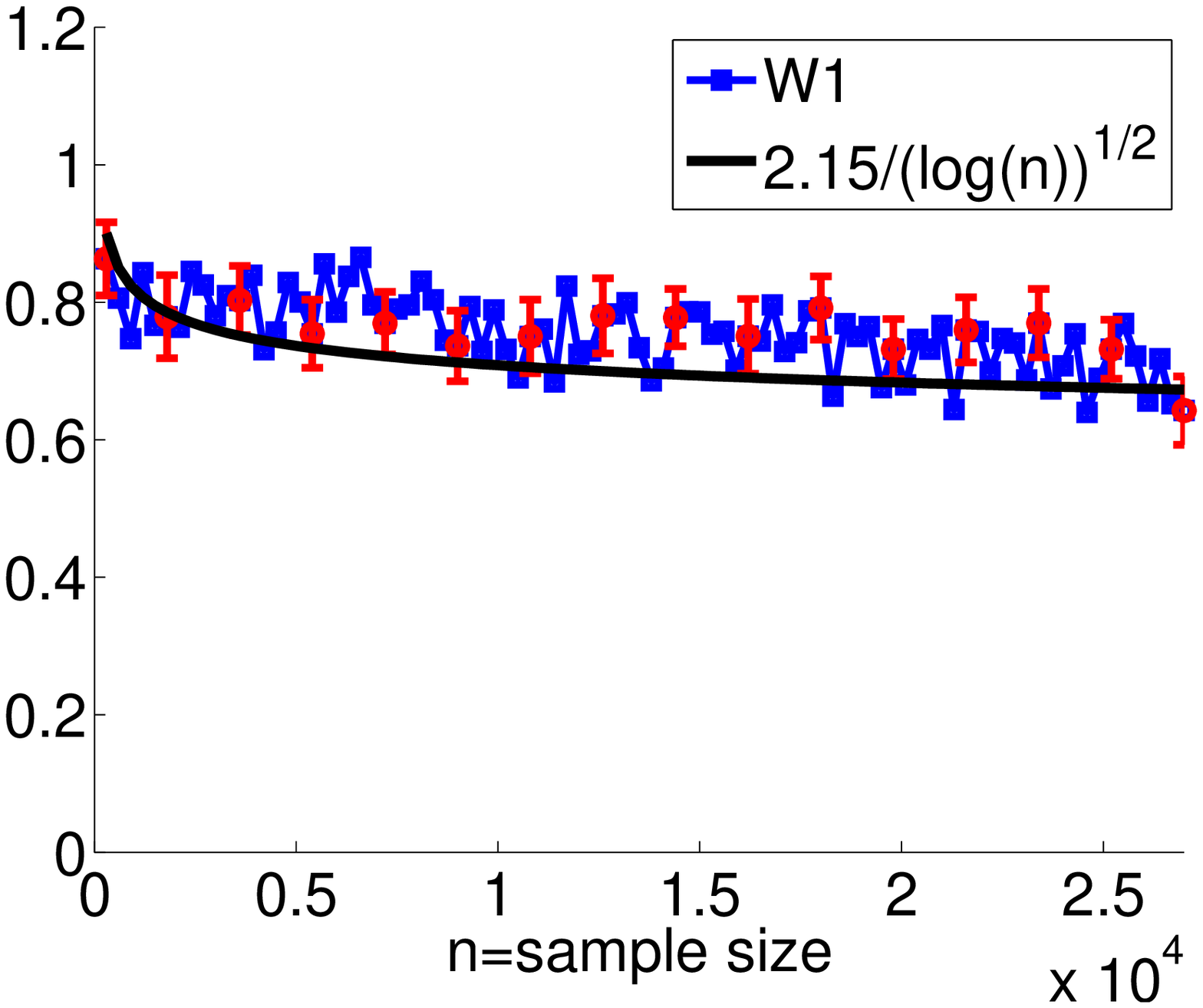}
\end{minipage}
\caption{\footnotesize{MLE rates for shape-rate mixtures of Gamma distributions.
L to R:
(1) Generic/Exact-fitted: $W_1(\widehat{G}_n,G_0) \sim n^{-1/2}$.
(2) Generic/Over-fitted: $W_2 \sim n^{-1/4}$.
(3) Pathological/Exact-fitted: $W_1 \approx 1/(\log(n)^{1/2}$.
(4) Pathological/Over-fitted:  $W_1 \approx 1/(\log(n)^{1/2}$.}}
\label{figure-exactfitted-overfitted-Gamma}
\end{figure*}

It is remarkable to see the wild swing in behaviors within this same class. 
See Figure~\ref{figure-exactfitted-overfitted-Gamma}. Even for
exact-fitted finite mixtures of Gamma, one can achieve very fast convergence rate of $n^{-1/2}$
in the generic case, or sink into a logarithmic rate if the true mixing measure $G_0$
takes on one of the pathological values.

\section{Proofs of representative theorems}
\label{Sec-proofs}
There are two types of theorems proved in this paper.
The first type are sharp inequalities of the form
$V(p_G,p_{G_0}) \gtrsim W_r^r(G,G_0)$ for some precise order $r > 0$ depending on 
the specific setting of the mixture models.
The second type of results are characterization theorems presented in 
Section~\ref{Sec:Characterization}. 

In this section we present
the proofs for three representative theorems: Theorem~\ref{theorem-firstorder} 
for strongly identifiable mixtures in the exact-fitted setting, Theorem ~\ref{theorem-secondorder} for strongly identifiable mixtures in the over-fitted setting,
and Theorem~\ref{theorem:generaloverfittedGaussian}
for over-fitted Gaussian mixtures (i.e., a weakly identifiable class) as well as Proposition \ref{proposition-specificvalueoverliner}
These proofs carry important insights underlying the theory
--- they are organized in a sequence of steps to help the reader.
For other density classes (e.g., second order identifiable,
Gamma and skew-Gaussian classes) the proofs are similar in spirit to these two, 
but they are of interest in their own right due to special 
and rich structures of each density class. 
Due to space constraints the proofs for these and all other theorems
are deferred to the Appendix.

\subsection{Strong identifiability in exact-fitted mixtures}
\paragraph{PROOF OF THEOREM~\ref{theorem-firstorder}}

It suffices to show that
\begin{eqnarray}
\mathop {\lim }\limits_{\epsilon \to 0}{\mathop {\inf }
{\biggr \{ V(p_{G},p_{G_0}/W_1(G,G_0) | W_1(G,G_0) \leq \epsilon 
\biggr \}}}>0, \label{eq:theorem1}
\end{eqnarray}
where the infimum is taken over all $G \in \mathcal{E}_{k_{0}}(\Theta \times \Omega)$.

\paragraph{Step 1.} Suppose that~\eqref{eq:theorem1} does not hold, which implies that we have sequence
of $G_{n}=\mathop {\sum }\limits_{i=1}^{k_{0}}{p_{i}^{n}\delta_{(\theta_{i}^{n},\Sigma_{i}^{n})}} \in \mathcal{E}_{k_{0}}(\Theta \times \Omega)$ converging to $G_{0}$ in $W_{1}$ distance  
such that $V(p_{G_n},p_{G_0})/W_1(G_{n},G_{0}) \to 0$ as $n \to \infty$. 
As $W_{1}(G_{n},G_{0}) \to 0$, the support points of $G_n$ must converge to that
of $G_0$. By permutation of the labels $i$, 
it suffices to assume that for each $i=1,\ldots, k_0$, 
$(\theta_i^n,\Sigma_i^n) \rightarrow (\theta_i^0,\Sigma_i^0)$. 
For each pair $(G_n,G_0)$, 
let $\{q_{ij}^n \}$ denote the corresponding probabilities of the
optimal coupling for $(G_n,G_0)$ pair, so we can write:
\[W_{1}(G_{n},G_{0})=
\sum_{1 \leq i,j \leq k_{0}} 
q_{ij}^{n}(\|\theta_{i}^{n}-\theta_{j}^{0}\|+\|\Sigma_{i}^{n}-\Sigma_{j}^{0}\|).\] 

Since $G_n$ and $G_0$ have the same number of support points, it is an easy observation
that for sufficiently large $n$, $q_{ii}^{n} = \min (p_i^n,p_i^0)$. And so,
$\sum_{i\neq j} q_{ij}^n = \sum_{i=1}^{k_0} |p_i^n-p_i^0|$. 
Adopting the notations that
$\Delta \theta_{i}^{n} := \theta_{i}^{n}-\theta_{i}^{0}$,
$\Delta \Sigma_{i}^{n} := \Sigma_{i}^{n}-\Sigma_{i}^{0}$, and
$\Delta p_{i}^{n} := p_{i}^{n}-p_{i}^{0}$ for all $1 \leq i \leq k_{0}$, we have
\begin{eqnarray}
W_{1}(G_{n},G_{0}) & = & \mathop {\sum }\limits_{i=1}^{k_{0}}{q_{ii}^{n}
(\|\theta_{i}^{n}-\theta_{i}^{0}\|+\|\Sigma_{i}^{n}-\Sigma_{i}^{0}\|)}+\mathop {\sum }\limits_{i \neq j}{q_{ij}^{n}(\|\theta_{i}^{n}-\theta_{j}^{0}\|+\|\Sigma_{i}^{n}-\Sigma_{j}^{0}\|)}\nonumber \\
& \lesssim & 
\mathop {\sum }\limits_{i=1}^{k_{0}}{p_{i}^{n}(\|\Delta \theta_{i}^{n}\|+\|\Delta \Sigma_{i}^{n}\|)+|\Delta p_{i}^{n}|} =: d(G_n,G_0). \nonumber
\end{eqnarray}
The inequality in the above display is due to
$q_{ii}^{n} \leq p_{i}^{n}$, and the observation that
$\|\theta_{i}^{n}-\theta_{j}^{0}\|, \|\Sigma_{i}^{n}-\Sigma_{i}^{0}\|$ are 
bounded for all $1 \leq i,j \leq k_{0}$ for sufficiently large $n$. 
Thus, we have $V(p_{G_n},p_{G_0})/d(G_{n},G_{0}) \to 0$. 

\paragraph{Step 2.} Now, consider the following important identity:
\begin{eqnarray}
p_{G_{n}}(x) - p_{G_{0}}(x)
= \mathop {\sum }\limits_{i=1}^{k_{0}}{\Delta p_{i}^{n}f(x|\theta_{i}^{0},\Sigma_{i}^{0})}+ \mathop {\sum }\limits_{i=1}^{k_{0}}{p_{i}^{n}(f(x|\theta_{i}^{n},\Sigma_{i}^{n})-f(x|\theta_{i}^{0},\Sigma_{i}^{0}))}. \nonumber
\end{eqnarray}
For each $x$, applying Taylor expansion to function $f$ to the first order to obtain
\begin{eqnarray}
\mathop {\sum }\limits_{i=1}^{k_{0}}{p_{i}^{n}(f(x|\theta_{i}^{n},\Sigma_{i}^{n})-f(x|\theta_{i}^{0},\Sigma_{i}^{0})} = \mathop {\sum }\limits_{i=1}^{k_{0}}{p_{i}^{n}\left[(\Delta \theta_{i}^{n})^{T}\dfrac{\partial{f}}{\partial{\theta}}(x|\theta_{i}^{0},\Sigma_{i}^{0})+\trace\left(\dfrac{\partial{f}}{\partial{\Sigma}}(x|\theta_{i}^{0},\Sigma_{i}^{0})^{T}\Delta \Sigma_{i}^{n}\right)\right]} \nonumber \\
+ R_{n}(x), \nonumber
\end{eqnarray}
where $R_{n}(x)= O\left(\mathop {\sum }\limits_{i=1}^{k_{0}}{p_{i}^{n}(\|\Delta \theta_{i}^{n}\|^{1+\delta_{1}}+\|\Delta \Sigma_{i}^{n}\|^{1+\delta_{2}}})\right)$, where the appearance of $\delta_1$
and $\delta_2$ are due the assumed Lipschitz conditions, and the big-O constant does not depend on
$x$. It is clear that
$\sup_{x} |R_{n}(x)/d(G_{n},G_{0}| \to 0$ as $n \to \infty$.

Denote $A_{n}(x)= \mathop {\sum }\limits_{i=1}^{k_{0}}{p_{i}^{n}\left[(\Delta \theta_{i}^{n})^{T}\dfrac{\partial{f}}{\partial{\theta}}(x|\theta_{i}^{0},\Sigma_{i}^{0})+\trace\left(\dfrac{\partial{f}}{\partial{\Sigma}}(x|\theta_{i}^{0},\Sigma_{i}^{0})^{T}\Delta \Sigma_{i}^{n}\right)\right]}$ 
and \\
$B_{n}(x) = \mathop {\sum }\limits_{i=1}^{k}{\Delta p_{i}^{n}f(x|\theta_{i}^{0},\Sigma_{i}^{0})}$. Then, we can rewrite
\begin{eqnarray}
(p_{G_{n}}(x) - p_{G_{0}}(x))/d(G_{n},G_{0}) = (A_{n}(x)+ B_{n}(x) + R_{n}(x))/d(G_{n},G_{0}). \nonumber
\end{eqnarray}

\paragraph{Step 3.}
We see that $A_{n}(x)/d(G_{n},G_{0})$ and $B_{n}(x)/d(G_{n},G_{0})$ are the linear combination of the
scalar elements of $f(x|\theta,\Sigma)$, $\dfrac{\partial{f}}{\partial{\theta}}(x|\theta,\Sigma)$
and $\dfrac{\partial{f}}{\partial{\Sigma}}(x|\theta,\Sigma)$ such that the coefficients do not
depend on $x$. 
We shall argue that \emph{not} all such coefficients in the linear combination converge to 0 
as $n \to \infty$. Indeed, if the opposite is true, then the summation of the 
absolute values of these coefficients must also tend to 0:
\[\biggr \{
\sum_{i=1}^{k_0}|\Delta p_i^n | + p_i^n (\|\Delta \theta_i^n\|_1 + \|\Delta \Sigma_i^n\|_1)
\biggr \}/d(G_n,G) \rightarrow 0.\]
Since the entrywise $\ell_1$ and $\ell_2$ norms are equivalent, the above entails
$\biggr \{
\sum_{i=1}^{k_0}|\Delta p_i^n | + p_i^n (\|\Delta \theta_i^n\| + \|\Delta \Sigma_i^n\|)
\biggr \}/d(G_n,G_0) \rightarrow 0$,
which contradicts with the definition of $d(G_n,G_0)$.
As a consequence, we can find at least one coefficient of the elements of
$A_n(x)/d(G_n,G_0)$ or $B_n(x)/d(G_n,G_0)$ that does not vanish 
as $n \to \infty$. 

\paragraph{Step 4.} 
Let $m_{n}$ be the maximum of the absolute value of the scalar coefficients of 
$A_{n}(x)/d(G_{n},G_{0})$, $B_{n}(x)/d(G_{n},G_{0})$ and $d_{n}=1/m_{n}$, then 
$d_{n}$ is uniformly bounded from above for all $n$. 
Thus, as $n \to \infty$, $d_{n}A_{n}(x)/d(G_{n},G_{0}) \to \mathop {\sum }\limits_{i=1}^{k_{0}}{\beta_{i}^{T}\dfrac{\partial{f}}{\partial{\theta}}(x|\theta_{i}^{0},\Sigma_{i}^{0})}+\trace\left(\dfrac{\partial{f}}{\partial{\Sigma}}(x|\theta_{i}^{0},\Sigma_{i}^{0})^{T}\gamma_{i}\right)$ and $d_{n}B_{n}(x)/d(G_{n},G_{0}) \to \mathop {\sum }\limits_{i=1}^{k_{0}}{\alpha_{i}f(x|\theta_{i}^{0},\Sigma_{i}^{0})}$, 
such that \emph{not} all scalar elements of $\alpha_i,\beta_i$ and $\gamma_i$ vanish. Moreover,
$\gamma_{i}$ are symmetric matrices because
$\Sigma_{i}^{n}$ are symmetric matrices for all $n,i$. 
Note that
\begin{eqnarray*}
d_n V(p_{G_n},p_{G_0})/d(G_n,G_0) & = &
\int d_n |p_{G_n}(x)-p_{G_0}(x)|/d(G_n,G_0) \\
& = & \int d_n|A_n(x) + B_n(x) + R_n(x)|/d(G_n,G_0) \; \mathrm{d}x \rightarrow 0.
\end{eqnarray*}
By Fatou's lemma, the integrand in the above display vanishes for almost all $x$. Thus,
\begin{eqnarray}
\sum_{i=1}^{k_{0}}
\alpha_{i}f(x|\theta_{i}^{0},\Sigma_{i}^{0})+ \beta_{i}^{T} \dfrac{\partial{f}}
{\partial{\theta}}(x|\theta_{i}^{0},\Sigma_{i}^{0})+\trace \left(\dfrac{\partial{f}}{\partial{\Sigma}}(x|\theta_{i}^{0},\Sigma_{i}^{0})^{T}\gamma_{i}\right)=0 \ \ \text{for almost all} \ x. \nonumber
\end{eqnarray}
By the first-order identifiability criteria of $f$, we have $\alpha_{i}=0, \beta_{i}=\vec{0} \in \mathbb{R}^{d_{1}}$, and $\gamma_{i}=\vec{0} \in \mathbb{R}^{d_{2} \times d_{2}}$ for all $i=1,2,...,k$, which is a contradiction. Hence, \eqref{eq:theorem1} is proved. 


\subsection{Strong identifiability in over-fitted mixtures}
\paragraph{PROOF OF THEOREM~\ref{theorem-secondorder}} 
(a) We only need to establish that
\begin{eqnarray}
\label{eqn:overfittedzero-pr}
\mathop {\lim }\limits_{\epsilon \to 0}{\mathop {\inf }\limits_{G \in \mathcal{O}_{k}(\Theta)}{\left\{\mathop {\sup }\limits_{x \in \mathcal{X}}{|p_{G}(x)-p_{G_{0}}(x)|}/W_{2}^{2}(G,G_{0}):W_{2}(G,G_{0}) \leq \epsilon \right\}}} > 0. 
\end{eqnarray}
The conclusion of the theorem follows from an application of Fatou's lemma in the same manner
as Step 4 in the proof of Theorem~\ref{theorem-firstorder}.
\paragraph{Step 1.} Suppose that~\eqref{eqn:overfittedzero-pr} does not hold, then 
we can find a sequence $G_{n} \in \mathcal{O}_{k}(\Theta)$ tending to
$G_{0}$ in $W_{2}$ distance and 
$\mathop {\sup }\limits_{x \in \mathcal{X}}{\left|p_{G_{n}}(x)- p_{G_{0}}(x)\right|}/W_{2}^{2}(G_{n},G_{0}) \to 0$ as $n \to \infty$. Since $k$ is finite,
there is some $k^* \in [k_0,k]$
such that there exists a subsequence of $G_{n}$ having exactly $k^{*}$ support points.
We cannot have $k^{*}=k_{0}$, due to Theorem~\ref{theorem-firstorder} and the fact that
$W_{2}^{2}(G_{n},G_{0}) \lesssim W_{1}(G_{n},G_{0})$ for all $n$. 
Thus, $k_{0}+1 \leq k^{*} \leq k$. 

Write $G_{n}=\mathop {\sum }\limits_{i=1}^{k^{*}}{p_{i}^{n}\delta_{(\theta_{i}^{n},\Sigma_{i}^{n})}}$ and $G_{0}=\mathop {\sum }\limits_{i=1}^{k_{0}}{p_{i}^{0}\delta_{(\theta_{i}^{0},\Sigma_{i}^{0})}}$. 
Since $W_{2}(G_{n},G_{0}) \to 0$, 
there exists a subsequence of $G_{n}$ such that each support point 
$(\theta_i^0,\Sigma_i^0)$ of
$G_{0}$ is the limit of a subset of $s_i \geq 1$ support points of $G_{n}$.
There may also a subset of support points of $G_{n}$ whose limits 
are not among the support points of $G_{0}$ --- we assume there are $m \geq 0$ such limit points.
To avoid notational cluttering, we replace the subsequence of $G_{n}$ by the whole sequence 
$\left\{G_{n}\right\}$. By re-labeling the support points, $G_n$ can be expressed by
\[G_{n}= \sum_{i=1}^{k_{0}+m} \sum_{j=1}^{s_{i}} p_{ij}^{n}
\delta_{(\theta_{ij}^{n},\Sigma_{ij}^{n})} 
\stackrel{W_2}{\longrightarrow} G_0 = \sum_{i=1}^{k_0+m} p_i^0
\delta_{(\theta_{i}^{0},\Sigma_{i}^{0})} \]
where $(\theta_{ij}^{n}, \Sigma_{ij}^{n}) \to (\theta_{i}^{0},\Sigma_{i}^{0})$ 
for each $i=1,\ldots,k_{0}+m$, $j=1,\ldots,s_i$,
$p_i^0 = 0$ for $i < k_0$, and 
we have that $p_{i\cdot}^n := \sum_{j=1}^{s_i}p_{ij}^n \rightarrow p_i^0$ for all $i$.
Moreover, the constraint $k_0+1 \leq \sum_{i=1}^{k_0+m} s_i \leq k$ must hold.

We note that if matrix $\Sigma$ is (strictly) positive definite whose maximum eigenvalue is bounded
(from above) by constant $M$, then $\Sigma$ is also bounded under the entrywise $\ell_2$ norm.
However if $\Sigma$ is only positive semidefinite, it can be singular and its $\ell_2$ norm
potentially unbounded. In our context, for $i \geq k_0+1$ it is possible that the limiting
matrices
$\Sigma_{i}^{0}$ can be singular.
It comes from the fact that the some eigenvalues of $\Sigma_{ij}^{n}$ can go to $0$ as 
$n \to \infty$, which implies $\textrm{det}(\Sigma_{ij}^{n}) \to 0$ and hence 
$\textrm{det}(\Sigma_{i}^{0})=0$. 
By re-labeling the support points, we may assume without loss of generality that
$\Sigma_{k_{0}+1}^{0},\ldots,\Sigma_{k_{0}+m_{1}}^{0}$ are (strictly) positive definite
matrices and $\Sigma_{k_{0}+m_{1}+1}^{0},\ldots,\Sigma_{k_{0}+m}^{0}$ 
are singular and positive semidefinite matrices for some $m_1 \in [0,m]$.
For those singular matrices, we shall make use of the assumption that
$\mathop {\lim }\limits_{\lambda_{1}(\Sigma) \to 0}{f(x|\theta,\Sigma)}=0$: 
accordingly, for each $x$, $f(x|\theta_{ij}^{n},\Sigma_{ij}^{n}) \to 0$ as $n \to \infty$ for all 
$k_{0}+m_{1}+1 \leq i \leq k_{0}+m, 1 \leq j \leq s_{i}$. 

\paragraph{Step 2.} Using shorthand notations 
$\Delta \theta_{ij}^{n} := \theta_{ij}^{n}-\theta_{i}^{0}$,
$\Delta \Sigma_{ij}^{n} := \Sigma_{ij}^{n} - \Sigma_{i}^{0}$ 
for $i=1,\ldots,k_{0}+m_{1}$ and $j=1,\ldots,s_i$, 
it is simple to see that
 \begin{eqnarray}
W_2^2(G_n,G_0) \lesssim
 d(G_{n},G_{0}) := 
\mathop {\sum }\limits_{i=1}^{k_{0}+m_{1}}{\mathop {\sum }\limits_{j=1}^{s_{i}}{p_{ij}^{n}(\|\Delta \theta_{ij}^{n}\|^{2}+\|\Delta \Sigma_{ij}^{n}\|^{2}})}+ \mathop {\sum }\limits_{i=1}^{k_{0}+m}{\left|p_{i.}^{n}-p_{i}^{0}\right|},
 \end{eqnarray}
because $W_2^2(G_n,G_0)$ is the optimal transport cost with respect to $\ell_2^2$,
while $d(G_n,G_0)$ corresponds to a multiple of the cost of a 
possibly non-optimal transport plan, which is achieved by coupling the 
atoms $(\theta_{ij}^n,\Sigma_{ij}^{n})$ for $j=1,\ldots,s_i$
with $(\theta_i^0,\Sigma_i^0)$ by mass $\min(p_{i\cdot}^n, p_i^0)$, while the 
remaining masses are coupled arbitrarily.
\comment{
As $G_{n},G_{0}$ are both finite discrete probability measures, from (cf.~\citep{Villani-2009}), we can find an optimal plan $q^{n}=\left\{q_{ij}^{n}\right\}_{i,j}$ such that $W_{2}^{2}(G_{n},G_{0})=\mathop {\sum }\limits_{i,j}{q_{ij}^{n}(\|\theta_{i}^{n}-\theta_{j}^{0}\|^{2}+\|\Sigma_{i}^{n}-\Sigma_{j}^{0}\|^{2})}$. Denote $M_{1}=\text{diam}(\Theta)$ and $M_{2}=\text{diam}(\Omega)$, then with the same argument as that of Theorem \ref{theorem-firstorder}, as $n$ is sufficiently large, $W_{2}^{2}(G_{n},G_{0}) \leq (M_{1}+M_{2}) d(G_{n},G_{0})$.

\begin{eqnarray}
 W_{2}^{2}(G_{n},G_{0})= \mathop {\sum }\limits_{i=1}^{k_{0}}{\mathop {\sum }\limits_{j=s_{i}+1}^{s_{i+1}}{q_{ij}^{(n)}\|\theta_{j}^{n}-\theta_{i}^{0}\|^{2}}} + R \leq  \mathop {\sum }\limits_{i=1}^{k_{0}}{\mathop {\sum }\limits_{j=s_{i}+1}^{s_{i+1}}{p_{i}^{n}\|\theta_{j}^{n}-\theta_{i}^{0}\|^{2}}} + R, \nonumber 
 \end{eqnarray}
 where we can check that $R \leq M \mathop {\sum }\limits_{i=1}^{m}{\mathop {\sum }\limits_{j=\tau_{i}+1}^{\tau_{i+1}}{p_{j}^{n}}}+M \mathop {\sum }\limits_{i=1}^{k_{0}}{\left|\mathop {\sum }\limits_{j=s_{i}+1}^{s_{i+1}}{p_{j}^{n}}-p_{i}^{0} \right|} + \mathop {\sum }\limits_{i=1}^{m}{\mathop {\sum }\limits_{j= \tau_{i}+1}^{\tau_{i+1}}{p_{j}^{n}\|\theta_{j}^{n}-\theta_{i}^{*}\|^{2}}}$. Therefore, we obtain $W_{2}^{2}(G_{n},G_{0}) \leq M d(G_{n},G_{0})$ as $n$ is sufficiently large.
 
 It implies that 
 \begin{eqnarray}
 \lim \limits_{n \to \infty} {d(G_{n},G_{0})/W_{2}^{2}(G_{n},G_{0})} \geq 1/(M_{1}+M_{2}). \nonumber
 \end{eqnarray}
}
Since $\mathop {\sup }\limits_{x \in \mathcal{X}}{\left|p_{G_{n}}(x) - p_{G_{0}}(x)\right|}/W_{2}^{2}(G_{n},G_{0})$ vanishes in the limit, so does
$\mathop {\sup }\limits_{x \in \mathcal{X}}{\left|p_{G_{n}}(x)- p_{G_{0}}(x)\right|}/d(G_{n},G_{0})$.

For each $x$, we make use of the key identity:
\begin{eqnarray}
p_{G_{n}}(x) - p_{G_{0}}(x) & = & 
\sum_{i=1}^{k_{0}+m_{1}}
\sum_{j=1}^{s_{i}}
p_{ij}^{n}(f(x|\theta_{ij}^{n},\Sigma_{ij}^{n})- f(x|\theta_{i}^{0},\Sigma_{i}^{0}))+ 
\sum_{i=1}^{k_{0}+m_{1}}(p_{i.}^{n}-p_{i}^{0})f(x|\theta_{i}^{0},\Sigma_{i}^{0}) \nonumber \\
& & + 
\sum_{i=k_{0}+m_{1}+1}^{k_{0}+m}
\sum_{j=1}^{s_{i}} p_{ij}^{n} f(x|\theta_{ij}^{n},\Sigma_{ij}^{n}) \nonumber \\
&:= & A_{n}(x)+B_{n}(x)+C_{n}(x). \label{eqn:overfittedone}
\end{eqnarray}

\paragraph{Step 3.}
By means of Taylor expansion up to the second order:
\begin{eqnarray}
A_{n}(x)=\mathop {\sum }\limits_{i=1}^{k_{0}+m_{1}}{\mathop {\sum }\limits_{j=1}^{s_{i}}{p_{ij}^{n}(f(x|\theta_{ij}^{n},\Sigma_{ij}^{n})-f(x|\theta_{i}^{0},\Sigma_{i}^{0}))}}  =  \mathop {\sum }\limits_{i=1}^{k_{0}+m_{1}}{\mathop {\sum }\limits_{\alpha} A_{\alpha_{1},\alpha_{2}}^{n}(\theta_{i}^{0},\Sigma_{i}^{0})}+R_{n}(x), \nonumber 
 \end{eqnarray}
where $\alpha=(\alpha_{1},\alpha_{2})$ such that 
$\alpha_{1}+\alpha_{2} \in \left\{1,2\right\}$. Specifically,
\begin{eqnarray}
A_{1,0}^{n}(\theta_{i}^{0},\Sigma_{i}^{0}) & = & \mathop {\sum }\limits_{j=1}^{s_{i}}{p_{ij}^{n}
(\Delta \theta_{ij}^{n}})^{T}\dfrac{\partial{f}}{\partial{\theta}}(x|\theta_{i}^{0},\Sigma_{i}^{0}), \nonumber \\
A_{0,1}^{n}(\theta_{i}^{0},\Sigma_{i}^{0}) & = & \mathop {\sum }\limits_{j=1}^{s_{i}}{p_{ij}^{n} \trace \left(\dfrac{\partial{f}}{\partial{\Sigma}}(x|\theta_{i}^{0},\Sigma_{i}^{0})^{T} \Delta \Sigma_{ij}^{n} \right)}, \nonumber \\
A_{2,0}^{n}(\theta_{i}^{0},\Sigma_{i}^{0}) & = & \dfrac{1}{2}\mathop {\sum }\limits_{j=1}^{s_{i}}{p_{ij}^{n}(\Delta \theta_{ij}^{n})^{T} \dfrac{\partial^{2}{f}}{\partial{\theta}^{2}}(x|\theta_{i}^{0},\Sigma_{i}^{0}) \Delta \theta_{ij}^{n}}, \nonumber \\
A_{0,2}^{n}(\theta_{i}^{0},\Sigma_{i}^{0}) & = & \dfrac{1}{2}\mathop {\sum }\limits_{j=1}^{s_{i}}{p_{ij}^{n}\trace\left(\dfrac{\partial{}}{\partial{\Sigma}}\left(\trace\left(\dfrac{\partial{}}{\partial{\Sigma}}(x|\theta_{i}^{0},\Sigma_{i}^{0})^{T}\Delta \Sigma_{ij}^{n}\right)\right)^{T} \Delta \Sigma_{ij}^{n}\right)}, \nonumber \\
A_{1,1}^{n}(\theta_{i}^{0},\Sigma_{i}^{0}) & = & 2\mathop {\sum }\limits_{j=1}^{s_{i}}{(\Delta \theta_{ij}^{n})^{T} \left[\dfrac{\partial{}}{\partial{\theta}}\left(\trace\left(\dfrac{\partial{f}}{\partial{\Sigma}}(x|\theta_{i}^{0},\Sigma_{i}^{0})^{T} \Delta \Sigma_{ij}^{n} \right)\right)\right]}. \nonumber
\end{eqnarray}
In addition, $R_{n}(x)=O \biggr (\sum_{i=1}^{k_{0}+m_{1}} \sum_{j=1}^{s_{i}}
p_{ij}^{n}(\|\Delta \theta_{ij}^{n}\|^{2+\delta}+\|\Delta \Sigma_{ij}^{n}\|^{2+\delta}) \biggr)$
due to the second-order Lipschitz condition. 
It is clear that $\sup_{x}|R_{n}(x)|/d(G_{n},G_{0}) \to 0$ as $n \to \infty$. 

\paragraph{Step 4.}
Write $D_n := d(G_n,G_0)$ for short. Note that 
$(p_{G_{n}}(x)-p_{G_{0}}(x))/D_{n}$ is a linear combination of
the scalar elements of $f(x|\theta,\Sigma)$ and its derivatives taken with respect to
$\theta$ and $\Sigma$ up to the second order, and evaluated at the distinct pairs
$(\theta_i^0,\Sigma_i^0)$ for $i=1,\ldots,k_0+m_1$. (To be specific, the elements of
$f(x|\theta,\Sigma),\dfrac{\partial{f}}{\partial{\theta}}(x|\theta,\Sigma)$, $\dfrac{\partial{f}}{\partial{\Sigma}}(x|\theta,\Sigma)$, $\dfrac{\partial^{2}{f}}{\partial{\theta}^{2}}(x|\theta,\Sigma)$, $\dfrac{\partial^{2}{f}}{\partial{\theta}^{2}}(x|\theta,\Sigma),\dfrac{\partial^{2}{f}}{\partial{\Sigma}^{2}}(x|\theta,\Sigma)$, and $\dfrac{\partial^{2}{f}}{\partial{\theta}\partial{\Sigma}}(x|\theta,\Sigma)$).
In addition, the coefficients
associated with these elements do not depend on $x$.
As in the proof of Theorem~\ref{theorem-firstorder}, we shall argue that \emph{not} all
such coefficients vanish as $n \to \infty$. 
Indeed, if this is not true, then by taking the summation of all the absolute value of 
the coefficients associated with the elements of 
$\dfrac{\partial^{2}{f}}{\partial{\theta}^{2}_{l}}(x|\theta)$ as $1 \leq l \leq d_{1}$ 
and $\dfrac{\partial^{2}{f}}{\partial{\Sigma}^{2}_{uv}}$ for $1 \leq u,v \leq d_{2}$, we obtain
\begin{eqnarray}
\mathop {\sum }\limits_{i=1}^{k_{0}+m_{1}}{\mathop {\sum }\limits_{j=1}^{s_{i}}{p_{ij}^{n}(\|\Delta \theta_{ij}^{n}\|^{2}+\|\Delta  \Sigma_{ij}\|^{2})}} / D_{n} \to 0. \nonumber
\end{eqnarray}
Therefore, $\mathop {\sum }\limits_{i=1}^{k_{0}+m}{|p_{i.}^{n}-p_{i}^{0}|}/D_{n} \to 1$ as $n \to \infty$. It implies that we should have at least one coefficient associated with
a $f(x|\theta)$ (appearing in $B_n(x)/D_n$) does not converge to 0 as $n \to \infty$, 
which is a contradiction. As a consequence, not all the coefficients 
vanish to $0$.

\paragraph{Step 5.}
Let $m_{n}$ be the maximum of the absolute value of the aforementioned coefficients.
and set $d_{n}=1/m_{n}$. Then, $d_{n}$ is uniformly bounded above when $n$ is sufficiently large. 
Therefore, as $n \to \infty$, we obtain 
\begin{eqnarray}
d_{n}B_{n}(x)/D_{n} & \to & \mathop {\sum }\limits_{i=1}^{k_{0}+m_{1}}{\alpha_{i}f(x|\theta_{i}^{0},\Sigma_{i}^{0})}, \nonumber \\
d_{n}\mathop {\sum }\limits_{i=1}^{k_{0}+m_{1}}{A_{1,0}^{n}(\theta_{i}^{0},\Sigma^{0})}/D_{n} & \to & \mathop {\sum }\limits_{i=1}^{k_{0}+m_{1}}{\beta_{i}^{T}\dfrac{\partial{f}}{\partial{\theta}}(x|\theta_{i}^{0},\Sigma_{i}^{0})}, \nonumber \\ 
d_{n}\mathop {\sum } \limits_{i=1}^{k_{0}+m_{1}}{A_{0,1}^{n}(\theta_{i}^{0},\Sigma_{i}^{0})}/D_{n} & \to &\mathop {\sum }\limits_{i=1}^{k_{0}+m_{1}}{\trace\left(\dfrac{\partial{f}}{\partial{\Sigma}}(x|\theta_{i}^{0},\Sigma_{i}^{0})^{T}\gamma_{i}\right)}, \nonumber \\
d_{n}\mathop {\sum } \limits_{i=1}^{k_{0}+m_{1}}{A_{2,0}^{n}(\theta_{i}^{0},\Sigma_{i}^{0})}/D_{n} & \to & \mathop {\sum }\limits_{i=1}^{k_{0}+m_{1}}{\mathop {\sum }\limits_{j=1}^{s_{i}}{\nu_{ij}^{T}\dfrac{\partial^{2}{f}}{\partial{\theta}^{2}}(x|\theta_{i}^{0},\Sigma_{i}^{0})\nu_{ij}}}, \nonumber \\
 d_{n}\mathop {\sum } \limits_{i=1}^{k_{0}+m_{1}}{A_{0,2}^{n}(\theta_{i}^{0},\Sigma_{i}^{0})}/D_{n} & \to & \mathop {\sum }\limits_{i=1}^{k_{0}+m_{1}}{\mathop {\sum }\limits_{j=1}^{s_{i}}{\trace\left(\dfrac{\partial{}}{\partial{\Sigma}}\left(\trace\left(\dfrac{\partial{f}}{\partial{\Sigma}}(x|\theta_{i}^{0},\Sigma_{i}^{0})^{T}\eta_{ij}\right)\right)^{T}\eta_{ij}\right)}}, \nonumber \\
 d_{n}\mathop {\sum } \limits_{i=1}^{k_{0}+m_{1}}{A_{1,1}^{n}(\theta_{i}^{0},\Sigma_{i}^{0})}/D_{n} & \to & \mathop {\sum }\limits_{i=1}^{k_{0}+m_{1}}{\mathop {\sum }\limits_{j=1}^{s_{i}}{\nu_{ij}^{T}\left[\dfrac{\partial{}}{\partial{\theta}}\left(\trace\left(\dfrac{\partial{}}{\partial{\Sigma}}(x|\theta_{i}^{0},\Sigma_{i}^{0})^{T}\eta_{ij}\right)\right)\right]}}, \nonumber
 \end{eqnarray}
 where $\alpha_{i} \in \mathbb{R}, \beta_{i},\nu_{i1},\ldots,\nu_{is_{i}} \in \mathbb{R}^{d_{1}}$, $\gamma_{i}, \eta_{i1}, \ldots, \eta_{is_{i}}$ are symmetric matrices in $\mathbb{R}^{d_{2} \times d_{2}}$ for all $1 \leq i \leq k_{0}+m_{1}, 1 \leq j \leq s_{i}$. Additionally, $d_{n}C_{n}(x)/D_{n}=D_{n}^{-1}\mathop {\sum }\limits_{i=k_{0}+m_{1}+1}^{k_{0}+m}{\mathop {\sum }\limits_{j=1}^{s_{i}}{d_{n}p_{ij}^{n}f(x|\theta_{ij}^{n},\Sigma_{ij}^{n})}} \to 0$ due to the fact that $f(x|\theta_{ij}^{n},\Sigma_{ij}^{n}) \to 0$ for all $k_{0}+m_{1}+1 \leq i \leq k_{0}+m, 1 \leq j \leq s_{i}$. As a consequence, we obtain for all $x$ that
 \begin{eqnarray}
\sum_{i=1}^{k_{0}+m_{1}} \biggr \{ \alpha_{i}f(x|\theta_{i}^{0},\Sigma_{i}^{0}) +
\beta_{i}^{T} \dfrac{\partial{f}}{\partial{\theta}}(x|\theta_{i}^{0},\Sigma_{i}^{0}) +
\mathop {\sum }\limits_{j=1}^{s_{i}}{\nu_{ij}^{T} \dfrac{\partial^{2}{f}}
{\partial{\theta}^{2}} (x|\theta_{i}^{0},\Sigma_{i}^{0}) \nu_{ij}} & + &  \notag \\
\trace \left(\dfrac{\partial{f}}{\partial{\Sigma}}(x|\theta_{i}^{0},\Sigma_{i}^{0})^{T}
\gamma_{i}\right ) 
+ 2\mathop {\sum }\limits_{j=1}^{s_{i}}{\nu_{ij}^{T}\left[\dfrac{\partial}{\partial{\theta}}\left(\trace \left(\dfrac{\partial{f}}{\partial{\Sigma}}(x|\theta_{i}^{0},\Sigma_{i}^{0})^{T}\eta_{ij}\right)\right)\right]}  & + & \notag \\
 \mathop {\sum }\limits_{j=1}^{s_{i}}{\trace \left(\dfrac{\partial{}}{\partial{\Sigma}}\left(\trace\left(\dfrac{\partial{f}}{\partial{\Sigma}}(x|\theta_{i}^{0},\Sigma_{i}^{0})^{T}\eta_{ij}\right)\right)^{T}\eta_{ij}\right) }
\biggr \}
& = & 0. \nonumber
 \end{eqnarray}
From the second-order identifiability of $\left\{f(x|\theta,\Sigma),\theta \in \Theta,\Sigma \in \Omega \right\}$, we obtain $\alpha_{i}=0,\beta_{i}=\nu_{i1}=\ldots=\nu_{is_{i}}=\vec{0} \in \mathbb{R}^{d_{1}}, \gamma_{i}=\eta_{i1}=\ldots=\eta_{is_{i}}=\vec{0} \in \mathbb{R}^{d_{2} \times d_{2}}$ for all $1 \leq i \leq k_{0}+m_{1}$, which is a contradiction to the fact that not all coefficients go to 0 as $n \to \infty$. This concludes the proof of Eq.~\eqref{eqn:overfittedzero-pr} and that of the theorem.

(b) Recall
$G_{0}=\mathop {\sum }\limits_{i=1}^{k_{0}}{p_{i}^{0}\delta_{(\theta_{i}^{0},\Sigma_{i}^{0})}}$. 
Construct a sequence of probability measures $G_n$ having exactly $k_0+1$ support points as follows:
$G_{n}=\mathop {\sum }\limits_{i=1}^{k_{0}+1}{p_{i}^{n}\delta_{(\theta_{i}^{n},\Sigma_{i}^{n})}}$,
where
$\theta_{1}^{n}=\theta_{1}^{0}-\dfrac{1}{n}\vec{1}_{d_{1}}, 
\theta_{2}^{n}=\theta_{1}^{0}+\dfrac{1}{n}\vec{1}_{d_{1}}, 
\Sigma_{1}^{n}=\Sigma_{1}^{0}-\dfrac{1}{n}I_{d_{2}}$ and 
$\Sigma_{2}^{n}=\Sigma_{1}^{0}+\dfrac{1}{n}I_{d_{2}}$. Here,
$I_{d_{2}}$ denotes identity matrix in $\mathbb{R}^{d_{2} \times d_{2}}$ and 
$\vec{1}_{n}$ a vector with all elements being equal to 1.
In addition, $(\theta_{i+1}^{n},\Sigma_{i+1}^{n}) = (\theta_{i}^{0},\Sigma_{i}^{0})$ 
for all $i=2,\ldots, k_{0}$. Also,
$p_{1}^{n}=p_{2}^{n}=\dfrac{p_{1}^{0}}{2}$ and $p_{i+1}^{n}=p_{i}^{0}$ 
for all $i=2,\ldots,  k_{0}$. It is simple to verify that
$E_{n} := W_{1}^{r}(G_{n},G_{0}) = \dfrac{(p_{1}^{0})^{r}}{2^{r}}(\|\theta_{1}^{n}-\theta_{1}^{0}\|+\|\theta_{2}^{n}-\theta_{2}^{0}\|+\|\Sigma_{1}^{n}-\Sigma_{1}^{0}\|+\|\Sigma_{2}^{n}-\Sigma_{1}^{0}\|)^{r}=\dfrac{(p_{1}^{0})^{r}}{2^{r}}(\sqrt{d_{1}}+\sqrt{d_{2}})^{r}\dfrac{1}{n^{r}} \asymp \dfrac{1}{n^{r}}$. 

By means of Taylor's expansion up to the first order, we get that as $n \to \infty$
\begin{eqnarray}
V(p_{G_{n}},p_{G_{0}})  & = & \dfrac{p_{1}^{0}}{2} \int_{x \in \mathcal{X}}
\left |
\mathop {\sum }\limits_{i=1}^{2}{\mathop {\sum }\limits_{\alpha_{1},\alpha_{2}}{(\Delta \theta_{1i}^{n})^{\alpha_{1}}(\Delta \Sigma_{1i}^{n})^{\alpha_{2}}\dfrac{\partial{f}}{\partial{\theta^{\alpha_{1}}}\partial{\Sigma^{\alpha_{2}}}}(x|\theta_{1}^{0},\Sigma_{1}^{0})+R_{1}(x)}}
\right| \; \textrm{d} x \nonumber \\
& = & \int \limits_{x \in \mathcal{X}} |R_1(x)| \; \textrm{d} x, \nonumber
\end{eqnarray}
where $\alpha_{1} \in \mathbb{N}^{d_{1}},\alpha_{2} \in \mathbb{N}^{d_{2} \times d_{2}}$ in the sum 
such that $|\alpha_{1}|+|\alpha_{2}|=1$, $R_{1}$ is Taylor expansion's remainder. 
The second equality in the above equation is due to 
$\mathop {\sum }\limits_{i=1}^{2}{(\Delta \theta_{1i}^{n})^{\alpha_{1}}(\Delta \Sigma_{1i}^{n})^{\alpha_{2}}}=0$ for each $\alpha_{1},\alpha_{2}$ such that $|\alpha_{1}|+|\alpha_{2}|=1$. Since $f$ is second-order differentiable with respect to $\theta,\Sigma$, $R_{1}(x)$ takes the form
\begin{eqnarray}
R_{1}(x)=\mathop {\sum }\limits_{i=1}^{2}{\mathop {\sum }\limits_{|\alpha|=2}{\dfrac{2}{\alpha !} (\Delta \theta_{1i}^{n})^{\alpha_{1}}(\Delta \Sigma_{1i}^{n})^{\alpha_{2}} \int \limits_{0}^{1}{(1-t)\dfrac{\partial^{2}{f}}{\partial{\theta^{\alpha_{1}}}\partial{\Sigma^{\alpha_{2}}}}(x|\theta_{1}^{0}+t\Delta \theta_{1i}^{n}, \Sigma_{1}^{0}+t \Delta \Sigma_{1i}^{n})}dt}}, \nonumber
\end{eqnarray}
where $\alpha=(\alpha_{1},\alpha_{2})$. 
Note that, $\mathop {\sum }\limits_{i=1}^{2}{|\Delta_{1i}^{n}|^{\alpha_{1}}|\Delta \Sigma_{1i}^{n}|^{\alpha_{2}}}
= O(n^{-2})$. Additionally, from the hypothesis, ${\displaystyle \mathop {\sup }\limits_{t \in [0,1]}{\int \limits_{x \in \mathcal{X}}{\left|\dfrac{\partial^{2}{f}}{\partial{\theta^{\alpha_{1}}}\partial{\Sigma^{\alpha_{2}}}}(x|\theta_{1}^{0}+t\Delta \theta_{1i}^{n}, \Sigma_{1}^{0}+t \Delta \Sigma_{1i}^{n})\right|}dx}} < \infty$. 
\comment{
As a consequence, from triangle inequality
\begin{eqnarray}
\int \limits_{x \in \mathcal{X}}{\dfrac{|R_{1}(x)|}{W_{1}^{r}(G_{n},G_{0})}} \; \textrm{d}x & \leq & \mathop {\sum }\limits_{|\alpha|=2}{\mathop {\sum }\limits_{i=1}^{2}{\dfrac{2}{\alpha !} \dfrac{|\Delta \theta_{1i}^{n}|^{\alpha_{1}}|\Delta \Sigma_{1i}^{n}|^{\alpha_{2}}}{W_{1}^{r}(G_{n},G_{0})} \times }} \nonumber \\
& \times & \int \limits_{0}^{1}{(1-t)\int \limits_{x \in \mathcal{X}}{\left|\dfrac{\partial^{2}{f}}{\partial{\theta^{\alpha_{1}}}\partial{\Sigma^{\alpha_{2}}}}(x|\theta_{1}^{0}+t\Delta \theta_{1i}^{n}, \Sigma_{1}^{0}+t \Delta \Sigma_{1i}^{n})\right|}dx}dt \to 0, \nonumber
\end{eqnarray}
}
It follows that $\int |R_1(x)|\; \textrm{d}x = O(n^{-2})$. 
So for any $r<2$, $V(p_{G_n},p_{G_0}) = o(W_1^r(G_n,G_0))$. This concludes the proof.

(c) Continuing with the same sequence $G_{n}$ constructed in part (b), we have
\begin{eqnarray}
h^{2}(p_{G_{n},p_{G_{0}}}) \leq \dfrac{1}{2p_{1}^{0}}\int \limits_{x \in \mathcal{X}}{\dfrac{(p_{G_{n}}(x)-p_{G_{0}}(x))^{2}}{f(x|\theta_{1}^{0},\Sigma_{1}^{0})}} \; \textrm{d}x \lesssim \int \limits_{x \in \mathcal{X}}{\dfrac{R_{1}^{2}(x)}
{f(x|\theta_{1}^{0},\Sigma_{1}^{0})}}\; \textrm{d}x. \nonumber
\end{eqnarray}
where the first inequality is due to $\sqrt{p_{G_{n}}(x)}+\sqrt{p_{G_{0}}(x)}>\sqrt{p_{G_{0}}(x)} >\sqrt{p_{1}^{0}f(x|\theta_{1}^{0},\Sigma_{1}^{0})}$ and the second inequality is because of Taylor expansion 
taken to the first order. 
\comment{
From here, by means of Cauchy-Schwartz's inequality and Holder's inequality, we obtain
\begin{eqnarray}
\int \limits_{x \in \mathcal{X}}{\dfrac{R_{1}^{2}(x)}{W_{1}^{2r}(G_{n},G_{0})f(x|\theta_{1}^{0},\Sigma_{1}^{0})}}dx & \lesssim & \mathop {\sum }\limits_{|\alpha|=2}{\mathop {\sum }\limits_{i=1}^{2}{\dfrac{2}{\alpha !} \dfrac{|\Delta \theta_{1i}^{n}|^{2\alpha_{1}}|\Delta \Sigma_{1i}^{n}|^{2\alpha_{2}}}{W_{1}^{2r}(G_{n},G_{0})} \times }} \nonumber \\
& \times & \int \limits_{0}^{1}{(1-t)\int \limits_{x \in \mathcal{X}}{\dfrac{\left(\dfrac{\partial^{2}{f}}{\partial{\theta^{\alpha_{1}}}\partial{\Sigma^{\alpha_{2}}}}(x|\theta_{1}^{0}+t\Delta \theta_{1i}^{n}, \Sigma_{1}^{0}+t \Delta \Sigma_{1i}^{n})\right)^{2}}{f(x|\theta_{1}^{0},\Sigma_{1}^{0})}}dx}dt, \nonumber
\end{eqnarray}
which converges to 0. Therefore, we obtain the conclusion of part (c) of our theorem.
}
The proof proceeds in the same manner as that of part (b).


\subsection{Proofs for over-fitted Gaussian mixtures}
\paragraph{Proof of Theorem~\ref{theorem:generaloverfittedGaussian}.}

For the ease of exposition, we consider the setting of univariate location-scale Gaussian distributions,
i.e., both $\theta$ and $\Sigma=\sigma^{2}$ are scalars.
The proof for general $d \geq 1$ is pretty similar and can be found in Appendix II. 
Let $v=\sigma^{2}$, so we write
$G_{0}=\mathop {\sum }\limits_{i=1}^{k_{0}}{p_{i}^{0}\delta_{(\theta_{i}^{0},v_{i}^{0})}}$. 

\paragraph{Step 1.}
For any sequence $G_{n} \in \mathcal{O}_{k,c_{0}}(\Theta \times \Omega) \to G_{0}$ in $W_{r}$, 
by employing the same subsequencing
argument in the second paragraph in the proof of Theorem~\ref{theorem-secondorder},
we can represent without loss of generality 
\begin{equation}
\label{eqn-Gn}
G_{n}=\mathop {\sum }\limits_{i=1}^{k_{0}}{\mathop {\sum }\limits_{j=1}^{s_{i}}{p_{ij}^{n}\delta_{(\theta_{ij}^{n},v_{ij}^{n})}}},
\end{equation}
where $(p_{ij}^{n},\theta_{ij}^{n},v_{ij}^{n}) \to (p_{i}^{0},\theta_{i}^{0},v_{i}^{0})$ 
for all $i=1,\ldots, k_{0}$ and $j=1,\ldots,s_{i}$, where $s_1,\ldots, s_{k_0}$
are some natural constants less than $k$. All $G_n$ have exactly the same $\sum s_i \leq k$ number
of support points. 
\comment{For sufficiently large $n$, it is simple to see
that 
\[W_r^r(G_n,G_0) \gtrsim \sum_{i=1}^{k_0}\sum_{j=1}^{s_i} p_{ij}^{n} 
(|\Delta\theta_{ij}^{n}| + |\Delta v_{ij}^{n}|)^r.\]
Here we have adopted the notation that
$\Delta \theta_{ij}^{n} := \theta_{ij}^{n}-\theta_{i}^{0}$, $\Delta v_{ij}^{n} := v_{ij}^{n}-v_{i}^{0}$.}

\paragraph{Step 2.} For any $x \in \mathbb{R}$,
\[p_{G_{n}}(x)-p_{G_{0}}(x) =  \mathop {\sum }\limits_{i=1}^{k_{0}}{\mathop {\sum }\limits_{j=1}^{s_{i}}{p_{ij}^{n}(f(x|\theta_{ij}^{n},v_{ij}^{n})-f(x|\theta_{i}^{0},v_{i}^{0}))}}+\mathop {\sum }\limits_{i=1}^{k_{0}}{(p_{i.}^{n}-p_{i}^{0})f(x|\theta_{i}^{0},v_{i}^{0})},\]
where $p_{i\cdot}^{n} := \mathop {\sum }\limits_{j=1}^{s_{i}}{p_{ij}^{n}}$, and
$p_i^0 = 0$ for any $i\geq k_0+1$.
For any $r \geq 1$, integer $N \geq r$
and $x \in \mathbb{R}$, by means of Taylor expansion up to order $N$, we obtain
\begin{eqnarray}
p_{G_n}(x)-p_{G_0}(x) &=&
\mathop {\sum }\limits_{i=1}^{k_{0}}{\mathop {\sum }\limits_{j=1}^{s_{i}}{p_{ij}^{n}\mathop {\sum }\limits_{|\alpha|=1}^{N}{(\Delta\theta_{ij}^{n})^{\alpha_{1}}(\Delta v_{ij}^{n})^{\alpha_{2}}\dfrac{D^{|\alpha|}f(x|\theta_{i}^{0},v_{i}^{0})}{\alpha!}}}} +
A_{1}(x)+
R_{1}(x). \label{eqn:generaloverfittedGaussiangeneralone}
\end{eqnarray}
Here, $\alpha=(\alpha_{1},\alpha_{2})$, $|\alpha|=\alpha_{1}+\alpha_{2}$, $\alpha!=\alpha_{1}!\alpha_{2}!$. 
Additionally,
$A_{1}(x)=\mathop {\sum }\limits_{i=1}^{k_{0}}{(p_{i\cdot}^{n}-p_{i}^{0})f(x|\theta_{i}^{0},v_{i}^{0})}$, 
 and 
$R_{1}(x) = O(\mathop {\sum }\limits_{i=1}^{k_{0}}{\mathop {\sum }\limits_{j=1}^{s_{i}}{p_{ij}^{n}(|\Delta\theta_{ij}^{n}|^{N+\delta}+|\Delta v_{ij}^{n}|^{N+\delta}}})$.

\paragraph{Step 3.}
Enter the key identity~\eqref{key-gaussian} (cf. Lemma \ref{lemma:multivariatenormaldistribution}):
$\dfrac{\partial^{2}{f}}{\partial{\theta^{2}}}(x|\theta,v)=2\dfrac{\partial{f}}{\partial{v}}(x|\theta,v)$ 
for all $x$. This entails, for any natural orders $n_{1},n_{2}$, that
$\dfrac{\partial^{n_{1}+n_{2}}{f}}{\partial{\theta}^{n_{1}}\partial{v}^{n_{2}}}(x|\theta,v)=\dfrac{1}{2^{n_{2}}}\dfrac{\partial^{n_{1}+2n_{2}}{f}}{\theta^{n_{1}+2n_{2}}}(x|\theta,v)$. Thus,
by converting all derivatives to those taken with respect to only $\theta$, we may
rewrite \eqref{eqn:generaloverfittedGaussiangeneralone} as
\begin{eqnarray}
p_{G_{n}}(x)-p_{G_{0}}(x) & = & \mathop {\sum }\limits_{i=1}^{k_{0}}{\mathop {\sum }\limits_{j=1}^{s_{i}}{p_{ij}^{n}{\mathop {\sum }\limits_{\alpha \geq 1}{\mathop {\sum }\limits_{n_{1}, n_{2}}{\dfrac{(\Delta \theta_{ij}^{n})^{n_{1}}(\Delta v_{ij}^{n})^{n_{2}}}{2^{n_{2}}n_{1}!n_{2}!}}\dfrac{\partial^{\alpha}{f}}{\partial{\theta}^{\alpha}}(x|\theta_{i}^{0},v_{i}^{0})}}}} \nonumber \\
& + & A_{1}(x) + R_{1}(x) \nonumber \\
& := & A_{1}(x)+B_{1}(x)+R_{1}(x), \label{eqn:generaloverfittedGaussiangeneralsecond}
\end{eqnarray}
where $n_{1},n_{2}$ in the sum satisfy $n_{1}+2n_{2}=\alpha, n_{1}+n_{2} \leq N$.


\paragraph{Step 4.}
We proceed to proving part (a) of the theorem. From the definition of $\overline{r}$, 
by setting $r = \overline{r}-1$, there exist non-trivial solutions 
$(c_{i}^{*},a_{i}^{*},b_{i}^{*})_{i=1}^{k-k_{0}+1}$ for the system of equations~\eqref{eqn:generalovefittedGaussianzero}. Construct a sequence of probability measures $G_{n} \in \Ocal_{k}(\Theta\times\Omega)$
under the representation given by Eq.~\eqref{eqn-Gn} as follows:
\[\theta_{1j}^{n}=\theta_{1}^{0}+\dfrac{a_{j}^{*}}{n},\;
v_{1j}^{n}=v_{1}^{0}+\dfrac{2b_{j}^{*}}{n^{2}},\; 
p_{1j}^{n}=\frac{p_{1}^{0}(c_{j}^{*})^{2}}{\mathop {\sum }\limits_{j=1}^{k-k_{0}+1}{(c_{j}^{*})^{2}}},\;
\text{for all}\; j=1,\ldots,k-k_{0}+1,\] 
and 
$\theta_{i1}^{n}=\theta_{i}^{0}, \; 
v_{i1}^{n}=v_{i}^{0}, \;
p_{i1}^{n}=p_{i}^{0}$
for all $i = 2,\ldots, k_{0}$. (That is, we set $s_{1}=k-k_{0}+1$, $s_{i}=1$ for all $2 \leq i \leq k_{0}$).
Note that $b_j^*$ may be negative, but we are guaranteed that $v_{1j}^n > 0$ for sufficiently large $n$.
It is easy to verify that
$W_{1}(G_{n},G_{0})= \mathop {\sum }\limits_{i=1}^{k-k_{0}+1}{p_{1i}^{n}\left(\dfrac{|a_{i}^{*}|}{n}+\dfrac{2|b_{i}^{*}|}{n^{2}}\right)} \asymp \dfrac{1}{n}$, because at least one of the $a_i^*$ is non-zero.

\paragraph{Step 5.}
Select $N=\overline{r}$ in Eq.~\eqref{eqn:generaloverfittedGaussiangeneralsecond}.
By our construction of $G_n$, clearly $A_{1}(x)=0$. Moreover, 
\begin{eqnarray}
B_{1}(x) &=&  \mathop {\sum }\limits_{i=1}^{k-k_{0}+1}{p_{1i}^{n}\mathop {\sum }\limits_{\alpha=1}^{\overline{r}-1}{\mathop {\sum }\limits_{n_{1}, n_{2}}{\dfrac{(\Delta \theta_{1i}^{n})^{n_{1}}(\Delta v_{1i}^{n})^{n_{2}}}{2^{n_{2}}n_{1}!n_{2}!}}\dfrac{\partial^{\alpha}{f}}{\partial{\theta}^{\alpha}}(x|\theta_{1}^{0},v_{1}^{0})}} \nonumber \\
&+& \mathop {\sum }\limits_{i=1}^{k-k_{0}+1}{p_{1i}^{n}\mathop {\sum }\limits_{\alpha = \overline{r}}^{2\overline{r}}{\mathop {\sum }\limits_{n_{1}, n_{2}}{\dfrac{(\Delta \theta_{1i}^{n})^{n_{1}}(\Delta v_{1i}^{n})^{n_{2}}}{2^{n_{2}}n_{1}!n_{2}!}}\dfrac{\partial^{\alpha}{f}}{\partial{\theta}^{\alpha}}(x|\theta_{1}^{0},v_{1}^{0})}} \nonumber \\
&:= & \mathop {\sum }\limits_{\alpha=1}^{\overline{r}-1}{B_{\alpha n}\dfrac{\partial^{\alpha}{f}}{\partial{\theta}^{\alpha}}(x|\theta_{1}^{0},v_{1}^{0})}+\mathop {\sum }\limits_{\alpha \geq \overline{r}}{C_{\alpha n}\dfrac{\partial^{\alpha}{f}}{\partial{\theta}^{\alpha}}(x|\theta_{1}^{0},v_{1}^{0})}. \nonumber
\end{eqnarray}
In the above display, 
for each $\alpha \geq \overline{r}$, observe that $C_{\alpha n} = O(n^{-\alpha})$. Moreover,
for each $1 \leq \alpha \leq \overline{r}-1$,
\begin{eqnarray}
B_{\alpha n} =\dfrac{1}{n^{\alpha}\mathop {\sum }\limits_{i=1}^{k-k_{0}+1}{(c_{i}^{*})^{2}}} \mathop {\sum }\limits_{i=1}^{k-k_{0}+1}{(c_{i}^{*})^{2} \mathop {\sum }\limits_{n_{1}+2n_{2}=\alpha}{\dfrac{(a_{i}^{*})^{n_{1}}(b_{i}^{*})^{n_{2}}}{n_{1}!n_{2}!}}} = 0, \nonumber
\end{eqnarray}
because $(c_i^*,a_i^*,b_i^*)_{i=1}^{k-k_0+1}$ form a non-trivial solution to system~\eqref{eqn:generalovefittedGaussianzero}. 

\paragraph{Step 6.} We arrive at an upper bound the Hellinger distance of mixture densities.
\comment{
\begin{eqnarray}
C_{\alpha n}/W_{1}^{r}(G_{n},G_{0})=An^{2r-\alpha}/\left(\mathop {\sum }\limits_{i=1}^{k-k_{0}+1}{p_{1i}^{0}(n|a_{i}^{*}|+|b_{i}^{*}|)}\right)^{r}\to 0, \nonumber
\end{eqnarray}
where $A=\mathop {\sum }\limits_{\substack{n_{1}+2n_{2}=\alpha \\ n_{1}+n_{2} \leq \overline{r}-1}}{\dfrac{(a_{i}^{*})^{n_{1}}(b_{i}^{*})^{n_{2}}}{n_{1}!n_{2}!}}$ and the last result is due to $r < \overline{r}$. 
}
\begin{eqnarray*}
h^{2}(p_{G_{n}},p_{G_{0}})
& \leq & \frac{1}{2p_1^0}\int_{\mathbb{R}} \frac{(p_{G_n}(x)-p_{G_0}(x))^2}{f(x|\theta_1^0,v_1^0)} 
\; \textrm{d}x \\
& \lesssim & 
\int \limits_{\mathbb{R}}{\dfrac{
\mathop {\sum }\limits_{\alpha = \overline{r}}^{2\overline{r}}
{C_{\alpha n}^{2}\left(\dfrac{\partial^{\alpha}{f}}{\partial{\theta^{\alpha}}}(x|\theta_{1}^{0},v_{1}^{0})\right)^{2}}+R_{1}^{2}(x)}{
f(x|\theta_{1}^{0},v_{1}^{0})}} \;\textrm{d} x, \nonumber
\end{eqnarray*}
For Gaussian densities, it can be verified that 
$\left(\dfrac{\partial^{\alpha}{f}}{\partial{\theta^{\alpha}}}(x|\theta_{1}^{0},v_{1}^{0})\right)^{2}/f(x|\theta_{1}^{0},v_{1}^{0})$ is integrable for all $1 \leq \alpha \leq 2\overline{r}$. So,
$h^{2}(p_{G_{n}},p_{G_{0}}) \leq O(n^{-2\rbar}) + \int R_1^2(x)/f(x|\theta_1^0,v_1^0) \; \textrm{d} x$. 
Turning to the Taylor remainder $R_1(x)$, note that
\begin{eqnarray}
|R_{1}(x)| \lesssim \mathop {\sum }\limits_{i=1}^{k-k_{0}+1}{\mathop {\sum }\limits_{|\beta| 
= \overline{r}+1}{\dfrac{(\overline{r}+1)}{\beta!} 
|\Delta \theta_{1i}^{n}|^{\beta_{1}}
|\Delta v_{1i}^{n}|^{\beta_{2}} 
\int \limits_{0}^{1}{(1-t)^{\overline{r}}\left |\dfrac{\partial^{\overline{r}+1}{f}}{\partial{\theta^{\beta_{1}}}\partial{v^{\beta_{2}}}}(x|\theta_{1}^{0}+t\Delta \theta_{1i}^{n},v_{1}^{0}+t\Delta v_{1i}^{n})\right|}
\; \textrm{d}t}}. \nonumber
\end{eqnarray}
Now,
$(\Delta \theta_{1i}^{n})^{\beta_{1}}(\Delta v_{1i}^{n})^{\beta_{2}} \asymp
n^{-\beta_1-2\beta_2} = o(n^{-2\rbar})$. In addition, as $n$ is sufficiently large, we have for all $|\beta|=\overline{r}+1$ that
\begin{eqnarray}
{\displaystyle \mathop {\sup }\limits_{t \in [0,1]}{\int \limits_{x \in \mathbb{R}}{\left(\dfrac{\partial^{\overline{r}+1}{f}}{\partial{\theta^{\beta_{1}}\partial{v^{\beta_{2}}}}}(x|\theta_{1}^{0}+t\Delta \theta_{1i}^{n},v_{1}^{0}+t\Delta v_{1i}^{n})\right)^{2}/f(x|\theta_{1}^{0},v_{1}^{0})}dx}}< \infty. \nonumber
\end{eqnarray}
It follows that $h(p_{G_n},p_{G_0}) = O(n^{-\rbar})$. As noted above,
$W_1(G_n,G_0) \asymp n^{-1}$, so the claim of part (a) is established.


\paragraph{Step 7.} Turning to part (b) of Theorem~\ref{theorem:generaloverfittedGaussian}, 
it suffices to show that
\begin{eqnarray}
\mathop {\lim }\limits_{\epsilon \to 0}{\mathop {\inf }\limits_{G \in \mathcal{O}_{k,c_{0}}(\Theta)}{\left\{\mathop {\sup }\limits_{x \in \mathcal{X}}{|p_{G}(x)-p_{G_{0}}(x)|}/W_{\overline{r}}^{\overline{r}}(G,G_{0}):W_{\overline{r}}(G,G_{0}) \leq \epsilon \right\}}} > 0. \label{eqn:generaloverfittedGaussiansecond}
\end{eqnarray}
Then one can arrive at theorem's claim by passing through a standard argument using Fatou's lemma
(cf. Step 4 in the proof of Theorem~\ref{theorem-firstorder}). 
Suppose that \eqref{eqn:generaloverfittedGaussiansecond} does not hold.
Then we can find a sequence of probability measures $G_n \in \Ocal_{k,c_0}(\Theta\times \Omega)$ 
that are represented by Eq.~\eqref{eqn-Gn}, such that
$W_\rbar^\rbar(G_n,G_0) \rightarrow 0$ and $\sup_{x}|p_{G_n}(x)-p_{G_0}(x)|/W_\rbar^\rbar(G_n,G_0)
\rightarrow 0$.  Define
\begin{eqnarray}
D_n := d(G_{n},G_{0}) := \mathop {\sum }\limits_{i=1}^{k_{0}}{\mathop {\sum }\limits_{j=1}^{s_{i}}{p_{ij}^{n}(|\Delta \theta_{ij}^{n}|^{\overline{r}}+|\Delta v_{ij}^{n}|^{\overline{r}})}}+\mathop {\sum }\limits_{i=1}^{k_{0}}{|p_{i\cdot}^{n}-p_{i}^{0}|}. \nonumber
\end{eqnarray}
Since $W_\rbar^\rbar(G_n,G) \lesssim D_n$, for all $x \in \mathbb{R}$
$(p_{G_{n}}(x)-p_{G_{0}}(x))/D_n \to 0$.
Combining this fact with \eqref{eqn:generaloverfittedGaussiangeneralsecond},
where $N=\overline{r}$, we obtain
\begin{eqnarray}
(A_{1}(x)+B_{1}(x)+R_{1}(x))/D_n \to 0. \label{eqn:generaloverfittedGaussiangeneralthird}
\end{eqnarray}
We have $R_1(x)/D_n = o(1)$ as $n\rightarrow \infty$. 

\paragraph{Step 8.} $A_{1}(x)/D_n$ and
$B_{1}(x)/D_n$ are the linear combination of elements of $\dfrac{\partial^{\alpha}{f}}{\partial{\theta}^{\alpha}}(x|\theta,v)$ where $\alpha=n_{1}+2n_{2}$ and $n_{1}+n_{2} \leq \overline{r}$. 
Note that the natural order $\alpha$ ranges in $[0,2\overline{r}]$. 
Let $E_{\alpha}(\theta,v)$ denote the corresponding coefficient of 
$\dfrac{\partial^{\alpha}{f}}{\partial{\theta}^{\alpha}}(x|\theta, v)$. 
Extracting from \eqref{eqn:generaloverfittedGaussiangeneralsecond}, 
for $\alpha = 0$, $E_0(\theta_i^0,v_i^0) = (p_{i\cdot}^n - p_i^0)/D_n$. For $\alpha \geq 1$,
\begin{eqnarray}
E_{\alpha}(\theta_{i}^{0},v_{i}^{0})=\left[\mathop {\sum }\limits_{j=1}^{s_{i}}{p_{ij}^{n}\mathop {\sum }\limits_{\substack{n_{1}+2n_{2}=\alpha \\ n_{1}+n_{2} \leq \overline{r}}}{\dfrac{(\Delta \theta_{ij}^{n})^{n_{1}}(\Delta v_{ij}^{n})^{n_{2}}}{2^{n_{2}}n_{1}!n_{2}!}}}\right]/D_n. \nonumber
\end{eqnarray}

Suppose that $E_{\alpha}(\theta_{i}^{0},v_{i}^{0}) \to 0$ for all $i=1,\ldots,k_{0}$ and 
$0 \leq \alpha \leq 2\overline{r}$ as $n \to \infty$. 
By taking the summation of all $|E_{0}(\theta_{i}^{0},v_{i}^{0})|$, we get 
$\mathop {\sum }\limits_{i=1}^{k_{0}}{|p_{i.}^{n}-p_{i}^{0}|}/D_n \to 0  \ \ \text{as } \ n \to \infty$. As a consequence, we get
\begin{eqnarray}
\mathop {\sum }\limits_{i=1}^{k_{0}}{\mathop {\sum }\limits_{j=1}^{s_{i}}{p_{ij}^{n}(|\Delta \theta_{ij}^{n}|^{\overline{r}}+|\Delta v_{ij}^{n}|^{\overline{r}})}}/D_n \to 1 \ \ \text{as } \ n \to \infty. \nonumber
\end{eqnarray}
Hence, we can find an index $i^{*} \in \left\{1,2,\ldots,k_{0}\right\}$ such that $\mathop {\sum }\limits_{j=1}^{s_{i^{*}}}{p_{i^{*}j}^{n}(|\Delta \theta_{i^{*}j}^{n}|^\rbar +|\Delta v_{i^{*}j}^{n})|^{\overline{r}})}/D_n \not \to 0$ as $n \to \infty$. Without loss of generality, we assume that $i^{*}=1$. Accordingly,
\begin{eqnarray}
F_{\alpha}(\theta_{1}^{0},v_{1}^{0})
:= \dfrac{D_nE_{\alpha}(\theta_{1}^{0},\sigma_{1}^{0}) }{\mathop {\sum }\limits_{j=1}^{s_{1}}{p_{1j}^{n}(|\Delta \theta_{1j}^{n}|^{\overline{r}}+|\Delta v_{1j}^{n})|^{\overline{r}})}}= \dfrac{\mathop {\sum }\limits_{j=1}^{s_{1}}{p_{1j}^{n}\mathop {\sum }\limits_{\substack{n_{1}+2n_{2}=\alpha \\ n_{1}+n_{2} \leq \overline{r}}}{\dfrac{(\Delta \theta_{1j}^{n})^{n_{1}}(\Delta v_{1j}^{n})^{n_{2}}}{2^{n_{2}}n_{1}!n_{2}!}}}}{\mathop {\sum }\limits_{j=1}^{s_{1}}{p_{1j}^{n}(|\Delta \theta_{1j}^{n}|^{\overline{r}}+|\Delta v_{1j}^{n})|^{\overline{r}})}} \to 0. \nonumber
\end{eqnarray}
If $s_{1}=1$ then $F_{1}(\theta_{1}^{0},\nu_{1}^{0})$ and $F_{2\overline{r}}(\theta_{1}^{0},\nu_{1}^{0})$ yield $|\Delta \theta_{11}^{n}|^{\overline{r}}/(|\Delta \theta_{11}^{n}|^{\overline{r}}+|\Delta v_{11}^{n}|^{r}), |\Delta v_{11}^{n}|^{\overline{r}}/(|\Delta \theta_{11}^{n}|^{\overline{r}}+|\Delta v_{11}^{n}|^{\overline{r}}) \to 0$ --- a contradiction. As a consequence, $s_{1} \geq 2$. 

Denote $\overline{p}_{n}=\mathop {\max }\limits_{1 \leq j \leq s_{1}}{\left\{p_{1j}^{n}\right\}}$, $\overline{M}_{n}=\mathop {\max }{\left\{|\Delta \theta_{11}^{n}|,\ldots,|\Delta \theta_{1s_{1}}^{n})|,|\Delta v_{11}^{n}|^{1/2},\ldots,|\Delta v_{1s_{1}}^{n}|^{1/2}\right\}} \rightarrow 0$. 
Since $0 < p_{1j}^{n}/\overline{p}_{n} \leq 1$ for all $1 \leq j \leq s_{1}$, by a subsequence argument,
there exist $c_j^2 := \mathop {\lim }\limits_{n \to \infty}{p_{1j}^{n}/\overline{p}_{n}}$ for all 
$j=1,\ldots,s_{1}$. Similarly, define $a_j:= \mathop {\lim }\limits_{n \to \infty}{\Delta \theta_{1j}^{n}/\overline{M}_{n}}$, and $2b_j:= \mathop {\lim }\limits_{n \to \infty}{\Delta v_{1j}^{n}/\overline{M}_{n}^2}$ 
for each $j=1,\ldots, s_{1}$.
By the constraints of $\Ocal_{k,c_0}$, $p_{1j}^{n} \geq c_{0}$,
so all of $c_{j}^{2}$ differ from 0 and at least one of them equals to 1. Likewise, at least one element of $\left(a_{j},b_{j}\right)_{j=1}^{s_{1}}$ equal to -1 or 1. Now, for each 
$\alpha = 1,\ldots,\overline{r}$, divide both the numerator and denominator of $F_{\alpha}(\theta_{1}^{0},v_{1}^{0})$ by 
$\overline{p}_n$ and then $\overline{M}_{n}^{\alpha}$ and let $n \to \infty$, 
we obtain the following system of polynomial equations
\begin{eqnarray}
\mathop {\sum }\limits_{j=1}^{s_{1}}{\mathop {\sum }\limits_{n_{1}+2n_{2} = \alpha}{\dfrac{c_{j}^{2}a_{j}^{n_{1}}b_{j}^{n_{2}}}{n_{1}!n_{2}!}}}=0 \ \ \text{for each } \ \alpha = 1,\ldots,\overline{r}. \nonumber
\end{eqnarray}
Since $s_{1} \geq 2$, we get $\overline{r} \geq 4$. If $a_{i}=0$ for all $1 \leq i \leq s_{1}$ then by choosing $\alpha=4$, we obtain $\mathop {\sum }\limits_{j=1}^{s_{1}}{c_{j}^{2}b_{j}^{2}}=0$. However, it demonstrates that $b_{i}=0$ for all $1 \leq i \leq s_{1}$ --- a contradiction to the fact that at least one element of $(a_{i},b_{i})_{i=1}^{s_{1}}$ is different from 0. Therefore, at least one element of $(a_{i})_{i=1}^{s_{1}}$ is not equal to 0. Observe that $s_{i} \leq k-k_{0}+1$(because the number of distinct atoms of $G_n$ is
$\sum_{i=1}^{k_0}s_i \leq k$ and all $s_i\geq 1$). Thus, the existence of non-trivial
solutions for the system of equations given in the above display entails the existence
of non-trivial solutions for system of 
equations~\eqref{eqn:generalovefittedGaussianzero}. This contradicts with the definition
of $\rbar$. Therefore, our hypothesis that all coefficients 
$E_{\alpha}(\theta_{i}^{0},v_{i}^{0})$ vanish does not hold --- there must be
at least one which does not converge to 0 as $n \to \infty$.

\paragraph{Step 9.}
Let $m_{n}$ to the maximum of the absolute values of $E_{\alpha}(\theta_{i}^{0},v_{i}^{0})$ where $0 \leq \alpha \leq 2\overline{r}$, $1 \leq i \leq k_{0}$ and $d_{n}=1/m_{n}$. As $m_{n} \not \to 0$ as $n \to \infty$, $d_{n}$ is uniformly bounded above for all $n$. As $d_{n}|E_{\alpha}(\theta_{i}^{0},v_{i}^{0})| \leq 1$, 
we have $d_{n}E_{\alpha}(\theta_{i}^{0},v_{i}^{0}) \to \beta_{i,\alpha}$ for all $0 \leq \alpha \leq 2\overline{r}$, $1 \leq i \leq k_{0}$ where at least one of $\beta_{i\alpha}$ differs from 0. 
Incorporating these limits to Eq.\eqref{eqn:generaloverfittedGaussiangeneralthird}, we 
obtain that for all $x \in \mathbb{R}$,
\begin{eqnarray}
(p_{G_{n}}(x)-p_{G_{0}}(x))/D_n \to \mathop {\sum }\limits_{i=1}^{k_{0}}{\mathop {\sum }\limits_{\alpha = 0}^{2\overline{r}}{\beta_{i\alpha}\dfrac{\partial^{\alpha}{f}}{\partial{\theta}^{\alpha}}(x|\theta_{i}^{0},v_{i}^{0})}} = 0. \nonumber 
\end{eqnarray}
By direct calculation, we can rewrite the above equation as
\begin{eqnarray}
\mathop {\sum }\limits_{i=1}^{k_{0}}{\left(\mathop {\sum }\limits_{j=1}^{2\overline{r}+1}{\gamma_{ij}(x-\theta_{i}^{0})^{j-1}}\right)\exp\left(-\dfrac{(x-\theta_{i}^{0})^{2}}{2v_{i}^{0}}\right)} =0 \ \ \text{for all } \  x \in \mathbb{R}, \nonumber
\end{eqnarray}
where $\gamma_{ij}$ for odd $j$ are linear combinations of 
$\beta_{i(2l_{1})}$, for $(j-1)/2 \leq l_{1} \leq \overline{r}$,
such that all of the coefficients are functions of  $v_{i}^{0}$ differing from 0.
For even $j$, $\gamma_{ij}$ are linear combinations of 
$\beta_{i(2l_{2}+1)}$, for $i/2 \leq l_{2} \leq \overline{r}$, 
such that all of the coefficients are functions of 
$v_{i}^{0}$ differing from 0. Employing the 
same argument as that of part (a) of Theorem \ref{identifiability-multivariatecharacterization}, 
we obtain $\gamma_{ij}=0$ for all $i=1,\ldots, k_{0}$, $j=1,\ldots, 2\overline{r}+1$. 
This entails that
$\beta_{i\alpha}=0$ for all $i=1,\ldots, k_{0}$, $\alpha = 0,\ldots, 2\overline{r}$ ---
a contradiction. Thus we achieve the conclusion of~\eqref{eqn:generaloverfittedGaussiansecond}.


\paragraph{PROOF OF PROPOSITION \ref{proposition-specificvalueoverliner}.}
Our proof is based on Groebner bases method for
determining solutions for a system of polynomial equations. 
(i) For the case  $k-k_{0}=1$, the system~\eqref{eqn:generalovefittedGaussianzero} 
when $r=4$ can be written as
\begin{eqnarray}
& &c_{1}^{2}a_{1}+c_{2}^{2}a_{2}=0 \label{S.1.1} \\ 
& &\dfrac{1}{2}(c_{1}^{2}a_{1}^{2}+c_{2}^{2}a_{2}^{2})+ c_{1}^{2}b_{1}+c_{2}^{2}b_{2} = 0 \label{S.1.2} \\ 
& &\dfrac{1}{3!}(c_{1}^{2}a_{1}^{3}+c_{2}^{2}a_{2}^{3})+ c_{1}^{2}a_{1}b_{1}+c_{2}^{2}a_{2}b_{2} =0 \label{S.1.3} \\ 
& &\dfrac{1}{4!}(c_{1}^{2}a_{1}^{4}+c_{2}^{2}a_{2}^{4})+\dfrac{1}{2!}(c_{1}^{2}a_{1}^{2}b_{1}+c_{2}^{2}a_{2}^{2}b_{2})+\dfrac{1}{2!}(c_{1}^{2}b_{1}^{2}+c_{2}^{2}b_{2}^{2})=0 \label{S.1.4}
\end{eqnarray}
Suppose that the above system has non-trivial solution. If $c_{1}a_{1}=0$, then equation \eqref{S.1.1} implies $c_{2}a_{2}=0$. Since $c_{1},c_{2} \neq 0$, we have 
$a_{1}=a_{2}=0$. 
This violates the constraint that one of $a_{1},a_{2}$ is non-zero. 
Hence, $c_{1}a_{1},c_{2}a_{2} \neq 0$. Divide both sides of \eqref{S.1.1},\eqref{S.1.2},\eqref{S.1.3},\eqref{S.1.4} by $c_{1}^{2}a_{1}$, $c_{1}^{2}a_{1}^{2}$, $c_{1}^{2}a_{1}^{3}$, $c_{1}^{2}a_{1}^{4}$ respectively, we obtain the following system of polynomial equations
\begin{eqnarray}
& &1+x^{2}a=0 \nonumber \\
& &1+x^{2}a^{2}+2(b+x^{2}c)=0 \nonumber \\
& &1+x^{2}a^{3}+6(b+x^{2}ac)=0 \nonumber \\
& & 1+x^{2}a^{4}+12(b+x^{2}a^{2}c)+12(b^{2}+x^{2}c^{2})=0 \nonumber
\end{eqnarray}
where $x=c_{2}/c_{1}, a= a_{2}/a_{1}, b= b_{1}/a_{1}, c=b_{2}/a_{1}$. By taking the lexicographical order $a \succ b \succ c \succ x$, the Groebner basis of the above system contains $x^{6}+2x^{4}+2x^{2}+1>0$ for all 
$x \in \mathbb{R}$. Therefore, the above system of polynomial equations does not have real solutions. 
As a consequence, the original system of polynomial equations does not have non-trivial solution, which means that $\overline{r} \leq 4$. However, we have already shown
that as $r=3$, Eq.\eqref{eqn:generalovefittedGaussianzero} has non-trivial solution. 
Therefore, $\overline{r}=4$.

(ii) The case $k-k_{0}=2$. System~\eqref{eqn:generalovefittedGaussianzero} when $r=6$ takes the form:
\begin{eqnarray}
& &\mathop {\sum }\limits_{i=1}^{3}{c_{i}^{2}a_{i}}=0 \label{S.2.1} \\
& &\dfrac{1}{2}\mathop {\sum }\limits_{i=1}^{3}{c_{i}^{2}a_{i}^{2}}+\mathop {\sum }\limits_{i=1}^{3}{c_{i}^{2}b_{i}}=0 \label{S.2.2} \\
& &\dfrac{1}{6}\mathop {\sum }\limits_{i=1}^{3}{c_{i}^{2}a_{i}^{3}}+\dfrac{1}{2}\mathop {\sum }\limits_{i=1}^{3}{c_{i}^{2}a_{i}b_{i}}=0 \label{S.2.3} \\
& &\dfrac{1}{24}\mathop {\sum }\limits_{i=1}^{3}{c_{i}^{2}a_{i}^{4}}+\dfrac{1}{2}\mathop {\sum }\limits_{i=1}^{3}{c_{i}^{2}a_{i}^{2}b_{i}}+\dfrac{1}{2}\mathop {\sum }\limits_{i=1}^{3}{c_{i}^{2}b_{i}^{2}}=0 \label{S.2.4} \\
& & \dfrac{1}{120}\mathop {\sum }\limits_{i=1}^{3}{c_{i}^{2}a_{i}^{5}}+\dfrac{1}{6}\mathop {\sum }\limits_{i=1}^{3}{c_{i}^{2}a_{i}^{3}b_{i}}+\dfrac{1}{2}\mathop {\sum }\limits_{i=1}^{3}{c_{i}^{2}a_{i}b_{i}^{2}}=0 \label{S.2.5} \\
& & \dfrac{1}{720}\mathop {\sum }\limits_{i=1}^{3}{c_{i}^{2}a_{i}^{6}}+ \dfrac{1}{24}\mathop {\sum }\limits_{i=1}^{3}{c_{i}^{2}a_{i}^{4}b_{i}}+\dfrac{1}{4}\mathop {\sum }\limits_{i=1}^{3}{c_{i}^{2}a_{i}^{2}b_{i}^{2}}+\dfrac{1}{6}\mathop {\sum }\limits_{i=1}^{3}{c_{i}^{2}b_{i}^{3}}=0 \label{S.2.6}
\end{eqnarray}
Non-trivial solution constraints require that
$c_{1},c_{2},c_{3} \neq 0$ and without loss of generality, $a_1 \neq 0$.
Dividing both sides of 
of the six equations above by 
$c_{1}^{2}a_{1},c_{1}^{2}a_{1}^{2},c_{1}^{2}a_{1}^{3},c_{1}^{2}a_{1}^{4},c_{1}^{2}a_{1}^{5},c_{1}^{2}a_{1}^{6}$,
respectively, we obtain
\begin{eqnarray}
& &1+x^{2}a+y^{2}b=0 \nonumber \\
& &\dfrac{1}{2}(1+x^{2}a^{2}+y^{2}b^{2})+c+x^{2}d+y^{2}e=0 \nonumber \\
& &\dfrac{1}{3}(1+x^{2}a^{3}+y^{2}b^{3})+c+x^{2}ad+y^{2}be=0 \nonumber \\
& &\dfrac{1}{12}(1+x^{2}a^{4}+y^{2}b^{4})+c+x^{2}a^{2}d+y^{2}b^{2}e + c^{2}+x^{2}d^{2}+y^{2}e^{2}=0 \nonumber \\
& &\dfrac{1}{60}(1+x^{2}a^{5}+y^{2}b^{5})+ \dfrac{1}{3}(c+x^{2}a^{3}d+y^{2}b^{3}e)+c^{2}+x^{2}ad^{2}+y^{2}be^{2}=0 \nonumber \\
& &\dfrac{1}{360}(1+x^{2}a^{6}+y^{2}b^{6})+\dfrac{1}{12}(c+x^{2}a^{4}d+y^{2}b^{4}e)+\dfrac{1}{2}(c^{2}+x^{2}a^{3}d+y^{2}b^{3}e)+\dfrac{1}{3}(c^{3}+x^{2}d^{3}+y^{2}e^{3})=0 \nonumber
\end{eqnarray}
where $x=c_{2}/c_{1}, y=c_{3}/c_{1}, a=a_{2}/a_{1}, b=a_{3}/a_{1}, c=b_{1}/a_{1}^{2}, d=b_{2}/a_{1}^{2}, e=b_{3}/a_{1}^{2}$. 
By taking the lexicographical order $a \succ b \succ c \succ d \succ x \succ y$, we can verify
that the Groebner bases of the above system of polynomial equations contains a polynomial 
in terms of $x^{2},y^{2}$ with all of the coefficients positive numbers, which cannot be 0 when 
$x,y \in \mathbb{R}$. Therefore, the original system of polynomial equations does not have 
a non-trivial solution. It follows that $\overline{r} \leq 6$. 

When $r=5$, we retain the first five equations in the system described in the above display.
By choosing $x=y=1$, under lexicographical order $a \succ b \succ c \succ d \succ e$, 
we can verify that the Groebner bases contains a polynomial of $e$ with 
roots $e=\pm \sqrt{2}/3$ or $e=(-3 \pm \sqrt{2})/6$ while $a,b,c,d$ can be uniquely 
determined by $e$. Thus, system of polynomial equations 
\eqref{eqn:generalovefittedGaussianzero} has a non-trivial solution. 
It follows that $\overline{r}=6$. 

(iii) For the case $k-k_{0} \geq 3$, we choose $c_{1}=c_{2}=\ldots=c_{k-k_{0}+1}=1$, 
$a_{i}=b_{i}=0$ for all $4 \leq i \leq k-k_{0}+1$. Additionally, take $a_{1}=a_{2}=1$. 
Now, by choosing $r=6$ in system~\eqref{eqn:generalovefittedGaussianzero}, we 
can check by Groebner bases that this system of polynomial equations has a 
non-trivial solution. As a result, $\overline{r} \geq 7$.

\newpage

\bibliography{Nhat,NPB,Nguyen}

\newpage

\begin{center}
\textbf{APPENDIX I}
\end{center}
In this appendix, we give proofs of the following results: Theorem \ref{identifiability-multivariatecharacterization} regarding the characterization of strong identifiability 
in mixture models with matrix-variate parameters, 
Theorem \ref{theorem:identifiabilitytransformation} regarding
preservability of strong identifiability under transformation, 
Theorem \ref{theorem:exactfittedGamma} for exact-fitted Gamma mixtures, 
Theorem \ref{theorem:exactlocationgamma} for exact-fitted location-exponential mixtures, 
Theorem \ref{theorem:exactfittedskewnormal} for exact-fitted skew-Gaussian mixtures, 
Theorem \ref{theorem:Overfittedskewnormal} for over-fitted skew-Gaussian mixtures. 
Proofs of most propositions and some corollaries are also given. 
Proofs of Theorem~\ref{identifiability-univariatecharacterization},
Theorem~\ref{theorem:overfittedGamma}, 
and 
Theorem~\ref{theorem:nonconformantexactfittedskewnormal}
are quite similar to the ones that we have already mentioned above, and are deferred to Appendix II.
\section{Proofs of other main results}


\paragraph{PROOF OF PROPOSITION~\ref{proposition:unboundenesssupport}}
We choose $G_{n}=\mathop {\sum }\limits_{i=1}^{k_{0}+1}{p_{i}^{n}\delta_{(\theta_{i}^{n},\Sigma_{i}^{n})}} \in \mathcal{O}_{k}(\Theta \times \Omega)$ such that $(\theta_{i}^{n},\Sigma_{i}^{n})=(\theta_{i}^{0},\Sigma_{i}^{0})$ for $i=1,\ldots,k_0$, $\theta_{k_0+1}^{n}=\theta_{1}^{0}$, 
$\Sigma_{k_{0}+1}^{n} =\Sigma_{1}^{0}+\dfrac{\exp(n/r)}{n^{\alpha}}I_{d_{2}}$ where $\alpha=\dfrac{1}{2\beta}$. Additionally, $p_{1}^{n}=p_{1}^{0}-\exp(-n), p_{i}^{n}=p_{i}^{0}$ for all $2 \leq i \leq k_{0}$, and $p_{k_{0}+1}^{n}=\exp(-n)$. With this construction, we can check that $W_{r}^{\beta}(G,G_{0}) = d_{2}^{\beta/2}/\sqrt{n}$. Now, as $h^{2}(p_{G_{n}},p_{G_{0}}) \lesssim V(p_{G_{n}},p_{G_{0}})$, we have
\begin{eqnarray}
\exp\left(\dfrac{2}{W_{r}^{\beta}(G_{n},G_{0})}\right) h^{2}(p_{G},p_{G_{0}}) \lesssim \exp\left(-n+\dfrac{2\sqrt{n}}{d_{2}^{\beta/2}}\right)\int \limits_{x \in \mathcal{X}}{|f(x|\theta_{1}^{0},\Sigma_{k_{0}+1}^{n})-f(x|\theta_{1}^{0},\Sigma_{1}^{0})|}dx, \nonumber
\end{eqnarray}
which converges to 0 as $n \to \infty$. The conclusion of our proposition is proved.
 
\paragraph{PROOF OF COROLLARY~\ref{corollary-lowerbound}.}
By Theorem \ref{theorem-firstorder}, 
there are positive constants $\epsilon=\epsilon(G_{0})$ and
$C_0 = C_0(G_0)$ such that $V(p_G, p_{G_0}) \geq C_0 W_1(G,G_0)$
when $W_1(G,G_0) \leq \epsilon$. It remains to show that \\
$\inf \limits_{G \in \Gcal: W_1(G,G_0) > \epsilon} V(p_G,p_{G_0})/W_1(G,G_0) > 0$.
Assume the contrary, then we can find a sequence of 
$G_{n} \in \mathcal{G}$ and $W_{1}(G_{n},G_{0})>\epsilon$ such that 
$\dfrac{V(p_{G_{n}},p_{G_{0}})}{W_{1}(G_{n},G_{0})} \to 0$ as $n \to \infty$.
Since $\mathcal{G}$ is a compact set, we can find $G' \in \mathcal{G}$ and $W_{1}(G',G_{0})>\epsilon$ such that $G_{n} \to G'$ under $W_{1}$ metric. It implies that $W_{1}(G_{n},G_{0}) \to W_{1}(G',G_{0})$ as $n \to \infty$. As $G' \not \equiv G_{0}$, we have $\lim \limits_{n \to \infty}{W_{1}(G_{n},G_{0})}>0$. As a consequence, $V(p_{G_{n}},p_{G_{0}}) \to 0$ as $n \to \infty$.
From the hypothesis, $V(p_{G_{n}},p_{G'}) \leq C(\Theta,\Omega) W_{1}^{\alpha}(G_{n},G')$, 
so $V(p_{G_{n}},p_{G'}) \to 0$ as $W_{1}(G_{n},G') \to 0$. 
Thus, $V(p_{G'},p_{G_{0}})=0$ or equivalently $p_{G_{0}}=p_{G'}$ almost surely. 
From the first-order identifiability of family of density functions 
$\left\{f(x|\theta,\Sigma),\theta \in \Theta,\Sigma \in \Omega \right\}$, it 
implies that $G' \equiv G_{0}$, which is a contradiction. 
This completes the proof.

\comment{So far, under the exact-fitted case, we achieve the result regarding the lower bound of $V(p_{G},p_{G_{0}})$ in terms of first order Wasserstein distance between $G$ and $G_{0}$ by means of the first-order identifiabily of class of density functions $\left\{f(x|\theta,\Sigma),\theta \in \Theta,\Sigma \in \Omega \right\}$. However, when moving to the over-fitted case,  the first-order identifiability appears to be not sufficient to guarantee the lower bound of $V(p_{G},p_{G_{0}})$ in terms of Wasserstein distance between $G$ and $G_{0}$. Therefore, we need to impose stronger identifiability condition into the class of density functions $\left\{f(x|\theta,\Sigma),\theta \in \Theta,\Sigma \in \Omega \right\}$. 
}


\subsection{Characterization of strong identifiability}

\label{proof-identifiability}

\paragraph{PROOF OF THEOREM~\ref{identifiability-multivariatecharacterization}.} 
We present the proof for part (a). The proof for other parts are similar and left to
Appendix II.
Assume that for given $k \geq 1$ and $k$ different tuples $(\theta_{1},\Sigma_{1},m_{1}),\ldots,
(\theta_{k},\Sigma_{k},m_{k})$, we can find $\alpha_{j} \in \mathbb{R}$, $\beta_{j} \in \mathbb{R}^{d}$, symmetric matrices $\gamma_{j} \in \mathbb{R}^{d \times d}$, and $\eta_{j} \in \mathbb{R}$, for $j=1,\ldots,k$ 
such that:
\begin{eqnarray}
\mathop {\sum }\limits_{j=1}^{k}{\alpha_{j}f(x|\theta_{j},\Sigma_{j},m_{j})+\beta_{j}^{T}\dfrac{\partial{f}}{\partial{\theta}}(x|\theta_{j},\Sigma_{j},m_{j})+\trace 
\biggr (\dfrac{\partial{f}}{\partial{\Sigma}}(x|\theta_{j},\Sigma_{j},m_{j})^{T}\gamma_{j}
\biggr )
+\dfrac{\partial{f}}{\partial{m}}(x|\theta_{j},\Sigma_{j},m_{j})}=0, \nonumber
\end{eqnarray}
Substituting the first derivatives of $f$ to get
\begin{eqnarray}
\mathop {\sum }\limits_{j=1}^{k}{
\biggr \{\alpha_{j}'+ 
\biggr ((\beta_{j}')^{T}(x-\theta_{j})+(x-\theta_{j})^{T}\gamma_{j}'(x-\theta_{j})
\biggr )
\biggr [(x-\theta_{j})^{T}\Sigma_{j}^{-1}(x-\theta_{j}) 
\biggr ]^{m_{j}-1}} & + & \nonumber \\
\eta_{j}'\log((x-\theta_{j})^{T}\Sigma_{j}^{-1}(x-\theta_{j}))\biggr \}
\exp \left(-\biggr [(x-\theta_{j})^{T}\Sigma_{j}^{-1}(x-\theta_{j})
\biggr ]^{m_{j}}\right) & = & 0, \label{eq:mulgaussian1}
\end{eqnarray} 
where \\
$\alpha_{j}'=\dfrac{2\alpha_{j}m_{j}\Gamma(d/2)-m_{j}\Gamma(d/2)\trace(\Sigma_{j}^{-1}\gamma_{j})+2\eta_{j}\Gamma(d/2)\left(1-\dfrac{d}{2m_{j}}\psi\left(\dfrac{d}{2m_{j}}\right)\right)}{2\pi^{d/2}\Gamma(d/(2m_{j}))|\Sigma_{j}|^{1/2}}$, \\ $\beta_{j}'=\dfrac{2m_{j}^{2}\Gamma(d/2)}{\pi^{d/2}\Gamma(d/(2m_{j}))|\Sigma_{j}|^{1/2}}\Sigma_{j}^{-1}\beta_{j}$,  $\gamma_{j}'=\dfrac{m_{j}^{2}\Gamma(d/2)}{\pi^{d/2}\Gamma(d/(2m_{j}))|\Sigma_{j}|^{1/2}}\Sigma_{j}^{-1}\gamma_{j}\Sigma_{j}^{-1}$, and \\ $\eta_{j}'=\dfrac{-m_{j}\eta_{j}\Gamma(d/2)}{\pi^{d/2}\Gamma(d/(2m_{j}))|\Sigma_{j}|^{1/2}}$.\\

Without loss of generality, assume $m_{1} \leq m_{2} \leq \ldots \leq m_{k}$. 
Let $\overline{i} \in [1,k]$ be the maximum index such that $m_{1}=m_{\overline{i}}$. 
As the tuples $(\theta_{i},\Sigma_{i},m_{i})$ are distinct, so are the pairs
$(\theta_{1},\Sigma_{1}),\ldots,(\theta_{\overline{i}},\Sigma_{\overline{i}})$.
In what follows, we represent $x$ by
$x=x_{1}x'$ where $x_{1}$ is scalar and $x' \in \mathbb{R}^{d}$.
Define 
\[a_{i}=(x')^{T}\gamma_{i}'x', \;\;
b_{i}=\left[(\beta_{i}')^{T}-2\theta_{i}^{T}\gamma_{i}'\right]x', \;\;
c_{i}=\theta_{i}^{T}\gamma_{i}'\theta_{i}-(\beta_{i}')^{T}\theta_{i},\]
\[d_{i}=(x')^{T}\Sigma_{i}^{-1}x',\;\; 
e_{i}=-2(x')^{T}\Sigma_{i}^{-1}\theta_{i}, \;\;
f_{i}=\theta_{i}^{T}\Sigma_{i}^{-1}\theta_{i}.\]
Borrowing a technique from~\cite{Yakowitzspragins-1968}, since $(\theta_{1},\Sigma_{1}),\ldots,(\theta_{\overline{i}},\Sigma_{\overline{i}})$ are distinct, we have two possibilities:
\paragraph{(i)} If $\Sigma_{j}$ are the same for all $1 \leq j \leq \overline{i}$, then $\theta_{1},\ldots,\theta_{\overline{i}}$ are distinct. For any $i<j$, denote $\Delta_{ij}=\theta_{i}-\theta_{j}$. Note that if $\displaystyle{x' \notin \mathop {\bigcup }\limits_{1 \leq i<j \leq \overline{i}}{\left\{u \in \mathbb{R}^{d}:u^{T}\Delta_{ij}=0\right\}}}$, which is a finite union of hyperplanes, then $(x')^{T}\theta_{1},\ldots,(x')^{T}\theta_{\overline{i}}$ are distinct. Hence, if we choose $x'  \in \mathbb{R}^{d}$ outside this finite union of hyperplanes, we have $((x')^{T}\theta_{1},(x')^{T}\Sigma_{1}x'),\ldots,((x')^{T}\theta_{\overline{i}},(x')^{T}\Sigma_{\overline{i}}x')$ are distinct.
\paragraph{(ii)} If $\Sigma_{j}$ are not the same for all $1 \leq j \leq \overline{i}$, then we assume without loss of generality that $\Sigma_{1},\ldots,\Sigma_{m}$ are the only distinct matrices from $\Sigma_{1},\ldots,\Sigma_{\overline{i}}$, where $m \leq \overline{i}$. Denote $\delta_{ij}=\Sigma_{i}-\Sigma_{j}$ as $1 \leq i<j \leq m$, then as $\displaystyle{x' \notin \mathop {\bigcup }\limits_{1 \leq i<j \leq m}{\left\{u \in \mathbb{R}^{d}:u^{T}\delta_{ij}u=0 \right\}}}$, we have $(x')^{T}\Sigma_{1}x',\ldots,(x')^{T}\Sigma_{m}x'$ are distinct. Therefore, if $\displaystyle{x' \notin \mathop {\bigcup }\limits_{1 \leq i<j \leq m}{\left\{u \in \mathbb{R}^{d}:u^{T}\delta_{ij}u=0 \right\}}}$, which is finite union of conics, $((x')^{T}\theta_{1},(x')^{T}\Sigma_{1}x'),\ldots,
((x')^{T}\theta_{m},(x')^{T}\Sigma_{m}x')$ are distinct. Additionally, 
for any $\theta_{j}$ where $m+1 \leq j \leq \overline{i}$ that shares 
the same $\Sigma_{i}$ where $1 \leq i \leq m$, using the argument in the first case, we can choose 
$x'$ outside a finite hyperplane such that these $(x')^{T}\theta_{j}$ are again distinct. Hence, for 
$x'$ outside a finite union of conics and hyperplanes,
$((x')^{T}\theta_{1}, (x')^{T}\Sigma_{1}x'),\ldots, ((x')^{T}\theta_{\overline{i}},(x')^{T}\Sigma_{\overline{i}}x')$ are all different.

Combining these two cases, we can find a set 
$D$, which is a finite union of conics and hyperplanes, such that for 
$x' \notin D$, $((x')^{T}\theta_{1},(x')^{T}\Sigma_{1}x'),\ldots((x')^{T}\theta_{\overline{i}},(x')^{T}\Sigma_{\overline{i}}x')$ are distinct. Thus, $(d_{i},e_{i})$ are different as $1 \leq i \leq \overline{i}$. 

Choose $d_{i_{1}}=\mathop {\min }\limits_{1 \leq i \leq \overline{i}}{\left\{d_{i}\right\}}$. Denote $J=\left\{1 \leq i \leq \overline{i}: d_{i}=d_{i_{1}}\right\}$. Choose $1 \leq i_{2} \leq \overline{i}$ such that $e_{i_{2}}=\mathop {\max }\limits_{i \in J}{\left\{e_{i}\right\}}$. Multiply both sides of \eqref{eq:mulgaussian1} with $\exp-(d_{i_{2}}x_{1}^{2}+e_{i_{2}}x_{1}+f_{i_{2}})^{m_{i_{2}}}$, we get
\begin{eqnarray}
\alpha_{i_{2}}'+(a_{i_{2}}x_{1}^{2}+b_{i_{2}}x_{1}+c_{i_{2}})(d_{i_{2}}x_{1}^{2}+e_{i_{2}}x_{1}+f_{i_{2}})^{m_{i_{2}}-1}+\eta_{i_{2}}'\log(d_{i_{2}}x_{1}^{2}+e_{i_{2}}x_{1}+f_{i_{2}})  & + & \nonumber \\
\mathop {\sum }\limits_{j \neq i_{2}} \biggr \{\alpha_{j}' +(a_{j}x_{1}^{2}+b_{j}x_{1}+c_{j})(d_{j}x_{1}^{2}+e_{j}x_{1}+f_{j})^{m_{i_{2}}-1}+\eta_{j}'\log(d_{j}x_{1}^{2}+e_{j}x+f_{j})\biggr \} & \times & \nonumber \\
\exp\left[(d_{i_{2}}x_{1}^{2}+e_{i_{2}}x_{1}+f_{i_{2}})^{m_{i_{2}}}-(d_{j}x_{1}^{2}+e_{j}x_{1}+f_{j})^{m_{j}}\right] & = & 0. \quad \quad \label{eq:mulgaussian12}
\end{eqnarray}

Note that if $j \in J \backslash {\left\{i_{2}\right\}}$, $d_{j}=d_{i_{2}}$, $m_{j}=m_{i_{2}}$, and $e_{j} > e_{i_{2}}$. So,
\begin{eqnarray}
(d_{i_{2}}x_{1}^{2}+e_{i_{2}}x_{1}+f_{i_{2}})^{m_{i_{2}}}-(d_{j}x_{1}^{2}+e_{j}x_{1}+f_{j})^{m_{j}} \lesssim -x_{1} \ \text{as} \ x_{1} \ \text{is large enough}. \nonumber
\end{eqnarray}
This implies that when $x_{1} \to \infty$, 
\begin{eqnarray}
A_{1}(x) = \mathop {\sum }\limits_{j \neq J \backslash {\left\{i_{2}\right\}}} \biggr \{\alpha_{j}' +(a_{j}x_{1}^{2}+b_{j}x_{1}+c_{j})(d_{j}x_{1}^{2}+e_{j}x_{1}+f_{j})^{m_{i_{2}}-1}+\eta_{j}'\log(d_{j}x_{1}^{2}+e_{j}x +  f_{j})\biggr \} \times  \nonumber \\
\exp\left[(d_{i_{2}}x_{1}^{2}+e_{i_{2}}x_{1}+f_{i_{2}})^{m_{i_{2}}}-(d_{j}x_{1}^{2}+e_{j}x_{1}+f_{j})^{m_{j}}\right] \to  0. \nonumber
\end{eqnarray}
On the other hand, if 
$j \notin J$ and $1 \leq j \leq \overline{i}$, then $d_{j}>d_{i_{2}}$ and $m_{i_{2}}=m_{j}$. So,
\begin{eqnarray}
(d_{i_{2}}x_{1}^{2}+e_{i_{2}}x_{1}+f_{i_{2}})^{m_{i_{2}}}-(d_{j}x_{1}^{2}+e_{j}x_{1}+f_{j})^{m_{j}} \lesssim -x_{1}^{2m_{i_{2}}} \ \text{as} \ x_{1} \ \text{is large enough}. \nonumber
\end{eqnarray}
This implies that when $x_{1} \to \infty$, 
\begin{eqnarray}
A_{2}(x) = \mathop {\sum }\limits_{\substack{j \notin J, \\ 1 \leq j \leq \overline{i}}} \biggr \{\alpha_{j}' +(a_{j}x_{1}^{2}+b_{j}x_{1}+c_{j})(d_{j}x_{1}^{2}+e_{j}x_{1}+f_{j})^{m_{i_{2}}-1}+\eta_{j}'\log(d_{j}x_{1}^{2}+e_{j}x+f_{j})\biggr \}  \times  \nonumber \\
\exp\left[(d_{i_{2}}x_{1}^{2}+e_{i_{2}}x_{1}+f_{i_{2}})^{m_{i_{2}}}-(d_{j}x_{1}^{2}+e_{j}x_{1}+f_{j})^{m_{j}}\right]  \to  0. \nonumber
\end{eqnarray}
Or else, if $j>\overline{i}$, then $m_{j}>m_{i_{2}}$. 
So, $(d_{i_{2}}x_{1}^{2}+e_{i_{2}}x_{1}+f_{i_{2}})^{m_{i_{2}}}-(d_{j}x_{1}^{2}+e_{j}x_{1}+f_{j})^{m_{j}} \lesssim -x_{1}^{2m_{j}}$. As a result,
\begin{eqnarray}
A_{3}(x) = \mathop {\sum }\limits_{j > \overline{i}} \biggr \{\alpha_{j}' +(a_{j}x_{1}^{2}+b_{j}x_{1}+c_{j})(d_{j}x_{1}^{2}+e_{j}x_{1}+f_{j})^{m_{i_{2}}-1}+\eta_{j}'\log(d_{j}x_{1}^{2}+e_{j}x+f_{j})\biggr \}  \times  \nonumber \\
\exp\left[(d_{i_{2}}x_{1}^{2}+e_{i_{2}}x_{1}+f_{i_{2}})^{m_{i_{2}}}-(d_{j}x_{1}^{2}+e_{j}x_{1}+f_{j})^{m_{j}}\right]  \to  0. \nonumber
\end{eqnarray}
Now, by letting $x_1 \to \infty$,
\begin{eqnarray}
\mathop {\sum }\limits_{j \neq i_{2}} \biggr \{\alpha_{j}' +(a_{j}x_{1}^{2}+b_{j}x_{1}+c_{j})(d_{j}x_{1}^{2}+e_{j}x_{1}+f_{j})^{m_{i_{2}}-1}+\eta_{j}'\log(d_{j}x_{1}^{2}+e_{j}x+f_{j})\biggr \}  & \times & \nonumber \\
\exp\left[(d_{i_{2}}x_{1}^{2}+e_{i_{2}}x_{1}+f_{i_{2}})^{m_{i_{2}}}-(d_{j}x_{1}^{2}+e_{j}x_{1}+f_{j})^{m_{j}}\right] = A_{1}(x) + A_{2}(x) + A_{3}(x) & \to & 0. \quad  \label{eq:mulgaussian13}
\end{eqnarray}
Combing \eqref{eq:mulgaussian12} and \eqref{eq:mulgaussian13}, we obtain that as $x_{1} \to \infty$
\begin{eqnarray}
\alpha_{i_{2}}'+(a_{i_{2}}x_{1}^{2}+b_{i_{2}}x_{1}+c_{i_{2}})(d_{i_{2}}x_{1}^{2}+e_{i_{2}}x_{1}+f_{i_{2}})^{m_{i_{2}}-1}+\eta_{i_{2}}'\log(d_{i_{2}}x_{1}^{2}+e_{i_{2}}x_{1}+f_{i_{2}}) \to 0.  \nonumber
\end{eqnarray}

The only possibility for this result to happen is $a_{i_{2}}=b_{i_{2}}=\eta_{i_{2}}'=0$. Or,
equivalently, $(x')^{T}\gamma_{i_{2}}'x'=\left[(\beta_{i}')^{T}-2\theta_{i_{2}}^{T}\gamma_{i_{2}}'\right]x'=0$. If $\gamma_{i_{2}}' \neq 0$, we can choose the element $x' \notin D$ lying outside the hyperplane $\left\{u \in \mathbb{R}^{d}:u^{T}\gamma_{i_{2}}'u=0\right\}$. It means that $(x')^{T}\gamma_{i_{2}}'x' \neq 0$, which is a contradiction. Therefore, $\gamma_{i_{2}}'=0$. It implies that $(\beta_{i_{2}}')^{T}x'=0$. If $\beta_{i_{2}}' \neq 0$, we can choose $x' \notin D$ such that $(\beta_{i_{2}}')^{T}x' \neq 0$. Hence, $\beta_{i_{2}}'=0$. With these results, $\alpha_{i_{2}}'=0$. Overall, we obtain $\alpha_{i_{2}}'=\beta_{i_{2}}'=\gamma_{i_{2}}'=\eta_{i_{2}}'=0$. Repeating the same argument to the remained parameters $\alpha_{j}',\beta_{j}',\gamma_{j}',\eta_{j}'$ and we get $\alpha_{j}'=\beta_{j}'=\gamma_{j}'=\eta_{j}'=0$ for $1 \leq j \leq k$. It is also equivalent that $\alpha_{j}=\beta_{j}=\gamma_{j}=\eta_{j}=0$ for all $1 \leq j \leq k$.

This concludes the proof of part (a) of our theorem.

\paragraph{PROOF OF THEOREM~\ref{theorem:identifiabilitytransformation}.}
The proof is a straightforward application of the chain rule.
\paragraph{``If'' direction:} 
Let $k \geq 1$ and let $(\eta_{1}^{*},\Lambda_{1}^{*}),(\eta_{2}^{*},\Lambda_{2}^{*})\ldots,(\eta_{k}^{*},\Lambda_{k}^{*}) \in \Theta^{*} \times \Omega^{*}$ be $k$ different pairs. 
Suppose there are
$\alpha_{i} \in \mathbb{R},\beta_{i} \in \mathbb{R}^{d_{1}}$, and symmetric matrices $\gamma_{i} \in \mathbb{R}^{d_{2} \times d_{2}}$ such that
\begin{eqnarray}
\mathop {\sum }\limits_{i=1}^{k}{\alpha_{i}g(x|\eta_{i}^{*},\Lambda_{i}^{*})+\beta_{i}^{T}\dfrac{\partial{g}}{\partial{\eta}}(x|\eta_{i}^{*},\Lambda_{i}^{*})+\trace\left(\dfrac{\partial{g}}{\partial{\Lambda}}(x|\eta_{i}^{*},\Lambda_{i}^{*})^{T}\gamma_{i}\right)}=0 \ \ \text{for almost all} \ x. \quad  \label{eq:bijective1}
\end{eqnarray}

Let $(\theta_{i},\Sigma_{i}):= T(\eta_{i}^{*},\Lambda_{i}^{*})$ for $i=1,\ldots, k$. 
Since $T$ is bijective,
$(\theta_{1},\Sigma_{1}),(\theta_{2},\Sigma_{2}),\ldots,(\theta_{k},\Sigma_{k})$ are distinct.
By the chain rule,
\begin{eqnarray}
\dfrac{\partial{g}}{\partial{\eta_{i}}}(x|\eta,\Lambda) & = &\mathop {\sum }\limits_{l=1}^{d_{1}}\dfrac{\partial{f}}{\partial{\theta_{l}}}(x|\theta,\Sigma)\dfrac{\partial{\theta_{l}}}{\partial{\eta_{i}}}+\mathop {\sum }\limits_{1 \leq u,v \leq d_{2}}{\dfrac{\partial{f}}{\partial{\Sigma_{uv}}}(x|\theta,\Sigma)\dfrac{\partial{\Sigma_{uv}}}{\partial{\eta_{i}}}} \nonumber \\
& = & \mathop {\sum }\limits_{l=1}^{d_{1}}\dfrac{\partial{f}}{\partial{\theta_{l}}}(x|\theta,\Sigma)\dfrac{\partial{\left[T_{1}(\eta,\Lambda)\right]_{l}}}{\partial{\eta_{i}}}+\mathop {\sum }\limits_{1 \leq u,v \leq d_{2}}{\dfrac{\partial{f}}{\partial{\Sigma_{uv}}}(x|\theta,\Sigma)\dfrac{\partial{\left[T_{2}(\eta,\Lambda)\right]_{uv}}}{\partial{\eta_{i}}}} \nonumber
\end{eqnarray}
and similarly,
\begin{eqnarray}
\dfrac{\partial{g}}{\partial{\Lambda_{ij}}}(x|\eta,\Lambda)= \mathop {\sum }\limits_{l=1}^{d_{1}}\dfrac{\partial{f}}{\partial{\theta_{l}}}(x|\theta,\Sigma)\dfrac{\partial{\left[T_{1}(\eta,\Lambda)\right]_{l}}}{\partial{\Lambda_{ij}}}+\mathop {\sum }\limits_{1 \leq u,v \leq d_{2}}{\dfrac{\partial{f}}{\partial{\Sigma_{uv}}}(x|\theta,\Sigma)\dfrac{\partial{\left[T_{2}(\eta,\Lambda)\right]_{uv}}}{\partial{\Lambda_{ij}}}}, \nonumber
\end{eqnarray}
where $\eta=(\eta_{1},\ldots,\eta_{d_{1}})$ and $\Sigma=[\Sigma_{ij}]$ where $1 \leq i,j \leq d_{2}$.
Equation \eqref{eq:bijective1} can be rewritten accordingly as follows
\begin{eqnarray}
\mathop {\sum }\limits_{i=1}^{k}{\alpha_{i}f(x|\theta_{i},\Sigma_{i})+(\beta_{i}')^{T}\dfrac{\partial{f}}{\partial{\theta}}(x|\theta_{i},\Sigma_{i})+\trace\left(\dfrac{\partial{f}}{\partial{\Sigma}}(x|\theta_{i},\Sigma_{i})^{T}\gamma_{i}'\right)}=0 \ \ \text{for almost all} \ x. \quad  \label{eq:bijective2}
\end{eqnarray}
where $\beta_{i}'=((\beta_{i}')^{1},\ldots,(\beta_{i}')^{d_{1}})$, $\gamma_{i}'=[\gamma_{i}']^{uv}$, $\eta_{i}=((\eta_{i})^{1},\ldots,(\eta_{i})^{d_{1}})$, $\Lambda_{i}=[\Lambda_{i}]^{uv}$, $\beta_{i}=(\beta_{i}^{1},\ldots,\beta_{i}^{d_{1}})$, $\gamma_{i}=[\gamma_{i}]^{uv}$, and for all $1 \leq j \leq d_{1}$
\begin{eqnarray}
(\beta_{i}')^{j}=\mathop {\sum }\limits_{h=1}^{d_{1}}{\beta_{i}^{h}\dfrac{\partial{\left[T_{1}(\eta_{i}^{*},\Lambda_{i}^{*})\right]_{j}}}{\partial{(\eta_{i})^{h}}}}+\mathop {\sum }\limits_{1 \leq u,v \leq d_{2}}{\gamma_{i}^{uv}\dfrac{\partial{\left[T_{1}(\eta_{i}^{*},\Lambda_{i}^{*})\right]_{j}}}{\partial{(\Lambda_{i})^{uv}}}} \nonumber
\end{eqnarray}
and for all $1 \leq j,l \leq d_{2}$
\begin{eqnarray}
(\gamma_{i}')^{jl}=\mathop {\sum }\limits_{h=1}^{d_{1}}{\beta_{i}^{h}\dfrac{\partial{\left[T_{2}(\eta_{i}^{*},\Lambda_{i}^{*})\right]_{jl}}}{\partial{(\eta_{i})^{h}}}}+\mathop {\sum }\limits_{1 \leq u,v \leq d_{2}}{\gamma_{i}^{uv}\dfrac{\partial{\left[T_{2}(\eta_{i}^{*},\Lambda_{i}^{*})\right]_{jl}}}{\partial{(\Lambda_{i})^{uv}}}}. \nonumber
\end{eqnarray}

Given that
$\left\{f(x|\theta,\Sigma),\theta \in \Theta, \Sigma \in \Omega\right\}$
is identifiable in the first order, Eq.~\eqref{eq:bijective2} 
entails that $\alpha_{i}=0,\beta_{i}'=\vec{0} \in \mathbb{R}^{d_{1}}$, and $\gamma_{i}'=\vec{0} \in \mathbb{R}^{d_{2} \times d_{2}}$. From the definition of modified Jacobian matrix $J$, the equations $\beta_{i}'=\vec{0}$ and $\gamma_{i}'=\vec{0}$ are equivalent to system of equations $J(\eta_{i}^{*},\Lambda_{i}^{*})\tau_{i}=0$, where $\tau_{i}^{T}=(\beta_{i},\gamma_{i}^{11},\ldots,\gamma_{i}^{1d_{2}},\gamma_{i}^{21},\ldots.,\gamma_{i}^{2d_{2}},\ldots.,\gamma_{i}^{d_{2}1},\ldots,\gamma_{i}^{d_{2}d_{2}}) \in \mathbb{R}^{d_{1}+d_{2}^{2}}$. Since $|J(\eta_{i}^{*},\Lambda_{i}^{*})| \neq 0$, the above system of equations has unique solution $\tau_{i}=0$ for all $1 \leq i \leq k$. These results imply that $\beta_{i}=\vec{0} \in \mathbb{R}^{d_{1}}$ and $\gamma_{i}=\vec{0} \in \mathbb{R}^{d_{2} \times d_{2}}$. Thus, $g$ is also identifiable in the first order.

\paragraph{``Only if'' direction.}
Assume by contrary that the modified Jacobian matrix $J(\eta,\Lambda)$ is not non-singular for all $(\eta,\Lambda) \in \Theta^{*} \times \Omega^{*}$. Then, we can find $(\eta_{0},\Lambda_{0}) \in \Theta^{*} \times \Omega^{*}$ such that $J(\eta_{0},\Lambda_{0})$ is singular matrix. 
Choose $k=1$ and assume that we can find $\alpha_{1} \in \mathbb{R},\beta_{1} \in \mathbb{R}^{d_{1}}$, and symmetric matrix $\gamma_{1} \in \mathbb{R}^{d_{2} \times d_{2}}$ such that:
\begin{eqnarray}
\alpha_{1}g(x|\eta_{0},\Lambda_{0})+\beta_{1}^{T}\dfrac{\partial{g}}{\partial{\eta}}(x|\eta_{0},\Lambda_{0})+\trace\left(\dfrac{\partial{g}}{\partial{\Lambda}}(x|\eta_{0},\Lambda_{0})^{T}\gamma_{1}\right)=0 \ \text{ for almost all} \ x. \nonumber
\end{eqnarray} 
The first-order identifiability of class $\left\{g(x|\eta,\Lambda),\eta \in \Theta^{*},\Lambda \in \Omega^{*}\right\}$ implies that $\alpha_{1}=0$, $\beta_{1}=\vec{0} \in \mathbb{R}^{d_{1}}$, and $\gamma_{1}=\vec{0} \in \mathbb{R}^{d_{2} \times d_{2}}$ are the only possibility for the above equation to hold. However, 
by the same argument as in the first part of the proof, we may rewrite the above equation as
\begin{eqnarray}
\alpha_{1}f(x|\theta_{0},\Sigma_{0})+(\beta_{1}')^{T}\dfrac{\partial{f}}{\partial{\theta}}(x|\theta_{0},\Sigma_{0})+\trace\left(\dfrac{\partial{f}}{\partial{\Sigma}}(x|\theta_{0},\Sigma_{0})^{T}\gamma_{1}'\right)=0 \ \text{for almost all} \ x,\nonumber
\end{eqnarray}
where $T(\eta_{0},\Lambda_{0})=(\theta_{0},\Sigma_{0})$, and $\beta_{1}'$, $\gamma_{1}'$ 
have the same formula as given above. 
The first-order identifiability of $\left\{f(x|\theta,\Sigma),\theta \in \Theta,\Sigma \in \Omega\right\}$ implies that $\beta_{1}'=\vec{0} \in \mathbb{R}^{d_{1}}$ and $\gamma_{1}'=\vec{0} \in \mathbb{R}^{d_{2} \times d_{2}}$. The last equation leads to the system of equations $J(\eta_{0},\Lambda_{0})\tau=0$, where 
\[\tau^{T}=\left(\beta_{1},\gamma_{1}^{11},\ldots,\gamma_{1}^{1d_{2}},\gamma_{1}^{21},\ldots,\gamma_{1}^{2d_{2}},\ldots,\gamma_{1}^{d_{2}1},\ldots,\gamma_{1}^{d_{2}d_{2}}\right).\]
However, the non-singularity of matrix $J(\eta_{0},\Lambda_{0})$ leads to 
non-uniquesness of the solution $\tau$ of this system of equations.
This contradicts with the uniqueness of the solution $\alpha_{1}=0$, $\beta_{1}=\vec{0} \in \mathbb{R}^{d_{1}}$, and $\gamma_{1}=\vec{0} \in \mathbb{R}^{d_{2} \times d_{2}}$. 
The proof is complete.

\subsection{Over-fitted location-covariance Gaussian mixtures}

\begin{lemma} \label{lemma:multivariatenormaldistribution} Let $\left\{f(x|\theta, \Sigma), \theta \in \mathbb{R}^{d}, \Sigma \in S_{d}^{++} \right\}$ be a class of multivariate Gaussian distribution. Then, $\dfrac{\partial^{2}{f}}{\partial{\theta}^{2}}(x|\theta,\Sigma)=2\dfrac{\partial{f}}{\partial{\Sigma}}(x|\theta,\Sigma)$ for all $\theta \in \mathbb{R}^{d}$ and $\Sigma \in S_{d}^{++}$.
\end{lemma}
\begin{proof}
Direct calculation yields
\begin{eqnarray}
\dfrac{\partial^{2}{f}}{\partial{\theta^{2}}}(x|\theta,\Sigma) & = & \dfrac{1}{(\sqrt{2\pi})^{d}|\Sigma|^{1/2}}\left[-\Sigma^{-1}+\Sigma^{-1}(x-\theta)(x-\theta)^{T}\Sigma^{-1}\right]\exp\left(-\dfrac{(x-\theta)^{T}\Sigma^{-1}(x-\theta)}{2}\right), \nonumber \\
\dfrac{\partial{f}}{\partial{\Sigma}}(x|\theta,\Sigma) & = & \dfrac{1}{2(\sqrt{2\pi})^{d}|\Sigma|^{1/2}}\left[-\Sigma^{-1}+\Sigma^{-1}(x-\theta)(x-\theta)^{T}\Sigma^{-1}\right]\exp\left(-\dfrac{(x-\theta)^{T}\Sigma^{-1}(x-\theta)}{2}\right). \nonumber
\end{eqnarray}
From these results, we can easily check the conclusion of our lemma.
\end{proof}

\paragraph{PROOF OF PROPOSITION~\ref{theorem:Gaussianoverfittedbytwo}.}
We only consider the case $k-k_{0}=1$ (the proof for the case $k-k_{0}=2$ 
is rather similar, and deferred to Appendix II). 
As in the proof of Theorem \ref{theorem:generaloverfittedGaussian}, 
it suffices to show for $d=1$ that
\begin{eqnarray}
\mathop {\lim }\limits_{\epsilon \to 0}{\mathop {\inf }\limits_{G \in \mathcal{O}_{k}(\Theta)}{\left\{\mathop {\sup }\limits_{x \in \mathcal{X}}{|p_{G}(x)-p_{G_{0}}(x)|}/W_{4}^{4}(G,G_{0}):W_{4}(G,G_{0}) \leq \epsilon \right\}}} > 0. \label{eqn:overfittedGaussianbytwofirst}
\end{eqnarray}
Denote $v=\sigma^{2}$. Assume that the above result does not hold, i.e we can find a sequence of $G_{n}=\mathop {\sum }\limits_{i=1}^{k_{0}+m}{\mathop {\sum }\limits_{j=1}^{s_{i}}{p_{ij}^{n}\delta_{(\theta_{ij}^{n},v_{ij}^{n})}}} \to G_{0}$ in $W_{4}$ where $(p_{ij}^{n},\theta_{ij}^{n},v_{ij}^{n}) \to (p_{i}^{0},\theta_{i}^{0},v_{i}^{0})$ for all $1 \leq i \leq k_{0}+m$, $1 \leq j \leq s_{i}$ and $p_{i}^{0}=0$ as $k_{0}+1 \leq i \leq k_{0}+m$. As $k-k_{0}=1$, we have $m \leq 1$. 
Repeating the same arguments as the
proof of Theorem \ref{theorem:generaloverfittedGaussian} up to Step 8, and noting
that $\mathop {\sum }\limits_{i=1}^{k_{0}+m}{\mathop {\sum }\limits_{j=1}^{s_{i}}{p_{ij}^{n}|\Delta v_{ij}^{n}|^{4}}}/d(G_{n},G_{0}) \to 0$, we can find $i^{*} \in \left\{1,2,\ldots,k_{0}+m\right\}$ such that as long as $1 \leq \alpha \leq 4$
\begin{eqnarray}
F_{\alpha}'(\theta_{i^{*}}^{0},v_{i^{*}}^{0})=\dfrac{\mathop {\sum }\limits_{j=1}^{s_{i^{*}}}{p_{i^{*}j}^{n}(|\Delta \theta_{i^{*}j}^{n}|^{4}+|\Delta v_{i^{*}j}^{n}|^{4})}}{\mathop {\sum }\limits_{j=1}^{s_{i^{*}}}{p_{i^{*}j}^{n}|\Delta \theta_{i^{*}j}^{n}|^{4}}}F_{\alpha}(\theta_{i^{*}}^{0},v_{i^{*}}^{0}) = \dfrac{\mathop {\sum }\limits_{j=1}^{s_{i^{*}}}{p_{i^{*}j}^{n}\mathop {\sum }\limits_{n_{1},n_{2}}{\dfrac{(\Delta \theta_{i^{*}j}^{n})^{n_{1}}(\Delta v_{i^*j}^{n})^{n_{2}}}{n_{1}!n_{2}!}}}}{\mathop {\sum }\limits_{j=1}^{s_{i^{*}}}{p_{i^{*}j}^{n}|\Delta \theta_{i^{*}j}^{n}|^{4}}} \to 0, \label{eqn:overfittedGaussianbytwo-one}
\end{eqnarray}
where $n_{1}+2n_{2}=\alpha$ and $1 \leq \alpha \leq 4$. As $i^{*} \in \left\{1,2,\ldots,k_{0}+m\right\}$, we have $i^{*} \in \left\{1,\ldots,k_{0}\right\}$ or $i^{*} \in \left\{k_{0}+1,\ldots,k_{0}+m\right\}$. Firstly, we assume that $i^{*} \in \left\{1,\ldots,k_{0}\right\}$. Without loss of generality, let $i^{*}=1$. 
Since $s_{1} \leq k-k_0+1 = 2$, there are two possibilities.

\paragraph{Case 1.} If $s_{1}=1$, then $F_{1}'(\theta_{1}^{0},v_{1}^{0})=\Delta \theta_{11}^{n}/|\Delta \theta_{11}^{n}|^{4} \not \to 0$, which is a contradiction. 

\paragraph{Case 2.} If $s_{1}=2$, without loss of generality, we assume that $p_{11}^{n}|\Delta \theta_{11}^{n}| \leq p_{12}^{n} |\Delta \theta_{12}^{n}|$ for infinitely many $n$, 
which we can assume to hold for all $n$ (by choosing the subsequence). 
Since $p_{11}^{n}(\Delta \theta_{11}^{n})^{4}+p_{12}^{n}(\Delta \theta_{12}^{n})^{4}>0$, we obtain $\theta_{12}^{n} \neq 0$ for all $n$. If $\Delta \theta_{11}^{n}=0$ for infinitely many $n$, then $F_{1}'(\theta_{1}^{0},v_{1}^{0})=\Delta \theta_{12}^{n}/(\Delta \theta_{12}^{n})^{4} \not \to 0$, which is a contradiction. 
Therefore, we may assume $\theta_{11}^{n} \neq 0$ for all $n$.
Let $a: = \mathop {\lim }\limits_{n \to \infty}{p_{11}^{n}\Delta \theta_{11}^{n}/p_{12}^{n}
\Delta \theta_{12}^{n}}\in [-1,1]$.
Dividing both the numerator and denominator of $F_{1}'(\theta_{1}^{0},v_{1}^{0})$ by $p_{12}^{n}\Delta \theta_{12}^{n}$ and letting $n \to \infty$, we obtain $a=-1$. Consider
the following scenarios regarding $p_{11}^{n}/p_{12}^{n}$:

\paragraph{(i)} If $p_{11}^{n}/p_{12}^{n} \to \infty$, then $\Delta \theta_{11}^{n}/\Delta \theta_{12}^{n} \to 0$. Since $\Delta \theta_{11}^{n},\Delta \theta_{12}^{n} \neq 0$, denote $\Delta v_{11}^{n}=k_{1}^{n}(\Delta \theta_{11}^{n})^{2}$, $\Delta v_{12}^{n}=k_{2}^{n} (\Delta \theta_{12}^{n})^{2}$ for all $n$. Now, by dividing the numerator and denominator of $F_{2}'(\theta_{1}^{0},v_{1}^{0}), F_{3}'(\theta_{1}^{0},v_{1}^{0}),$\\$ F_{4}'(\theta_{1}^{0},v_{1}^{0})$ by $p_{12}^{n}(\Delta \theta_{12}^{n})^{2}$, $p_{12}^{n}(\Delta \theta_{12}^{n})^{3}$, and $p_{12}^{n}(\Delta \theta_{12}^{n})^{4}$ respectively, we obtain
\begin{eqnarray}
M_{n,1} & = &\dfrac{1}{2}+k_{2}^{n}+k_{1}^{n}\dfrac{p_{11}^{n}(\Delta \theta_{11}^{n})^{2}}{p_{12}^{n}(\Delta \theta_{12}^{n})^{2}}\to 0, \nonumber \\
M_{n,2} & = & \dfrac{1}{3!}+k_{2}^{n}+k_{1}^{n}\dfrac{p_{11}^{n}(\Delta \theta_{11}^{n})^{3}}{p_{12}^{n}(\Delta \theta_{12}^{n})^{3}}\to 0, \nonumber \\
M_{n,3} & = & \dfrac{1}{4!}+\dfrac{k_{2}^{n}}{2}+\dfrac{(k_{2}^{n})^{2}}{2}+\left(\dfrac{k_{1}^{n}}{2}+\dfrac{(k_{1}^{n})^{2}}{2}\right)\dfrac{p_{11}^{n}(\Delta \theta_{11}^{n})^{4}}{p_{12}^{n}(\Delta \theta_{12}^{n})^{4}}\to 0. \nonumber
\end{eqnarray}
If $|k_{1}^{n}|,|k_{2}^{n}| \to \infty$ then $M_{n,3}>\dfrac{1}{4!}$ for sufficiently
large $n$, which is a contradiction. Therefore, at least one of $|k_{1}^{n}|$, $|k_{2}^{n}|$ does not converge to $\infty$. If $|k_{1}^{n}| \to \infty$ and $|k_{2}^{n}| \not\to \infty$ then $M_{n,1}$ implies that $|k_{1}^{n}\dfrac{p_{11}^{n}(\Delta \theta_{11}^{n})^{2}}{p_{12}^{n}(\Delta \theta_{12}^{n})^{2}}| \not\to \infty$. Therefore, $|k_{1}^{n}\dfrac{p_{11}^{n}(\Delta \theta_{11}^{n})^{3}}{p_{12}^{n}(\Delta \theta_{12}^{n})^{3}}| \to 0$ as $\Delta \theta_{11}^{n}/\Delta \theta_{12}^{n} \to 0$ and $k_{1}^{n}\dfrac{(\Delta \theta_{11}^{n})^{2}}{\Delta \theta_{12}^{n})^{2}} \to 0$ as $p_{11}^{n}/p_{12}^{n} \to \infty$. Combining these results with $M_{n,3}, M_{n,4}$, we get $k_{2}^{n} + \dfrac{1}{3!} \to 0$ and $\dfrac{1}{4!}+\dfrac{k_{2}^{n}}{2}+\dfrac{(k_{2}^{n})^{2}}{2} \to 0$, which cannot happen. If $|k_{1}^{n}| \not \to \infty$, then $M_{n,1}$ and $M_{n,2}$ implies that $k_{2}^{n}+1/2 \to 0$ and $k_{2}^{n}+1/6 \to 0$, which cannot happen either. 
As a consequence, $p_{11}^{n}/p_{12}^{n} \not \to \infty$.

\paragraph{(ii)} If $p_{11}^{n}/p_{12}^{n} \to 0$ then $p_{12}^{n}/p_{11}^{n} \to \infty$. Since $p_{11}^{n}\Delta \theta_{11}^{n}/p_{12}^{n} \Delta \theta_{12}^{n} \to -1$, we have $|\Delta \theta_{11}^{n}/\Delta \theta_{12}^{n}| \to \infty$ or equivalently $\Delta \theta_{12}^{n}/\Delta \theta_{11}^{n} \to 0$. From here, using the same argument 
as that above, we are also led to a contradiction. So, $p_{11}^{n}/p_{12}^{n} \not \to 0$.

\paragraph{(iii)} If $p_{11}^{n}/p_{12}^{n} \to b \not\in \left\{0,\infty\right\}$. It also means that $\Delta \theta_{11}^{n}/\Delta \theta_{12}^{n} \to -1/b$. Therefore, by dividing the numerator and denominator of $F_{2}'(\theta_{1}^{0},v_{1}^{0}), F_{3}'(\theta_{1}^{0},v_{1}^{0}), F_{4}'(\theta_{1}^{0},v_{1}^{0})$ by $p_{12}^{n}(\Delta \theta_{12}^{n})^{2}$, $p_{12}^{n}(\Delta \theta_{12}^{n})^{3}$, and $p_{12}^{n}(\Delta \theta_{12}^{n})^{4}$ and let $n \to \infty$, we arrive at the scaling system of equations~\eqref{eqn:generalovefittedGaussianzero}
when $r=4$ for which we already know that non-trivial solution does not exist. 
Therefore, the case $s_{1}=2$ cannot happen.\\

As a consequence, $i^{*} \not \in \left\{1,\ldots,k_{0} \right\}$. However, since $m \leq 1$, we 
have $i^{*}=k_{0}+1$. This implies that $s_{k_{0}+1}=1$, which we already know from 
Case 1 that \eqref{eqn:overfittedGaussianbytwo-one} cannot hold. 
This concludes the proof.

\subsection{Mixture of Gamma distributions}

\paragraph{PROOF OF THEOREM~\ref{theorem:exactfittedGamma}.}
(a) For the range of generic parameter values of $G_0$, we shall show that
the first-order identifiability still holds for Gamma mixtures, so that the conclusion
can be drawn immediately from Theorem \ref{theorem-firstorder}.
It suffices to show that for 
any $\alpha_{ij} \in \mathbb{R}(1 \leq i \leq 3, 1 \leq j \leq k_{0})$ such that for almost sure $x>0$
\begin{eqnarray}
\mathop {\sum }\limits_{i=1}^{k_{0}}{\alpha_{1i}f(x|a_{i}^{0},b_{i}^{0})+\alpha_{2i}\dfrac{\partial{f}}{\partial{a}}(x|a_{i}^{0},b_{i}^{0})+\alpha_{3i}\dfrac{\partial{f}}{\partial{b}}(x|a_{i}^{0},b_{i}^{0})}=0 \label{eqn:exactfittedGammazero}
\end{eqnarray}
then $\alpha_{ij}=0$ for all $i,j$. Equation \eqref{eqn:exactfittedGammazero} is
rewritten as
\begin{eqnarray}
\mathop {\sum }\limits_{i=1}^{k_{0}}{\left(\beta_{1i}x^{a_{i}^{0}-1}+\beta_{2i}\log(x)x^{a_{i}^{0}-1}+\beta_{3i}x^{a_{i}^{0}}\right)\exp(-b_{i}^{0}x)}=0, \label{eqn:exactfittedGammaone}
\end{eqnarray}
where $\beta_{1i}=\alpha_{1i}\dfrac{(b_{i}^{0})^{a_{i}^{0}}}{\Gamma(a_{i}^{0})}+\alpha_{2i}\dfrac{(b_{i}^{0})^{a_{i}^{0}}(\log(b_{i}^{0})-\psi(a_{i}^{0}))}{\Gamma(a_{i}^{0})}+\alpha_{3i}\dfrac{a_{i}^{0}(b_{i}^{0})^{a_{i}^{0}-1}}{\Gamma(a_{i}^{0})}$, $\beta_{2i}=\alpha_{2i}\dfrac{(b_{i}^{0})^{a_{i}^{0}}}{\Gamma(a_{i}^{0})}$, and $\beta_{3i}=-\alpha_{3i}\dfrac{(b_{i}^{0})^{a_{i}^{0}}}{\Gamma(a_{i}^{0})}$. Without loss of generality, we assume that $b_{1}^{0} \leq b_{2}^{0} \leq \ldots \leq b_{k_{0}}^{0}$. Denote $\overline{i}$ to be the maximum index $i$ such that $b_{i}^{0}=b_{1}^{0}$. Multiply both sides of \eqref{eqn:exactfittedGammaone} with $\exp(b_{\overline{i}}^{0}x)$ and let $x \to +\infty$, we obtain
\begin{eqnarray}
\mathop {\sum }\limits_{i=1}^{\overline{i}}{\beta_{1i}x^{a_{i}^{0}-1}+\beta_{2i}\log(x)x^{a_{i}^{0}-1}+\beta_{3i}x^{a_{i}^{0}}} \to 0. \nonumber
\end{eqnarray}
Since $|a_{i}^{0}-a_{j}^{0}| \neq 1$ and $a_{i}^{0} \geq 1$ for all $1 \leq i,j \leq \overline{i}$, the above result implies that $\beta_{1i}=\beta_{2i}=\beta_{3i}=0$ for all $1 \leq i \leq \overline{i}$ or equivalently $\alpha_{1i}=\alpha_{2i}=\alpha_{3i}$ for all $1 \leq i \leq \overline{i}$. Repeat the same argument for the remained indices, we obtain $\alpha_{1i}=\alpha_{2i}=\alpha_{3i}=0$ for all $1 \leq i \leq k_{0}$.
This concludes the proof.

(b) Without loss of generality, we assume that $\left\{|a_{2}^{0}-a_{1}^{0}|,|b_{2}^{0}-b_{1}^{0}|\right\}=\left\{1,0\right\}$.
In particular,  $b_{1}^{0}=b_{2}^{0}$ and assume $a_{2}^{0}=a_{1}^{0}-1$. 
We construct the following sequence of measures
 $G_{n} = \mathop {\sum }\limits_{i=1}^{k_{0}}{p_{i}^{n}\delta_{(a_{i}^{n},b_{i}^{n})}}$, where
$a_{i}^{n}=a_{i}^{0}$ for all $1 \leq i \leq k_{0}$, $b_{1}^{n}=b_{1}^{0}, b_{2}^{n}=b_{1}^{0}(1+\dfrac{1}{a_{2}^{0}(np_{2}^{0}-1)}), b_{i}^{n}=b_{i}^{0}$ for all $3 \leq i \leq k_{0}$, $p_{1}^{n}=p_{1}^{0}+1/n, p_{2}^{n}=p_{2}^{0}-1/n, p_{i}^{n}=p_{i}^{0}$ for all $3 \leq i \leq k_{0}$. 
We can check that $W_{r}^{r}(G,G_{0}) \asymp 1/n+(p_{2}^{0}-1/n)|b_{2}^{n}-b_{1}^{0}|^r \asymp n^{-1}$ 
as $n \to \infty$. 
For any natural order $r \geq 1$, 
by applying Taylor's expansion up to $([r]+1)$th-order, we obtain:
\begin{eqnarray}
p_{G_{n}}(x)-p_{G_{0}}(x)
=  \mathop {\sum }\limits_{i=1}^{k_{0}}{p_{i}^{n}(f(x|a_{i}^{n},b_{i}^{n})-f(x|a_{i}^{0}, b_{i}^{0}))+(p_{i}^{n}-p_{i}^{0})f(x|a_{i}^{0},b_{i}^{0})} \nonumber \\
=  (p_{1}^{n}-p_{1}^{0})f(x|a_{1}^{0},b_{1}^{0})+(p_{2}^{n}-p_{2}^{0})f(x|a_{2}^{0},b_{2}^{0})
+
\mathop {\sum }\limits_{j=1}^{[r]+1}{p_{2}^{n}\dfrac{(b_{2}^{n}-b_{2}^{0})^{j}}{j!}\dfrac{\partial^{j}{f}}{\partial{b}^{j}}(x|a_{2}^{0},b_{2}^{0})}+R_{n}(x). \label{eqn:Gammapartbone}
\end{eqnarray}
The Taylor expansion remainder
$|R_{n}(x)|=O(p_{2}^{n}|b_{2}^{n}-b_{2}^{0}|^{[r]+1+\delta})$ for some $\delta>0$ 
due to $a_{2}^{0} \geq 1$. Therefore, $R_{n}(x)= o(W_{r}^{r}(G_{n},G_{0}))$ as $n \to \infty$. 
For the choice of $p_{2}^{n}, b_{2}^{n}$, we can check that as 
$j \geq 2$, $p_{2}^{n}(b_{2}^{n}-b_{2}^{0})^{j} = o(W_{r}^{r}(G_{n},G_{0}))$.
Now, we can rewrite \eqref{eqn:Gammapartbone} as
\begin{eqnarray}
p_{G_{n}}(x)-p_{G_{0}}(x)
= A_{n} x^{a_{2}^{0}}\exp(-b_{1}^{0}x)+B_{n}x^{a_{2}^{0}-1}\exp(-b_{1}^{0}x) + \nonumber \\ 
\mathop {\sum }\limits_{j=2}^{[r]+1}{p_{2}^{(n)}\dfrac{(b_{2}^{n}-b_{2}^{0})^{j}}{j!}\dfrac{\partial^{j}{f}}{\partial{b}^{j}}(x|a_{2}^{0},b_{2}^{0})}+R_{n}(x),\nonumber
\end{eqnarray}
where we have $A_{n}=
\dfrac{(b_{1}^{0})^{a_{1}^{0}}}{\Gamma(a_{1}^{0})}(p_{1}^{n}-p_{1}^{0})-\dfrac{(b_{1}^{0})^{a_{2}^{0}}}{\Gamma(a_{2}^{0})}p_{2}^{n}(b_{2}^{n}-b_{1}^{0})= 0$ and similarly 
$B_{n} = \dfrac{(b_{1})^{a_{2}^{0}}}{\Gamma(a_{2}^{0})} (p_{2}^{n}-p_{2}^{0})+\dfrac{a_{2}^{0}(b_{1}^{0})^{a_{2}^{0}-1}}{\Gamma(a_{2}^{0})}p_{2}^{n}(b_{2}^{n}-b_{1}^{0}) = 0$ for all $n$. 
Since $a_{2}^{0} \geq 1$, $\left|\dfrac{\partial^{j}{f}}{\partial{b}^{j}}(x|a_{2}^{0},b_{2}^{0})\right|$ 
is bounded for all $2 \leq j \leq r+1$. 
It follows that $\sup_{x>0}|p_{G_n}(x) - p_{G_0}(x)| = O(n^{-2})$.
Observe that
\begin{eqnarray}
V(p_{G_{n}},p_{G_{0}}) 
= 2{\displaystyle \int \limits_{p_{G_{n}}(x)<p_{G_{0}}(x)}{(p_{G_{0}}(x)-p_{G_{n}}(x)) \; \textrm{d}(x)} }
\leq 2{\displaystyle \int \limits_{x \in (0, a_{2}^{0}/b_{1}^{0})}{|p_{G_{n}}(x)-p_{G_{0}}(x)|}
\textrm{d}x}. \nonumber
\end{eqnarray}
As a consequence $V(p_{G_{n}},p_{G_{0}}) = O(n^{-1/2})$ so for any $r \geq 1$,
$V(p_{G_{n}},p_{G_{0}}) = o(W_r^r(G_n,G_0))$ as $n\rightarrow \infty$.

\paragraph{PROOF OF THEOREM~\ref{theorem:overfittedGamma}.}

(a) By the same argument as the beginning of the proof of 
Theorem \ref{theorem-secondorder}, it suffices to show that
\begin{eqnarray}
\mathop {\lim }\limits_{\epsilon \to 0}{\mathop {\inf }\limits_{G \in \mathcal{O}_{k,c_{0}}(\Theta)}{\left\{\mathop {\sup }\limits_{x \in \mathcal{X}}{|p_{G}(x)-p_{G_{0}}(x)|}/W_{2}^{2}(G,G_{0}):W_{2}(G,G_{0}) \leq \epsilon 
\right\}}} > 0. \label{eqn:overfittedzero}
\end{eqnarray}
Suppose this does not hold, by repeating the arguments of the aforementioned proof,
there is a sequence
$G_{n}=\mathop {\sum }\limits_{i=1}^{k^{*}}{\mathop {\sum }\limits_{j=1}^{s_{i}}{p_{ij}^{n}\delta_{(a_{ij}^{n},b_{i}^{n})}}} \to G_{0}$ such that $(a_{ij}^{n},b_{ij}^{n}) \to (a_{i}^{0},b_{i}^{0})$ for all $1 \leq i \leq k^{*}$ where $p_{i}^{0}=0$ as $k_{0}+1 \leq i \leq k^{*}$. 
Invoke the Taylor expansion up to second order, as we let $n \to \infty$, we have for almost surely $x$ 
\begin{eqnarray}
\dfrac{p_{G_{n}}(x)-p_{G_{0}}(x)}{d(G_{n},G_{0})} \to \mathop {\sum }\limits_{i=1}^{k^{*}}\biggr \{\alpha_{1i}f(x|a_{i}^{0},b_{i}^{0})+\alpha_{2i}\dfrac{\partial{f}}{\partial{a}}(x|a_{i}^{0},b_{i}^{0})+\alpha_{3i}\dfrac{\partial{f}}{\partial{b}}(x|a_{i}^{0},b_{i}^{0})+ \nonumber \\
\mathop {\sum }\limits_{j=1}^{s_{i}}{\alpha_{4ij}^{2}}\dfrac{\partial^{2}{f}}{\partial{a}^{2}}(x|a_{i}^{0},b_{i}^{0})+ \mathop {\sum }\limits_{j=1}^{s_{i}}{\alpha_{5ij}^{2}}\dfrac{\partial^{2}{f}}{\partial{b}^{2}}(x|a_{i}^{0},b_{i}^{0})+2\mathop {\sum }\limits_{j=1}^{s_{i}}{\alpha_{4ij}\alpha_{5ij}}\dfrac{\partial^{2}{f}}{\partial{a}\partial{b}}(x|a_{i}^{0},b_{i}^{0})\biggr \}=0,
\end{eqnarray} 
where at least one of $\alpha_{1i},\alpha_{2i},\alpha_{3i},\mathop {\sum }\limits_{j=1}^{s_{i}}{\alpha_{4ij}^{2}}, \mathop {\sum }\limits_{j=1}^{s_{i}}{\alpha_{5ij}^{2}}, 2\mathop {\sum }\limits_{j=1}^{s_{i}}{\alpha_{4ij}\alpha_{5ij}}$ differs from 0.
We can rewrite the above equation as
\begin{eqnarray}
\mathop {\sum }\limits_{i=1}^{k^{*}}\biggr \{\beta_{1i}x^{a_{i}^{0}-1}+\beta_{2i}x^{a_{i}^{0}}+\beta_{3i}x^{a_{i}^{0}+1}+\beta_{4i}\log(x)x^{a_{i}^{0}-1}+\beta_{5i}\log(x)^{2}x^{a_{i}^{0}-1}+ 
\beta_{6i}\log(x)x^{a_{i}^{0}} \biggr \}e^{-b_{i}^{0}x}=0, \nonumber
\end{eqnarray}
where $\beta_{1i}=\alpha_{1i}\dfrac{b_{i}^{0}}{\Gamma(a_{i}^{0})}+\beta_{i}^{0}\dfrac{\partial{}}{\partial{a}}\left(\dfrac{(b_{i}^{0})^{a_{i}^{0}}}{\Gamma(a_{i}^{0})}\right)+\alpha_{3i}\dfrac{a_{i}^{0}(b_{i}^{0})^{a_{i}^{0}-1}}{\Gamma(a_{i}^{0}}+ \mathop {\sum }\limits_{j=1}^{s_{i}}{\alpha_{5ij}^{2}}\dfrac{a_{i}^{0}(a_{i}^{0}-1)(b_{i}^{0})^{a_{i}^{0}-2}}{\Gamma(a_{i}^{0})}+$ \\$ \mathop {\sum }\limits_{j=1}^{s_{i}}{\alpha_{4ij}^{2}}\dfrac{\partial}{\partial{a}^{2}}\left(\dfrac{(b_{i}^{0})^{a_{i}^{0}}}{\Gamma(a_{i}^{0})}\right)+ 2\mathop {\sum }\limits_{j=1}^{s_{i}}{\alpha_{4ij}\alpha_{5ij}}\dfrac{\partial{}}{\partial{a}}\left(\dfrac{a_{i}^{0}(b_{i}^{0})^{a_{i}^{0}-1}}{\Gamma(a_{i}^{0})}\right)$, \;\;
$\beta_{2i}=-\alpha_{3i}\dfrac{(b_{i}^{0})^{a_{i}^{0}}}{\Gamma(a_{i}^{0})}+2\mathop {\sum }\limits_{j=1}^{s_{i}}{\alpha_{5ij}^{2}}\dfrac{a_{i}^{0}(b_{i}^{0})^{a_{i}^{0}-1}}{\Gamma(a_{i}^{0})}+2\mathop {\sum }\limits_{j=1}^{s_{i}}{\alpha_{4ij}\alpha_{5ij}}\dfrac{\partial{}}{\partial{a}}\left(\dfrac{(b_{i}^{0})^{a_{i}^{0}}}{\Gamma(a_{i}^{0})}\right)$, \;\;
$\beta_{3i}=\mathop {\sum }\limits_{j=1}^{s_{i}}{\alpha_{5ij}^{2}}\dfrac{(b_{i}^{0})^{a_{i}^{0}}}{\Gamma(a_{i}^{0})}$, \;\;
$\beta_{4i}=\alpha_{2i}\dfrac{(b_{i}^{0})}{\Gamma(a_{i}^{0})}+2\mathop {\sum }\limits_{j=1}^{s_{i}}{\alpha_{4ij}^{2}}\dfrac{\partial{}}{\partial{a}}\left(\dfrac{(b_{i}^{0})^{a_{i}^{0}}}{\Gamma(a_{i}^{0})}\right)+2\mathop {\sum }\limits_{j=1}^{s_{i}}{\alpha_{4ij}\alpha_{5ij}}\dfrac{a_{i}^{0}(b_{i}^{0})^{a_{i}^{0}-1}}{\Gamma(a_{i}^{0})}$, $\beta_{5i}= \mathop {\sum }\limits_{j=1}^{s_{i}}{\alpha_{4ij}^{2}}\dfrac{(b_{i}^{0})^{a_{i}^{0}}}{\Gamma(a_{i}^{0})}$, 
and $\beta_{6i}=-2\mathop {\sum }\limits_{j=1}^{s_{i}}{\alpha_{4ij}\alpha_{5ij}}\dfrac{(b_{i}^{0})^{a_{i}^{0}}}{\Gamma(a_{i}^{0})}$. 
Using the same argument as that of the proof of part (a) of 
Theorem \ref{theorem:exactfittedGamma}, by multiplying both sides of the above equation 
with $\exp(b_{\overline{i}}^{0}x)$ and let $x \to +\infty$, we obtain
\begin{eqnarray}
\mathop {\sum }\limits_{i=1}^{\overline{i}}{\beta_{1i}x^{a_{i}^{0}-1}+\beta_{2i}x^{a_{i}^{0}}+\beta_{3i}x^{a_{i}^{0}+1}+\beta_{3i}\log(x)x^{a_{i}^{0}-1}+\beta_{4i}\log(x)^{2}x^{a_{i}^{0}-1} + \beta_{5i}\log(x)x^{a_{i}^{0}}} \to 0. \nonumber
\end{eqnarray}
By the constraints of $\mathcal{O}_{k,c_{0}}$, 
we have $|a_{i}^{0}-a_{j}^{0}| \not \in \left\{1,2\right\}$ for all $1 \leq i,j \leq k^{*}$. Therefore, this limit yields $\beta_{1i}=\beta_{2i}=\beta_{3i}=\beta_{4i}=\beta_{5i}=0$ for all $1 \leq i \leq \overline{i}$ or equivalently $\alpha_{1i}=\alpha_{2i}=\alpha_{3i}=\alpha_{4ij}=\alpha_{5ij}=0$ for all $1 \leq i \leq \overline{i}, 1 \leq j \leq s_{i}$. The same argument yields $\alpha_{1i}=\alpha_{2i}=\alpha_{3i}=\alpha_{4ij}=\alpha_{5ij}=0$ for all $1 \leq i \leq k_{0}, 1 \leq j \leq s_{i}$, which leads to contradiction. 
This concludes the proof.

(b) The proof is similar to part (b) of Theorem \ref{theorem:exactfittedGamma}. We choose sequence 
$G_{n}=\mathop {\sum }\limits_{i=1}^{k_{0}+1}{p_{i}^{n}\delta_{(a_{i}^{n},b_{i}^{n})}}$ 
by letting $a_{i}^{n}=a_{i}^{0}$ for all $2 \leq i \leq k_{0}+1$, $a_{1}^{n}=a_{1}^{0}+1$, $b_{1}^{n}=b_{1}^{0},b_{2}^{n}=b_{1}^{0}(1+\dfrac{1}{a_{1}^{0}(np_{1}^{0}-1)})$, $b_{i}^{n}=b_{i-1}^{0}$ for all $3 \leq i \leq k_{0}+1$, $p_{1}^{n}=1/n$, $p_{2}^{n}=p_{1}^{0}-1/n$, $p_{i}^{n}=p_{i-1}^{0}$ for all $3 \leq i \leq k_{0}+1$.
Given this construction, we can check that as $r \geq 1$, $W_{r}^{r}(G_{n},G_{0})=1/n+(p_{1}^{0}-1/n)|b_{2}^{n}-b_{1}^{0}|^{r}$. The remainder of the proof is 
proceeds in the same way as that of Theorem \ref{theorem:exactfittedGamma}.

c) If there exists $(i,j)$ such that $\left\{|a_{i}^{0}-a_{j}^{0}|,|b_{i}^{0}-b_{j}^{0}|\right\} \equiv \left\{1,0\right\}$, then we can use the same way of construction as that of part (b). Now, the only case of interest is when we have some $(i,j)$ such that $\left\{|a_{i}^{0}-a_{j}^{0}|,|b_{j}^{0}-b_{j}^{0}|\right\} \equiv \left\{2,0\right\}$. Without loss of generality, assume that $a_{2}^{0}=a_{1}^{0}-2$. We construct the sequence $G_{n}=\mathop {\sum }\limits_{i=1}^{k_{0}+1}{p_{i}^{n}\delta_{(a_{i}^{n},b_{i}^{n})}}$ as $a_{1}^{n}=a_{1}^{0}, a_{2}^{n}=a_{3}^{n}=a_{2}^{0},a_{i}^{n}=a_{i-1}^{0}$ for all $4 \leq i \leq k_{0}+1$, $b_{1}^{n}=b_{1}^{0}, b_{2}^{n}-b_{1}^{0}=b_{1}^{0}-b_{3}^{n}=\dfrac{b_{1}^{0}}{a_{2}^{0}n}$, $b_{i}^{n}=b_{i-1}^{0}$ for all $4 \leq i \leq k_{0}+1$, $p_{1}^{n}=p_{1}^{0}-c_{n}$, $p_{2}^{n}=\dfrac{p_{2}^{0}}{2}+\dfrac{1}{2}\left(c_{n}+\dfrac{1}{n}\right), p_{3}^{n}=\dfrac{p_{2}^{0}}{2}+\dfrac{1}{2}\left(c_{n}-\dfrac{1}{n}\right)$, $p_{i}^{n}=p_{i-1}^{0}$ for all $4 \leq i \leq k_{0}+1$.  where $c_{n}=\dfrac{(a_{2}^{0}+1)p_{2}^{0}}{(2n^{2}-1)a_{2}^{0}-1}$. Now, we can check that for any $r \geq 1$, $W_{r}^{r}(G_{n},G_{0}) \gtrsim c_{n}+\dfrac{1}{n^{r}}$. As $r \geq 2$, by means of Taylor expansions up to $([r]+1)$-th order, we obtain
\begin{eqnarray}
p_{G_{n}}(x)-p_{G_{0}}(x) 
= (p_{1}^{n}-p_{1}^{0})f(x|a_{1}^{0},b_{1}^{0})+(\mathop{\sum }\limits_{i=2}^{3}{p_{i}^{n}}-p_{2}^{0})f(x|a_{2}^{0},b_{2}^{0}) \nonumber \\ 
+ \mathop {\sum }\limits_{j=1}^{r+1}{\dfrac{\mathop {\sum }\limits_{i=2}^{3}{p_{i}^{n}(b_{i}^{n}-b_{i}^{0})^{j}}}{j!}\dfrac{\partial^{j}{f}}{\partial{b^{j}}}(x|a_{2}^{0},b_{2}^{0})+R_{n}(x)}, 
\label{eqn:overfittedGammatheoremone}
\end{eqnarray}
where $R_{n}(x)$ is the remainder term and therefore $|R_{n}(x)|/W_{r}^{r}(G_{n},G_{0}) \to 0$. We can check that as $j \geq 3$, $\mathop {\sum }\limits_{i=2}^{3}{p_{i}^{n}(b_{i}^{n}-b_{i}^{0})^{j}}/W_{r}^{r}(G_{n},G_{0}) \to 0$ as $n \to \infty$. Additionally, direct computation demonstrates that
\begin{eqnarray}
(p_{1}^{n}-p_{1}^{0})f(x|a_{1}^{0},b_{1}^{0})+(\mathop{\sum }\limits_{i=2}^{3}{p_{i}^{n}}-p_{2}^{0})f(x|a_{2}^{0},b_{2}^{0})+\mathop {\sum }\limits_{j=1}^{2}{\dfrac{\mathop {\sum }\limits_{i=2}^{3}{p_{i}^{n}(b_{i}^{n}-b_{i}^{0})^{j}}}{j!}\dfrac{\partial^{j}{f}}{\partial{b^{j}}}(x|a_{2}^{0},b_{2}^{0})}= 0. \nonumber
\end{eqnarray}
The rest of the proof goes through in the same way as that 
of Theorem \ref{theorem:exactfittedGamma} part (b).

\paragraph{PROOF OF THEOREM~\ref{theorem:exactlocationgamma}.}
%
%
Choose the sequence $G_{n}=\mathop {\sum }\limits_{i=1}^{k_{0}}{p_{i}^{n}\delta_{(\theta_{i}^{n},\sigma_{i}^{n})}}$ such that $\sigma_{i}^{n}=\sigma_{i}^{0}$ for all $1 \leq i \leq k_{0}$, $(p_{i}^{n},\theta_{i}^{n})=(p_{i}^{0},\theta_{i}^{0})$ for all $3 \leq i \leq k_{0}$. The parameters $p_{1}^{n},p_{2}^{n},\theta_{1}^{n},\theta_{2}^{n}$ are to be determined. 
With this construction of $G_n$, we obtain 
$W_{1}(G_{n},G_{0})\asymp 
|p_{1}^{n}-p_{1}^{0}|+|p_{2}^{n}-p_{2}^{0}|+p_{1}^{0}|\theta_{1}^{n}-\theta_{1}^{0}|+p_{2}^{0}|\theta_{2}^{n}-\theta_{2}^{0}|$. 
Now, for any 
$x \not \in \left\{\theta_{1}^{0},\theta_{2}^{0}\right\}$ and for any $r \geq 1$, taking the Taylor expansion with respect to $\theta$ up to $([r]+1)$-th order, we obtain
\begin{eqnarray}
p_{G_{n}}(x)-p_{G_{0}}(x) 
& = & \mathop {\sum }\limits_{i=1}^{2}{p_{i}^{0}(f(x|\theta_{i}^{n},\sigma_{i}^{0})-f(x|\theta_{i}^{0},\sigma_{i}^{0}))+(p_{i}^{n}-p_{i}^{0})f(x|\theta_{i}^{n},\sigma_{i}^{0})} \nonumber \\
& = & \mathop {\sum }\limits_{i=1}^{2}{(p_{i}^{n}-p_{i}^{0}) f(x|\theta_i^n,\sigma_i^0)
- p_{i}^{0}\left[\mathop {\sum }\limits_{j=1}^{[r]+1}{\dfrac{(\theta_{i}^{0}-\theta_{i}^{n})^{j}}{j!}\dfrac{\partial^{j}{f}}{\partial{\theta}^{j}}(x|\theta_{i}^{n},\sigma_{i}^{0})}\right]}+ R(x)  \nonumber \\
& = & \mathop {\sum }\limits_{i=1}^{2}{\left[(p_{i}^{n}-p_{i}^{0})-p_{i}^{0}\mathop {\sum }\limits_{j=1}^{[r]+1}{\dfrac{(\theta_{i}^{0}-\theta_{i}^{n})^{j}}{j!(\sigma_{i}^{0})^{j}}}\right]f(x|\theta_{i}^{n},\sigma_{i}^{0})}+ R(x), \nonumber
\end{eqnarray}
where the last inequality is due to the identity \eqref{eqn:identitylocationexponential} and $R(x)$ is remainder of Taylor expansion. Note that
\begin{eqnarray}
\mathop {\sup }\limits_{x \not \in \left\{\theta_{1}^{0},\theta_{2}^{0}\right\}}|R(x)|/W_{1}^{r}(G_{n},G_{0})\leq \mathop {\sum }\limits_{i=1}^{2}{O(|\theta_{i}^{n}-\theta_{i}^{0}|^{r+1+\delta})/|\theta_{i}^{n}-\theta_{i}^{0}|^{r}} \to 0. \nonumber 
\end{eqnarray}
Now, we choose $p_{1}^{n}=p_{1}^{0}+1/n$, $p_{2}^{n}=p_{2}^{0}-1/n$, which means $p_{1}^{n}+p_{2}^{n}=p_{1}^{0}+p_{2}^{0}$ and $p_{1}^{n} \to p_{1}^{0}, p_{2}^{n} \to p_{2}^{0}$. As $p_{i}^{0}/j!(\sigma_{i}^{0})^{j}$ are fixed positive constants for all $1 \leq j \leq r+1$.
It is clear that there exists sequences $\theta_1^n$ and $\theta_2^n$
such that for both $i=1$ and $i=2$,
$\theta_{i}^{n}-\theta_{i}^{0} \to 0$,
the identity $p_{i}^{0}\mathop {\sum }\limits_{j=1}^{[r]+1}{\dfrac{(\theta_{i}^{0}-\theta_{i}^{n})^{j}}{j!(\sigma_{i}^{0})^{j}}}=p_{i}^{n}-p_{i}^{0}$ 
holds for all $n$ (sufficiently large). 
With these choices of $p_{1}^{n},p_{2}^{n},\theta_{1}^{n},\theta_{2}^{n}$, we have
\begin{eqnarray}
\mathop {\sup }\limits_{x \not \in {\left\{\theta_{1}^{0},\theta_{2}^{0}\right\}}}{|p_{G_{n}}(x)-p_{G_{0}}(x)|}/W_{1}^{r}(G_{n},G_{0})=\mathop {\sup }\limits_{x \not \in \left\{\theta_{1}^{0},\theta_{2}^{0}\right\}}|R(x)|/W_{1}^{r}(G_{n},G_{0}) \to 0. \nonumber
\end{eqnarray}
To conclude the proof, note that there exists a positive constant $m_1$ 
such that $m_{1}>\mathop {\min }{\left\{\theta_{1}^{0},\theta_{2}^{0}\right\}}$ 
and for sufficiently large $n$,
\begin{eqnarray}
V(p_{G_{n}},p_{G_{0}})/W_{1}^{r}(G_{n},G_{0}) \lesssim \int \limits_{x \in (\mathop {\min }{\left\{\theta_{1}^{0},\theta_{2}^{0}\right\}},m_{1})\backslash \left\{\theta_{1}^{0},\theta_{2}^{0}\right\}}{|p_{G_{n}}(x)-p_{G_{0}}(x)|}/W_{1}^{r}(G_{n},G_{0}) \to 0. \nonumber
\end{eqnarray}

\subsection{Mixture of skew-Gaussian distributions}

\begin{lemma} \label{lemma:skewnormaldistribution} Let $\left\{f(x|\theta,\sigma,m),(\theta,m) \in \mathbb{R}^{2}, \sigma \in \mathbb{R}_{+}\right\}$ be a class of skew normal distribution. Then $\dfrac{\partial^{2}{f}}{\partial{\theta}^{2}}(x|\theta,\sigma^{2},m)-2\dfrac{\partial{f}}{\partial{\sigma}^{2}}(x|\theta,\sigma^{2},m)+\dfrac{m^{3}+m}{\sigma^{2}}\dfrac{\partial{f}}{\partial{m}}(x|\theta,\sigma^{2},m)=0$.
\end{lemma}
\begin{proof}
Direct calculation yields
\begin{eqnarray}
\dfrac{\partial^{2}{f}}{\partial{\theta^{2}}}(x|\theta,\sigma,m) = \biggr \{\left(-\dfrac{2}{\sqrt{2\pi}\sigma^{3}}+\dfrac{2(x-\theta)^{2}}{\sqrt{2\pi}\sigma^{5}}\right)\Phi\left(\dfrac{m(x-\theta)}{\sigma}\right)- \nonumber \\
\dfrac{2m(m^{2}+2)(x-\theta)}{\sqrt{2\pi}\sigma^{4}}f\left(\dfrac{m(x-\theta)}{\sigma}\right)\biggr \}\exp\left(-\dfrac{(x-\theta)^{2}}{2\sigma^{2}}\right), \nonumber \\
\dfrac{\partial{f}}{\partial{\sigma^{2}}}(x|\theta,\sigma,m)  =  \biggr \{\left(-\dfrac{1}{\sqrt{2\pi}\sigma^{3}}+\dfrac{(x-\theta)^{2}}{\sqrt{2\pi}\sigma^{5}}\right)\Phi\left(\dfrac{m(x-\theta)}{\sigma}\right)-\nonumber \\
\dfrac{m(x-\theta)}{\sqrt{2\pi}\sigma^{4}}f\left(\dfrac{m(x-\theta)}{\sigma}\right)\biggr \} \exp\left(-\dfrac{(x-\theta)^{2}}{2\sigma^{2}}\right), \nonumber \\
\dfrac{\partial{f}}{\partial{m}}(x|\theta,\sigma,m)  =  \dfrac{2(x-\theta)}{\sqrt{2\pi}\sigma^{2}}f\left(\dfrac{m(x-\theta)}{\sigma}\right)\exp\left(-\dfrac{(x-\theta)^{2}}{2\sigma^{2}}\right). \nonumber
\end{eqnarray}
From these equations, we can easily verify the conclusion of our lemma.
\end{proof}

\paragraph{PROOF OF PROPOSITION~\ref{proposition-notskewnormal}.} 
For any $k \geq 1$ and $k$ different pairs $(\theta_{1},\sigma_{1},m_{1}),\ldots,(\theta_{k},\sigma_{k},m_{k})$, 
let $\alpha_{ij} \in \mathbb{R}$ for $i=1,\ldots,4,\; j=1,\ldots, k$
such that for almost all $x \in \mathbb{R}$
\begin{eqnarray}
\mathop {\sum }\limits_{j=1}^{k}{\alpha_{1j}f(x|\theta_{j},\sigma_{j},m_{j})+\alpha_{2j}\dfrac{\partial{f}}{\partial{\theta}}(x|\theta_{j},\sigma_{j},m_{j})+\alpha_{3j}\dfrac{\partial{f}}{\partial{\sigma^{2}}}(x|\theta_{j},\sigma_{j},m_{j})+\alpha_{4j}\dfrac{\partial{f}}{\partial{m}}(x|\theta_{j},\sigma_{j},m_{j})}=0. \nonumber
\end{eqnarray}
We can rewrite the above equation as
\begin{eqnarray}
\mathop {\sum }\limits_{j=1}^{k} \biggr \{[\beta_{1j}+\beta_{2j}(x-\theta_{j})+\beta_{3j}(x-\theta_{j})^{2}]\Phi\left(\dfrac{m_{j}(x-\theta_{j})}{\sigma_{j}}\right)\exp\left(-\dfrac{(x-\theta_{j})^{2}}{2\sigma_{j}^{2}}\right)+ \nonumber \\
(\gamma_{1j}+\gamma_{2j}(x-\theta_{j}))f\left(\dfrac{m_{j}(x-\theta_{j})}{\sigma_{j}}\right)\exp\left(-\dfrac{(x-\theta_{j})^{2}}{2\sigma_{j}^{2}}\right) \biggr \}=0, \label{eqn:notidentifiableskewnormalone}
\end{eqnarray}
where $\beta_{1j}=\dfrac{2\alpha_{1j}}{\sqrt{2\pi}\sigma_{j}}-\dfrac{\alpha_{3j}}{\sqrt{2\pi}\sigma_{j}^{3}}$,\;\; $\beta_{2j}=\dfrac{2\alpha_{2j}}{\sqrt{2\pi}\sigma_{j}^{3}}$,\;\; $\beta_{3j}=\dfrac{\alpha_{3j}}{\sqrt{2\pi}\sigma_{j}^{5}}$,\;\; $\gamma_{1j}=-\dfrac{2\alpha_{2j}m_{j}}{\sqrt{2\pi}\sigma_{j}^{2}}$, and $\gamma_{2j}=-\dfrac{\alpha_{3j}m_{j}}{\sqrt{2\pi}\sigma_{j}^{4}}+\dfrac{2\alpha_{4j}}{\sqrt{2\pi}\sigma_{j}^{2}}$ 
for all $j=1,\ldots,k$. 
Now, we identify two scenarios 
in which the first order identifiability of skew-normal distribution fails to hold.

\paragraph{Case 1:} There exists some $m_{j}=0$ as $1 \leq j \leq k$. In this case, we choose $k=1$, $m_{1}=0$. Equation \eqref{eqn:notidentifiableskewnormalone} can be rewritten as
\begin{eqnarray}
\dfrac{\beta_{11}}{2}+\dfrac{\gamma_{11}}{\sqrt{2\pi}}+\left(\dfrac{\beta_{21}}{2}+\dfrac{\gamma_{21}}{\sqrt{2\pi}}\right)(x-\theta_{1})+\dfrac{\beta_{31}}{2}(x-\theta_{1})^{2}=0. \nonumber
\end{eqnarray}
By choosing $\alpha_{31}=0$, $\alpha_{11}=0$, $\alpha_{21}=-\dfrac{\alpha_{41}\sigma_{1}}{\sqrt{2\pi}}$, the above equation always equal to 0. Since $\alpha_{21},\alpha_{41}$ are not necessarily zero, first-order identifiability condition is violated.

\paragraph{Case 2:} There exists two indices $1 \leq i \neq j \leq k$ such that $\left(\dfrac{\sigma_{i}^{2}}{1+m_{i}^{2}},\theta_{i}\right)=\left(\dfrac{\sigma_{j}^{2}}{1+m_{j}^{2}},\theta_{j} \right)$. Now, we choose $k=2$, $i=1,j=2$. Equation in \eqref{eqn:notidentifiableskewnormalone} can be rewritten as
\begin{eqnarray}
\mathop {\sum }\limits_{j=1}^{2} \biggr \{[\beta_{1j}+\beta_{2j}(x-\theta_{j})+\beta_{3j}(x-\theta_{j})^{2}]\Phi\left(\dfrac{m_{j}(x-\theta_{j})}{\sigma_{j}}\right)\exp\left(-\dfrac{(x-\theta_{j})^{2}}{2\sigma_{j}^{2}}\right) \biggr \}+ \nonumber \\
\dfrac{1}{\sqrt{2\pi}}\left(\mathop {\sum }\limits_{j=1}^{2}{\gamma_{1j}}+\mathop {\sum }\limits_{j=1}^{2}{\gamma_{2j}}(x-\theta_{1})^{2}\right)\exp\left(-\dfrac{(m_{1}^{2}+1)(x-\theta_{1})^{2}}{2\sigma_{1}^{2}}\right)=0. \nonumber
\end{eqnarray}
Now, we choose $\alpha_{1j}=\alpha_{2j}=\alpha_{3j}=0$ for all $1 \leq j \leq 2$, $\dfrac{\alpha_{41}}{\sigma_{1}^{2}}+\dfrac{\alpha_{42}}{\sigma_{2}^{2}}=0$ then the above equation always hold. 
Since $\alpha_{41}$ and $\alpha_{42}$ need not be zero, first-order identifiability is
again violated.
\paragraph{PROOF OF THEOREM~\ref{theorem:exactfittedskewnormal}.}
(a) According to the conclusion of Theorem \ref{theorem-firstorder}, to get the conclusion of part a), it is sufficient to demonstrate that for any $\alpha_{ij} \in \mathbb{R} (1 \leq i \leq 4, 1 \leq j \leq k)$ such that for almost sure $x \in \mathbb{R}$
\begin{eqnarray}
\mathop {\sum }\limits_{j=1}^{k_{0}}{\alpha_{1j}f(x|\theta_{j}^{0},\sigma_{j}^{0},m_{j}^{0})+\alpha_{2j}\dfrac{\partial{f}}{\partial{\theta}}(x|\theta_{j}^{0},\sigma_{j}^{0},m_{j}^{0})+\alpha_{3j}\dfrac{\partial{f}}{\partial{\sigma^{2}}}(x|\theta_{j}^{0},\sigma_{j}^{0},m_{j}^{0})+\alpha_{4j}\dfrac{\partial{f}}{\partial{m}}(x|\theta_{j}^{0},\sigma_{j}^{0},m_{j}^{0})}=0. \nonumber
\end{eqnarray}
then $\alpha_{ij}=0$ for all $1 \leq i \leq 4$ and $1 \leq j \leq k_{0}$. In fact, using the result from Proposition \eqref{proposition:generaloverfittedGaussian}, we can rewrite the above equation as
\begin{eqnarray}
\mathop {\sum }\limits_{j=1}^{k} \biggr \{[\beta_{1j}+\beta_{2j}(x-\theta_{j}^{0})+\beta_{3j}(x-\theta_{j}^{0})^{2}]\Phi\left(\dfrac{m_{j}^{0}(x-\theta_{j}^{0})}{\sigma_{j}^{0}}\right)\exp\left(-\dfrac{(x-\theta_{j}^{0})^{2}}{2(\sigma_{j}^{0})^{2}}\right)+ \nonumber \\
(\gamma_{1j}+\gamma_{2j}(x-\theta_{j}^{0}))f\left(\dfrac{m_{j}(x-\theta_{j}^{0})}{\sigma_{j}^{0}}\right)\exp\left(-\dfrac{(x-\theta_{j}^{0})^{2}}{2(\sigma_{j}^{0})^{2}}\right) \biggr \}=0, \label{eqn:exactfittedidentifiableskewnormalone}
\end{eqnarray}
where $\beta_{1j}=\dfrac{2\alpha_{1j}}{\sqrt{2\pi}\sigma_{j}}-\dfrac{\alpha_{3j}}{\sqrt{2\pi}\sigma_{j}^{3}}$, $\beta_{2j}=\dfrac{2\alpha_{2j}}{\sqrt{2\pi}\sigma_{j}^{3}}$, $\beta_{3j}=\dfrac{\alpha_{3j}}{\sqrt{2\pi}\sigma_{j}^{5}}$, $\gamma_{1j}=-\dfrac{2\alpha_{2j}m_{j}}{\sqrt{2\pi}\sigma_{j}^{2}}$, and $\gamma_{2j}=-\dfrac{\alpha_{3j}m_{j}}{\sqrt{2\pi}\sigma_{j}^{4}}+\dfrac{2\alpha_{4j}}{\sqrt{2\pi}\sigma_{j}^{2}}$ for all $1 \leq j \leq k_{0}$. Denote $\sigma_{j+k_{0}}^{0}=\dfrac{(\sigma_{j}^{0})^{2}}{1+(m_{j}^{0})^{2}}$ for all $1 \leq j \leq k_{0}$. From the assumption that $\sigma_{i}^{0}$ are pairwise different and $\dfrac{v_{i}^{0}}{1+(m_{i}^{0})^{2}} \not \in \left\{(\sigma_{j}^{0})^{2}: 1 \leq j \leq k_{0} \right\}$ for all $1 \leq i \leq k_{0}$, we achieve $\sigma_{j}^{0}$ are pairwise different as $1 \leq j \leq 2k_{0}$. The equation \eqref{eqn:exactfittedidentifiableskewnormalone} can be rewritten as
\begin{eqnarray}
\mathop {\sum }\limits_{j=1}^{2k_{0}}\biggr \{[\beta_{1j}+\beta_{2j}(x-\theta_{j}^{0})+\beta_{3j}(x-\theta_{j}^{0})^{2}]\Phi\left(\dfrac{m_{j}^{0}(x-\theta_{j}^{0})}{\sigma_{j}^{0}}\right)\exp\left(-\dfrac{(x-\theta_{j}^{0})^{2}}{2(\sigma_{j}^{0})^{2}}\right)\biggr \}=0, \label{eqn:exactfittedidentifiableskewnormalsecond}
\end{eqnarray}
where $m_{j}^{0}=0$, $\theta_{j+k_{0}}^{0}=\theta_{j}^{0}$, $\beta_{1(j+k_{0})}=\dfrac{2\gamma_{1j}}{\sqrt{2\pi}},\beta_{2(j+k_{0})}=\dfrac{2\gamma_{2j}}{\sqrt{2\pi}}, \beta_{3j}=0$ as $k_{0}+1 \leq j \leq 2k_{0}$. Denote $\overline{i}=\mathop {\arg \max }\limits_{1 \leq i \leq 2k_{0}}{\left\{\sigma_{i}^{0}\right\}}$. Multiply both sides of \eqref{eqn:exactfittedidentifiableskewnormalsecond} with $\exp\left(\dfrac{(x-\theta_{\overline{i}})^{2}}{2\sigma_{\overline{i}}^{2}}\right)/\Phi\left(\dfrac{m_{\overline{i}}^{0}(x-\theta_{\overline{i}}^{0})}{\sigma_{\overline{i}}^{0}}\right)$ and let $x \to +\infty$ if $m_{\overline{i}} \geq 0$ or let $x \to -\infty$ if $m_{\overline{i}}<0$ on both sides of new equation, we obtain $\beta_{1\overline{i}}+\beta_{2\overline{i}}(x-\theta_{\overline{i}}^{0})+\beta_{2\overline{i}}(x-\theta_{\overline{i}}^{0})^{2} \to 0$. It implies that $\beta_{1\overline{i}}=\beta_{2\overline{i}}=\beta_{3\overline{i}}=0$. Keep repeating the same 
argument to the remained $\sigma_{i}$ until we obtain $\beta_{1i}=\beta_{2i}=\beta_{3i}=0$ for all $1 \leq i \leq 2k_{0}$. It is equivalent to $\alpha_{1i}=\alpha_{2i}=\alpha_{3i}=\alpha_{4i}=0$ for all 
$1 \leq i \leq k_{0}$. This concludes the proof for part (a).

(b) In this section, we denote $v=\sigma^{2}$. Without loss of generality, we assume $m_{1}^{0},m_{2}^{0},\ldots,m_{\overline{i}}^{0}=0$ where $1 \leq \overline{i}_{1} \leq k_{0}$ denotes the largest index $i$ such that $m_{i}^{0}=0$. Denote $s_{1}=\overline{i}_{1}+1 < s_{2} < \ldots < s_{\overline{i}_{2}} \in [\overline{i}_{1}+1,k_{0}]$ such that $(\dfrac{v_{j}^{0}}{1+(m_{j}^{0})^{2}},\theta_{j}^{0})= (\dfrac{v_{l}}{1+(m_{l}^{0})^{2}},\theta_{l}^{0})$ and $m_{j}^{0}m_{l}^{0}>0$ for all $s_{i} \leq j,l \leq s_{i+1}-1$, $1 \leq i \leq \overline{i}_{2}-1$. From that definition, we have $|I_{s_{i}}|=s_{i+1}-s_{i}$ for all $1 \leq i \leq \overline{i}_{2}-1$. 
In order to establish part (b) of 
Theorem \ref{theorem:exactfittedskewnormal}, it suffices to show
\begin{eqnarray}
\mathop {\lim }\limits_{\epsilon \to 0}{\mathop {\inf }\limits_{G \in \mathcal{E}_{k_{0}}(\Theta \times \Omega)}{\left\{\dfrac{\mathop {\sup }\limits_{x \in \mathcal{X}}{|p_{G}(x)-p_{G_{0}}(x)|}}{W_{2}^{2}(G,G_{0})}: W_{2}(G,G_{0}) \leq \epsilon \right\}}} > 0. \label{eqn:exactfittedskewnormalone}
\end{eqnarray}
Assume by contrary that \eqref{eqn:exactfittedskewnormalone} does not hold. It means that we can find a sequence $G_{n} \in \mathcal{E}_{k}(\Theta \times \Omega)$ such that $W_{2}(G_{n},G_{0}) \to 0 $ as $n \to \infty$ and for all $x \in \mathcal{X}$, $(p_{G_{n}}(x)-p_{G_{0}}(x))/W_{2}^{2}(G_{n},G_{0}) \to 0$ as $n \to \infty$. Denote $G_{n}=\mathop {\sum }\limits_{i=1}^{k_{0}}{p_{i}^{n}\delta_{(\theta_{i}^{n},v_{i}^{n},m_{i}^{n})}}$ and assume that $(p_{i}^{n},\theta_{i}^{n},v_{i}^{n},m_{i}^{n}) \to (p_{i}^{0},\theta_{i}^{0},v_{i}^{0},m_{i}^{0})$ for all $1 \leq i \leq k_{0}$. Denote $d(G_{n},G_{0})=\mathop {\sum }\limits_{i=1}^{k_{0}}{p_{i}^{n}(|\Delta \theta_{i}^{n}|^{2}+|\Delta v_{i}^{n}|^{2}+|\Delta m_{i}^{n}|^{2})+|\Delta p_{i}^{n}|}$ where $\Delta \theta_{i}^{n}=\theta_{i}^{n}-\theta_{i}^{0}, \Delta v_{i}^{n}=v_{i}^{n}-v_{i}^{0}, \Delta m_{i}^{n}=m_{i}^{n}-m_{i}^{0}$, and $\Delta p_{i}^{n}=p_{i}^{n}-p_{i}^{0}$ for all $1 \leq i \leq n$. According to the argument of the proof of Theorem \ref{theorem-firstorder}, we have $(p_{G_{n}}(x)-p_{G_{0}}(x))/d(G_{n},G_{0}) \to 0 $ as $n \to \infty$ for all $x \in \mathcal{X}$. By means of Taylor expansion up to second order, we can write $(p_{G_{n}}(x)-p_{G_{0}}(x))/d(G_{n},G_{0})$ as the summation of four parts, which we denote by $A_{n,1}(x)$, $A_{n,2}(x)$, $A_{n,3}(x)$, and $A_{n,4}(x)$.

Regarding $A_{n,4}(x)$, it is the remainder of Taylor expansion, which means as $n \to \infty$
\begin{eqnarray}
A_{n,4}(x)=O(\mathop {\sum }\limits_{i=1}^{k_{0}}{p_{i}^{n}(|\Delta\theta_{i}^{n}|^{2+\delta}+|\Delta\sigma_{i}^{n}|^{2+\delta}+|\Delta m_{i}^{n}|^{2+\delta}}))/d(G_{n},G_{0}) \to 0, \nonumber
\end{eqnarray} 
for some constant $\delta>0$.

Regarding $A_{n,1}(x), A_{n,2}(x)$, $A_{n,3}(x)$, these are linear combinations of $f(x|\theta_{i}^{0},v_{i}^{0},m_{i}^{0})$, $\dfrac{\partial{f}}{\partial{\theta}}(x|\theta_{i}^{0},v_{i}^{0},m_{i}^{0})$, 
$\dfrac{\partial{f}}{\partial{v}}(x|\theta_{i}^{0},v_{i}^{0},m_{i}^{0})$, $\dfrac{\partial{f}}{\partial{m}}(x|\theta_{i}^{0},v_{i}^{0},m_{i}^{0})$, $\dfrac{\partial^{2}{f}}{\partial{\theta^{2}}}(x|\theta_{i}^{0},v_{i}^{0},m_{i}^{0})$, $\dfrac{\partial^{2}{f}}{\partial{v^{2}}}(x|\theta_{i}^{0},v_{i}^{0},m_{i}^{0})$, $\dfrac{\partial^{2}{f}}{\partial{m^{2}}}(x|\theta_{i}^{0},v_{i}^{0},m_{i}^{0})$, \\ 
$\dfrac{\partial^{2}{f}}{\partial{\theta}\partial{v}}(x|\theta_{i}^{0},v_{i}^{0},m_{i}^{0})$, $\dfrac{\partial^{2}{f}}{\partial{\theta}\partial{m}}(x|\theta_{i}^{0},v_{i}^{0},m_{i}^{0})$, $\dfrac{\partial^{2}{f}}{\partial{v}\partial{m}}(x|\theta_{i}^{0},v_{i}^{0},m_{i}^{0})$. However, in $A_{n,1}(x)$, the index $i$ ranges from $1$ to $\overline{i}_{1}$ while in $A_{n,2}(x)$ and $A_{n,3}$, the index $i$ 
ranges from $\overline{i}_{1}+1$ to  $s_{\overline{i}_{2}}-1$ and from $s_{\overline{i}_{2}}$ to $k_{0}$,
 respectively.

Regarding $A_{n,3}(x)$, we denote $B_{\alpha_{1}\alpha_{2}\alpha_{3}}(\theta_{i}^{0},v_{i}^{0},m_{i}^{0})$ to be the coefficient of $\dfrac{\partial^{\alpha}{f}}{\theta^{\alpha_{1}}v^{\alpha_{2}}m^{\alpha_{3}}}(x|\theta_{i}^{0},v_{i}^{0},m_{i}^{0})$ for any $s_{\overline{i}_{2}} \leq i \leq k_{0}$, $0 \leq \alpha \leq 2$ and $\alpha_{1}+\alpha_{2}+\alpha_{3}=\alpha$, $\alpha_{j} \geq 0$ for all $1 \leq j \leq 3$.

Regarding $A_{n,2}(x)$, the structure $(\dfrac{v_{j}^{0}}{1+(m_{j}^{0})^{2}},\theta_{j}^{0})= (\dfrac{v_{l}^{0}}{1+(m_{l}^{0})^{2}},\theta_{l}^{0})$ for all $s_{i} \leq j,l \leq s_{i+1}-1$, $1 \leq i \leq \overline{i}_{2}-1$, 
allows us to rewrite $A_{n,2}(x)$ as  
\begin{eqnarray}
A_{n,2}(x)= \sum_{i=1}^{\overline{i}_{2}-1} \biggr \{\mathop {\sum }\limits_{j=s_{i}}^{s_{i+1}-1}{\left[\alpha_{1ji}^{n}+\alpha_{2ji}^{n}(x-\theta_{s_{i}}^{0})+\alpha_{3ji}^{n}(x-\theta_{s_{i}}^{0})^{2}+\alpha_{4ji}^{n}(x-\theta_{s_{i}}^{0})^{3}+\alpha_{5ji}^{n}(x-\theta_{s_{i}}^{0})^{4}\right] \times } \nonumber \\
f(\dfrac{x-\theta_{s_{i}}^{0}}{\sigma_{j}^{0}})\Phi(\dfrac{m_{j}^{0}(x-\theta_{s_{i}}^{0})}{\sigma_{j}^{0}})\biggr\}+\left[\beta_{1i}^{n}+\beta_{2i}^{n}(x-\theta_{s_{i}}^{0})+\beta_{3i}^{n}(x-\theta_{s_{i}}^{0})^{2}+\beta_{4i}^{n}(x-\theta_{s_{i}}^{0})^{3}\right] \times \nonumber \\
\exp\left(-\dfrac{(m_{s_{i}}^{0})^{2}+1}{2v_{s_{i}}^{0}}(x-\theta_{s_{i}}^{0})^{2}\right), \nonumber
\end{eqnarray}
where $f(x)=\dfrac{1}{\sqrt{2\pi}}\exp(-\dfrac{x^{2}}{2})$.
Moreover,  $d(G_{n},G_{0})\alpha_{l_{1}ji}^{n}$ is a linear combination of elements 
of $\Delta p_{j}^{n}, (\Delta \theta_{j}^{n})^{\alpha_{1}}(\Delta v_{j}^{n})^{\alpha_{2}}$ 
for each $i=1,\ldots, \overline{i}_{2}-1$, $s_{i} \leq j \leq s_{i+1}-1$, 
$1 \leq l_{1} \leq 5$ and $1 \leq  \alpha_{1}+\alpha_{2} \leq 2$. 
Additionally, $d(G_{n},G_{0})\beta_{l_{2}i}^{n}$ is a linear combination of elements of 
$\mathop  {\sum }\limits_{j=s_{i}}^{s_{i+1}-1}{(\Delta \theta_{j}^{n})^{\alpha_{1}}(\Delta v_{j}^{n})^{\alpha_{2}}(\Delta m_{j}^{n})^{\alpha_{3}}}$ for each 
$1 \leq l_{2} \leq 4$, $1 \leq i \leq \overline{i}_{2}-1$, and 
$1 \leq \alpha_{1}+ \alpha_{2} +\alpha_{3} \leq 2$. 
The detailed formula 
of $d(G_{n},G_{0})\alpha_{l_{1}ji}^{n}, d(G_{n},G_{0})\beta_{l_{2}i}^{n}$ are given in 
Appendix II.

Regarding $A_{n,1}(x)$, the structure $m_{1}^{0},m_{2}^{0},\ldots,m_{\overline{i}}^{0}=0$
allow us to rewrite $A_{n,1}(x)$ as
\begin{eqnarray}
A_{n,1}(x)=\sum_{j=1}^{\overline{i}_{1}} {\left[\gamma_{1j}^{n}+\gamma_{2j}^{n}(x-\theta_{j}^{0})+\gamma_{3j}^{n}(x-\theta_{j}^{0})^{2}+\gamma_{4j}^{n}(x-\theta_{j}^{0})^{3}+\gamma_{5j}^{n}(x-\theta_{j}^{0})^{4}\right]f\left(\dfrac{x-\theta_{j}^{0}}{\sigma_{j}^{0}}\right)}, \nonumber
\end{eqnarray}
where $d(G_{n},G_{0})\gamma_{lj}^{n}$ are linear combination of elements of $\Delta p_{j}^{n}, (\Delta \theta_{j}^{n})^{\alpha_{1}}(\Delta v_{j}^{n})^{\alpha_{2}}(\Delta m_{j}^{n})^{\alpha_{3}}$  for all $1 \leq j \leq \overline{i}_{1}$ and $\alpha_{1}+\alpha_{2}+\alpha_{3} \leq 2$. The detail formulae of $d(G_{n},G_{0})\gamma_{lj}^{n}$ are in Appendix II. 

Now, suppose that all $\gamma_{ij}^{n}$ ($1 \leq i \leq 5$,$1 \leq j \leq \overline{i}_{1}$), 
$\beta_{ij}^{n}$ ($1 \leq i \leq 4$, $1 \leq j \leq \overline{i}_{2}-1$), $\alpha_{ijl}^{n}$
($1 \leq i \leq 5$, $s_{l} \leq j \leq s_{l+1}-1$, $1 \leq l \leq \overline{i}_{2}-1$), $B_{\alpha_{1}\alpha_{2}\alpha_{3}}(\theta_{i}^{0},v_{i}^{0},m_{i}^{0})$
(for all $\alpha_{1}+\alpha_{2}+\alpha_{3} \leq 2$) go to $0$ as $n \to \infty$. 
We can find at least one index $1 \leq i^{*} \leq k_{0}$ such that $(|\Delta p_{i^{*}}^{n}|+p_{i^{*}}^{n}(|\Delta \theta_{i^{*}}^{n}|^{2}+|\Delta v_{i^{*}}^{n}|^{2}+|\Delta m_{i^{*}}^{n}|^{2}))/d(G_{n},G_{0}) \not \to 0$ as $n \to \infty$. Define $d(p_{i^{*}}^{n},\theta_{i^{*}}^{n},v_{i^{*}}^{n},m_{i^{*}}^{n})=|\Delta p_{i^{*}}^{n}|+p_{i^{*}}^{n}(|\Delta \theta_{i^{*}}^{n}|^{2}+|\Delta v_{i^{*}}^{n}|^{2}+|\Delta m_{i^{*}}^{n}|^{2})$. There are three possible cases for $i^{*}$:
\paragraph{Case 1:} $1 \leq i^{*} \leq \overline{i}_{1}$. Since $d(p_{i^{*}}^{n},\theta_{i^{*}}^{n},v_{i^{*}}^{n},m_{i^{*}}^{n})/d(G_{n},G_{0}) \not \to 0$, we obtain that for all $1 \leq j \leq 5$
\begin{eqnarray}
C_{j}^{n}:=\dfrac{d(G_{n},G_{0})}{d(p_{i^{*}}^{n},\theta_{i^{*}}^{n},v_{i^{*}}^{n},m_{i^{*}}^{n})}\gamma_{ji^{*}}^{n} \to 0 \ \text{as } \ n \to \infty. \nonumber
\end{eqnarray}

Within this scenario our argument is organized into four steps.
\paragraph{Step 1.1:} We can argue that $\Delta \theta_{i^{*}}^{n}, \Delta v_{i^{*}}^{n}, \Delta m_{i^{*}}^{n} \neq 0$ for infinitely many $n$. The detailed argument is left to Appendix II. 

\paragraph{Step 1.2:} If $|\Delta \theta_{i^{*}}^{n}|$ is the maximum among $|\Delta \theta_{i^{*}}^{n}|$, $|\Delta v_{i^{*}}^{n}|$, $|\Delta m_{i^{*}}^{n}|$ for infinitely many $n$, 
then we can assume that it holds for all $n$. Denote $\Delta v_{i^{*}}^{n}=k_{1}^{n}\Delta \theta_{i^{*}}^{n}$ and $\Delta m_{i^{*}}^{n}=k_{2}^{n}\Delta \theta_{i^{*}}^{n}$ where $k_{1}^{n},k_{2}^{n} \in [-1,1]$. Assume that 
$k_{1}^{n} \to k_{1}$ and $k_{2}^{n} \to k_{2}$ as $n \to \infty$. As $C_{5}^{n} \to 0$, 
dividing both the numerator and denominator by $(\Delta \theta_{i^{*}}^{n})^{2}$, we obtain that as $n \to \infty$
\begin{eqnarray}
\dfrac{(k_{1}^{n})^{2}}{\dfrac{|\Delta p_{i^{*}}^{n}|}{(\Delta \theta_{i^{*}}^{n})^{2}}+1+(k_{1}^{n})^{2}+(k_{2}^{n})^{2}} \to 0. \label{eqn:skewnormaldistributionone}
\end{eqnarray}
If $\dfrac{|\Delta p_{i^{*}}^{n}|}{(\Delta \theta_{i^{*}}^{n})^{2}} \to \infty$ as $n \to \infty$, then $C_{1}^{n}+(\sigma_{i^{*}}^{0})^{2}C_{3}^{n} \not \to 0$ as $n \to \infty$, which is a contradiction to the fact that $C_{1}^{n},C_{3}^{n} \to 0$ as $n \to \infty$. Therefore, $\dfrac{|\Delta p_{i^{*}}^{n}|}{(\Delta \theta_{i^{*}}^{n})^{2}} \not \to \infty$ as $n \to \infty$. Combining this result with \eqref{eqn:skewnormaldistributionone}, we obtain $k_{1}=0$. Similarly, by dividing both the numerator and denominator of $C_{2}^{n}$ and $C_{3}^{n}$, we obtain the following equations $\dfrac{1}{2\sqrt{2\pi}\sigma_{i^{*}}^{0}}+\dfrac{2k_{2}}{\pi}=0$ and $\dfrac{1}{\sqrt{2\pi}\sigma_{i^{*}}^{0}}+\dfrac{k_{2}}{\pi}=0$. These equations imply that $1/\sigma_{i^{*}}^{0}=0$, which is a contradiction. 

\paragraph{Step 1.3:} If $|\Delta v_{i^{*}}^{n}|$ is the maximum among $|\Delta \theta_{i^{*}}^{n}|$, $|\Delta v_{i^{*}}^{n}|$, $|\Delta m_{i^{*}}^{n}|$ for infinitely many $n$, then we can assume that it holds for all $n$.  However, the formation of $C_{5}^{n}$ implies that $\dfrac{|\Delta p_{i^{*}}^{n}|}{(\Delta v_{i^{*}}^{n})^{2}} \to \infty$ as $n \to \infty$. It again leads to $C_{1}^{n}+(\sigma_{i^{*}}^{0})^{2}C_{3}^{n} \not \to 0$ as $n \to \infty$, which is a contradiction.

\paragraph{Step 1.4:} If $|\Delta m_{i^{*}}^{n}|$ is  the maximum among $|\Delta \theta_{i^{*}}^{n}|$, $|\Delta v_{i^{*}}^{n}|$, $|\Delta m_{i^{*}}^{n}|$ for infinitely many $n$, then we can assume that it holds for all $n$. Denote $\Delta \theta_{i^{*}}^{n}=k_{3}^{n}\Delta m_{i^{*}}^{n}$ and $\Delta v_{i^{*}}^{n}=k_{4}^{n}\Delta m_{i^{*}}^{n}$. Let $k_{3}^{n} \to k_{3}$ and $k_{4}^{n} \to k_{4}$. With the same argument as the case $|\Delta \theta_{i^{*}}^{n}|$ is the maximum, we obtain $k_{4}=0$. By dividing both the numerator and denominator of $C_{2}^{n}$ and $C_{3}^{n}$ by $(\Delta m_{i^{*}}^{n})^{2}$, we obtain the following equations $\dfrac{k_{3}}{\sqrt{2\pi}\sigma_{i^{*}}^{0}}+\dfrac{1}{\pi}=0$ and $\dfrac{2k_{3}}{\pi}+\dfrac{k_{3}^{2}}{2\sqrt{2\pi}\sigma_{i^{*}}^{0}}=0$, 
for which there is no real solution.

\noindent In sum, Case 1 cannot happen.

\paragraph{Case 2:} $s_{1} \leq i^{*} \leq s_{\overline{i}_{2}}-1$. Without loss of generality, we assume that $s_{1} \leq i^{*} \leq s_{2}-1$. Denote 
\begin{eqnarray}
d_{\text{new}}(p_{i^{*}}^{n},\theta_{i^{*}}^{n},v_{i^{*}}^{n},m_{i^{*}}^{n})=\mathop {\sum }\limits_{j=s_{1}}^{s_{2}-1}{|\Delta p_{j}^{n}|+p_{j}^{n}(|\Delta \theta_{j}^{n}|^{2}+|\Delta v_{j}^{n}|^{2}+|\Delta m_{j}^{n}|^{2})}, \nonumber
\end{eqnarray}
Since $d(p_{i^{*}}^{n},\theta_{i^{*}}^{n},v_{i^{*}}^{n},m_{i^{*}}^{n})/d(G_{n},G_{0}) \not \to 0$, we have $d_{\text{new}}(p_{i^{*}}^{n},\theta_{i^{*}}^{n},v_{i^{*}}^{n},m_{i^{*}}^{n})/d(G_{n},G_{0}) \not \to 0$ as $n \to \infty$. Therefore, for $1 \leq j \leq 5$ and $s_{1} \leq i \leq s_{2}-1$,
\begin{eqnarray}
D_{j}^{n}:=\dfrac{d(G_{n},G_{0})}{d_{\text{new}}(p_{i^{*}}^{n},\theta_{i^{*}}^{n},v_{i^{*}}^{n},m_{i^{*}}^{n})}\alpha_{ji1}^{n} \to 0 \ \text{as } \ n \to \infty. \nonumber
\end{eqnarray}
Our argument is organized into three steps.
\paragraph{Step 2.1:} From $D_{2}^{n}$ and $D_{4}^{n}$, we obtain $p_{i}^{n}\Delta \theta_{i}^{n}/d_{\text{new}}(p_{i^{*}}^{n},\theta_{i^{*}}^{n},v_{i^{*}}^{n},m_{i^{*}}^{n}) \to 0$ as $n \to \infty$ for all $s_{1} \leq i \leq s_{2}-1$. Combining with $D_{1}^{n}$ and $D_{5}^{n}$, we achieve 
\begin{eqnarray}
\Delta p_{i}^{n}/d_{\text{new}}(p_{i^{*}}^{n},\theta_{i^{*}}^{n},v_{i^{*}}^{n},m_{i^{*}}^{n}), p_{i}^{n}\Delta v_{i}^{n}/d_{\text{new}}(p_{i^{*}}^{n},\theta_{i^{*}}^{n},v_{i^{*}}^{n},m_{i^{*}}^{n}) \to 0 \ \text{as} \ n \to \infty \ \text{for all} \ s_{1} \leq i \leq s_{2}-1. \nonumber
\end{eqnarray}
Therefore, we also have $p_{i}^{n}(\Delta \theta_{i}^{n})^{2}/d_{\text{new}}(p_{i^{*}}^{n},\theta_{i^{*}}^{n},v_{i^{*}}^{n},m_{i^{*}}^{n}) \to 0$ and $p_{i}^{n}(v_{i}^{n})^{2}/d_{\text{new}}(p_{i^{*}}^{n},\theta_{i^{*}}^{n},v_{i^{*}}^{n},m_{i^{*}}^{n}) \to 0$ as $n \to \infty$ for all $s_{1} \leq i \leq s_{2}-1$. These results show that 
\begin{eqnarray}
U_{n}=\left[\mathop {\sum }\limits_{j=s_{1}}^{s_{2}-1}{p_{j}^{n}(\Delta m_{j}^{n})^{2}}\right]/d_{\text{new}}(p_{i^{*}}^{n},\theta_{i^{*}}^{n},v_{i^{*}}^{n},m_{i^{*}}^{n}) \not \to 0 \ \text{as} \ n \to \infty. \nonumber
\end{eqnarray}

\paragraph{Step 2.2:} Now for $1 \leq j \leq 4$, we also have
\begin{eqnarray}
E_{j}^{n} :=\dfrac{d(G_{n},G_{0})}{d_{\text{new}}(p_{i^{*}}^{n},\theta_{i^{*}}^{n},v_{i^{*}}^{n},m_{i^{*}}^{n})}\beta_{j1}^{n} \to 0 \ \text{as } \ n \to \infty. \nonumber
\end{eqnarray}
Since $p_{i}^{n}\Delta v_{i}^{n}/d_{\text{new}}(p_{i^{*}}^{n},\theta_{i^{*}}^{n},v_{i^{*}}^{n},m_{i^{*}}^{n})$ and $p_{i}^{n} (\Delta v_{i}^{n})^{2}/d_{\text{new}}(p_{i^{*}}^{n},\theta_{i^{*}}^{n},v_{i^{*}}^{n},m_{i^{*}}^{n})$ go to $0$ as $n \to \infty$ for all $s_{1} \leq i \leq s_{2}-1$, we obtain that as $n \to \infty$
\begin{eqnarray}
\left[\mathop {\sum }\limits_{j=s_{1}}^{s_{2}-1}{\dfrac{p_{j}^{n}((m_{j}^{0})^{2}+1)\Delta m_{j}^{n}\Delta v_{j}^{n}}{\pi(\sigma_{j}^{0})^{6}}}\right]/d_{\text{new}}(p_{i^{*}}^{n},\theta_{i^{*}}^{n},v_{i^{*}}^{n},m_{i^{*}}^{n})\to 0, \nonumber
\end{eqnarray}
and
\begin{eqnarray}
\left[\mathop {\sum }\limits_{j=s_{i}}^{s_{i+1}-1}{-\dfrac{p_{j}^{n}((m_{j}^{0})^{3}+2m_{j}^{0})(\Delta v_{j}^{n})^{2}}{8\pi(\sigma_{i}^{0})^{8}}}\right]/d_{\text{new}}(p_{i^{*}}^{n},\theta_{i^{*}}^{n},v_{i^{*}}^{n},m_{i^{*}}^{n}) \to 0. \nonumber
\end{eqnarray}
Combining these results with $E_{4}^{n}$, we have 
\begin{eqnarray}
V_{n}=\left[\mathop {\sum }\limits_{j=s_{1}}^{s_{2}-1}{p_{j}^{n}m_{j}^{0}(\Delta m_{j}^{n})^{2}}\right]/d_{\text{new}}(p_{i^{*}}^{n},\theta_{i^{*}}^{n},v_{i^{*}}^{n},m_{i^{*}}^{n}) \to 0. \nonumber
\end{eqnarray}
\paragraph{Step 2.3:} As $U_{n} \not \to 0$ as $n \to \infty$, we obtain
\begin{eqnarray}
V_{n}/U_{n}=\mathop {\sum }\limits_{j=s_{1}}^{s_{2}-1}{p_{j}^{n}m_{j}^{0}(\Delta m_{j}^{n})^{2}}/\mathop {\sum }\limits_{j=s_{1}}^{s_{2}-1}{p_{j}^{n}(\Delta m_{j}^{n})^{2}} \to 0. \label{eqn:exactfittedskewnormalgeneralkey}
\end{eqnarray}
Since $m_{i}^{0}m_{j}^{0}>0 $ for all $s_{1} \leq i,j \leq s_{2}-1$, without loss of generality we assume that $m_{j}^{0}>0$ for all $s_{1} \leq j \leq s_{2}-1$. However, it implies that
\begin{eqnarray}
\mathop {\sum }\limits_{j=s_{1}}^{s_{2}-1}{p_{j}^{n}m_{j}^{0}(\Delta m_{j}^{n})^{2}}/\mathop {\sum }\limits_{j=s_{1}}^{s_{2}-1}{p_{j}^{n}(\Delta m_{j}^{n})^{2}} \geq \mathop {\min }\limits_{s_{1} \leq j \leq s_{2}-1}{\left\{m_{j}^{0}\right\}} \mathop {\sum }\limits_{j=s_{1}}^{s_{2}-1}{p_{j}^{n}(\Delta m_{j}^{n})^{2}}/\mathop {\sum }\limits_{j=s_{1}}^{s_{2}-1}{p_{j}^{n}(\Delta m_{j}^{n})^{2}} , \label{eqn:exactfittedskewnormalgeneral}
\end{eqnarray}
which means $\mathop {\min }\limits_{s_{1} \leq j \leq s_{2}-1}{\left\{m_{j}^{0}\right\}}=0$.
This is a contradiction. 
\noindent In sum, Case 2 cannot happen.

\paragraph{Case 3:} $s_{\overline{i}_{2}} \leq i^{*} \leq k_{0}$. Since $d(p_{i^{*}}^{n},\theta_{i^{*}}^{n},v_{i^{*}}^{n},m_{i^{*}}^{n})/d(G_{n},G_{0}) \not \to 0$, we obtain
\begin{eqnarray}
\tau(p_{i^{*}}^{n},\theta_{i^{*}}^{n},v_{i^{*}}^{n},m_{i^{*}}^{n})/d(G_{n},G_{0}) \not \to 0 \ \text{as } \ n \to \infty, \nonumber
\end{eqnarray} 
where $\tau(p_{i^{*}}^{n},\theta_{i^{*}}^{n},v_{i^{*}}^{n},m_{i^{*}}^{n})=|\Delta p_{i^{*}}^{n}|+p_{i^{*}}^{n}(|\Delta \theta_{i^{*}}^{n}|+|\Delta v_{i^{*}}^{n}|+|\Delta m_{i^{*}}^{n}|) \gtrsim d(p_{i^{*}}^{n},\theta_{i^{*}}^{n},v_{i^{*}}^{n},m_{i^{*}}^{n})$. As a consequence, for any $\alpha_{1}+\alpha_{2}+\alpha_{3} \leq 1$, as $n \to \infty$
\begin{eqnarray}
\dfrac{d(G_{n},G_{0})}{\tau(p_{i^{*}}^{n},\theta_{i^{*}}^{n},v_{i^{*}}^{n},m_{i^{*}}^{n})}B_{\alpha_{1}\alpha_{2}\alpha_{3}}(\theta_{i^{*}}^{0},v_{i^{*}}^{0},m_{i^{*}}^{0}) \to 0. \nonumber
\end{eqnarray}
However, from the proof of part (a), at least one of the above coefficients does not go to 0, which is a contradiction. Therefore, Case 3 cannot happen either.

Summarizing from the arguments with the three cases above, we conclude that
not all of 
$\gamma_{ij}^{n}$ ($1 \leq i \leq 5$, $1 \leq j \leq \overline{i}_{1}$), $\beta_{ij}^{n}$ ($1 \leq i \leq 4$, $1 \leq j \leq \overline{i}_{2}-1$), $\alpha_{ijl}^{n}$ ($1 \leq i \leq 5$, $s_{1} \leq j \leq s_{\overline{i}_{2}}-1$, $1 \leq l \leq \overline{i}_{2}-1$), $B_{\alpha_{1}\alpha_{2}\alpha_{3}}(\theta_{i}^{0},v_{i}^{0},m_{i}^{0})$ ($\alpha_{1}+\alpha_{2}+\alpha_{3} \leq 2$) go to $0$ as $n \to \infty$. Denote $m_{n}$ to be the the maximum of the absolute values of these coefficients and $d_{n}=1/m_{n}$. Then, $d_{n}\alpha_{ijl}^{n} \to \alpha_{ijl}$ for all $1 \leq i \leq 5, s_{l} \leq j \leq s_{j+1}-1$, $1 \leq l \leq \overline{i}_{2}-1$, $d_{n}\beta_{ij}^{n} \to \beta_{ij}$ for all $1 \leq i \leq 4$, $1 \leq j \leq \overline{i}_{2}-1$, $d_{n}\gamma_{ij}^{n} \to \gamma_{ij}$ for all $1 \leq i \leq 5$, $1 \leq j \leq \overline{i}_{1}$, and $d_{n}B_{\alpha_{1}\alpha_{2}\alpha_{3}}(\theta_{i}^{0},\sigma_{i}^{0},m_{i}^{0}) \to \lambda_{\alpha_{1}\alpha_{2}\alpha_{3}i}$ for all $s_{\overline{i}_{2}} \leq i \leq k_{0}$. Therefore, by letting $n \to \infty$, we obtain for all $x \in \mathbb{R}$ that
\begin{eqnarray}
\dfrac{d_{n}(p_{G_{n}}(x)-p_{G_{0}}(x))}{d(G_{n},G_{0})} \to A_{1}(x)+A_{2}(x)+A_{3}(x)=0, \nonumber
\end{eqnarray}
where $A_{1}(x)=\sum \limits_{j=1}^{\overline{i}_{1}} {\left(\mathop {\sum }\limits_{i=1}^{5}{\gamma_{ij}(x-\theta_{j}^{0})^{i-1}}\right)f\left(\dfrac{x-\theta_{j}^{0}}{\sigma_{j}^{0}}\right)}$,  $A_{2}(x)=\sum \limits_{l=1}^{\overline{i}_{2}-1} \biggr \{\mathop {\sum }\limits_{j=s_{i}}^{s_{i+1}-1}{\mathop {\sum }\limits_{i=1}^{5}{\alpha_{ijl}(x-\theta_{s_{l}}^{0})^{i-1}}}$ \\
$\times f\left(\dfrac{x-\theta_{s_{l}}^{0}}{\sigma_{j}^{0}}\right)\Phi\left(\dfrac{m_{j}^{0}(x-\theta_{s_{l}}^{0})}{\sigma_{j}^{0}}\right)+\mathop {\sum }\limits_{i=1}^{4}{\beta_{il}(x-\theta_{s_{l}}^{0})^{i-1}}\exp\left(-\dfrac{(m_{s_{l}}^{0})^{2}+1}{2v_{s_{l}}^{0}}(x-\theta_{s_{l}}^{0})^{2}\right)\biggr\}$, and \\ $A_{3}(x)=\mathop {\sum }\limits_{i=s_{\overline{i}_{2}}}^{k_{0}}{\mathop {\sum }\limits_{|\alpha| \leq 2}{\lambda_{\alpha_{1}\alpha_{2}\alpha_{3}i}\dfrac{\partial^{|\alpha|}{f}}{\partial{\theta}^{\alpha_{1}}v^{\alpha_{2}}m^{\alpha_{3}}}(x|\theta_{i}^{0},v_{i}^{0},m_{i}^{0})}}$.

Using the same argument as that of part (a), we obtain $\alpha_{ijl}=0$ for all $1 \leq i \leq 5$, $s_{l} \leq j \leq s_{l+1}-1$, $1 \leq l \leq \overline{i}_{2}-1$, $\beta_{ij}=0$ for all $1 \leq i \leq 4$, $1 \leq j \leq \overline{i}_{2}-1$, and $\gamma_{ij}=0$ for all $1 \leq i \leq 5$ and $1 \leq j \leq \overline{i}_{1}$. However, we do not have $\lambda_{\alpha_{1}\alpha_{2}\alpha_{3}i}=0$ for all $s_{\overline{i}_{2}} \leq i \leq k_{0}$ and $0 \leq |\alpha| \leq 2$. It comes from the identity in Lemma \ref{lemma:skewnormaldistribution}, which implies that all $\dfrac{\partial^{|\alpha|}{f}}{\partial{\theta}^{\alpha_{1}}v^{\alpha_{2}}m^{\alpha_{3}}}(x|\theta_{i}^{0},v_{i}^{0},m_{i}^{0})$ are not linear independent as $0 \leq |\alpha| \leq 2$. 
Therefore, this case needs a new treatment, which is divided into three steps

\paragraph{Step F.1:} From the definition of $m_{n}$, at least one coefficient $\alpha_{ijl}, \beta_{ij}, \gamma_{ij}, \lambda_{\alpha_{1}\alpha_{2}\alpha_{3}i}$ equals to 1. As all $\alpha_{ijl}, \beta_{ij}, \gamma_{ij}$ equal to 0, this result implies that at least one coefficient $\lambda_{\alpha_{1}\alpha_{2}\alpha_{3}i}$ equal to 1. Therefore, $m_{n} = |B_{\alpha_{1}^{*}\alpha_{2}^{*}\alpha_{3}^{*}}(\theta_{i^{'}}^{0},v_{i^{'}}^{0},m_{i^{'}}^{0})|$ for some $\alpha_{1}^{*},\alpha_{2}^{*},\alpha_{3}^{*}$ and $s_{\overline{i}_{2}} \leq i^{'} \leq k_{0}$. As $\Delta \theta_{i^{'}}^{n},\Delta v_{i^{'}}^{n},\Delta m_{i^{'}}^{n} \to 0$, $|B_{\alpha_{1}\alpha_{2}\alpha_{3}}(\theta_{i^{'}}^{0},v_{i^{'}}^{0},m_{i^{'}}^{0})|$ when $\alpha_{1}+\alpha_{2}+\alpha_{3}=2$ is dominated by $|B_{\alpha_{1}\alpha_{2}\alpha_{3}}(\theta_{i^{'}}^{0},v_{i^{'}}^{0},m_{i^{'}}^{0})|$ when $\alpha_{1}+\alpha_{2}+\alpha_{3}\leq 1$.Therefore, $\alpha_{1}^{*}+\alpha_{2}^{*}+\alpha_{3}^{*} \leq 1$, i.e,
 at most first order derivative. 

\paragraph{Step F.2:} As $(p_{G_{n}}(x)-p_{G_{0}}(x))/W_{2}^{2}(G_{n},G_{0}) \to 0$, we also have $(p_{G_{n}}(x)-p_{G_{0}}(x))/W_{1}(G_{n},G_{0}) \to 0$. From here, by applying Taylor expansion up to first order, we can write $(p_{G_{n}}(x)-p_{G_{0}}(x))/W_{1}(G_{n},G_{0})$ as $L_{n,1}(x)+L_{n,2}(x)+L_{n,3}(x)+L_{n,4}(x)$ where $L_{n,4}(x)$ is Taylor's remainder term, which means that $L_{n,4}(x)/W_{1}(G_{n},G_{0}) \to 0$. Additionally. $L_{n,1}(x),L_{n,2}(x),L_{n,3}(x)$ are the linear combinations of elements of $f(x|\theta_{i}^{0},v_{i}^{0},m_{i}^{0})$, $\dfrac{\partial{f}}{\partial{\theta}}(x|\theta_{i}^{0},v_{i}^{0},m_{i}^{0})$, $\dfrac{\partial{f}}{\partial{v}}(x|\theta_{i}^{0},v_{i}^{0},m_{i}^{0})$, $\dfrac{\partial{f}}{\partial{m}}(x|\theta_{i}^{0},v_{i}^{0},m_{i}^{0})$. In $L_{n,1}(x)$, the index $i$ ranges from $1$ to $\overline{i}_{1}$ while in $L_{n,2}(x)$,$L_{n,3}$, the index $i$ ranges from $\overline{i}_{1}+1$ to  $s_{\overline{i}_{2}}-1$ and from $s_{\overline{i}_{2}}$ to $k_{0}$ respectively. Assume that all of these coefficients go to $0$ as $n \to +\infty$, then we have
\begin{eqnarray} 
|B_{\alpha_{1}^{*}\alpha_{2}^{*}\alpha_{3}^{*}}(\theta_{i^{'}}^{0},v_{i^{'}}^{0},m_{i^{'}}^{0})|d(G_{n},G_{0})/W_{1}(G_{n},G_{0}) \to 0, \label{eqn:exactfittedskewnormalgeneralcasef2}
\end{eqnarray}
where the limit is due to the fact that $|B_{\alpha_{1}^{*}\alpha_{2}^{*}\alpha_{3}^{*}}(\theta_{i^{'}}^{0},v_{i^{'}}^{0},m_{i^{'}}^{0})|d(G_{n},G_{0})/W_{1}(G_{n},G_{0})$ is the maximum coefficient of $L_{n,1}(x),L_{n,2}(x),L_{n,3}(x)$. However, from the result of the proof of Theorem \ref{theorem-firstorder}, we have
\begin{eqnarray}
W_{1}(G_{n},G_{0}) \lesssim \mathop {\sum }\limits_{i=1}^{k_{0}}{p_{i}^{n}(|\Delta \theta_{i}^{n}|+|\Delta v_{i}^{n}|+ |\Delta m_{i}^{n}|)+|\Delta p_{i}^{n}|} & \lesssim & \mathop {\max }\limits_{1 \leq i \leq k_{0}}{\left\{|\Delta p_{i}^{n}|, |\Delta \theta_{i}^{n}|,|\Delta v_{i}^{n}|, \Delta m_{i}^{n}|\right\}} \nonumber \\
& = & |B_{\alpha_{1}^{*}\alpha_{2}^{*}\alpha_{3}^{*}}(\theta_{i^{'}}^{0},\sigma_{i^{'}}^{0},m_{i^{'}}^{0})|d(G_{n},G_{0}), \nonumber
\end{eqnarray}   
which contradicts to \eqref{eqn:exactfittedskewnormalgeneralcasef2}. Therefore, at least one coefficient does not vanish to 0. 

\paragraph{Step F.3:} Denote $m_{n}'$ to be the maximum among the absolute values of these coefficients and $d_{n}'=1/m_{n}'$. Then, we achieve
\begin{eqnarray}
d_{n}^{'}|B_{\alpha_{1}^{*}\alpha_{2}^{*}\alpha_{3}^{*}}(\theta_{i^{'}}^{0},v_{i^{'}}^{0},m_{i^{'}}^{0})|d(G_{n},G_{0})/W_{1}(G_{n},G_{0})=1 \ \text{for all} \ n.\nonumber
\end{eqnarray}
Therefore, as $n \to \infty$
\begin{eqnarray}
\mathop {\sum }\limits_{i=1}^{3}{d_{n}^{'}L_{n,i}(x)} \to \sum \limits_{i=1}^{k_{0}}\biggr \{\alpha_{1i}'f(x|\theta_{i}^{0},v_{i}^{0},m_{i}^{0})+\alpha_{2i}'\dfrac{\partial{f}}{\partial{\theta}}(x|\theta_{i}^{0},v_{i}^{0},m_{i}^{0})+\alpha_{3i}'\dfrac{\partial{f}}{\partial{\sigma_{i}^{2}}}(x|\theta_{i}^{0},v_{i}^{0},m_{i}^{0}) \nonumber \\
+\alpha_{4i}'\dfrac{\partial{f}}{\partial{m}}(x|\theta_{i}^{0},v_{i}^{0},m_{i}^{0})\biggr \}=0. \nonumber
\end{eqnarray}
where one of $\alpha_{1i'}',\alpha_{2i'}',\alpha_{3i'}',\alpha_{4i'}'$ differs from 0.
However, using the same argument as that of part (a), this equation will imply that $\alpha_{ji}'=0$ for all $1 \leq j \leq 4$ and $s_{\overline{i}_{2}} \leq i \leq k_{0}$, which is a contradiction. 

\noindent We have reached the
conclusion \eqref{eqn:exactfittedskewnormalone} which completes the proof.

\paragraph{Best lower bound of $V(p_{G},p_{G_{0}})$ as $G_{0}$ satisfies condition (S.2):} We have two cases
\paragraph{Case b.1:} There exists $m_{i}^{0}=0$ for some $1 \leq i \leq k_{0}$. Without loss of generality, we assume $m_{1}^{0}=0$. We construct the sequence $G_{n} \in \mathcal{E}_{k_{0}}(\Theta \times \Omega)$ as $(\Delta p_{i}^{n}, \Delta \theta_{i}^{n},\Delta v_{i}^{n}, \Delta m_{i}^{n})=(0,0,0,0)$ for all $2 \leq i \leq k_{0}$ and $\Delta p_{1}^{n}=\Delta v_{1}^{n}=0$, $\Delta \theta_{1}^{n}=\dfrac{1}{n}$, $\Delta m_{1}^{n}=-\dfrac{\sqrt{2\pi}}{\sigma_{1}^{0}n}$. With this construction, we can check that $\Delta \theta_{1}^{n}+\Delta m_{1}^{n} \sigma_{1}^{0}/\sqrt{2\pi}=0$. Using the same argument as that of part (b) of the proof of Theorem \ref{theorem-secondorder} with the notice that $V(p_{G_{n}},p_{G_{0}})={\displaystyle \int \limits_{\mathbb{R}}{|R(x)|}dx}$ where $R(x)$ is Taylor expansion's remainder in the first order, we readily achieve the conclusion of our theorem.

\paragraph{Case b.2:} There exists conformant cousin set $I_{i}$ for some $1 \leq i \leq k_{0}$. Without loss of generality, we assume $i=1$ and $j=2 \in I_{1}$. Now, we choose $G_{n}$ such that $\Delta p_{i}^{n}=\Delta \theta_{i}^{n}=\Delta v_{i}^{n}=0$ for all $1 \leq i \leq k_{0}$, $\Delta m_{i}^{n}=0$ for all $3 \leq i \leq k_{0}$, $\Delta m_{1}^{n}=\dfrac{1}{n}, \Delta m_{2}^{n}=-\dfrac{v_{2}^{0}}{v_{1}^{0}n}$. Then, we can guarantee that $\Delta m_{1}^{n}/v_{1}^{0}+\Delta m_{2}^{n}/v_{2}^{0}=0$. By means of Taylor expansion up to first order, we can check that $V(p_{G_{n}},p_{G_{0}})={\displaystyle \int \limits_{\mathbb{R}}{|R(x)|}dx}$ where $R(x)$ is Taylor remainder. From then, using the same argument as case b.1, we get the conclusion of our theorem.

\paragraph{Remark:} With extra hard work, we can also prove that $W_{2}^{2}$ is the best lower bound of $h(p_{G},p_{G_{0}})$ as $G_{0}$ satisfies condition (S.2). Therefore, for any standard estimation method ( such as the MLE) which yields $n^{-1/2}$ convergence rate for $p_{G}$, the induced rate of convergence for the mixing measure $G$ is the minimax optimal $n^{-1/4}$ under $W_{2}$ when $G_{0}$ satisfies condition (S.2) while it is the minimax optimal $n^{-1/2}$ under $W_{1}$ when $G_{0}$ satisfies condition (S.1).

\paragraph{PROOF OF THEOREM \ref{theorem:Overfittedskewnormal}}
This proof is quite similar to that of Theorem~\ref{theorem:generaloverfittedGaussian},
so we shall give only a sketch. It is sufficient to demonstrate that 
\begin{eqnarray}
\mathop {\lim }\limits_{\epsilon \to 0}{\mathop {\inf }\limits_{G \in \mathcal{O}_{k}(\Theta \times \Omega)}{\left\{\dfrac{\mathop {\sup }\limits_{x \in \mathcal{X}}{|p_{G}(x)-p_{G_{0}}(x)|}}{W_{\overline{m}}^{\overline{m}}(G,G_{0})}: W_{\overline{m}}(G,G_{0}) \leq \epsilon \right\}}} > 0. \label{eqn:overfittedskewnormalgeneralone}
\end{eqnarray}
Assume by contrary that \eqref{eqn:overfittedskewnormalgeneralone} does not hold. Here, 
we assume $\overline{r}$ is even (the case $\overline{r}$ is odd number can be addressed in the same way). 
In this case, $\overline{m}=\overline{r}$. Denote $v=\sigma^{2}$. 
Then, there is a sequence 
$G_{n}=\mathop {\sum }\limits_{i=1}^{k_{0}}{\mathop {\sum }\limits_{j=1}^{s_{i}}{p_{ij}^{n}\delta_{(\theta_{ij}^{n},v_{ij}^{n},m_{ij}^{n})}}}$ such that $(p_{ij}^{n}, \theta_{ij}^{n},v_{ij}^{n},m_{ij}^{n}) \to (p_{i}^{0}, \theta_{i}^{0},v_{i}^{0},m_{i}^{0})$ for all $1 \leq i \leq k_{0}, 1 \leq j \leq s_{i}$. 
Define
\begin{eqnarray}
d(G_{n},G_{0})=\mathop {\sum }\limits_{i=1}^{k_{0}}{\mathop {\sum }\limits_{j=1}^{s_{i}}{p_{ij}^{n}(|\Delta \theta_{ij}^{n}|^{\overline{r}}+|\Delta v_{ij}^{n}|^{\overline{r}}+|\Delta m_{ij}^{n}|^{\overline{r}})}+|p_{i.}^{n}-p_{i}^{0}|}, \nonumber
\end{eqnarray}
where $\Delta \theta_{ij}^{n}=\theta_{ij}^{n}-\theta_{i}^{0}, \Delta v_{ij}^{n}=v_{ij}^{n}-v_{i}^{0}, \Delta m_{ij}^{n}= m_{ij}^{n} - m_{i}^{0}$. 
Now, applying Taylor's expansion up to $\overline{r}$-th order, we obtain
\begin{eqnarray}
p_{G_{n}}(x)-p_{G_{0}}(x)= 
\mathop {\sum }\limits_{i=1}^{k_{0}}{\mathop {\sum }\limits_{j=1}^{s_{i}}{p_{ij}^{n}\dfrac{(\Delta \theta_{ij}^{n})^{\alpha_{1}}(\Delta v_{ij}^{n})^{\alpha_{2}}(\Delta m_{ij}^{n})^{\alpha_{3}}}{\alpha_{1}!\alpha_{2}!\alpha_{3}!}\dfrac{\partial^{|\alpha|}{f}}{\partial{\theta^{\alpha_{1}}\partial{v^{\alpha_{2}}}\partial{m^{\alpha_{3}}}}}(x|\theta_{i}^{0},v_{i}^{0},m_{i}^{0})}} + \nonumber \\
+ \mathop {\sum }\limits_{i=1}^{k_{0}}{(p_{i.}-p_{i}^{0})f(x|\theta_{i}^{0},v_{i}^{0},m_{i}^{0})}+R_{1}(x)
 := A_{1}(x)+B_{1}(x)+R_{1}(x), \nonumber 
\end{eqnarray}
where $\alpha=(\alpha_{1},\alpha_{2},\alpha_{3})$, $R_{1}(x)$ is Taylor remainder and $R_{1}(x)/d(G_{n},G_{0}) \to 0$. 

Now we invoke the key identity (cf. Lemma \ref{lemma:skewnormaldistribution})
\begin{eqnarray}
\dfrac{\partial{f}}{\partial{v}}(x|\theta,v,m)=\dfrac{1}{2}\dfrac{\partial^{2}{f}}{\partial{\theta^{2}}}(x|\theta,v,m)+\dfrac{m^{3}+m}{2v}\dfrac{\partial{f}}{\partial{m}}(x|\theta,v,m). \nonumber
\end{eqnarray}
It follows by induction that, for any $\alpha_{2} \geq 1$
\begin{eqnarray}
\dfrac{\partial^{\alpha_{2}}f}{\partial{v^{\alpha_{2}}}}=\dfrac{1}{2^{\alpha_{2}}}\dfrac{\partial^{2\alpha_{2}}{f}}{\partial{\theta^{2\alpha_{2}}}}+\mathop {\sum }\limits_{i=1}^{\alpha_{2}}{\dfrac{1}{2^{\alpha_{2}-i}}\dfrac{\partial^{i-1}{f}}{\partial{v^{i-1}}}\left(\dfrac{m^{3}+m}{2v}\dfrac{\partial^{2\alpha_{2}-2i+1}{f}}{\partial{\theta^{2(\alpha_{2}-i)}}\partial{m}}\right)}. \nonumber
\end{eqnarray}
Therefore, for any $\alpha=(\alpha_{1},\alpha_{2},\alpha_{3})$ such that $\alpha_{2} \geq 1$, we have
\begin{eqnarray}
\dfrac{\partial^{|\alpha|}{f}}{\partial{\theta^{\alpha_{1}}\partial{v^{\alpha_{2}}}\partial{m^{\alpha_{3}}}}}=\dfrac{1}{2^{\alpha_{2}}}\dfrac{\partial^{\alpha_{1}+2\alpha_{2}+\alpha_{3}}{f}}{\partial{\theta^{\alpha_{1}+2\alpha_{2}}}\partial{m^{\alpha_{3}}}}+\mathop {\sum }\limits_{i=1}^{\alpha_{2}}{\dfrac{1}{2^{\alpha_{2}-i}}\dfrac{\partial^{\alpha_{1}+\alpha_{3}+i-1}}{\partial{\theta^{\alpha_{1}}\partial{m^{\alpha_{3}}}\partial{v^{i-1}}}}\left(\dfrac{m^{3}+m}{2v}\dfrac{\partial^{2\alpha_{2}-2i+1}{f}}{\partial{\theta^{2(\alpha_{2}-i)}}\partial{m}}\right)}. \nonumber
\end{eqnarray}
Continue this identity until the right hand side of this equation only contains derivatives in terms of $\theta$ and $m$, which means all the derivatives involving $v$ can be reduced to the derivatives with only $\theta$ and $m$. As a consequence, $A_{1}(x)/d(G_{n},G_{0})$ is the linear combination of elements of $\dfrac{\partial^{|\beta|}{f}}{\partial{\theta^{\beta_{1}}m^{\beta_{2}}}}(x|\theta,v,m)$ where $0 \leq |\beta| \leq 2\overline{r}$ (not necessarily all the value of $\beta$ in this range). 
We can check that 
for each $\gamma = 1,\ldots, 2\overline{r}$, the
coefficient of $\dfrac{\partial^{\gamma}{f}}{\partial{\theta^{\gamma}}}(x|\theta_{i}^{0},v_{i}^{0},m_{i}^{0})$ is 
\begin{eqnarray}
E_{\gamma}(\theta_{i}^{0},v_{i}^{0},m_{i}^{0})=\left[\mathop {\sum }\limits_{j=1}^{s_{i}}{p_{ij}^{n}\mathop {\sum }\limits_{\substack{n_{1}+2n_{2}=\gamma \\ n_{1}+n_{2} \leq \overline{r}}}{\dfrac{(\Delta \theta_{ij}^{n})^{n_{1}}(\Delta v_{ij}^{n})^{n_{2}}}{2^{n_{2}}n_{1}!n_{2}!}}}\right]/d(G_{n},G_{0}). \nonumber
\end{eqnarray}
Additionally, the coefficient of the $\overline{r}$-th order
derivative with respect to $m$, $\dfrac{\partial^{\overline{r}}{f}}{\partial{m^{\overline{r}}}}(x|\theta_{i}^{0},v_{i}^{0},m_{i}^{0})$, 
is $\mathop {\sum }\limits_{j=1}^{s_{i}}{p_{ij}^{n}(\Delta m_{ij})^{\overline{r}}}/d(G_{n},G_{0})$. Therefore, if all of the coefficients of $A_{1}(x)/d(G_{n},G_{0}),B_{1}(x)/d(G_{n},G_{0})$ go to $0$, then as $\overline{r}$ is even, we obtain $\mathop {\sum }\limits_{j=1}^{s_{i}}{p_{ij}^{n}|\Delta m_{ij}|^{\overline{r}}}/d(G_{n},G_{0}) \to 0$ for all $1 \leq i \leq k_{0}$ and $\mathop {\sum }\limits_{i=1}^{k_{0}}{|p_{i.}^{n}-p_{i}^{0}|}/d(G_{n},G_{0}) \to 0$. It implies that 
\begin{eqnarray}
\mathop {\sum }\limits_{i=1}^{k_{0}}{\mathop {\sum }\limits_{j=1}^{s_{i}}{p_{ij}^{n}(|\Delta \theta_{ij}^{n}|^{\overline{r}}+|\Delta m_{ij}|^{\overline{r}})}}/d(G_{n},d_{G_{0}}) \to 1. \nonumber
\end{eqnarray}
Therefore, we can find an index $i^{*} \in \left\{1,\ldots,k_{0}\right\}$ such that $\mathop {\sum }\limits_{j=1}^{s_{i^{*}}}{p_{i^{*}j}^{n}(|\Delta \theta_{i^{*}j}^{n}|^{\overline{r}}+|\Delta m_{i^{*}j}|^{\overline{r}})}/d(G_{n},d_{G_{0}}) \not \to 0$. By multiply this term with $E_{\gamma}(\theta_{i^{*}}^{0},v_{i^{*}}^{0},m_{i^{*}}^{0})$ as $1 \leq \gamma \leq \overline{r}$, we obtain
\begin{eqnarray}
\left[\mathop {\sum }\limits_{j=1}^{s_{i^{*}}}{p_{i^{*}j}^{n}\mathop {\sum }\limits_{\substack{n_{1}+2n_{2}=\gamma \\ n_{1}+n_{2} \leq \overline{r}}}{\dfrac{(\Delta \theta_{i^{*}j}^{n})^{n_{1}}(\Delta v_{i^{*}j}^{n})^{n_{2}}}{2^{n_{2}}n_{1}!n_{2}!}}}\right]/\mathop {\sum }\limits_{j=1}^{s_{i^{*}}}{p_{i^{*}j}^{n}(|\Delta \theta_{i^{*}j}^{n}|^{\overline{r}}+|\Delta m_{i^{*}j}|^{\overline{r}})} \to 0,\nonumber
\end{eqnarray}
which is a contradiction due to the proof of Theorem \ref{theorem:generaloverfittedGaussian}. Therefore, not all the coefficients of $A_{1}(x),B(x)$ go to 0. As a consequence, for all $x \in \mathbb{R}$, $(p_{G_{n}}(x)-p_{G_{0}}(x))/d(G_{n},G_{0})$ converges to the linear combinations of $\dfrac{\partial^{|\beta|}{f}}{\partial{\theta^{\beta_{1}}m^{\beta_{2}}}}(x|\theta_{I}^{0},v_{i}^{0},m_{i}^{0})$ where at least one coefficient differs from 0. 
However, due to Assumption (S1) on $G_{0}$, 
the collection of 
$\dfrac{\partial^{|\beta|}{f}}{\partial{\theta^{\beta_{1}}\partial{m^{\beta_{2}}}}}(x|\theta_{i}^{0},v_{i}^{0},m_{i}^{0})$ are linearly independent, which is a contradiction. This concludes our proof. 

The following addresses the remarks following the statement of Theorem
\ref{theorem:Overfittedskewnormal}.

\paragraph{Best lower bound when $k-k_{0}=1$:} The remark 
regarding the removal of the constraint $\mathcal{O}_{k,c_{0}}$ is immediate from (the proof of)
Proposition \ref{theorem:Gaussianoverfittedbytwo}.
To show the bound is sharp in this case, 
we construct sequence $G_{n}=\mathop {\sum }\limits_{i=1}^{k}{\mathop {\sum }\limits_{j=1}^{s_{i}}{p_{ij}^{n}\delta_{(\theta_{i}^{n},v_{i}^{n},m_{i}^{n})}}}$ as follows $s_{1}=2,s_{i}=1$ for all $2 \leq i \leq k_{0}$, $p_{11}^{n}=p_{12}^{n}=p_{1}^{0}/2$, $\Delta \theta_{11}^{n}=1/n, \Delta \theta_{12}^{n}=-1/n$, $\Delta v_{11}^{n}=\Delta v_{12}^{n}=-1/n^{2}$, $\Delta m_{11}^{n}=\Delta m_{12}^{n}=a_{n}$ where $a_{n}$ is the solution of following equation
\begin{eqnarray}
\dfrac{3m_{1}^{0}}{n^{2}v_{1}^{0}}a_{n}^{2}-\left(2-\dfrac{3(m_{1}^{0})^{2}+1}{n^{2}v_{i}^{0}}-\dfrac{3(m_{1}^{0})^{2}+1}{n^{4}(v_{i}^{0})^{2}}\right)a_{n}+\dfrac{(m_{1}^{0})^{3}+m_{1}^{0}}{n^6(v_{1}^{0})^{3}}+\dfrac{(m_{1}^{0})^{3}+m_{1}^{0}}{n^{4}(v_{1}^{0})^{2}} & + & \nonumber \\
\dfrac{(m_{1}^{0})^{3}+ \nonumber m_{1}^{0}}{n^{2}v_{1}^{0} }  -  \dfrac{m_{1}^{0}((m_{1}^{0})^{2}+1)^{2}}{4n^{4}(v_{1}^{0})^{2}} & = & 0,  \label{eqn:bestlowerboundoverfittedskewnormal}
\end{eqnarray}
which has the solution when $n$ is sufficiently large. Additionally, $|a_{n}| \asymp 1/n^{2} \to 0$ when $n \to \infty$. The choice of $a_{n}$ will be discussed in the sequel. 
Now, for any $1 \leq r <4$, we have $W_{1}^{r}(G_{n},G_{0}) \gtrsim 1/n^{r}$. By using Taylor expansion up to the fourth order, we can write $(p_{G_{n}}(x)-p_{G_{0}}(x))/W_{1}^{r}(G_{n},G_{0})$ as the linear combination of the first part, which consists of $\dfrac{\partial{f}}{\partial{m}}(x|\theta_{i}^{0},v_{i}^{0},m_{i}^{0})$, $\dfrac{\partial{f}}{\partial{v^{2}}}(x|\theta_{i}^{0},v_{i}^{0},m_{i}^{0})$, $\dfrac{\partial{f}}{\partial{\theta}}(x|\theta_{i}^{0},v_{i}^{0},m_{i}^{0})$, $\dfrac{\partial^{2}{f}}{\partial{\theta}\partial{m}}(x|\theta_{i}^{0},v_{i}^{0},m_{i}^{0})$, $\dfrac{\partial^{2}{f}}{\partial{\theta}\partial{v}}(x|\theta_{i}^{0},v_{i}^{0},m_{i}^{0})$ plus the second part, which consists of the remaining 
derivatives and the Taylor remainder. 
Note that the second part always converges to 0. 
For $i=2, \ldots, k_0$, the coefficients of the derivatives in the first part are 0,
thanks to our construction of $G_n$.
Thus, only the case left is when $i=1$. 
By direct computation, the 
coefficient of $\dfrac{\partial{f}}{\partial{m}}(x|\theta_{1}^{0},v_{1}^{0},m_{1}^{0})$ is
\begin{eqnarray*}
\dfrac{(m_{1}^{0})^{3} + m_{1}^{0}}{2v_{1}^{0}}\mathop{\sum}\limits_{i=1}^{2}{p_{1i}^{n}(\Delta \theta_{1i}^{n})^{2}}+
\dfrac{(m_{1}^{0})^{3}+m_{1}^{0}}{2(v_{1}^{0})^{2}}\mathop {\sum }\limits_{i=1}^{2}{p_{1i}^{n}(\Delta \theta_{1i}^{n})^{2}\Delta v_{1i}^{n}} \\
-\dfrac{3(m_{1}^{0})^{2}+1}{2v_{1}^{0}}\mathop {\sum }\limits_{i=1}^{2}{p_{1i}^{n}(\Delta \theta_{1i}^{n})^{2}\Delta m_{1i}^{n}}+ \nonumber 
\dfrac{m_{1}^{0}((m_{1}^{0})^{2}+1)^{2}}{8(v_{1}^{0})^{2}}\mathop {\sum }\limits_{i=1}^{2}{p_{i}^{n}(\Delta \theta_{1i}^{n})^{4}} - \\
\dfrac{(m_{1}^{0})^{3}+m_{1}^{0}}{2(v_{1}^{0})^{3}}\mathop {\sum }\limits_{i=1}^{2}{p_{i}^{n}(\Delta \theta_{1i}^{n})^{2}(\Delta v_{1i}^{n})^{2}} + 
\dfrac{3(m_{1}^{0})^{2}+1}{2(v_{1}^{0})^{2}}\mathop {\sum }\limits_{i=1}^{2}{p_{1i}^{n}(\Delta \theta_{1i}^{n})^{2}\Delta v_{1i}^{n} \Delta m_{1i}^{n}}- \\
\dfrac{3m_{1}^{0}}{2v_{1}^{0}}\mathop {\sum }\limits_{i=1}^{2}{p_{1i}^{n}(\Delta \theta_{1i}^{n})^{2}(\Delta m_{1i}^{n})^{2}}+ 
 \mathop {\sum }\limits_{i=1}^{2}{p_{i}^{n}\Delta m_{1i}^{n}}=0, \nonumber
\end{eqnarray*}
where the equality is due to the fact that the left hand side of this equation is equal to the left hand side of equation \eqref{eqn:bestlowerboundoverfittedskewnormal}. Therefore, the choice of $a_{n}$ is to guarantee the coefficient of $\dfrac{\partial{f}}{\partial{m}}$ to be 0. With similar calculation, we can easily check that all the coefficients of $\dfrac{\partial{f}}{\partial{v}}$, $\dfrac{\partial{f}}{\partial{\theta}}$, $\dfrac{\partial{f}}{\partial{\theta}\partial{m}}$, $\dfrac{\partial{f}}{\partial{\theta}\partial{v}}$ are also 0. Therefore, the assertion about the best lower bound immediately follows. 

\paragraph{Case $k-k_{0}=2$:} In this scenario,
we conjecture that $W_{4}^{4}(G,G_{0})$ is still the best lower bound of $V(p_{G},p_{G_{0}})$. 
Following the same proof recipe as above, such a conclusion follows from the hypothesis
that for any fixed value $m \neq 0$, $\sigma^{2}>0$, the following system of 8 polynomial equations
\begin{eqnarray*}
\mathop {\sum_{i=1}^{3} }{d_{i}^{2}a_{i}} = 0, \;\;
\mathop {\sum_{i=1}^{3} }{d_{i}^{2}(a_{i}^{2}+b_{i})}  =  0, \;\;
\mathop {\sum_{i=1}^{3} }{d_{i}^{2}(\dfrac{a_{i}^{3}}{3}+a_{i}b_{i})}  =  0, \;\; 
\mathop {\sum_{i=1}^{3} }{d_{i}^{2}(\dfrac{a_{i}^{4}}{6}+a_{i}^{2}b_{i}+\dfrac{b_{i}^{2}}{2})} = 0 \nonumber \\
\sum_{i=1}^{3} d_{i}^{2} 
\biggr ( -\dfrac{m^{3}+m}{2\sigma^{2}} a_{i}^{2} 
+ \dfrac{m^{3}+m}{2\sigma^{4}}a_{i}^{2}b_{i} 
-\dfrac{3m^{2}+1}{2\sigma^{2}}a_{i}^{2}c_{i}  \\
+ \dfrac{3(m^{2}+1)(m^{3}+m)}{4!\sigma^{4}}a_{i}^{4} 
- \dfrac{m^{3}+m}{2\sigma^{6}}a_{i}^{2}b_{i}^{2}   
+ \dfrac{3m^{2}+1}{2\sigma^{4}}a_{i}^{2}b_{i}c_{i}
- \dfrac{3m}{2\sigma^{2}} a_{i}^{2}c_{i}^{2} 
+ c_{i}   \biggr )
  = 0  \\
\mathop {\sum_{i=1}^{3} }{d_{i}^{2}(-\dfrac{m^{3}+m}{6\sigma^{2}}a_{i}^{3}+\dfrac{m^{3}+m}{6\sigma^{4}}a_{i}^{3}b_{i}-\dfrac{3m^{2}+1}{6\sigma^{2}}a_{i}^{3}c_{i}+a_{i}c_{i})}=0 \nonumber \\
\mathop {\sum_{i=1}^{3} }{d_{i}^{2}(-\dfrac{m^{3}+m}{6\sigma^{2}}a_{i}^{4}+\dfrac{m^{3}+m}{2\sigma^{4}}a_{i}^{2}b_{i}^{2}-\dfrac{3m^{2}+1}{2\sigma^{2}}a_{i}^{2}b_{i}c_{i}+\dfrac{b_{i}c_{i}}{2})}=0 \nonumber \\
\mathop {\sum_{i=1}^{3} }{d_{i}^{2}(\dfrac{(m^{3}+m)^{2}}{4!\sigma^{4}}a_{i}^{4}+\dfrac{m^{3}+m}{4\sigma^{4}}a_{i}^{2}b_{i}c_{i}-\dfrac{3m^{2}+1}{2\sigma^{2}}a_{i}^{2}c_{i}^{2}+\dfrac{c_{i}^{2}}{2})}=0 \nonumber
\end{eqnarray*}
does not have any non-trivial solution, i.e $d_{i} \neq 0$ for all $1 \leq i \leq 3$ and at least one among $a_{1},\ldots,a_{3},b_{1},$\\ $\ldots,b_{3},c_{1},\ldots,c_{3}$ is non-zero.

\newpage
\begin{center}
\textbf{APPENDIX II}
\end{center}

For the sake of completeness, we collect herein the proof of technical results and auxiliary arguments that
were left out of the main text and Appendix I. 

\paragraph{PROOF OF COROLLARY \ref{corollary:lowerboundtransformation}} From Theorem \ref{theorem:identifiabilitytransformation}, the class $\left\{g(x|\eta,\Lambda),\eta \in \Theta^{*}, \Lambda \in \Omega^{*} \right\}$ is identifiable in the first order. From the proof of Theorem \ref{theorem-firstorder}, in order to achieve the conclusion of our theorem, it remains to verify that $g(x|\eta,\Lambda)$ satisfies conditions \eqref{Eqn:Lipschitz1} and \eqref{Eqn:Lipschitz2}. As the first derivative of $f$ in terms of $\theta$ and $\Sigma$ is $\alpha$-Holder continuous, $f(x|\theta,\Sigma)$ satisfies conditions \eqref{Eqn:Lipschitz1} and \eqref{Eqn:Lipschitz2} with $\delta_{1}=\delta_{2}=\alpha$.

Now, for any $\eta^{1},\eta^{2} \in \Theta^{*}$, $\Lambda \in \Omega^{*}$, we have $T(\eta^{1},\Lambda)=(\theta^{1},\Sigma)$ and $T(\eta^{2},\Lambda)=(\theta^{2},\Sigma)$. For any $1 \leq i \leq d_{1}$, we obtain
\begin{eqnarray}
\dfrac{\partial{}}{\partial{\eta_{i}}}(g(x|\eta^{1},\Lambda)-g(x|\eta^{2},\Lambda)) = \mathop {\sum }\limits_{l=1}^{d_{1}}\dfrac{\partial{f}}{\partial{\theta_{l}}}(x|\theta^{1},\Sigma)\dfrac{\partial{\left[T_{1}(\eta^{1},\Lambda)\right]_{l}}}{\partial{\eta_{i}}}-\mathop {\sum }\limits_{l=1}^{d_{1}}\dfrac{\partial{f}}{\partial{\theta_{l}}}(x|\theta^{2},\Sigma)\dfrac{\partial{\left[T_{1}(\eta^{2},\Lambda)\right]_{l}}}{\partial{\eta_{i}}}+ \nonumber \\
\mathop {\sum }\limits_{1 \leq u,v \leq d_{2}}{\dfrac{\partial{f}}{\partial{\Sigma_{uv}}}(x|\theta^{1},\Sigma)\dfrac{\partial{\left[T_{2}(\eta^{1},\Lambda)\right]_{uv}}}{\partial{\eta_{i}}}}-\mathop {\sum }\limits_{1 \leq u,v \leq d_{2}}{\dfrac{\partial{f}}{\partial{\Sigma_{uv}}}(x|\theta^{2},\Sigma)\dfrac{\partial{\left[T_{2}(\eta^{2},\Lambda)\right]_{uv}}}{\partial{\eta_{i}}}}. \nonumber
\end{eqnarray}
Notice that,
\begin{eqnarray}
\mathop {\sum }\limits_{l=1}^{d_{1}}\dfrac{\partial{f}}{\partial{\theta_{l}}}(x|\theta^{1},\Sigma)\dfrac{\partial{\left[T_{1}(\eta^{1},\Lambda)\right]_{l}}}{\partial{\eta_{i}}}-\mathop {\sum }\limits_{l=1}^{d_{1}}\dfrac{\partial{f}}{\partial{\theta_{l}}}(x|\theta^{2},\Sigma)\dfrac{\partial{\left[T_{1}(\eta^{2},\Lambda)\right]_{l}}}{\partial{\eta_{i}}} \leq  \|\dfrac{\partial{f}}{\partial{\theta}}(x|\theta^{1},\Sigma)- \dfrac{\partial{f}}{\partial{\theta}}(x|\theta^{2},\Sigma) \| & &\nonumber \\
 \times \|\dfrac{\partial{T_{1}}}{\partial{\eta_{i}}}(\eta^{1},\Lambda)\| + \|\dfrac{\partial{f}}{\partial{\theta}}(x|\theta^{2},\Sigma)\|\|\dfrac{\partial{T_{1}}}{\partial{\eta_{i}}}(\eta^{1},\Lambda)-\dfrac{\partial{T_{1}}}{\partial{\eta_{i}}}(\eta^{2},\Lambda)\| \nonumber \\
\leq  L_{1}\|\theta^{1}-\theta^{2}\|^{\alpha}+L_{2}\|\eta^{1}-\eta^{2}\|^{\alpha}, & & \nonumber
\end{eqnarray}
where $L_{1}$, $L_{2}$ are two positive constants from the $\alpha$-Holder continuity and the boundedness of the first derivative of $f(x|\theta,\Sigma)$ and $T(\eta,\Lambda)$. Moreover, since $T$ is Lipschitz continuous, it implies that $\|\theta^{1}-\theta^{2}\| \lesssim \|\eta^{1}-\eta^{2}\|$. Therefore, the above inequality can be rewritten as
\begin{eqnarray}
\mathop {\sum }\limits_{l=1}^{d_{1}}\dfrac{\partial{f}}{\partial{\theta_{l}}}(x|\theta^{1},\Sigma)\dfrac{\partial{\left[T_{1}(\eta^{1},\Lambda)\right]_{l}}}{\partial{\eta_{i}}}-\mathop {\sum }\limits_{l=1}^{d_{1}}\dfrac{\partial{f}}{\partial{\theta_{l}}}(x|\theta^{2},\Sigma)\dfrac{\partial{\left[T_{1}(\eta^{2},\Lambda)\right]_{l}}}{\partial{\eta_{i}}} \lesssim \|\eta^{1}-\eta^{2}\|^{\alpha}. \nonumber
\end{eqnarray}
With the similar argument, we get
\begin{eqnarray}
\mathop {\sum }\limits_{1 \leq u,v \leq d_{2}}{\dfrac{\partial{f}}{\partial{\Sigma_{uv}}}(x|\theta^{1},\Sigma)\dfrac{\partial{\left[T_{2}(\eta^{1},\Lambda)\right]_{uv}}}{\partial{\eta_{i}}}}-\mathop {\sum }\limits_{1 \leq u,v \leq d_{2}}{\dfrac{\partial{f}}{\partial{\Sigma_{uv}}}(x|\theta^{2},\Sigma)\dfrac{\partial{\left[T_{2}(\eta^{2},\Lambda)\right]_{uv}}}{\partial{\eta_{i}}}} \lesssim \|\eta^{1}-\eta^{2}\|^{\alpha}. \nonumber
\end{eqnarray}
Thus, for any $1 \leq i \leq d_{1}$,
\begin{eqnarray}
\left|\dfrac{\partial{}}{\partial{\eta_{i}}}(g(x|\eta^{1},\Lambda)-g(x|\eta^{2},\Lambda)) \right| \lesssim \|\eta^{1}-\eta^{2}\|^{\alpha}. \nonumber
\end{eqnarray}
As a consequence, for any $\gamma_{1} \in \mathbb{R}^{d_{1}}$,
\begin{eqnarray}
\left|\gamma_{1}^{T} \biggr ( \dfrac{\partial{g}}{\partial{\eta}}(x|\eta^{1},\Sigma)-\dfrac{\partial{g}}{\partial{\eta}}(x|\eta^{2},\Sigma) \biggr ) \right| \lesssim \|\dfrac{\partial{g}}{\partial{\eta}}(x|\eta^{1},\Sigma)-\dfrac{\partial{g}}{\partial{\eta}}(x|\eta^{2},\Sigma)\|\|\gamma_{1}\| \lesssim \|\eta^{1}-\eta^{2}\|^{\alpha}\|\gamma_{1}\|, \nonumber
\end{eqnarray}
which means that condition \eqref{Eqn:Lipschitz1} is satisfied by $g(x|\eta,\Lambda)$.
Likewise, we also can demonstrate that condition \eqref{Eqn:Lipschitz2} is satisfied by $g(x|\theta,\Lambda)$. Therefore, the conclusion of our corollary is achieved.

\paragraph{PROOF OF THEOREM \ref{identifiability-univariatecharacterization}}
(a)  Assume that we have $\alpha_{j},\beta_{j},\gamma_{j} \in \mathbb{R}$ as $1 \leq j \leq k$, $k \geq 1$ such that:
\begin{eqnarray}
\mathop {\sum }\limits_{j=1}^{k}{\alpha_{j}f(x|\theta_{j},\sigma_{j})+\beta_{j}\dfrac{\partial{f}}{\partial{\theta}}(x|\theta_{j},\sigma_{j})+\gamma_{j}\dfrac{\partial{f}}{\partial{\sigma}}(x|\theta_{j},\sigma_{j})}=0. \nonumber
\end{eqnarray}

Multiply both sides of the above equation with $\exp(itx)$ and take the integral in $\mathbb{R}$, we obtain the following result:
\begin{eqnarray}
\mathop {\sum }\limits_{j=1}^{k}{\left[(\alpha_{j}'+\beta_{j}'(it))\phi(\sigma_{j}t)+\gamma_{j}'\psi(\sigma_{j}t)\right]\exp(it\theta_{j})}=0, \label{eq:logistic1}
\end{eqnarray}
where $\alpha_{j}'=\alpha_{j}-\dfrac{\gamma_{j}}{\sigma_{j}},\beta_{j}'=\beta_{j},\gamma_{j}'=-\dfrac{\gamma_{j}}{\sigma_{j}}$, ${\displaystyle \phi(t)=\int \limits_{\mathbb{R}}{\exp(itx)f(x)}dx}$, and ${\displaystyle \psi(t)=\int \limits_{\mathbb{R}}{\exp(itx)xf'(x)}dx}$.\\

By direct calculation, we obtain $\phi(t)=\dfrac{\Gamma(p+it)\Gamma(q-it)}{\Gamma(p)\Gamma(q)}$.  Additionally, from the property of Gamma function and Euler's reflection formula, as $p,q$ are two positive integers, we have
\begin{eqnarray}
\Gamma(p+it)\Gamma(q-it)=\begin{cases} \prod \limits_{j=1}^{p-1}(p-j+it)\prod \limits_{j=1}^{q-1}(q-j-it)\dfrac{\pi t}{\sinh(\pi t)}, & \mbox{if } p,q \geq 2 \\ \prod \limits_{j=1}^{p-1}(p-j+it)\dfrac{\pi t}{\sinh(\pi t)}, & \mbox{if } p \geq 2, q=1 \\ \prod \limits_{j=1}^{q-1}(q-j-it)\dfrac{\pi t}{\sinh(\pi t)}, & \mbox{if } p=1, q \geq 2 \\ \dfrac{\pi t}{\sinh(\pi t)}, & \mbox{if } p=q=1 \end{cases}. \label{eq:logistic12}
\end{eqnarray}

From now, we only consider the case $p,q \geq 2$ as other cases can be argued in the same way. Denote $\prod \limits_{j=1}^{p-1}(p-j+it)\prod \limits_{j=1}^{q-1}(q-j-it)= \mathop {\sum }\limits_{u=0}^{p+q-2}{a_{u}t^{u}}$. It is clear that $a_{0}=\prod \limits_{j=1}^{p-1}{(p-j)}\prod \limits_{j=1}^{q-1}{(q-j)}$ and $a_{p+q-2}=(-1)^{q-1}.i^{p+q-2} \neq 0$.\\
From \eqref{eq:logistic12}, the characteristic function $\phi(t)$ can be rewritten as
\begin{eqnarray}
\phi(t)=\dfrac{2 \pi \exp(\pi t)(\mathop {\sum }\limits_{u=0}^{p+q-2}{a_{u}t^{u+1}})}{\Gamma(p)\Gamma(q)(\exp(2\pi t)-1)}. \label{eq:logistic13}
\end{eqnarray}
Additionally, since $xf'(x)$ and $f'(x)$ are integrable functions,
\begin{eqnarray}
{\displaystyle \psi(t)=\int \limits_{\mathbb{R}}{\exp(itx)xf'(x)}dx=-i\dfrac{\partial{}}{\partial{t}}\left(\int \limits_{\mathbb{R}}{\exp(itx)f'(x)}dx\right)=-i\dfrac{\partial{}}{\partial{t}}\left(it\phi(t)\right)=\phi(t)+t\phi'(t)}. \quad  \nonumber
\end{eqnarray}
By direct computation, we obtain
\begin{eqnarray}
\psi(t)=\dfrac{2\pi(\mathop {\sum }\limits_{u=0}^{p+q-2}{a_{u}(u+2)t^{u+1}})\exp(\pi t)}{\Gamma(p)\Gamma(q)(\exp(2\pi t)-1)}-\dfrac{2\pi^{2}(\mathop {\sum }\limits_{u=0}^{p+q-2}{a_{u}t^{u+2}})(\exp(2\pi t)+1)\exp(\pi t)}{\Gamma(p)\Gamma(q)(\exp(\pi t)-1)^{2}}. \label{eq:logistic16}
\end{eqnarray}
Combining \eqref{eq:logistic13} and \eqref{eq:logistic16}, we can rewrite \eqref{eq:logistic1} as
\begin{eqnarray}
\mathop {\sum }\limits_{j=1}^{k}{(\alpha_{j}'+\beta_{j}'(it))}\dfrac{(\mathop {\sum }\limits_{u=0}^{p+q-2}{a_{u}\sigma_{j}^{u+1}t^{u+1}})\exp((\pi\sigma_{j}+\theta_{j}) t)}{(\exp(2\pi \sigma_{j} t)-1)} & + & \nonumber \\ \dfrac{\gamma_{j}'(\mathop {\sum }\limits_{u=0}^{p+q-2}{a_{u}(u+2)\sigma_{j}^{u+1}t^{u+1}})\exp((\pi \sigma_{j}+i\theta_{j}) t)}{(\exp(2\pi \sigma_{j}t)-1)} & - & \nonumber \\
\dfrac{\gamma_{j}'\pi (\mathop {\sum }\limits_{u=0}^{p+q-2}{a_{u}\sigma_{j}^{u+2}t^{u+2}})(\exp(2\pi \sigma_{j}t)+1)\exp((\pi\sigma_{j}+i\theta_{j}) t)}{\Gamma(p)\Gamma(q)(\exp(\pi \theta_{j} t)-1)^{2}} & = &  0. \nonumber
\end{eqnarray}

Denote $t'=\pi t$, $\theta_{j}'=\dfrac{\theta_{j}}{\pi}$, $\beta_{j}^{''}=\dfrac{\beta_{j}'}{\pi}$, $a_{u}^{(j)}=\dfrac{a_{u}\sigma_{j}^{u+1}}{\pi^{u+1}}$, $b_{u}^{(j)}=\dfrac{a_{u}(u+2)\sigma_{j}^{u+1}}{\pi^{u+1}}$, and $c_{u}^{(j)}=\dfrac{a_{u}\sigma_{j}^{u+2}}{\pi^{u+2}}$ for all $1 \leq j \leq k$, $0 \leq u \leq p+q-2$ and multiply both sides of the above equation with $\displaystyle{\prod \limits_{j=1}^{k}{(\exp(2\sigma_{j}t)-1)^{2}}}$, we can rewrite it as
\begin{eqnarray}
\mathop {\sum }\limits_{j=1}^{k}{((\alpha_{j}'+\beta_{j}^{''}(it'))(\mathop {\sum }\limits_{u=0}^{p+q-2}{a_{u}^{(j)}(t')^{u+1}})} & + & \nonumber \\
\gamma_{j}'(\mathop {\sum }\limits_{u=0}^{p+q-2}{b_{u}^{(j)}(t')^{u+1}}))\exp((\sigma_{j}+i\theta_{j}')t')(\exp(2\sigma_{j}t')-1) \prod \limits_{l \neq j}{(\exp(2\sigma_{l}t')-1)^{2}} & - & \nonumber \\
\pi \gamma_{j}'(\mathop {\sum }\limits_{u=0}^{p+q-2}{c_{u}^{(j)}(t')^{u+2}})\exp((\sigma_{j}+i\theta_{j}')t')(\exp(2\sigma_{j}t')+1)\prod \limits_{l \neq j}{(\exp(2\sigma_{l}t')-1)^{2}} & = &  0. \label{eq:logistic17}
\end{eqnarray}

Without loss of generality, we assume that $\sigma_{1} \leq \sigma_{2} \leq ... \leq \sigma_{k}$. Note that, we can view $\exp(t'\sigma_{j})(\exp(2\sigma_{j}t')-1)\prod \limits_{l \neq j}{(\exp(2\sigma_{l}t')-1)^{2}}$ as $\mathop {\sum }\limits_{u=1}^{m_{j}}{d_{u}^{(j)}\exp(t'e_{u}^{(j)})}$ where $e_{1}^{(j)} < e_{2}^{(j)} < ... < e_{m_{j}}^{(j)}$ are just the combinations of $\sigma_{1},\sigma_{2},...,\sigma_{k}$ and $m_{j} \geq 1$ for all $1 \leq j \leq k$. Similarly, we can write $\exp(t'\sigma_{j})(\exp(2\sigma_{j}t')+1)\prod \limits_{l \neq j}{(\exp(2\sigma_{l}t')-1)^{2}}$ as $\mathop {\sum }\limits_{u=1}^{n_{j}}{k_{u}^{(j)}\exp(t'h_{u}^{(j)})}$, where $h_{1}^{(j)}<...<h_{n_{j}}^{(j)}$ and $n_{j} \geq 1$ for all $1 \leq j \leq k$.\\

Direct calculation yields $e_{m_{j}}^{(j)}=h_{n_{j}}^{(j)}=4\mathop {\sum }\limits_{l \neq j}{\sigma_{l}}+3\sigma_{j}$ and $e_{m_{j}}^{(j)}=h_{n_{j}}^{(j)}=1$ for all $1 \leq j \leq k$.  From the assumption, it is straightforward that $e_{m_{1}}^{(1)} \geq e_{m_{2}}^{(2)} \geq ... \geq e_{m_{k}}^{(k)}$. Additionally, by denoting $(\alpha_{j}'+\beta_{j}^{''}(it'))(\mathop {\sum }\limits_{u=0}^{p+q-2}{a_{u}^{(j)}(t')^{u+1}})+\gamma_{j}'(\mathop {\sum }\limits_{u=0}^{p+q-2}{b_{u}^{(j)}(t')^{u+1}})=\mathop {\sum }\limits_{u=0}^{p+q-1}{f_{u}^{(j)}(t')^{u+1}}$, we obtain $f_{0}^{(j)}=\alpha_{j}'a_{0}^{(j)}+\gamma_{j}'b_{0}^{(j)}$ and $f_{p+q-1}^{(j)}=i\beta_{j}^{''}a_{p+q-2}^{(j)}$ for all $1 \leq j \leq k$.\\

By applying the Laplace transformation in both sides of equation \eqref{eq:logistic17}, we get:
\begin{eqnarray}
\mathop {\sum }\limits_{j=1}^{k}{\mathop {\sum }\limits_{u=0}^{p+q-1}{f_{u}^{(j)}\mathop {\sum }\limits_{u_{1}=1}^{m_{j}}{\dfrac{d_{u_{1}}^{(j)}(u+1)!}{(s-z_{u_{1}}^{(j)})^{u+2}}}}-\mathop {\sum }\limits_{u=0}^{p+q-2}{\gamma_{j}'\pi c_{u}^{(j)}\mathop {\sum }\limits_{u_{1}=1}^{n_{j}}{\dfrac{k_{u_{1}}^{(j)}(u+2)!}{(s-w_{u_{1}}^{(j)})^{u+3}}}}}=0 \ \text{as Res} (s)>e_{m_{1}}^{(1)}.\label{eq:logistic18}
\end{eqnarray}
where $z_{u_{1}}^{(j)}=i\theta_{j}'+e_{u_{1}}^{(j)}$ as $1 \leq u_{1} \leq m_{j}$ and $w_{u_{1}}^{(j)}=i\theta_{j}'+h_{u_{1}}^{(j)}$ as $1 \leq u_{1} \leq n_{j}$.\\

Multiplying both sides of equation \eqref{eq:logistic18} with $(s-z_{m_{1}}^{(1)})^{p+q+1}$ and letting $s \to z_{m_{1}}^{(1)}$, as $e_{u_{1}}^{(j)}<e_{m_{1}}^{(1)}$ for all $(u_{1},j) \neq (m_{1},1)$ and $h_{u_{1}}^{(j)}<h_{n_{1}}^{(1)}=e_{m_{1}}^{(1)}$ for all $(u_{1},j) \neq (n_{1},1)$, we obtain $|f_{p+q-1}^{(1)}d_{m_{1}}^{(1)}-\gamma_{1}'\pi c_{p+q-2}^{(1)}k_{n_{1}}^{(1)}|=0$. Since $d_{m_{1}}^{(1)}=k_{n_{1}}^{(1)}=1$, $f_{p+q-1}^{(1)}=i\beta_{1}^{''}a_{p+q-2}^{(1)}$, $c_{p+q-2}^{(1)}=\dfrac{\sigma_{1}}{\pi}a_{p+q-2}^{(1)}$, and $a_{p+q-2}^{(1)}=\dfrac{a_{p+q-2}\sigma_{1}^{p+q-1}}{\pi^{p+q-1}} \neq 0$, it implies that $|i\beta_{1}^{''}-\gamma_{1}'\sigma_{1}|=0$ or equivalently $\beta_{1}^{''}=\gamma_{1}'=0$. Likewise, multiplying both sides of \eqref{eq:logistic18} with $(s-z_{m_{1}}^{(1)})^{p+q}$ and let $s \to z_{m_{1}}^{(1)}$, as $\gamma_{1}'=0$, we obtain $f_{p+q-2}^{(1)}=0$. Continue this fashion until we multiply both sides of \eqref{eq:logistic18} with $(s-z_{m_{1}}^{(1)})$ and let $s \to z_{m_{1}}^{(1)}$ to get $f_{0}^{(1)}=0$ or equivalently $\alpha_{1}'a_{0}^{(1)}=0$. As $a_{0}^{(1)}=\sigma_{1}\prod \limits_{j=1}^{p-1}{(p-j)}\prod \limits_{j=1}^{q-1}{(q-j)}/\pi \neq 0$, it implies that $\alpha_{1}'=0$. Overall, we achieve $\alpha_{1}'=\beta_{1}^{''}=\gamma_{1}'=0$. Repeat the same argument until we achieve $\alpha_{j}'=\beta_{j}^{''}=\gamma_{j}'=0$ for all $1 \leq j \leq k$ or equivalently $\alpha_{j}=\beta_{j}=\gamma_{j}=0$.\\

\noindent
(b) Assume that we can find $\alpha_{j},\beta_{j},\gamma_{j}.\eta_{j} \in \mathbb{R}$ such that
\begin{eqnarray}
\mathop {\sum }\limits_{j=1}^{k}{\alpha_{j}f(x|\theta_{j},\sigma_{j},\lambda_{j})+\beta_{j}\dfrac{\partial{f}}{\partial{\theta}}(x|\theta_{j},\sigma_{j},\lambda_{j})+\gamma_{j}\dfrac{\partial{f}}{\partial{\sigma}}(x|\theta_{j},\sigma_{j},\lambda_{j})+\eta_{j}\dfrac{\partial{f}}{\partial{\lambda}}(x|\theta_{j},\sigma_{j},\lambda_{j})}=0. \label{eq:gumbel1}
\end{eqnarray}
Applying the moment generating function to both sides of equation \eqref{eq:gumbel1}, we obtain
\begin{eqnarray}
\mathop {\sum }\limits_{j=1}^{k}{(\alpha_{j}'+\beta_{j}' t+\gamma_{j}'t\psi(\lambda_{j}-\sigma_{j} t)+\eta_{j}'\psi(\lambda_{j}-\sigma_{j}t))\exp(\theta_{j}'t)\Gamma(\lambda_{j}-\sigma_{j} t)}=0 \quad \text{as } t<\mathop {\min }\limits_{1 \leq j \leq k}{\left\{\dfrac{\lambda_{j}}{\sigma_{j}}\right\}}, \label{eq:gumbel12}
\end{eqnarray}
where $\alpha_{j}'=\dfrac{\alpha_{j}-\eta_{j}\psi(\lambda_{j})}{\Gamma(\lambda_{j})}$, $\beta_{j}'=\dfrac{\beta_{j}+\gamma_{j}\log(\lambda_{j})+\eta_{j}\sigma_{j}\lambda_{j}^{-1}}{\Gamma(\lambda_{j})}, \gamma_{j}'=-\dfrac{\gamma_{j}}{\Gamma(\lambda_{j})}$, $\eta_{j}'=\dfrac{\eta_{j}}{\Gamma(\lambda_{j})}$, and $\theta_{j}'=\theta_{j}+\log(\lambda_{j})\sigma_{j}$ as $\psi$ is di-gamma function.\\

Without loss of generality, we assume that $\sigma_{1} \leq \sigma_{2} \leq \ldots \leq \sigma_{k}$. We choose $\overline{i}$ to be minimum index such that $\sigma_{\overline{i}}=\sigma_{k}$. Denote $i_{1} \in [\overline{i},k]$ as the index such that $\theta_{i_{1}}'=\mathop {\min }\limits_{\overline{i} \leq i \leq k}{\left\{\theta_{i}'\right\}}$. Denote $I=\left\{i \in [\overline{i},k]:\theta_{i}'=\theta_{i_{1}}'\right\}$. From the formation of $\theta_{j}'$, it implies that $\lambda_{i}$ are pairwise different as $i \in I$. Choose $i_{2} \in I$ such that $\lambda_{i_{2}}=\mathop {\max }\limits_{i \in I}{\lambda_{i}}$, i.e $\lambda_{i_{2}} > \lambda_{i}$ for all $i \in I$. Divide both sides of equation \eqref{eq:gumbel12} by $t\Gamma(1-\sigma_{i_{2}}t)\psi(1-\sigma_{i_{2}}t)\exp(\theta_{i_{2}}'t)$, we get that as $t<\dfrac{1}{\sigma_{k}}$
\begin{eqnarray}
\dfrac{\alpha_{i_{2}}'}{t\psi(\lambda_{i_{2}}-\sigma_{i_{2}}t)}+\dfrac{\beta_{i_{2}}'}{\psi(\lambda_{i_{2}}-\sigma_{i_{2}}t)}+\gamma_{i_{2}}'+\dfrac{\eta_{i_{2}}'}{t}+\mathop {\sum }\limits_{j \neq i_{2}}{\dfrac{\alpha_{j}'\Gamma(\lambda_{j}-\sigma_{j}t)\exp(\theta_{j}'t)}{t\Gamma(\lambda_{i_{2}}-\sigma_{i_{2}}t)\psi(\lambda_{i_{2}}-\sigma_{i_{2}}t)\exp(\theta_{i_{2}}'t)}} & + & \nonumber \\
\dfrac{\beta_{j}\Gamma(\lambda_{j}-\sigma_{j}t)\exp(\theta_{j}'t)}{\Gamma(\lambda_{i_{2}}-\sigma_{i_{2}}t)\psi(\lambda_{i_{2}}-\sigma_{i_{2}}t)\exp(\theta_{i_{2}}'t)}
+\dfrac{\gamma_{j}'\exp(\theta_{j}'t)\Gamma(\lambda_{j}-\sigma_{j}t)\psi(\lambda_{j}-\sigma_{j}t)}{\exp(\theta_{i_{2}}'t)\Gamma(\lambda_{i_{2}}-\sigma_{i_{2}}t)\psi(\lambda_{i_{2}}-\sigma_{i_{2}}t)} & + &\nonumber\\ \dfrac{\eta_{j}'\exp(\theta_{j}'t)\Gamma(\lambda_{j}-\sigma_{j}t)\psi(\lambda_{j}-\sigma_{j}t)}{t\exp(\theta_{i_{2}}'t)\Gamma(\lambda_{i_{2}}-\sigma_{i_{2}}t)\psi(\lambda_{i_{2}}-\sigma_{i_{2}}t)} & = & 0. \quad  \quad  \label{eq:gumbel13}
\end{eqnarray}
Note that $\lim \limits_{t \to -\infty}{\psi(\lambda_{j}-\sigma_{j}t)/\psi(\lambda_{i_{2}}-\sigma_{i_{2}}t)}=1$ for all $1 \leq j \leq k$. Additionally, when $j \in I$ and $j \neq i_{2}$, as $\lambda_{j} <\lambda_{i_{2}}$, we see that $\Gamma(\lambda_{j}-\sigma_{j}t)/\Gamma(\lambda_{i_{2}}-\sigma_{i_{2}}t) \to 0$ as $t \to -\infty$ and $\exp((\theta_{j}'-\theta_{i_{2}}')(t))=1$. It implies that $\exp(\theta_{j}'t)\Gamma(\lambda_{j}-\sigma_{j}t)/\exp(\theta_{i_{2}}'t)\Gamma(\lambda_{i_{2}}-\sigma_{i_{2}}t) \to 0$ as $t \to -\infty$. Since $\psi(\lambda_{i_{2}}-\sigma_{i_{2}}t) \to +\infty$ as $t \to -\infty$, if we let $t \to -\infty$, we obtain
\begin{eqnarray}
\mathop {\sum }\limits_{j \in I \backslash {i_{2}}}{\dfrac{\alpha_{j}'\Gamma(\lambda_{j}-\sigma_{j}t)\exp(\theta_{j}'t)}{t\Gamma(\lambda_{i_{2}}-\sigma_{i_{2}}t)\psi(\lambda_{i_{2}}-\sigma_{i_{2}}t)\exp(\theta_{i_{2}}'t)}+\dfrac{\beta_{j}\Gamma(\lambda_{j}-\sigma_{j}t)\exp(\theta_{j}'t)}{\Gamma(\lambda_{i_{2}}-\sigma_{i_{2}}t)\psi(\lambda_{i_{2}}-\sigma_{i_{2}}t)\exp(\theta_{i_{2}}'t)}} & + & \nonumber \\
\dfrac{\gamma_{j}'\exp(\theta_{j}'t)\Gamma(\lambda_{j}-\sigma_{j}t)\psi(\lambda_{j}-\sigma_{j}t)}{\exp(\theta_{i_{2}}'t)\Gamma(\lambda_{i_{2}}-\sigma_{i_{2}}t)\psi(\lambda_{i_{2}}-\sigma_{i_{2}}t)} + \dfrac{\eta_{j}'\exp(\theta_{j}'t)\Gamma(\lambda_{j}-\sigma_{j}t)\psi(\lambda_{j}-\sigma_{j}t)}{t\exp(\theta_{i_{2}}'t)\Gamma(\lambda_{i_{2}}-\sigma_{i_{2}}t)\psi(\lambda_{i_{2}}-\sigma_{i_{2}}t)} & \to & 0 .\quad \quad \label{eq:gumbel14}
\end{eqnarray} 
Additionally, as $j \geq \overline{i}$ and $j \notin I$, we have $\sigma_{j}=\sigma_{i_{2}}$ and $\theta_{j}'>\theta_{i_{2}}'$. Therefore, we obtain $\exp((\theta_{j}'-\theta_{i_{2}}')t)\Gamma(\lambda_{j}-\sigma_{j}t)/\Gamma(\lambda_{i_{2}}-\sigma_{i_{2}}t) \to 0$ as $t \to -\infty$. As a consequence, if we let $t \to -\infty$, then
\begin{eqnarray}
\mathop {\sum }\limits_{j \notin I, j \geq \overline{i}}{\dfrac{\alpha_{j}'\Gamma(\lambda_{j}-\sigma_{j}t)\exp(\theta_{j}'t)}{t\Gamma(\lambda_{i_{2}}-\sigma_{i_{2}}t)\psi(\lambda_{i_{2}}-\sigma_{i_{2}}t)\exp(\theta_{i_{2}}'t)}+\dfrac{\beta_{j}\Gamma(\lambda_{j}-\sigma_{j}t)\exp(\theta_{j}'t)}{\Gamma(\lambda_{i_{2}}-\sigma_{i_{2}}t)\psi(\lambda_{i_{2}}-\sigma_{i_{2}}t)\exp(\theta_{i_{2}}'t)}} & + & \nonumber \\
\dfrac{\gamma_{j}'\exp(\theta_{j}'t)\Gamma(\lambda_{j}-\sigma_{j}t)\psi(\lambda_{j}-\sigma_{j}t)}{\exp(\theta_{i_{2}}'t)\Gamma(\lambda_{i_{2}}-\sigma_{i_{2}}t)\psi(\lambda_{i_{2}}-\sigma_{i_{2}}t)} + \dfrac{\eta_{j}'\exp(\theta_{j}'t)\Gamma(\lambda_{j}-\sigma_{j}t)\psi(\lambda_{j}-\sigma_{j}t)}{t\exp(\theta_{i_{2}}'t)\Gamma(\lambda_{i_{2}}-\sigma_{i_{2}}t)\psi(\lambda_{i_{2}}-\sigma_{i_{2}}t)} & \to & 0. \quad \quad \label{eq:gumbel15}
\end{eqnarray} 
Now, as $j< \overline{i}$, we have $\sigma_{j} <\sigma_{i_{2}}$. Therefore, as $\Gamma(\lambda_{j}-\sigma_{j}t)/\Gamma(\lambda_{i_{2}}-\sigma_{i_{2}}t) \sim (-t)^{(\sigma_{i_{2}}-\sigma_{j})t}$ when $t <0$, we get $\exp((\theta_{j}'-\theta_{i_{2}}')t)\Gamma(\lambda_{j}-\sigma_{j}t)/\Gamma(\lambda_{i_{2}}-\sigma_{i_{2}}t) \to 0$ as $t \to -\infty$. As a consequence, if we let $t \to -\infty$, then
\begin{eqnarray}
\mathop {\sum }\limits_{j <\overline{i}}{\dfrac{\alpha_{j}'\Gamma(\lambda_{j}-\sigma_{j}t)\exp(\theta_{j}'t)}{t\Gamma(\lambda_{i_{2}}-\sigma_{i_{2}}t)\psi(\lambda_{i_{2}}-\sigma_{i_{2}}t)\exp(\theta_{i_{2}}'t)}+\dfrac{\beta_{j}\Gamma(\lambda_{j}-\sigma_{j}t)\exp(\theta_{j}'t)}{\Gamma(\lambda_{i_{2}}-\sigma_{i_{2}}t)\psi(\lambda_{i_{2}}-\sigma_{i_{2}}t)\exp(\theta_{i_{2}}'t)}} & + & \nonumber \\
\dfrac{\gamma_{j}'\exp(\theta_{j}'t)\Gamma(\lambda_{j}-\sigma_{j}t)\psi(\lambda_{j}-\sigma_{j}t)}{\exp(\theta_{i_{2}}'t)\Gamma(\lambda_{i_{2}}-\sigma_{i_{2}}t)\psi(\lambda_{i_{2}}-\sigma_{i_{2}}t)} + \dfrac{\eta_{j}'\exp(\theta_{j}'t)\Gamma(\lambda_{j}-\sigma_{j}t)\psi(\lambda_{j}-\sigma_{j}t)}{t\exp(\theta_{i_{2}}'t)\Gamma(\lambda_{i_{2}}-\sigma_{i_{2}}t)\psi(\lambda_{i_{2}}-\sigma_{i_{2}}t)} & \to & 0. \quad \quad \label{eq:gumbel16}
\end{eqnarray} 
Combining \eqref{eq:gumbel14}, \eqref{eq:gumbel15}, and \eqref{eq:gumbel16}, by letting $t \to -\infty$ in \eqref{eq:gumbel13}, we get $\gamma_{i_{2}}'=0$.
With this result, we divide both sides of \eqref{eq:gumbel13} by $t\exp(\theta_{i_{2}}'t)\Gamma(\lambda_{i_{2}}-\sigma_{i_{2}}t)$, we obtain that as $t \to -\infty$
\begin{eqnarray}
\dfrac{\alpha_{i_{2}}'}{t}+\beta_{i_{2}}'+\dfrac{\eta_{i_{2}}'\psi(\lambda_{i_{2}}-\sigma_{i_{2}}t)}{t}+\mathop {\sum }\limits_{j \neq i_{2}}{\dfrac{\alpha_{j}'\Gamma(\lambda_{j}-\sigma_{j}t)\exp(\theta_{j}'t)}{t\Gamma(\lambda_{i_{2}}-\sigma_{i_{2}}t)\exp(\theta_{i_{2}}'t)}+\dfrac{\beta_{j}\Gamma(\lambda_{j}-\sigma_{j}t)\exp(\theta_{j}'t)}{\Gamma(\lambda_{i_{2}}-\sigma_{i_{2}}t)\exp(\theta_{i_{2}}'t)}} & + &\nonumber \\
\dfrac{\gamma_{j}'\exp(\theta_{j}'t)\Gamma(\lambda_{j}-\sigma_{j}t)\psi(\lambda_{j}-\sigma_{j}t)}{\exp(\theta_{i_{2}}'t)\Gamma(\lambda_{i_{2}}-\sigma_{i_{2}}t)}+ \dfrac{\eta_{j}'\exp(\theta_{j}'t)\Gamma(\lambda_{j}-\sigma_{j}t)\psi(\lambda_{j}-\sigma_{j}t)}{t\exp(\theta_{i_{2}}'t)\Gamma(\lambda_{i_{2}}-\sigma_{i_{2}}t)} & = & 0. \nonumber
\end{eqnarray}
Using the same argument with the notice that $\exp((\theta_{j}'-\theta_{i_{1}}')t)\psi(\lambda_{j}-\sigma_{j}t)\Gamma(\lambda_{j}-\sigma_{j}t)/\Gamma(\lambda_{i_{2}}-\sigma_{i_{2}}t) \to 0$ as $t \to -\infty$ for all $j \neq i_{1}$ and $\psi(\lambda_{i_{2}}-\sigma_{i_{2}}t)/t \to 0$ as $t \to -\infty$, we obtain $\beta_{i_{2}}'=0$. Continue in this fashion, we divide both sides of \eqref{eq:gumbel13} by $\psi(\lambda_{i_{2}}-\sigma_{i_{2}}t)\exp(\theta_{i_{2}}'t)\Gamma(\lambda_{i_{2}}-\sigma_{i_{2}}t)$ and $\exp(\theta_{i_{2}}'t)\Gamma(\lambda_{i_{2}}-\sigma_{i_{2}}t)$ respectively and by letting $t \to -\infty$, we get $\alpha_{i_{2}}'=\eta_{i_{2}}'=0$. Applying this argument to the remained indices $i$, we achieve $\alpha_{j}'=\beta_{j}'=\gamma_{j}'=\eta_{j}'=0$ for $1 \leq j \leq k$ or equivalently $\alpha_{j}=\beta_{j}=\gamma_{j}=\eta_{j}=0$ for $1 \leq j \leq k$.\\\\
(c) Assume that we can find $\alpha_{j},\beta_{j},\gamma_{j} \in \mathbb{R}$ such that
\begin{eqnarray}
\mathop {\sum }\limits_{j=1}^{k}{\alpha_{j}f_{X}(x|\nu_{j},\lambda_{j})+\beta_{j}\dfrac{\partial{f_{X}}}{\partial{\nu}}(x|\nu_{j},\lambda_{j})+\gamma_{j}\dfrac{\partial{f_{X}}}{\partial{\lambda}}(x|\nu_{j},\lambda_{j})}=0. \nonumber
\end{eqnarray}
It implies that by the transformation $Y=\log(X)$, we still have:
\begin{eqnarray}
\mathop {\sum }\limits_{j=1}^{k}{\alpha_{j}f_{Y}(y|\nu_{j},\lambda_{j})+\beta_{j}\dfrac{\partial{f_{Y}}}{\partial{\nu}}(y|\nu_{j},\lambda_{j})+\gamma_{j}\dfrac{\partial{f_{Y}}}{\partial{\lambda}}(y|\nu_{j},\lambda_{j})}=0. \label{eq:weibull1}
\end{eqnarray}
where $f_{Y}(y)$ is the density function of $Y$.\\

Applying the moment generating function to both sides of \eqref{eq:weibull1}, we obtain
\begin{eqnarray}
\mathop {\sum }\limits_{j=1}^{k}{\alpha_{j}\lambda_{j}^{t}\Gamma(\dfrac{t}{\nu_{j}}+1)-\dfrac{\beta_{j}t\lambda_{j}^{t}}{\nu_{j}^{2}}\Gamma(\dfrac{t}{\nu_{j}}+1)\psi(\dfrac{t}{\nu_{j}}+1)+\gamma_{j}t\lambda_{j}^{t-1}\Gamma(\dfrac{t}{\nu_{j}}+1)}=0 \ \text{as } \ t>-\mathop {\min }\limits_{1 \leq i \leq k}{\left\{\nu_{i}\right\}}. \label{eq:weibull5}
\end{eqnarray}
Without loss of generality, assume that $\nu_{1} \leq \nu_{2} \leq \ldots \leq \nu_{k}$. Denote $\overline{i}$ as the minimum index such that $\nu_{\overline{i}}=\nu_{1}$ and $i_{1}$ is index such that $\lambda_{i_{1}}=\mathop {\min }\limits_{1 \leq i \leq i_{1}}{\left\{\lambda_{i}\right\}}$, which implies that $\lambda_{i_{1}}<\lambda_{i}$ for all $1 \leq i \leq \overline{i}$. Using the same argument as that of generalized gumbel density function case, we firstly divide both sides of  \eqref{eq:weibull5} by $t\Gamma(t/\nu_{i_{1}}+1)\psi(t/\nu_{j}+1)$ and let $t \to +\infty$, we obtain $\beta_{i_{1}}=0$. Then, with this result, we divide both sides of \eqref{eq:weibull5} by $t\Gamma(t/\nu_{j}+1)$ and let $t \to +\infty$, we get $\gamma_{i_{1}}=0$. Finally, divide both sides of \eqref{eq:weibull5} by $\Gamma(t/\nu_{j}+1)$ and let $t \to +\infty$, we achieve $\alpha_{i_{1}}=0$. Repeat the same argument until we obtain $\alpha_{i}=\beta_{i}=\gamma_{i}=0$ for all $1 \leq i \leq k$.\\\\
(d) The idea of this proof is based on main theorem of \cite{Kent-1983}. Assume that we can find $\alpha_{j},\beta_{j},\gamma_{j} \in \mathbb{R}$ such that
\begin{eqnarray}
\mathop {\sum }\limits_{j=1}^{k}{\alpha_{j}f(x|\mu_{j},\kappa_{j})+\beta_{j}\dfrac{\partial{f}}{\partial{\mu}}(x|\mu_{j},\kappa_{j})+\gamma_{j}\dfrac{\partial{f}}{\partial{\kappa}}(x|\mu_{j},\kappa_{j})}=0. \nonumber
\end{eqnarray}
We can rewrite the above equation as
\begin{eqnarray}
\mathop {\sum }\limits_{j=1}^{k}{\left[\alpha_{j}'+\beta_{j}'\sin(x-\mu_{j})+\gamma_{j}'\cos(x-\mu_{j})\right]\exp(\kappa_{j}\cos(x-\mu_{j}))}=0 \ \text{for all } x \in [0,2\pi). \label{eq:vonmises1}
\end{eqnarray}
where $C(\kappa)=\dfrac{1}{2\pi I_{0}(k)}$, $\alpha_{j}'=C(\kappa_{j})\alpha_{j}+C'(\kappa_{j})\gamma_{j}$, $\beta_{j}'=-C(\kappa_{j})\beta_{j}$, and $\gamma_{j}'=C(\kappa_{j})\gamma_{j}$ for all $1 \leq j \leq k$. 

Since the functions $\exp(\kappa_{j}(x-\mu_{j}))$, $\cos(x-\mu_{j})\exp(\kappa_{j}(x-\mu_{j}))$, and $\sin(x-\mu_{j})\exp(\kappa_{j}(x-\mu_{j}))$ are analytic functions of $x$, we can extend equation\eqref{eq:vonmises1} to the whole range $x \in \mathbb{C}$. Denote $x=y+iz$, where $y,z  \in \mathbb{R}$. Direct calculation yields $\cos(x-\mu_{j})=\cos(y-\mu_{j})\cosh(z)-i\sin(y-\mu_{j})\sinh(z)$, $\sin(x-\mu_{j})=\sin(y-\mu_{j})\cosh(z)+i\cos(y-\mu_{j})\sinh(z)$, and
\begin{eqnarray}
\exp(\kappa_{j}\cos(x-\mu_{j}))=\exp(\kappa_{j}\left[\cos(y-\mu_{j})\cosh(z)-i\sin(y-\mu_{j})\sinh(z)\right]). \nonumber
\end{eqnarray}
Therefore, we can rewrite equation \eqref{eq:vonmises1} as
for all $y,z \in \mathbb{R}$ 
\begin{eqnarray}
\mathop {\sum }\limits_{j=1}^{k}{\left\{\alpha_{j}'+\left[\beta_{j}'\cos(y-\mu_{j})+\gamma_{j}'\sin(y-\mu_{j})\right]\cosh(z)-i\left[\beta_{j}'\sin(y-\mu_{j})-\gamma_{j}'\cos(y-\mu_{j})\right]\sinh(z)\right\}} \nonumber \\
\exp\left(\kappa_{j}\left[\cos(y-\mu_{j})\cosh(z)-i\sin(y-\mu_{j})\sinh(z)\right]\right)=0. \label{eq:vonmises2}
\end{eqnarray}
As $(\mu_{j},\kappa_{j})$ are pairwise different as $1 \leq j \leq k$, we can choose at least one $y^{*} \in [0,2\pi)$ such that $m_{j}=\kappa_{j}\cos(y^{*}-\mu_{j})$ are pairwise different as $1 \leq j \leq k$ and $\cos(y^{*}-\mu_{j})$, $\sin(y^{*}-\mu_{j})$ are all different from 0 for all $1 \leq j \leq k$. Without loss of generality, we assume that $m_{1}<m_{2}<\ldots<m_{k}$. Multiply both sides of \eqref{eq:vonmises2} with $\exp(-m_{k}+i\kappa_{k}\sin(y^{*}-\mu_{k})\sinh(z))$, we obtain
\begin{eqnarray}
\alpha_{k}'+\left[\beta_{k}'\cos(y^{*}-\mu_{k}) +  \gamma_{k}'\sin(y^{*}-\mu_{k})\right]\cosh(z) -  i(\beta_{k}'\sin(y^{*}-\mu_{k}) & - & \gamma_{k}'\cos(y^{*}-\mu_{k}))\sinh(z)| \nonumber \\
 = \mathop {\sum }\limits_{j=1}^{k-1}{|\alpha_{j}'+\left[\beta_{j}'\cos(y^{*}-\mu_{j}) + \gamma_{j}'\sin(y^{*}-\mu_{j})\right]\cosh(z)}  & - & \nonumber \\
 i\left[\beta_{j}'\sin(y^{*}-\mu_{j})-\gamma_{j}'\cos(y^{*}-\mu_{j})\right]\sinh(z)| & \times & \exp((m_{j}-m_{k})\cosh(z)). \nonumber
\end{eqnarray}
Noted that as $m_{j}<m_{k}$ for all $1 \leq j \leq k-1$, 
\begin{eqnarray}
\lim \limits_{z \to \infty}{\cosh(z)\exp((m_{j}-m_{k})\cosh(z))}=\lim \limits_{z \to \infty}{\sinh(z)\exp((m_{j}-m_{k})\cosh(z))}=0. \nonumber
\end{eqnarray}
Therefore, by letting $z \to \infty$ in both sides of the above equation, we obtain
\begin{eqnarray}
|\alpha_{k}'+\left[\beta_{k}'\cos(y^{*}-\mu_{k})+\gamma_{k}'\sin(y^{*}-\mu_{k})\right]\cosh(z)-i(\beta_{k}'\sin(y^{*}-\mu_{k})- \nonumber \\
\gamma_{k}'\cos(y^{*}-\mu_{k}))\sinh(z)|  \to 0. \nonumber
\end{eqnarray}
It implies that $\alpha_{k}'=0$, $\beta_{k}'\cos(y^{*}-\mu_{k})+\gamma_{k}'\sin(y^{*}-\mu_{k})=0$, and $\beta_{k}'\sin(y^{*}-\mu_{k})-\gamma_{k}'\cos(y^{*}-\mu_{k})=0$. These equations imply $\alpha_{k}'=\beta_{k}'=\gamma_{k}'=0$. Repeat the same argument for the remained $\alpha_{j}',\beta_{j}',\gamma_{j}'$ as $1 \leq j \leq k-1$, we eventually achieve $\alpha_{j}'=\gamma_{j}'=\gamma_{j}'=0$ for all $1 \leq j \leq k$ or equivalently $\alpha_{j}=\beta_{j}=\gamma_{j}=0$ for all $1 \leq j \leq k$.

\paragraph{PROOF OF THEOREM \ref{identifiability-multivariatecharacterization} (Continue)}
Part (a) was proved in Appendix I. The following is the proof for the remaining parts.

\noindent
(b)  Consider that for given $k \geq 1$ and $k$ different pairs $(\theta_{1},\Sigma_{1}),...,(\theta_{k},\Sigma_{k})$, where $\theta_{j} \in \mathbb{R}^{d}$, $\Sigma_{j} \in S_{d}^{++}$ for all $1 \leq j \leq k$, we can find $\alpha_{j} \in \mathbb{R},\beta_{j} \in \mathbb{R}^{d}$, and symmetric matrices $\gamma_{j} \in \mathbb{R}^{d \times d}$ such that:
\begin{eqnarray}
\mathop {\sum }\limits_{j=1}^{k}{\alpha_{j}f(x|\theta_{j},\Sigma_{j})+\beta_{j}^{T}\dfrac{\partial{f}}{\partial{\theta}}(x|\theta_{j},\Sigma_{j})+\trace(\dfrac{\partial{f}}{\partial{\Sigma}}(x|\theta_{j},\Sigma_{j})^{T}\gamma_{j})}=0. \label{eq:mustudent1}
\end{eqnarray}
Multiply both sides with $\exp(it^{T}x)$ and take the integral in $\mathbb{R}^{d}$, we get:
\begin{eqnarray}
\mathop {\sum }\limits_{j=1}^{k}{\int \limits_{\mathbb{R}^{d}}{\exp(it^{T}x)\left[\alpha_{j}f(x|\theta_{j},\Sigma_{j})+\beta_{j}^{T}\dfrac{\partial{f}}{\partial{\theta}}(x|\theta_{j},\Sigma_{j})+\trace(\dfrac{\partial{f}}{\partial{\Sigma}}(x|\theta_{j},\Sigma_{j})^{T}\gamma_{j})\right]}dx}=0. \label{eq:mustudent2}
\end{eqnarray}
Notice that
\begin{eqnarray}
\int \limits_{\mathbb{R}^{d}}{\exp(it^{T}x)f(x|\theta_{j},\Sigma_{j})}dx & = &\exp(it^{T}\theta_{j})\int \limits_{R^{d}}{\exp(i(\Sigma_{j}^{1/2}t)^{T}x)\dfrac{1}{(\nu+\|x\|^{2})^{(\nu+d)/2}}}dx. \nonumber \\
 \int \limits_{\mathbb{R}^{d}}{\exp(it^{T}x)\beta_{j}^{T}\dfrac{\partial{f}}{\partial{\theta}}(x|\theta_{j},\Sigma_{j})}dx & = & \dfrac{C(\nu+d)}{2}\int \limits_{R^{d}}{\dfrac{\exp(i(\Sigma_{j}^{1/2}t)^{T}x)\beta_{j}^{T}\Sigma_{j}^{-1/2}x}{(\nu+\|x\|^{2})^{(\nu+d+2)/2}}}dx. \nonumber
\end{eqnarray}
and
\begin{eqnarray}
\int \limits_{\mathbb{R}^{d}}{\exp(it^{T}x)\trace(\dfrac{\partial{f}}{\partial{\Sigma}}(x|\theta_{j},\Sigma_{j})^{T}\gamma_{j})}dx  =  -\dfrac{C}{2}\trace(\Sigma_{j}^{-1}\gamma_{j})\exp(it^{T}\theta_{j})\times \nonumber \\
\times \int \limits_{R^{d}}\exp(it^{T}\theta_{j}){\dfrac{\exp(i(\Sigma_{j}^{1/2}t)^{T}x)}{(\nu+\|x\|^{2})^{(\nu+d)/2}}}dx + \nonumber \\ {\displaystyle \dfrac{C(\nu+d)}{2}\exp(it^{T}\theta_{j})\int \limits_{R^{d}}{\dfrac{\exp(i(\Sigma_{j}^{1/2}t)^{T}x)\trace(\Sigma_{j}^{-1/2}xx^{T}\Sigma_{j}^{-1/2}\gamma_{j})}{(\nu+\|x\|^{2})^{(\nu+d+2)/2}}}dx}. \nonumber
 \end{eqnarray}
From the property of trace of matrix, $\trace(\Sigma_{j}^{-1/2}xx^{T}\Sigma_{j}^{-1/2})\gamma_{j})=x^{T}\Sigma_{j}^{-1/2}\gamma_{j}\Sigma_{j}^{-1/2}x$. Equation \eqref{eq:mustudent2} can be rewritten as
\begin{eqnarray}
\mathop {\sum }\limits_{j=1}^{k}{\left[\int \limits_{\mathbb{R}^{d}}{\left(\dfrac{\alpha_{j}'\exp(i(\Sigma_{j}^{1/2}t)^{T}x)}{(\nu+\|x\|^{2})^{(\nu+d)/2}}+\dfrac{\exp(i(\Sigma_{j}^{1/2}t)^{T}x)(\beta_{j}')^{T}x}{(\nu+\|x\|^{2})^{(\nu+d+2)/2}}+\dfrac{\exp(i(\Sigma_{j}^{1/2}t)^{T}x)x^{T}M_{j}x}{(\nu+\|x\|^{2})^{(\nu+d+2)/2}}\right)}dx\right]} \times \nonumber \\
\times \exp(it^{T}\theta_{j})=0, \label{eq:mustudent3}
\end{eqnarray}
where $\alpha_{j}'=\alpha_{j}-\dfrac{\trace(\Sigma_{j}^{-1}\gamma_{j})}{2},\beta_{j}'=\dfrac{(\nu+d)}{2}\Sigma^{-1/2}\beta_{j}$, and $M_{j}=\dfrac{\nu+d}{2}\Sigma_{j}^{-1/2}\gamma_{j}\Sigma_{j}^{-1/2}$.\\

To simplify the left hand side of equation \eqref{eq:mustudent3}, it is sufficient to calculate the following quantities $A={\displaystyle \int \limits_{\mathbb{R}^{d}}{\dfrac{\exp(it^{T}x)}{(\nu+\|x\|^{2})^{(\nu+d)/2}}}dx}$,$B={\displaystyle \int \limits_{\mathbb{R}^{d}}{\dfrac{\exp(it^{T}x)(\beta')^{T}x}{(\nu+\|x\|^{2})^{(\nu+d+2)/2}}}dx}$, and $C={\displaystyle \int \limits_{\mathbb{R}^{d}}{\dfrac{\exp(it^{T}x)x^{T}Mx}{(\nu+\|x\|^{2})^{(\nu+d+2)/2}}}dx}$, where $\beta' \in \mathbb{R}^{d}$ and $M=(M_{ij}) \in \mathbb{R}^{d \times d}$.\\
In fact, using orthogonal transformation $x=O.z$, where $O \in \mathbb{R}^{d \times d}$ and its first column to be $(\dfrac{t_{1}}{\|t\|},...,\dfrac{t_{d}}{\|t\|})^{T}$, then it is not hard to verify that $\exp(it^{T}x)=\exp(i\|t\|z_{1})$, $\|x\|^{2}=\|z\|^{2}$, and $dx=|\det(O)|dz=dz$, then we obtain the following results:
\begin{eqnarray}
A &= &\int \limits_{\mathbb{R}^{d}}{\dfrac{\exp(i\|t\|z_{1})}{(\nu+\|z\|^{2})^{(\nu+d)/2}}}dz \nonumber \\
&=& \int \limits_{\mathbb{R}}{\exp(i\|t\|z_{1})\int \limits_{\mathbb{R}}{...\int \limits_{\mathbb{R}}{\dfrac{1}{(\nu+\|z\|^{2})^{(\nu+d)/2}}}}}dz_{d}dz_{d-1}...dz_{1} \nonumber \\
&=& C_{1}A_{1}(\|t\|), \nonumber
\end{eqnarray}
where ${\displaystyle C_{1}=\prod \limits_{j=2}^{d}{\int \limits_{\mathbb{R}}{\dfrac{1}{(1+z^{2})^{(\nu+j)/2}}}dz}}$ and ${\displaystyle A_{1}(t')=\int \limits_{\mathbb{R}}{\dfrac{\exp(i|t'|z)}{(v+z^{2})^{(\nu+1)/2}}}dz}$ for any $t' \in \mathbb{R}$.\\
Hence, for all $1 \leq j \leq k$
\begin{eqnarray}
\int \limits_{\mathbb{R}^{d}}{\dfrac{\exp(i(\Sigma_{j}^{1/2}t)^{T}x)}{(\nu+\|x\|^{2})^{(\nu+d)/2}}}dx=C_{1}A_{1}(\|\Sigma_{j}^{1/2}t\|). \label{eq:mustudent4}
\end{eqnarray}
Turning to $B$:
\begin{eqnarray}
B = \mathop {\sum }\limits_{j=1}^{d}{\beta_{j}'\int \limits_{\mathbb{R}^d}{\dfrac{\exp(it^{T}x)x_{j}}{(\nu+\|x\|^{2})^{(\nu+d+2)/2}}}dx} = \mathop {\sum }\limits_{j=1}^{d}{\beta_{j}'\int \limits_{\mathbb{R}^d}{\dfrac{\exp(i\|t\|z_{1})(\mathop {\sum }\limits_{l=1}^{d}{O_{jl}z_{l}})}{(\nu+\|z\|^{2})^{(\nu+d+2)/2}}}dz}. \label{eq:mustudent5}
\end{eqnarray}
When $j \neq 1$, since $\dfrac{z_{j}}{(\nu+\|z\|^{2})^{(\nu+d+2)/2}}$ is an integrable odd function, ${\displaystyle \int \limits_{\mathbb{R}^{d}}{\dfrac{\exp(i\|t\|z_{1})z_{j}}{(\nu+\|z\|^{2})^{(\nu+d+2)/2}}}dz=0}$. Simultaneously, using the same argument as \eqref{eq:mustudent4}, we get
\begin{eqnarray}
\int \limits_{\mathbb{R}^{d}}{\dfrac{\exp(i\|t\|z_{1})z_{1}}{(\nu+\|z\|^{2})^{(\nu+d+2)/2}}}dz=C_{2}A_{2}(\|t\|), \nonumber
\end{eqnarray}
where ${\displaystyle C_{2}=\prod \limits_{j=2}^{d}{\int \limits_{\mathbb{R}}{\dfrac{1}{(1+z^{2})^{(\nu+2+j)/2}}}dz}}$ and ${\displaystyle A_{2}(t')=\int \limits_{\mathbb{R}}{\dfrac{\exp(i|t'|z)z}{(\nu+z^{2})^{(\nu+3)/2}}}dz}$ for any $t' \in \mathbb{R}$. \\
Therefore, we can rewrite \eqref{eq:mustudent5} as
\begin{eqnarray}
B = \left(\mathop {\sum }\limits_{j=1}^{d}{O_{j1}\beta_{j}'}\right)\int \limits_{\mathbb{R}^{d}}{\dfrac{\exp(it^{t}z_{1})z_{1}}{(\nu+\|z\|^{2})^{(\nu+d+2)/2}}}dz &=& \left(\mathop {\sum }\limits_{j=1}^{d}{O_{j1}\beta_{j}'}\right)C_{2}A_{2}(\|t\|) \nonumber \\
&=& \dfrac{C_{2}(\beta')^{T}tA_{2}(\|t\|)}{\|t\|}. \nonumber
\end{eqnarray}
It demonstrates that for all $1 \leq j  \leq k$
\begin{eqnarray}
\int \limits_{\mathbb{R}^{d}}{\dfrac{\exp(i(\Sigma_{j}^{1/2}t)^{T}x)(\beta_{j}')^{T}x}{(\nu+\|x\|^{2})^{(\nu+d+2)/2}}}dx=\dfrac{C_{2}(\beta_{j}')^{T}\Sigma_{j}^{1/2}tA_{2}(\|\Sigma_{j}^{1/2}t\|)}{\|t\|}. \label{eq:mustudent6}
\end{eqnarray}
Turning to $C$:
\begin{eqnarray}
C =\mathop {\sum }\limits_{j=1}^{d}{M_{jj}\int \limits_{\mathbb{R}^{d}}{\dfrac{\exp(it^{T}x)x_{j}^{2}}{(\nu+\|x\|^{2})^{(\nu+d+2)/2}}}dx}+2\mathop {\sum }\limits_{j < l}{M_{jl}\int \limits_{\mathbb{R}^{d}}{\dfrac{\exp(it^{T}x)x_{j}x_{l}}{(\nu+\|x\|^{2})^{(\nu+d+2)/2}}}dx}. \label{eq:mustudent7}
\end{eqnarray}
Notice that, for each $1 \leq j \leq d$:
\begin{eqnarray}
\int \limits_{\mathbb{R}^{d}}{\dfrac{\exp(it^{T}x)x_{j}^{2}}{(\nu+\|x\|^{2})^{(\nu+d+2)/2}}}dx & = & \int \limits_{\mathbb{R}^{d}}{\dfrac{\exp(i\|t\|z_{1})(\mathop {\sum }\limits_{l=1}^{d}{O_{jl}z_{l})^{2}})^{2}}{(\nu+\|z\|^{2})^{(\nu+d+2)/2}}}dz \nonumber \\
& = & \mathop {\sum }\limits_{l=1}^{d}{O_{jl}^{2}\int \limits_{\mathbb{R}^{d}}{\dfrac{\exp(i\|t\|z_{1})z_{l}^{2}}{(\nu+\|z\|^{2})^{(\nu+d+2)/2}}}dz} + \nonumber \\
& & 2\mathop {\sum }\limits_{u<v}{O_{ju}O_{jv}} \int \limits_{\mathbb{R}^{d}}{\dfrac{\exp(i\|t\|z_{1})z_{u}z_{v}}{(\nu+\|z\|^{2})^{(\nu+d+2)/2}}}dz. \nonumber
\end{eqnarray}
As $u < v$, then one of $u$,$v$ will differ from 1. It follows that ${\displaystyle \int \limits_{\mathbb{R}^{d}}{\dfrac{\exp(i\|t\|z_{1})z_{u}z_{v}}{(\nu+\|z\|^{2})^{(\nu+d+2)/2}}}dz=0}$. Additionally, as $l \neq 1$, we see that 
\begin{eqnarray}
{\displaystyle \int \limits_{\mathbb{R}^{d}}{\dfrac{\exp(i\|t\|z_{1})z_{l}^{2}}{(\nu+\|z\|^{2})^{(\nu+d+2)/2}}}dz=\int \limits_{\mathbb{R}^{d}}{\dfrac{\exp(i\|t\|z_{1})z_{2}^{2}}{(\nu+\|z\|^{2})^{(\nu+d+2)/2}}}dz=C_{3}A_{1}(\|t\|)}. \nonumber
\end{eqnarray}
where ${\displaystyle C_{3}=\int \limits_{\mathbb{R}}{\dfrac{z^{2}}{(1+z^{2})^{(\nu+4)/2}}}dz\prod \limits_{j=3}^{k}{\int \limits_{\mathbb{R}}{\dfrac{1}{(1+z^{2})^{(\nu+2+j)/2}}}dz}}$.\\ Similarly, ${\displaystyle \int \limits_{\mathbb{R}^{d}}{\dfrac{\exp(i\|t\|z_{1})z_{1}^{2}}{(\nu+\|z\|^{2})^{(\nu+d+2)/2}}}dz=C_{2}A_{3}(\|t\|)}$, where ${\displaystyle A_{3}(t')=\int \limits_{\mathbb{R}}{\dfrac{\exp(i|t'|z)z^{2}}{(\nu+z^{2})^{(\nu+3)/2}}}dz}$ for any $t' \in \mathbb{R}$.\\
Therefore,
\begin{eqnarray}
\int \limits_{\mathbb{R}}{\dfrac{\exp(it^{T}x)x_{j}^{2}}{(\nu+\|x\|^{2})^{(\nu+d+2)/2}}}dx  = C_{2}O_{j1}^{2}A_{3}(\|t\|)+C_{3}(1-O_{j1}^{2})A_{1}(\|t\|). \nonumber
\end{eqnarray}
As a consequence,
\begin{eqnarray}
\mathop {\sum }\limits_{j=1}^{d}{M_{jj}\int \limits_{\mathbb{R}^{d}}{\dfrac{\exp(it^{T}y)x_{j}^{2}}{(\nu+\|x\|^{2})^{(\nu+d+2)/2}}}dx} =  C_{2}(\mathop {\sum }\limits_{j=1}^{d}{M_{jj}O_{j1}^{2}})A_{3}(\|t\|) & + & \nonumber \\
C_{3}(\mathop {\sum }\limits_{j=1}^{d}{M_{jj}(1-O_{j1}^{2})})A_{1}(\|t\|) & &.  \label{eq:mustudent8}
\end{eqnarray}
Simultaneously, as $j \neq l$
\begin{eqnarray}
\int \limits_{\mathbb{R}^{d}}{\dfrac{\exp(it^{T}x)x_{j}x_{l}}{(\nu+\|x\|^{2})^{(\nu+d+2)/2}}}dx = \mathop {\sum }\limits_{u=1}^{d}{O_{ju}O_{lu}\int \limits_{\mathbb{R}^{d}}{\dfrac{\exp(i|t\|z_{1})z_{u}^{2}}{(\nu+\|z\|^{2})^{(\nu+d+2)/2}}}dz} = \nonumber \\ 
C_{2}O_{j1}O_{l1}A_{3}(\|t\|)+C_{3}(\mathop {\sum }\limits_{u=2}^{d}{O_{ju}O_{lu}})A_{1}(\|t\|)= O_{j1}O_{l1}(C_{2}A_{3}(\|t\|)-C_{3}A_{1}(\|t\|)). \label{eq:mustudent9}
\end{eqnarray}
Combining \eqref{eq:mustudent8} and \eqref{eq:mustudent9}, we can rewrite \eqref{eq:mustudent7} as:
\begin{eqnarray}
C &=& C_{3}(\mathop {\sum }\limits_{j=1}^{d}{M_{jj}})A_{1}(\|t\|)+(\mathop {\sum }\limits_{jl}{M_{jl}O_{j1}O_{l1})(C_{2}A_{3}(\|t\|)-C_{3}A_{1}(\|t\|))} \nonumber \\
&=& C_{3}(\mathop {\sum }\limits_{j=1}^{d}{M_{jj}})A_{1}(\|t\|)+\dfrac{1}{\|t\|^{2}}(\mathop {\sum }\limits_{j,l}{M_{jl}t_{j}t_{l}})(C_{2}A_{3}(\|t\|)-C_{3}A_{1}(\|t\|)). \nonumber
\end{eqnarray}
Thus, for all $1 \leq j \leq d$
\begin{eqnarray}
\int \limits_{\mathbb{R}^{d}}{\dfrac{\exp(i(\Sigma_{j}^{1/2}t)^{T}x)x^{T}M_{j}x}{(\nu+\|x\|^{2})^{(\nu+d+2)/2}}}dx=\dfrac{1}{\|\Sigma_{j}^{1/2}t\|^{2}}(\mathop {\sum }\limits_{u,v}{M_{uv}^{j}[\Sigma_{j}^{1/2}t]_{u}[\Sigma_{j}^{1/2}t]_{v}}) \times \nonumber \\ 
\times (C_{2}A_{3}(\|\Sigma_{j}^{1/2}t\|)-C_{3}A_{1}(\|\Sigma_{j}^{1/2}t\|))+ C_{3}(\mathop {\sum }\limits_{l=1}^{d}{M_{ll}^{j})A_{1}(\|\Sigma_{j}^{1/2}t\|)}, \label{eq:mustudent10}
\end{eqnarray}
where $M_{uv}^{j}$ indicates the element at $u$-th row and $v$-th column of $M_{j}$ and $[\Sigma_{j}^{1/2}t]_{u}$ simply means the $u$-th component of $\Sigma_{j}^{1/2}t$.\\
As a consequence, by combining \eqref{eq:mustudent4},\eqref{eq:mustudent6}, and \eqref{eq:mustudent10}, we can rewrite \eqref{eq:mustudent3} as:\\
\begin{eqnarray}
\mathop {\sum }\limits_{j=1}^{k}{[\alpha_{j}'A_{1}(\|\Sigma_{j}^{1/2}t\|)+C_{2}\dfrac{(\Sigma_{j}^{1/2}t)^{T}\beta_{j}'}{\|\Sigma_{j}^{1/2}t\|}A_{2}(\|\Sigma_{j}^{1/2}t\|)+C_{3}(\mathop {\sum }\limits_{l=1}^{d}{M_{ll}^{j}})A_{1}(\|\Sigma_{j}^{1/2}t\|)} & + & \nonumber \\
\left(\mathop {\sum }\limits_{u,v}{M_{uv}^{j}\dfrac{[\Sigma_{j}^{1/2}t]_{u}[\Sigma_{j}^{1/2}t]_{v}}{\|\Sigma_{j}^{1/2}t\|^{2}}}\right)(C_{2}A_{3}(\|\Sigma_{j}^{1/2}t\|)-C_{3}A_{1}(\|\Sigma_{j}^{1/2}t\|))]\exp(it^{T}\theta_{j}) & = & 0. \label{eq:mustudent11}
\end{eqnarray}

Define $t=t_{1}t'$, where $t_{1} \in \mathbb{R}$ and $t' \in \mathbb{R}^{d}$. By using the same argument as that of multivariate generalized Gaussian distribution, we can find $D$ to be the finite union of conics and hyperplanes such that as $t' \notin D$, $((t')^{T}\theta_{1},(t')^{T}\Sigma_{1}t'),...((t')^{T}\theta_{k},(t')^{T}\Sigma_{k}t')$ are pairwise distinct. By denoting $\theta_{j}'=(t')^{T}\theta_{j}$, $\sigma_{j}=(t')^{T}\Sigma_{j}t'$, we can rewrite \eqref{eq:mustudent11} as:
\begin{eqnarray}
\mathop {\sum }\limits_{j=1}^{k}{[\alpha_{j}'A_{1}(\sigma_{j}|t_{1}|)+C_{2}\dfrac{t_{1}(\Sigma_{j}^{1/2}t')^{T}\beta_{j}'}{|t_{1}|\sigma_{j}}A_{2}(\sigma_{j}|t_{1}|)+C_{3}(\mathop {\sum }\limits_{l=1}^{d}{M_{ll}^{j}})A_{1}(\sigma_{j}|t_{1}|)} & + &\nonumber \\
(\mathop {\sum }\limits_{u,v}{M_{uv}^{j}\dfrac{[\Sigma_{j}^{1/2}t']_{u}[\Sigma_{j}^{1/2}t']_{v}}{\sigma_{j}^{2}}})(C_{2}A_{3}(\sigma_{j}|t_{1}|)-C_{3}A_{1}(\sigma_{j}|t_{1}|)]\exp(i\theta_{j}'t_{1}) & = & 0.\nonumber
\end{eqnarray}
Since $A_{2}(\sigma_{j}|t_{1}|)=(i|t_{1}|)A_{1}(\sigma_{j}|t_{1}|)$, the above equation can be rewritten as:
\begin{eqnarray}
\mathop {\sum }\limits_{j=1}^{k}{[(\alpha_{j}'+C_{3}(\mathop {\sum }\limits_{l=1}^{d}{M_{ll}^{j}})-C_{3}(\mathop {\sum }\limits_{u,v}{M_{uv}^{j}\dfrac{[\Sigma_{j}^{1/2}t']_{u}[\Sigma_{j}^{1/2}t']_{v}}{\sigma_{j}^{2}}}))A_{1}(\sigma_{j}|t_{1}|)} & + &\nonumber \\
C_{2}(it_{1})\dfrac{(\Sigma_{j}^{1/2}t')^{T}\beta_{j}'}{\sigma_{j}}A_{1}(\sigma_{j}|t_{1}|)+C_{2}(\mathop {\sum }\limits_{u,v}{M_{uv}^{j}\dfrac{[\Sigma_{j}^{1/2}t']_{u}[\Sigma_{j}^{1/2}t']_{v}}{\sigma_{j}^{2}}})A_{3}(\sigma_{j}|t_{1}|)]\exp(i\theta_{j}'t_{1}) & = & 0.  \quad \quad \label{eq:mustudent12}
\end{eqnarray}
As $\nu$ is odd number, we assume $\nu=2l-1$. By applying Lemma \ref{lemma-complexformula} 
(stated and proved in the sequel), we obtain for any $m \in \mathbb{N}$ that
\begin{eqnarray}
\int \limits_{-\infty}^{+\infty}{\dfrac{\exp(i|t_{1}|z)}{(z^{2}+\nu)^{m}}}dz=\dfrac{2\pi\exp(-|t_{1}|\sqrt{2l-1})}{(2\sqrt{2l-1})^{2m-1}}\left[\mathop {\sum }\limits_{j=1}^{m}{\dbinom{2m-1-j}{m-j}\dfrac{(2|t_{1}|\sqrt{2l-1})^{j-1}}{(j-1)!}}\right]. \nonumber
\end{eqnarray}
It means that we can write 
\begin{eqnarray}
A_{1}(t_{1})=C_{4}\exp(-|t_{1}|\sqrt{2l-1})\mathop {\sum }\limits_{u=0}^{l-1}{a_{u}|t_{1}|^{u}},\nonumber
\end{eqnarray}
where $C_{4}=\dfrac{2\pi}{(2\sqrt{2l-1})^{2m-1}}$, $a_{u}=\dbinom {2l-u-2}{l-u-1}\dfrac{(2\sqrt{2l-1})^{u}}{u!}$.\\
Simultaneously, as $\displaystyle{A_{3}(t_{1})=A_{1}(t_{1})-\nu\int \limits_{\mathbb{R}}{\dfrac{\exp(i|t_{1}|z)}{(\nu+z^{2})^{(\nu+3)/2}}}dz}$, we can write 
\begin{eqnarray}
\displaystyle{A_{3}(t_{1})=C_{4}\exp(-|t_{1}|\sqrt{2l-1})\mathop {\sum }\limits_{u=0}^{l}{b_{u}|t_{1}|^{u}}},\nonumber
\end{eqnarray}
where $b_{u}=\left[\dbinom{2l-u-2}{l-u-1}-\dfrac{1}{4}\dbinom{2l-u}{l-u}\right]\dfrac{(2\sqrt{2l-1})^{u}}{u!}$ as $0 \leq u \leq l-1$, and $b_{l}=-\dfrac{1}{4}\dfrac{(2\sqrt{2l-1})^{l}}{l !}$. It is not hard to notice that $a_{0},a_{l-1},b_{l}\neq 0$. \\

Now, for all $t_{1} \in \mathbb{R}$, equation \eqref{eq:mustudent12} can be rewritten as:
\begin{eqnarray}
\mathop {\sum }\limits_{j=1}^{k}{\left[\left(\alpha_{j}^{''}+\beta_{j}^{''}(it_{1})\right)\mathop {\sum }\limits_{u=0}^{l-1}{a_{u}\sigma_{j}^{u}|t_{1}|^{u}}+\gamma_{j}^{''}\mathop {\sum }\limits_{u=0}^{l}{b_{u}\sigma_{j}^{u}|t_{1}|^{u}}\right]\exp(it\theta_{j}'-\sigma_{j}\sqrt{2l-1}|t_{1}|)}=0, \nonumber
\end{eqnarray}
where $\alpha_{j}^{''}=\alpha_{j}'+C_{3}(\mathop {\sum }\limits_{l=1}^{d}{M_{ll}^{j}})-C_{3}(\mathop {\sum }\limits_{u,v}{M_{uv}^{j}\dfrac{[\Sigma_{j}^{1/2}t']_{u}[\Sigma_{j}^{1/2}t']_{v}}{\sigma_{j}^{2}}})$, $\beta_{j}^{''}=C_{2}\dfrac{(\Sigma_{j}^{1/2}t')^{T}\beta_{j}'}{\sigma_{j}}$, and $\gamma_{j}^{''}=C_{2}(\mathop {\sum }\limits_{u,v}{M_{uv}^{j}\dfrac{[\Sigma_{j}^{1/2}t']_{u}[\Sigma_{j}^{1/2}t']_{v}}{\sigma_{j}^{2}}})$.\\
The above equation yields that for all $t_{1} \geq 0$
\begin{eqnarray}
\mathop {\sum }\limits_{j=1}^{k}{\left[\left(\alpha_{j}^{''}+\beta_{j}^{''}(it_{1})\right)\mathop {\sum }\limits_{u=0}^{l-1}{a_{u}\sigma_{j}^{u}t_{1}^{u}}+\gamma_{j}^{''}\mathop {\sum }\limits_{u=0}^{l}{b_{u}\sigma_{j}^{u}t_{1}^{u}}\right]\exp(it_{1}\theta_{j}'-\sigma_{j}\sqrt{2l-1}t_{1})}=0. \label{eq:mustudent13}
\end{eqnarray}
Using the Laplace transformation on both sides of \eqref{eq:mustudent13} and denoting $c_{j}=\sigma_{j}\sqrt{2l-1}-i\theta_{j}'$ as $1 \leq j \leq k$, we obtain that as $\text{Re}(s)>\mathop {\max }\limits_{1 \leq j \leq k}{\left\{-\sigma_{j}\sqrt{2l-1}\right\}}$
\begin{eqnarray}
\mathop {\sum }\limits_{j=1}^{k}{\alpha_{j}^{''}\mathop {\sum }\limits_{u=0}^{l-1}{\dfrac{u!a_{u}\sigma_{j}^{u}}{(s+c_{j})^{u+1}}}+i\beta_{j}^{''}\mathop {\sum }\limits_{u=1}^{l}{\dfrac{u!a_{u-1}\sigma_{j}^{u-1}}{(s+c_{j})^{u+1}}}+\gamma_{j}^{''}\mathop {\sum }\limits_{u=0}^{l}{\dfrac{u!b_{u}\sigma_{j}^{u}}{(s+c_{j})^{u+1}}}}=0.\label{eq:mustudent14}
\end{eqnarray}

Without loss of generality, we assume that $\sigma_{1} \leq \sigma_{2} \leq ... \leq \sigma_{k}$. It demonstrates that $-\sigma_{1}\sqrt{2l-1}=\mathop {\max }\limits_{1 \leq j \leq k}{\left\{-\sigma_{j}\sqrt{2l-1}\right\}}$. Denote $a_{u}^{(j)}=a_{u}\sigma_{j}^{u}$ and $b_{u}^{(j)}=b_{u}\sigma_{j}^{u}$ for all $u$. By multiplying both sides of \eqref{eq:mustudent14} with $(s+c_{1})^{l+1}$, as $\text{Re}(s)>-\sigma_{1}\sqrt{2l-1}$ and $s \to -c_{1}$, we obtain $|i\beta_{1}^{''}l!a_{l-1}^{(1)}+\gamma_{1}^{''}b_{l}l!b_{l}^{(1)}|=0$ or equivalently $\beta_{1}^{''}=\gamma_{1}^{''}=0$ since $a_{l-1}^{(1)},b_{l}^{(1)} \neq 0$. Likewise, multiply both sides of \eqref{eq:mustudent14} with $(s+c_{1})^{l}$ and using the same argument, as $s \to -c_{1}$, we obtain $\alpha_{1}^{''}=0$. Overall, we obtain $\alpha_{1}^{''}=\beta_{1}^{''}=\gamma_{1}^{''}=0$. Continue this fashion until we get $\alpha_{j}^{''}=\beta_{j}^{''}=\gamma_{j}^{''}=0$ for all $1 \leq j \leq k$ or equivalently $\alpha_{j}=\beta_{j}=\gamma_{j}=0$ for all $1 \leq j \leq k$.\\
 As a consequence, for all $1 \leq j \leq k$, we have
\begin{eqnarray}
 \alpha_{j}'+C_{3}(\mathop {\sum }\limits_{l=1}^{d}{M_{ll}^{j}})-C_{3}(\mathop {\sum }\limits_{u,v}{M_{uv}^{j}\dfrac{[\Sigma_{j}^{1/2}t']_{u}[\Sigma_{j}^{1/2}t']_{v}}{\sigma_{j}^{2}}})=0, \ \dfrac{(\Sigma_{j}^{1/2}t')^{T}\beta_{j}'}{\sigma_{j}}=0, \nonumber
 \end{eqnarray}
 and $\mathop {\sum }\limits_{u,v}{M_{uv}^{j}\dfrac{[\Sigma_{j}^{1/2}t']_{u}[\Sigma_{j}^{1/2}t']_{v}}{\sigma_{j}^{2}}}=0$. \\
 Since $\mathop {\sum }\limits_{u,v}{M_{uv}^{j}[\Sigma_{j}^{1/2}t']_{u}[\Sigma_{j}^{1/2}t']_{v}}=(t')^{T}\Sigma_{j}^{1/2}M_{j}\Sigma_{j}^{1/2}t'=(t')^{T}\gamma_{j}t'$, it is equivalent that 
\begin{eqnarray}
\alpha_{j}'+C_{3}(\mathop {\sum }\limits_{l=1}^{d}{M_{ll}^{j}})=0, (t')^{T}\Sigma_{j}^{1/2}\beta_{j}'=0, \ \text{and} \ (t')^{T}\gamma_{j}t'=0. \nonumber
\end{eqnarray}

With the same argument as the last paragraph of part (a) of Theorem \ref{identifiability-multivariatecharacterization}, we readily obtain that $\alpha_{j}'=0$, $\beta_{j}'=0 \in \mathbb{R}^{d}$, and $\gamma_{j}=0 \in \mathbb{R}^{d \times d}$. From the formation of $\alpha_{j}',\beta_{j}'$, it follows that $\alpha_{j}=0$, $\beta_{j}=0 \in \mathbb{R}^{d}$, and $\gamma_{j}=0 \in \mathbb{R}^{d \times d}$ for all $1 \leq j \leq k$.\\

\noindent
(c) Assume that we can find $\alpha_{i} \in \mathbb{R},\beta_{i} \in \mathbb{R}^{d},\eta_{i} \in \mathbb{R}^{d}$, and $\gamma_{i} \in \mathbb{R}^{d \times d}$ symmetric matrices such that:
\begin{eqnarray}
\mathop {\sum }\limits_{i=1}^{k}{\alpha_{i}f(x|\theta_{i},\Sigma_{i},\lambda_{i})+\beta_{i}^{T}\dfrac{\partial{f}}{\partial{\theta}}(x|\theta_{i},\Sigma_{i},\lambda_{i})+\trace((\dfrac{\partial{f}}{\partial{\Sigma}}(x|\theta_{i},\Sigma_{i},\lambda_{i}))^{T}\gamma_{i})+\eta_{i}^{T}\dfrac{\partial{f}}{\partial{\lambda}}(x|\theta_{i},\Sigma_{i},\lambda_{i})}=0. \label{eq:mmulstudent1}
\end{eqnarray}
where $\theta_{i} \in \mathbb{R}^{d}$, $\Sigma_{i} \in S_{d}^{++}$, and $\lambda_{i} \in \mathbb{R}^{d,+}$.

From the formation of $f$, we have $f=f_{Y}*f_{Z}$, where $f_{Y}(x|\theta,\Sigma)=\dfrac{1}{|\Sigma|^{1/2}}g((x-\theta)^{T}\Sigma^{-1}(x-\theta))$, $g(x)=C_{\nu}/(\nu+x)^{(\nu+d)/2}$, $C_{\nu}=\Gamma(\dfrac{\nu+d}{2})\nu^{\nu/2}/\Gamma(\dfrac{\nu}{2})\pi^{d/2}$, $f_{Z}(x|\lambda')=\prod \limits_{i=1}^{d}{\dfrac{(\lambda_{i}')^{b_{i}}}{\Gamma(b_{i})}x_{i}^{b_{i}-1}\exp(-\lambda_{i}'x_{i}).1_{\left\{x_{i}>0\right\}}}$ where $b_{1},\ldots,b_{k} \in \mathbb{N}$ are fixed number and $\lambda' \in \mathbb{R}^{d,+}$.\\\\
Denote ${\displaystyle \phi_{Z}(t|\lambda)= \int \limits_{\mathbb{R}^{d}}{\exp(it^{T}x)f_{Z}(x|\lambda)}dx}$. Multiplying both sides of \eqref{eq:mmulstudent1} with $\exp(it^{T}x)$ and take the integral in $\mathbb{R}^{d}$ , we have following results:
\begin{eqnarray}
\mathop {\sum }\limits_{j=1}^{k}{\alpha_{j}\int \limits_{\mathbb{R}^{d}}{\exp(it^{T}x)f(x|\theta_{j},\Sigma_{j},\lambda_{j})}} &=& \mathop {\sum }\limits_{j=1}^{k}{\alpha_{j}\sigma_{Z}(t|\lambda_{j})\int \limits_{\mathbb{R}^{d}}{\exp(it^{T}x)f_{Y}(x|\theta_{j},\Sigma_{j})}dx}. \nonumber \\
\mathop {\sum }\limits_{j=1}^{k}{\int \limits_{\mathbb{R}^{d}}{\exp(it^{T}x)\beta^{T}\dfrac{\partial{f}}{\partial{\theta}}(x|\theta_{j},\Sigma_{j},\lambda_{j})}dx}& = & \mathop {\sum }\limits_{j=1}^{k}{\sigma_{Z}(t|\lambda_{j})\int \limits_{\mathbb{R}^{d}}{\exp(it^{T}x)\beta_{j}^{T}\dfrac{\partial{f_{Y}}}{\partial{\theta}}(x|\theta_{j},\Sigma_{j})}dx}. \nonumber \\
\mathop {\sum }\limits_{j=1}^{d}{\int \limits_{\mathbb{R}^{d}}{\exp(it^{T}x)\trace((\dfrac{\partial{f}}{\partial{\Sigma}}(x|\theta_{j},\Sigma_{j},\lambda_{j})^{T}\gamma_{j})}dx}& = & \mathop {\sum }\limits_{j=1}^{d}{\sigma_{Z}(t|\lambda_{j})\int \limits_{\mathbb{R}^{d}}{\exp(it^{T}x)\trace((\dfrac{\partial{f_{Y}}}{\partial{\Sigma}}(x|\theta_{j},\Sigma_{j}))^{T}\gamma_{j})}dx}. \nonumber \\
\mathop {\sum }\limits_{j=1}^{d}{\int \limits_{\mathbb{R}^{d}}{\exp(it^{T}x)\eta_{j}^{T}\dfrac{\partial{f}}{\partial{\lambda}}(x|\theta_{j},\Sigma_{j},\lambda_{j})}dx} & = & \mathop {\sum }\limits_{j=1}^{d}{\int \limits_{\mathbb{R}^{d}}{\exp(it^{T}x)f_{Y}(x|\theta_{j},\Sigma_{j})}dx} \times \nonumber \\
& \times & \int \limits_{\mathbb{R}^{d}}{\exp(it^{T}x)\eta_{j}^{T}\dfrac{\partial{f_{Z}}}{\partial{\lambda}}(x|\lambda_{j})}dx. \nonumber
\end{eqnarray}
Therefore, under this transformation, equation \eqref{eq:mmulstudent1} can be rewritten as
\begin{eqnarray}
\mathop {\sum }\limits_{j=1}^{k}{\sigma_{Z}(t|\lambda_{j})(\alpha_{j}\int \limits_{\mathbb{R}^{d}}{\exp(it^{T}x)f_{Y}(x|\theta_{j},\Sigma_{j})}dx} & + & \nonumber \\
\int \limits_{\mathbb{R}^{d}}{\exp(it^{T}x)\beta_{j}^{T}\dfrac{\partial{f_{Y}}}{\partial{\theta}}(x|\theta_{j},\Sigma_{j})}dx + \int \limits_{\mathbb{R}^{d}}{\exp(it^{T}x)\trace((\dfrac{\partial{f_{Y}}}{\partial{\Sigma}}(x|\theta_{j},\Sigma_{j}))^{T}\gamma_{j})}dx) & + & \nonumber \\
\int \limits_{\mathbb{R}^{d}}{\exp(it^{T}x)f_{Y}(x|\theta_{j},\Sigma_{j})}dx\int \limits_{\mathbb{R}^{d}}{\exp(it^{T}x)\eta_{j}^{T}\dfrac{\partial{f_{Z}}}{\partial{\lambda}}(x|\lambda_{j})}dx  & = &  0. \quad \quad \label{eq:mmulstudent2}
\end{eqnarray}
Using \eqref{eq:mustudent11}, we have 
\begin{eqnarray}
{\displaystyle \int \limits_{\mathbb{R}^{d}}{\exp(it^{T}x)f_{Y}(x|\theta_{j},\Sigma_{j})}dx=C_{\nu}C_{1}\exp(it^{T}\theta_{j}A_{1}(||\Sigma_{j}^{1/2}t||))} \nonumber
\end{eqnarray}
and
\begin{eqnarray}
\mathop {\sum }\limits_{j=1}^{k}{\int \limits_{\mathbb{R}^{d}}{\left(\alpha_{j}f_{Y}(x|\theta_{j},\Sigma_{j})+\beta_{j}^{T}\dfrac{\partial{f_{Y}}}{\partial{\theta}}(x|\theta_{j},\Sigma_{j})+\exp(it^{T}x)\trace((\dfrac{\partial{f_{Y}}}{\partial{\Sigma}}(x|\theta_{j},\Sigma_{j}))^{T}\gamma_{j})\right)}}\exp(it^{T}x)dx & = & \nonumber \\
 \mathop {\sum }\limits_{j=1}^{k}{C_{\nu}\left[\left(\alpha_{j}'+C_{3}\trace(M_{j})-\dfrac{C_{3}t^{T}t^{T}\gamma_{j}t}{t^{T}\Sigma_{j}t}+\dfrac{i C_{2}(\Sigma_{j}^{1/2}t)^{T}\beta_{j}'}{\nu+1}\right)A_{1}(||\Sigma^{1/2}t||)\right]}\exp(it^{T}\theta_{j}) & + & \nonumber \\
 \left[\dfrac{C_{2}t^{T}\gamma_{j}t}{t^{T}\Sigma_{j}t}A_{3}(||\Sigma_{j}^{1/2}t||)\right]\exp(it^{T}\theta_{j}). \nonumber
\end{eqnarray} 
where ${\displaystyle A_{1}(t')=\int \limits_{\mathbb{R}}{\dfrac{\exp(i|t'|z)}{(v+z^{2})^{(\nu+1)/2}}}dz}$, ${\displaystyle A_{3}(t')=\int \limits_{\mathbb{R}}{\dfrac{\exp(i|t'|z)z^{2}}{(\nu+z^{2})^{(\nu+3)/2}}}dz}$ for any $t' \in \mathbb{R}$, and $\alpha_{j}'=\alpha_{j}-\dfrac{\trace(\Sigma_{j}^{-1}\gamma_{j})}{2}$, $\beta_{j}'=\dfrac{\nu+d}{2}\Sigma^{-1/2}\beta_{j}$, and $M_{j}=\dfrac{\nu+d}{2}\Sigma_{j}^{-1/2}\gamma_{j}\Sigma_{j}^{-1/2}$.\\
Denote $f_{Z_{l}}(x_{l}|\lambda_{l}')=\dfrac{(\lambda_{l}')^{b_{l}}}{\Gamma(b_{l})}x_{l}^{b_{l}-1}\exp(-\lambda_{l}'x_{l}).1_{\left\{x_{l}>0\right\}}$ and ${\displaystyle \phi_{Z_{l}}(t|\lambda_{l}')=\int \limits_{\mathbb{R}}{\exp(itx_{l})f_{Z_{l}}(x_{l}|\lambda_{l}')}dx_{l}}$ as $\lambda_{l}' \in \mathbb{R}$, we obtain 
\begin{eqnarray}
\phi_{Z}(t|\lambda_{j})=\prod \limits_{l=1}^{d}{\phi_{Z_{l}}(x_{l}|\lambda_{j}^{l})}=\prod \limits_{l=1}^{d}{\dfrac{(\lambda_{j}^{l})^{b_{l}}}{(\lambda_{j}^{l}-it)^{b_{l}}}}, \nonumber
\end{eqnarray} 
where $\lambda_{j}=(\lambda_{j}^{1},\ldots,\lambda_{j}^{d})$.\\
Additionally, by denoting $\eta_{j}=(\eta_{j}^{1},\ldots,\eta_{j}^{d})$
\begin{eqnarray}
\int \limits_{\mathbb{R}^{d}}{\exp(it^{T}x)\eta_{j}^{T}\dfrac{\partial{f_{Z}}}{\partial{\lambda}}(x|\lambda_{j})}dx  &=& \mathop {\sum }\limits_{l=1}^{d}{\eta_{j}^{l}\mathop {\prod }\limits_{u \neq l}{\phi_{Z_{u}}(t_{u}|\lambda_{j}^{u})}\int \limits_{R}{\exp(it_{l}x_{l})\dfrac{\partial{f_{Z_{l}}}}{\partial{\lambda_{l}'}}(x|\lambda_{j}^{l})}dx_{l}} \nonumber \\
&=& \mathop {\sum }\limits_{l=1}^{d}{\eta_{j}^{l}\mathop {\prod }\limits_{u \neq l}{\phi_{Z_{u}}(t_{u}|\lambda_{j}^{u})}\dfrac{\partial{\phi_{Z_{l}}}}{\partial{\lambda_{l}}}(t_{l}|\lambda_{j}^{l})} \nonumber \\
&=& -i \mathop {\sum }\limits_{l=1}^{d}{\eta_{j}^{l}\dfrac{\beta_{l}(\lambda_{j}^{l})^{b_{l}-1}t_{l}}{b_{j}^{l}-it_{l})^{b_{l}+1}}\prod \limits_{u \neq l}{\dfrac{(\lambda_{j}^{u})^{b_{u}}}{(\lambda_{j}^{u}-it_{u})^{b_{u}}}}}. \nonumber
\end{eqnarray}
Multiplying both sides of equation \eqref{eq:mmulstudent2} with $\prod \limits_{j=1}^{k}{\prod \limits_{u=1}^{d}{(\lambda_{j}^{u}-it_{u})^{b_{u}+1}}}$, we obtain:
\begin{eqnarray}
\mathop {\sum }\limits_{j=1}^{k}{\left[\left(\nu_{j}'-\dfrac{C_{3}t^{T}t^{T}\gamma_{j}t}{t^{T}\Sigma_{j}t}+\dfrac{i C_{2}(\Sigma_{j}^{1/2}t)^{T}\beta_{j}'}{\nu+1}\right)A_{1}(||\Sigma^{1/2}t||)+\dfrac{C_{2}t^{T}\gamma_{j}t}{t^{T}\Sigma_{j}t}A_{3}(||\Sigma_{j}^{1/2}t||)\right]} & \times & \nonumber \\
\exp(it^{T}\theta_{j})\prod \limits_{u=1}^{d}{(\lambda_{j}^{u})^{b_{u}}(\lambda_{j}^{u}-it_{u})}\prod \limits_{l \neq j}{\prod \limits_{u=1}^{d}{(\lambda_{l}^{u}-it_{u})^{b_{u}+1}}}-iC_{1}\exp(it^{T}\theta_{j})A_{1}(||\Sigma_{j}^{1/2}t||) & \times & \nonumber \\
\left(\mathop {\sum }\limits_{l=1}^{d}{\eta_{j}^{l}b_{l}(\lambda_{j}^{l})^{b_{l}-1}\prod \limits_{u \neq l}{(\lambda_{j}^{u})^{b_{u}}}t_{l}\prod \limits_{u \neq l}{(\lambda_{j}^{u}-it_{u})}}\right)\prod \limits_{l \neq j}{\prod \limits_{u=1}^{d}{(\lambda_{l}^{u}-it_{u})^{b_{u}+1}}} & = & 0, \quad \quad \label{eq:mmulstudent3}
\end{eqnarray}
where $\nu_{j}'=\alpha_{j}'+C_{3}(\mathop {\sum }\limits_{l=1}^{d}{M_{ll}^{j}})$. Using the same argument as that of multivariate generalized Gaussian distribution, we can find set $D$ being the union of finite hyperplanes and cones such that as $t' \notin D$, $((t')^{T}\theta_{1},(t')^{T}\Sigma_{1}t'),\ldots,((t')^{T}\theta_{k},(t')^{T}\Sigma_{k}t')$ are pairwise different. Denote $t=t_{1}t'$, where $t_{1} \in R$ and $t' \notin D$ and $\theta_{j}'=(t')^{T}\theta_{j}$, $\sigma_{j}^{2}=(t')^{T}\Sigma_{j}t'$.
For all $t_{1} \geq 0$, using the result from multivariate Student's t-distribution, we can denote $A_{1}(t_{1})=C_{1}'\exp(-t_{1}\sqrt{\nu})\mathop {\sum }\limits_{u=0}^{l_{1}-1}{a_{u}t_{1}^{u}}$ and $A_{3}(t_{1})=C_{1}'\exp(-t_{1}\sqrt{\nu})\mathop {\sum }\limits_{u=0}^{l_{1}}{b_{u}t_{1}^{u}}$, where $\nu=2l_{1}-1$ and $a_{0},a_{l_{1}-1},b_{0},b_{l_{1}} \neq 0$.\\
Define $\left(\mathop {\sum }\limits_{u=0}^{l_{1}-1}{a_{u}t_{1}^{u}}\right)\prod \limits_{u=1}^{d}{(\lambda_{j}^{u})^{b_{u}}(\lambda_{j}^{u}-it_{u}'t_{1})}\prod \limits_{l \neq j}{\prod \limits_{u=1}^{d}{(\lambda_{l}^{u}-it_{u}'t_{1})^{b_{u}+1}}}=\mathop {\sum }\limits_{u=0}^{m_{1}}{c_{u}^{j}t_{1}^{u}}$, where $m_{1}=l_{1}+d-2+(d+\mathop {\sum }\limits_{u=1}^{d}{b_{u}})(k-1)$. Additionally, we define
\begin{eqnarray}
\mathop {\sum }\limits_{u=0}^{m_{1}+1}{d_{u}^{j}t_{1}^{u}} : = \left(\mathop {\sum }\limits_{u=0}^{l_{1}}{b_{u}t_{1}^{u}}\right)\prod \limits_{u=1}^{d}{(\lambda_{j}^{u})^{b_{u}}(\lambda_{j}^{u}-it_{u}'t_{1})}\prod \limits_{l \neq j}{\prod \limits_{u=1}^{d}{(\lambda_{l}^{u}-it_{u}'t_{1})^{b_{u}+1}}} \nonumber
\end{eqnarray} 
and 
\begin{eqnarray}
\mathop {\sum }\limits_{u=1}^{m_{1}+1}{e_{u}^{j}t_{1}^{u}} := \left(\mathop {\sum }\limits_{u=0}^{l_{1}-1}{a_{u}t_{1}^{u}}\right)\left(\mathop {\sum }\limits_{l=1}^{d}{\eta_{j}^{l}b_{l}(\lambda_{j}^{l})^{b_{l}-1}\prod \limits_{u \neq l}{(\lambda_{j}^{u})^{b_{u}}}t_{l}'t_{1}\prod \limits_{u \neq l}{(\lambda_{j}^{u}-it_{u}'t_{1})}}\right)\prod \limits_{l \neq j}{\prod \limits_{u=1}^{d}{(\lambda_{l}^{u}-it_{u}'t_{1})^{b_{u}+1}}}.\nonumber
\end{eqnarray}
Equation \eqref{eq:mmulstudent3} can be rewritten as
\begin{eqnarray}
\mathop {\sum }\limits_{j=1}^{k}{\left[(\alpha_{j}^{''}+\beta_{j}^{''}(it_{1}))\mathop {\sum }\limits_{u=0}^{m_{1}}{c_{u}^{j}t_{1}^{u}}+\gamma_{j}^{''}\mathop {\sum }\limits_{u=0}^{m_{1}+1}{d_{u}^{j}t_{1}^{u}}-iC_{1}\mathop {\sum }\limits_{u=1}^{m_{1}+1}{e_{u}^{j}t_{1}^{u}}\right]\exp(i\theta_{j}'t_{1}-\sigma_{j}\sqrt{\nu})}=0, \label{eq:mmulstudent4}
\end{eqnarray}
where $\alpha_{j}^{''}=\alpha_{j}'+C_{3}\trace(M_{j})-\dfrac{C_{3}(t')^{T}\gamma_{j}t'}{\sigma_{j}^{2}}$, $\beta_{j}^{''}=\dfrac{(\Sigma_{j}^{1/2}t')^{T}\beta_{j}'}{\nu+1}$, and $\gamma_{j}^{''}=\dfrac{C_{2}(t')^{T}\gamma_{j}t'}{\sigma_{j}^{2}}$.\\

Without loss of generality, we assume $\sigma_{1} \leq \sigma_{2} \leq \ldots \leq \sigma_{k}$. Denote $h_{j}=\sigma_{j}\sqrt{\nu}-i\theta_{j}'$ and apply Laplace transformation to \eqref{eq:mmulstudent4}, we obtain that as $\text{Re}(s)>-\sigma_{1}\sqrt{\nu}$
\begin{eqnarray}
\mathop {\sum }\limits_{j=1}^{k}{\alpha_{j}^{''}\mathop {\sum }\limits_{u=0}^{m_{1}}{\dfrac{c_{u}^{j}u!}{(s+h_{j})^{u+1}}}+i\beta_{j}^{''}\mathop {\sum }\limits_{u=1}^{m_{1}+1}{\dfrac{c_{u-1}^{j}u!}{(s+h_{j})^{u+1}}}+\gamma_{j}^{''}\mathop {\sum }\limits_{u=0}^{m_{1}+1}{\dfrac{d_{u}^{j}u!}{(s+h_{j})^{u+1}}}} & - & \nonumber \\
iC_{1}\mathop {\sum }\limits_{u=1}^{m_{1}}{\dfrac{e_{u}^{j}u!}{(s+h_{j})^{u+1}}} & = & 0. \label{eq:mmulstudent5}
\end{eqnarray}
Using the same argument as that of multivariate Student's t-distribution, by multiplying both sides of equation \eqref{eq:mmulstudent5} with $(s+h_{1})^{m_{1}+2}$ and let $s \to -h_{1}$, we obtain $|i\beta_{1}^{''}c_{m_{1}}^{1}+\gamma_{1}^{''}d_{m_{1}+1}^{1}|=0$. Since 
\begin{eqnarray}
c_{m_{1}}^{1}=(-i)^{(d+\mathop {\sum }\limits_{u=0}^{d}{b_{u}})(k-1)+d}a_{l_{1}-1}\prod \limits_{u=1}^{d}{(\lambda_{1}^{''})^{b_{u}}(t_{u}')^{(b_{u}+1)(k-1)+1}} \nonumber
\end{eqnarray}
and 
\begin{eqnarray}
d_{m_{1}}^{1}=(-i)^{(d+\mathop {\sum }\limits_{u=0}^{d}{b_{u}})(k-1)+d}b_{l_{1}}\prod \limits_{u=1}^{d}{(\lambda_{1}^{''})^{b_{u}}(t_{u}')^{(b_{u}+1)(k-1)+1}}, \nonumber
\end{eqnarray} the equation $|i\beta_{1}^{''}c_{m_{1}}^{1}+\gamma_{1}^{''}d_{m_{1}+1}^{1}|=0$ is equivalent to $|i\beta_{1}^{''}a_{l_{1}-1}+\gamma_{1}^{''}b_{l_{1}}|=0$, which yields that $\beta_{1}^{''}a_{l_{1}-1}=\gamma_{1}^{''}b_{l_{1}}=0$. As $a_{l_{1}-1},b_{l_{1}} \neq 0$, we obtain $\beta_{1}^{''}=\gamma_{1}^{''}=0$.\\

With this result, we multiply two sides of \eqref{eq:mmulstudent5} with $(s+h_{1})^{m_{1}+1}$ and let $s \to -h_{1}$, we obtain $|a_{1}^{''}c_{m_{1}}^{1}-iC_{1}e_{m_{1}}^{1}|=0$. Then, we multiply both sides of \eqref{eq:mmulstudent5} with $(s+h_{1})^{m_{1}}$ and let $s \to -h_{1}$, we get $|\alpha_{1}^{''}c_{m_{1}-1}^{1}(m_{1}-1)!-iC_{1}e_{m_{1}-1}^{1}(m_{1}-1)!|=0$. Repeat this argument until we obtain $|\alpha_{1}^{''}c_{0}^{1}|=0$ and $|\alpha_{1}^{''}c_{1}^{1}-iC_{1}e_{1}^{1}|=0$, which implies that $\alpha_{1}^{''}=0$ as $c_{0}^{1}=a_{0}\prod \limits_{l=1}^{k}{\prod \limits_{u=1}^{d}{(\lambda_{l}^{u})^{b_{u}+1}}} \neq 0$ and $e_{1}^{1}=0$.\\
From the formation of $e_{1}^{1}$, it yields that
\begin{eqnarray} 
a_{0}\left(\mathop {\sum }\limits_{l=1}^{d}{\eta_{1}^{l}t_{l}'b_{l}(\lambda_{1}^{l})^{b_{l}-1}\prod \limits_{u \neq l}{(\lambda_{1}^{u})^{b_{u}+1}}}\right)\prod \limits_{l \neq 1}{\prod \limits_{u=1}^{d}{(\lambda_{l}^{u})^{b_{u}+1}}}=0. \nonumber
\end{eqnarray}
As $a_{0} \neq 0$, it implies that 
\begin{eqnarray}
\mathop {\sum }\limits_{l=1}^{d}{\eta_{1}^{l}t_{l}'b_{l}(\lambda_{1}^{l})^{b_{l}-1}\prod \limits_{u \neq l}{(\lambda_{1}^{u})^{b_{u}+1}}}=0. \nonumber
\end{eqnarray}Denote $\eta_{1}^{l}b_{l}(\lambda_{1}^{l})^{b_{l}-1}\prod \limits_{u \neq l}{(\lambda_{1}^{u})^{b_{u}+1}}=\psi_{1}^{l}$ for all $1 \leq l \leq d$ then we have $\mathop {\sum }\limits_{l=1}^{d}{\psi_{1}^{l}t_{l}'}=0$. If there is any $\psi_{1}^{l} \neq 0$, by choosing $t'$ to lie outside that hyperplane, we will not get the equality $\mathop {\sum }\limits_{l=1}^{d}{\psi_{1}^{l}t_{l}'}=0$. Therefore, $\psi_{1}^{l}=0$ for all $1 \leq l \leq d$, which implies that $\eta_{1}^{l}=0$ for all $1 \leq l \leq d$ or equivalently $\eta_{1}=0$. Repeating the above argument until we obtain $\alpha_{j}^{''}=\beta_{j}^{''}=\gamma_{j}^{''}=0 \in \mathbb{R}$ and $\eta_{j}=0 \in \mathbb{R}^{d}$ for all $1 \leq j \leq k$. From the formation of $\alpha_{j}^{''},\beta_{j}^{''},\gamma_{j}^{''}$, using the same argument as that of multivariate Student's t-distribution, by choosing $t'$ appropriately, we will have $\alpha_{j}=0$, $\beta_{j}=0 \in \mathbb{R}^{d}$, and $\gamma_{j}=0 \in \mathbb{R}^{d \times d}$ for all $1 \leq j \leq k$.\\

\noindent
(d) Assume that we can find $\alpha_{j} \in \mathbb{R}$,$\beta_{j} \in \mathbb{R}^{d}$, symmetric matrices $\gamma_{j} \in \mathbb{R}^{d \times d}$,$\eta_{j} \in \mathbb{R}^{d}$, and $\tau_{j} \in \mathbb{R}^{d}$ such that
\begin{eqnarray}
\mathop {\sum }\limits_{j=1}^{k}{\alpha_{j}f(x|\theta_{j},\Sigma_{j},a_{j},b_{j})+\beta_{j}^{T}\dfrac{\partial{f}}{\partial{\theta}}(x|\theta_{j},\Sigma_{j},a_{j},b_{j})+\trace(\dfrac{\partial{f}}{\partial{\Sigma}}(x|\theta_{j},\Sigma_{j},a_{j},b_{j})^{T}\gamma_{j})} & + & \nonumber \\
\eta_{j}^{T}\dfrac{\partial{f}}{\partial{a}}(x|\theta_{j},\Sigma_{j},a_{j},b_{j})+\tau_{j}^{T}\dfrac{\partial{f}}{\partial{b}}(x|\theta_{j},\Sigma_{j},a_{j},b_{j}) & = & 0. \quad \quad \label{eq:gamgauss1} 
\end{eqnarray}
Denote $Z=\prod \limits_{j=1}^{d}{Z_{j}}$, where $Z_{j} \sim \text{Gamma}(a_{j},b_{j})$. Let $\phi_{Z_{j}}(t_{j}|a_{j},b_{j})$ to be the moment generating function of $Z_{j}$, then $\phi_{Z_{j}}(t_{j}|a_{j},b_{j})=b_{j}^{a_{j}}/(b_{j}-a_{j})^{a_{j}}$ as $t_{j}<b_{j}$. Therefore, the moment generating function $\phi_{Z}(t|a,b)$ of $Z$ is $\prod \limits_{j=1}^{d}{\dfrac{b_{j}^{a_{j}}}{(b_{j}-t_{j})^{a_{j}}}}$ as $t_{j}<b_{j}$ for all $1 \leq j \leq d$.\\

Multiply both sides of \eqref{eq:gamgauss1} with $\exp(t^{T}x)$ and take the integral in $\mathbb{R}^{d}$, using the same argument as that of multivariate generalized Gaussian case, we obtain that as $t_{i}<\mathop {\min }\limits_{1 \leq j \leq k}{\left\{b_{j}^{i}\right\}}$ for all $1 \leq i \leq k$
\begin{eqnarray}
\mathop {\sum }\limits_{j=1}^{k}{\left(\alpha_{j}+\beta_{j}^{T}t +\dfrac{t^{T}\gamma_{j}t}{2}+\mathop {\sum }\limits_{l=1}^{d}{\eta_{j}^{l}\log\left(\dfrac{b_{j}^{l}}{b_{j}^{l}-t_{l}}\right)}-\mathop {\sum }\limits_{l=1}^{d}{\tau_{j}^{l}\dfrac{a_{j}^{l}t_{l}}{b_{j}^{l}(b_{j}^{l}-t_{l})}}\right)} & \times & \nonumber \\
\times \exp(t^{T}\theta_{j}+\dfrac{1}{2}t^{T}\Sigma_{j}t)\prod \limits_{i=1}^{d}{\dfrac{(b_{j}^{i})^{a_{j}^{i}}}{(b_{j}^{i}-t_{i})^{a_{j}^{i}}}} & = & 0. \nonumber
\end{eqnarray}
Multiply both sides of the above equation with $\prod \limits_{u=1}^{k}{\prod \limits_{i=1}^{d}{(b_{u}^{i}-t_{i})^{a_{u}^{i}+1}}}$, we can rewrite it as 
\begin{eqnarray}
\mathop {\sum }\limits_{j=1}^{k}{\left((\alpha_{j}+\beta_{j}^{T}t+\dfrac{t^{T}\gamma_{j}t}{2}+\mathop {\sum }\limits_{l=1}^{d}{\eta_{j}^{l}\log\left(\dfrac{b_{j}^{l}}{b_{j}^{l}-t_{l}}\right)}\right)\prod \limits_{i=1}^{d}{(b_{j}^{i}-t_{i})}} & - &\nonumber \\
-\mathop {\sum }\limits_{l=1}^{d}{\tau_{j}^{l}a_{j}^{l}t_{l}\prod \limits_{u \neq l}{(b_{j}^{u}-t_{u}}})) \exp(t^{T}\theta_{j}+\dfrac{1}{2}t^{T}\Sigma_{j}t)\prod \limits_{i=1}^{d}{(b_{j}^{i})^{a_{j}^{i}}}\prod \limits_{u \neq j}{\prod \limits_{i=1}^{d}{(b_{u}^{i}-t_{i})^{a_{u}^{i}+1}}} & = & 0. \label{eq:gamgauss3}
\end{eqnarray}
Put $t=t_{1}t'$ as $t_{1} \in \mathbb{R}$ and $t' \in \mathbb{R}^{d,+}$. We can find set $D$, which is the finite union of hyperplanes and cones such that as $t' \notin D$ and $t' \in \mathbb{R}^{d,+}$, we get that $((t')^{T}\theta_{1},(t')^{T}\Sigma_{1}t'),\ldots,((t')^{T}\theta_{k},(t')^{T}\Sigma_{k}t')$ are pairwise different. Therefore as $t_{i}< \mathop {\min }\limits_{1 \leq j \leq k}{\left\{b_{j}^{i}\right\}}$ for all $1 \leq i \leq k$, we get $t_{1}<t^{*}=\mathop {\min }\limits_{1 \leq j \leq k,1 \leq i \leq d}{\left\{\dfrac{b_{j}^{i}}{t_{i}'}\right\}}$. Denote $\theta_{j}'=t^{T}\theta_{j}$ and $\sigma_{j}^{2}=t^{T}\Sigma_{j}t$, as $t_{1} <t^{*}$, we can rewrite \eqref{eq:gamgauss3} as follows 
\begin{eqnarray}
\mathop {\sum }\limits_{j=1}^{k}{\left(\alpha_{j}+t_{1}\beta_{j}^{T}t'+t_{1}^{2}\dfrac{(t')^{T}\gamma_{j}t'}{2}+\mathop {\sum }\limits_{l=1}^{d}{\eta_{j}^{l}\log\left(\dfrac{b_{j}^{l}}{b_{j}^{l}-t_{l}'t_{1}}\right)}\right)\prod \limits_{i=1}^{d}{(b_{j}^{i}-t_{i}'t_{1})}} & - &\nonumber \\
\mathop {\sum }\limits_{l=1}^{d}{\tau_{j}^{l}a_{j}^{l}t_{l}'t_{1}\prod \limits_{u \neq l}{(b_{j}^{u}-t_{u}'t_{1}}})) \exp(\theta_{j}'t_{1}+\dfrac{\sigma_{j}^{2}t^{2}}{2})\prod \limits_{i=1}^{d}{(b_{j}^{i})^{a_{j}^{i}}}\prod \limits_{u \neq j}{\prod \limits_{i=1}^{d}{(b_{u}^{i}-t_{i}'t_{1})^{a_{u}^{i}+1}}} & = & 0. \quad \quad \label{eq:gamgauss4}
\end{eqnarray}

Without loss of generality, we assume that $\sigma_{1} \leq \sigma_{2} \leq \ldots \leq \sigma_{k}$. By using the same argument as that of multivariate generalized Gaussian distribution in Theorem \eqref{identifiability-multivariatecharacterization}, we denote $\overline{i}$ to be minimum index such that $\sigma_{\overline{i}}=\sigma_{k}$ and $i_{k}$ as the index such that $\theta_{i_{k}}'=\mathop {\min }\limits_{\overline{i} \leq j \leq k}{\left\{\theta_{j}'\right\}}$. Multiply both sides of \eqref{eq:mmulstudent5} with $\exp(-\theta_{i_{k}}'t_{1}-\dfrac{\sigma_{i_{k}}^{2}t_{1}^{2}}{2})$ and let $t_{1} \to -\infty$, using the convergence argument of generalized Gaussian case, we eventually obtain as $t_{1} \to -\infty$
\begin{eqnarray}
\left(\alpha_{i_{k}}+t_{1}\beta_{i_{k}}^{T}t'+t_{1}^{2}\dfrac{(t')^{T}\gamma_{i_{k}}t'}{2}+\mathop {\sum }\limits_{l=1}^{d}{\eta_{i_{k}}^{l}\log\left(\dfrac{b_{j}^{l}}{b_{i_{k}}^{l}-t_{l}'t_{1}}\right)}\right)\prod \limits_{i=1}^{d}{(b_{i_{k}}^{i}-t_{i}'t_{1})} & - &\nonumber \\
\mathop {\sum }\limits_{l=1}^{d}{\tau_{i_{k}}^{l}a_{i_{k}}^{l}t_{l}'t_{1}\prod \limits_{u \neq l}{(b_{i_{k}}^{u}-t_{u}'t_{1}}}))\prod \limits_{i=1}^{d}{(b_{i_{k}}^{i})^{a_{i_{k}}^{i}}}\prod \limits_{u \neq i_{k}}{\prod \limits_{i=1}^{d}{(b_{u}^{i}-t_{i}'t_{1})^{a_{u}^{i}+1}}} & \to & 0. \nonumber
\end{eqnarray}
Since $\prod \limits_{i=1}^{d}{(b_{i_{k}}^{i})^{a_{i_{k}}^{i}}}\prod \limits_{u \neq i_{k}}{\prod \limits_{i=1}^{d}{(b_{u}^{i}-t_{i}'t_{1})^{a_{u}^{i}+1}}} \to +\infty$ as $t_{1} \to -\infty$, the above result implies that as $t_{1} \to -\infty$,
\begin{eqnarray}
B(t_{1})=\left(\alpha_{i_{k}}+t_{1}\beta_{i_{k}}^{T}t'+t_{1}^{2}\dfrac{(t')^{T}\gamma_{i_{k}}t'}{2}+\mathop {\sum }\limits_{l=1}^{d}{\eta_{i_{k}}^{l}\log\left(\dfrac{b_{j}^{l}}{b_{i_{k}}^{l}-t_{l}'t_{1}}\right)}\right)  & \times & \nonumber \\
\times \prod \limits_{i=1}^{d}{(b_{i_{k}}^{i}-t_{i}'t_{1})} - \mathop {\sum }\limits_{l=1}^{d}{\tau_{i_{k}}^{l}a_{i_{k}}^{l}t_{l}'t_{1}\prod \limits_{u \neq l}{(b_{i_{k}}^{u}-t_{u}'t_{1}}}) & \to & 0. \quad \quad  \label{eq:gamgauss5}
\end{eqnarray}
Note that the highest degree in terms of $t_{1}$ in $B(t_{1})$ is $d+2$ and its corresponding coefficient is $(-1)^{d}\prod \limits_{i=1}^{d}{t_{i}'}\dfrac{(t')^{T}\gamma_{i_{k}}t'}{2}$. As $B(t_{1}) \to 0$ as $t_{1} \to -\infty$, it implies that $(t')^{T}\gamma_{i_{k}}t'=0$, which yields that $\gamma_{i_{k}}=0$ under appropriate choice of $t'$. Similarly, the coefficient of $t_{1}^{d+1}$ in $B(t_{1})$ is $(-1)^{d}\prod \limits_{i=1}^{d}{t_{i}'}\beta_{i_{k}}^{T}t'$. Therefore, $\beta_{i_{k}}^{T}t'=0$, which implies that $\beta_{i_{k}}=0$. With these results, from \eqref{eq:gamgauss5}, we see that 
\begin{eqnarray}
\left(\mathop {\sum }\limits_{l=1}^{d}{\eta_{i_{k}}^{l}\log(b_{i_{k}}^{l}-t_{l}'t_{1})}\right)\prod \limits_{i=1}^{d}{(b_{i_{k}}-t_{i}'t_{1})} \to 0 \ \ \text{as} \ t_{1} \to -\infty. \nonumber
\end{eqnarray}
It follows that $\eta_{i_{k}}^{l}=0$ for all $1 \leq l \leq d$.  Now, the coefficient of $t_{1}^{0}$ in $B(t_{1})$ is $\alpha_{i_{k}}\prod \limits_{i=1}^{d}{b_{i_{k}}^{i}}$; therefore, it implies that $\alpha_{i_{k}}=0$. Last but not least, the coefficient of $t_{1}$ now is $-\mathop {\sum }\limits_{l=1}^{d}{\tau_{i_{k}}^{l}a_{i_{k}}^{l}t_{l}'\prod \limits_{u \neq l}{b_{i_{k}}^{u}}}$. Thus, we have $\mathop {\sum }\limits_{l=1}^{d}{\tau_{i_{k}}^{l}a_{i_{k}}^{l}t_{l}'\prod \limits_{u \neq l}{b_{i_{k}}^{u}}}=0$. By an appropriate choice of $t'$, we obtain $\tau_{i_{k}}^{l}=0$ for all $1 \leq l \leq d$. Repeat the above argument until we get $\alpha_{i}=0$, $\beta_{i}=\eta_{i}=\tau_{i}=0 \in \mathbb{R}^{d}$, and $\gamma_{i}=0 \in \mathbb{R}^{d \times d}$, which yields the conclusion of our theorem.

\begin{lemma}\label{lemma-complexformula}
For any $m \in \mathbb{N}$, we have
\begin{eqnarray}
\int \limits_{-\infty}^{+\infty}{\dfrac{\exp(itx)}{(x^{2}+1)^{m}}}dx=\dfrac{2\pi\exp(-|t|)}{2^{2m-1}}\left[\mathop {\sum }\limits_{j=1}^{m}{\dbinom{2m-1-j}{m-j}\dfrac{(2|t|)^{j-1}}{(j-1)!}}\right]. \label{eq:complex}
\end{eqnarray}
\end{lemma}
\begin{proof}
Assume that $t>0$ and for any $R>0$, we define $C_{R}=I_{R} \cup \Gamma_{R}$, where $\Gamma_{R}$ is the upper half of the circle $|z|=R$ and $I_{R}=\left\{z \in C:|\text{Re}(z)| \leq R \quad \text{and }  \text{Im}(z)=0\right\}$. Now, we have the following formula:
\begin{eqnarray}
\oint \limits_{C_{R}}{\dfrac{\exp(itz)}{(z^{2}+1)^{m}}}dz=\oint \limits_{I_{R}}{\dfrac{\exp(itz)}{(z^{2}+1)^{m}}}dz+\oint \limits_{\Gamma_{R}}{\dfrac{\exp(itz)}{(z^{2}+1)^{m}}}dz. \nonumber
\end{eqnarray}
Notice that ${\displaystyle \oint \limits_{I_{R}}{\dfrac{\exp(itz)}{(z^{2}+1)^{m}}}dz=\int \limits_{-R}^{R}{\dfrac{\exp(itx)}{(x^{2}+1)^{m}}}dx}$, therefore
\begin{eqnarray}
 {\displaystyle \oint \limits_{C_{R}}{\dfrac{\exp(itz)}{(z^{2}+1)^{m}}}dz=\int \limits_{-R}^{R}{\dfrac{\exp(itx)}{(x^{2}+1)^{m}}}dx}+\oint \limits_{\Gamma_{R}}{\dfrac{\exp(itz)}{(z^{2}+1)^{m}}}dz. \nonumber
\end{eqnarray}
Regarding the term ${\displaystyle \oint \limits_{I_{R}}{\dfrac{\exp(itz)}{(z^{2}+1)^{m}}}dz}$, from residue's theorem, we have
\begin{eqnarray} 
\oint \limits_{I_{R}}{\dfrac{\exp(itz)}{(z^{2}+1)^{m}}}dz=2\pi i.\underset{z=i}{\text{Res }}\left(\dfrac{\exp(itz)}{(z^{2}+1)^{m}}\right)=\dfrac{2\pi i}{(m-1)!}\lim \limits_{z \to i}{\dfrac{d^{m-1}}{dz^{m-1}}\dfrac{\exp(itz)}{(z+i)^{m}}}.\nonumber
\end{eqnarray}
By direct calculations, we obtain
\begin{eqnarray}
\lim \limits_{z \to i}{\dfrac{d^{m-1}}{dz^{m-1}}\dfrac{\exp(itz)}{(z+i)^{m}}}=\dfrac{\exp(-t)}{i}\mathop {\sum} \limits_{j=1}^{m}{\dfrac{(2m-j-1)!}{2^{2m-j}}\dbinom{m-1}{j-1}t^{j-1}}. \nonumber
\end{eqnarray}
Thus, it yields that
\begin{eqnarray}
\oint \limits_{I_{R}}{\dfrac{\exp(itz)}{(z^{2}+1)^{m}}}dz &=& \dfrac{2\pi\exp(-t)}{(m-1)!}\mathop {\sum }\limits_{j=1}^{m}{\dfrac{(2m-j-1)!}{2^{2m-j}}\dbinom{m-1}{j-1}t^{j-1}} \nonumber \\
&=& \dfrac{2\pi\exp(-t)}{2^{2m-1}}\left[\mathop {\sum }\limits_{j=1}^{m}{\dbinom{2m-1-j}{m-j}\dfrac{(2t)^{j-1}}{(j-1)!}}\right]. \nonumber
\end{eqnarray} 
Additionally, ${\displaystyle \left|\oint \limits_{\Gamma_{R}}{\dfrac{\exp(itz)}{(z^{2}+1)^{m}}}dz\right| \leq \oint \limits_{\Gamma_{R}}{\dfrac{1}{|(z^{2}+1)^{m}|}|dz|}=\dfrac{\pi R}{(R^{2}+1)^{m}} \to 0}$ as $R \to \infty$.\\
As a consequence, as $t>0$, by letting $R \to \infty$, we get:
\begin{eqnarray}
\int \limits_{-\infty}^{+\infty}{\dfrac{\exp(itx)}{(x^{2}+1)^{m}}}dx=\dfrac{2\pi\exp(-t)}{2^{2m-1}}\left[\mathop {\sum }\limits_{j=1}^{m}{\dbinom{2m-1-j}{m-j}\dfrac{(2t)^{j-1}}{(j-1)!}}\right]. \nonumber
\end{eqnarray} 
For the case $t<0$, notice that $\int \limits_{-\infty}^{\infty}{\dfrac{itx}{(x^{2}+1)^{m}}}dx=\int \limits_{-\infty}^{\infty}{\dfrac{\exp(-itx)}{(x^{2}+1)^{m}}}dx$, we achieve
\begin{eqnarray}
\int \limits_{-\infty}^{+\infty}{\dfrac{\exp(itx)}{(x^{2}+1)^{m}}}dx=\dfrac{2\pi\exp(t)}{2^{2m-1}}\left[\mathop {\sum }\limits_{j=1}^{m}{\dbinom{2m-1-j}{m-j}\dfrac{(-2t)^{j-1}}{(j-1)!}}\right]. \nonumber
\end{eqnarray}
The lemma is proved completely.
\end{proof}


\paragraph{PROOF OF THEOREM \ref{theorem:generaloverfittedGaussian} (Continue)}
We present here the proof for general $d \geq 1$. This proof is 
similar to the case $d=1$, with extra care for handling matrix-variate parameters.
For any sequence $G_{n} \in \mathcal{O}_{k,c_{0}}(\Theta \times \Omega) \to G_{0}$ in $W_{\overline{r}}$, we can denote $G_{n}=\mathop {\sum }\limits_{i=1}^{k_{0}}{\mathop {\sum }\limits_{j=1}^{s_{i}}{p_{ij}^{n}\delta_{(\theta_{ij}^{n},\Sigma_{ij}^{n})}}}$ where $(p_{ij}^{n},\theta_{ij}^{n},\Sigma_{ij}^{n}) \to (p_{i}^{0},\theta_{i}^{0},\Sigma_{i}^{0})$ for all $1 \leq i \leq k_{0}$ and $1 \leq j \leq s_{i} \leq k-k_{0}+1$. Let $N$ be any positive integer. For any $r \geq 1$ and for each $x \in \mathbb{R}$, by means of Taylor expansion up to any $N$ order, we obtain
\begin{eqnarray}
p_{G_{n}}(x)-p_{G_{0}}(x) & = & \mathop {\sum }\limits_{i=1}^{k_{0}}{\mathop {\sum }\limits_{j=1}^{s_{i}}{p_{ij}^{n}(f(x|\theta_{ij}^{n},\Sigma_{ij}^{n})-f(x|\theta_{i}^{0},\Sigma_{i}^{0}))}}+\mathop {\sum }\limits_{i=1}^{k_{0}+m}{(p_{i.}^{n}-p_{i}^{0})f(x|\theta_{i}^{0},\Sigma_{i}^{0})} \nonumber \\
&=& A_{1}(x)+\mathop {\sum }\limits_{i=1}^{k_{0}}{\mathop {\sum }\limits_{j=1}^{s_{i}}{p_{ij}^{n}\mathop {\sum }\limits_{|\alpha|=1}^{N}{(\Delta\theta_{ij}^{n})^{\alpha_{1}}(\Delta \Sigma_{ij}^{n})^{\alpha_{2}}\dfrac{D^{|\alpha|}f(x|\theta_{i}^{0},\Sigma_{i}^{0})}{\alpha!}}}}+R_{1}(x), \label{eqn:generaloverfittedGaussiangeneralmultivariateone}
\end{eqnarray}
where $p_{i.}^{n}=\mathop {\sum }\limits_{j=1}^{s_{i}}{p_{ij}^{n}}$, $A_{1}(x)=\mathop {\sum }\limits_{i=1}^{k_{0}}{(p_{i.}^{n}-p_{i}^{0})f(x|\theta_{i}^{0},\Sigma_{i}^{0})}$, $\Delta \theta_{ij}^{n}=\theta_{ij}^{n}-\theta_{i}^{0}$, $\Delta \Sigma_{ij}^{n}=\Sigma_{ij}^{n}-\Sigma_{i}^{0}$ for all $1 \leq i \leq k_{0}$, $1 \leq j \leq s_{i}$, and $R_{1}(x) \leq O(\mathop {\sum }\limits_{i=1}^{k_{0}}{\mathop {\sum }\limits_{j=1}^{s_{i}}{p_{i,j}^{(n)}(|\Delta\theta_{ij}^{n}|^{N+\delta}+|\Delta \Sigma_{ij}^{n}|^{N+\delta}}})$. Additionally, $\alpha=(\alpha_{1},\alpha_{2})$, where $\alpha_{1}=(\alpha_{1}^{1},\ldots,\alpha_{d}^{1}) \in \mathbb{N}^{d}$, $\alpha_{2}=(\alpha_{uv}^{2})_{uv} \in \mathbb{N}^{d \times d}$, $|\alpha|=\mathop {\sum }\limits_{i=1}^{d}{\alpha_{i}^{1}}+\mathop {\sum }\limits_{1 \leq u,v \leq d}{\alpha_{uv}^{2}}$, and  $\alpha!=\prod \limits_{i=1}^{d}{\alpha_{i}^{1}!}\prod \limits_{1 \leq u,v \leq d}{\alpha_{uv}^{2}!}$. Moreover, $(\Delta \theta_{ij}^{n})^{\alpha_{1}}=\prod \limits_{l=1}^{d}{(\Delta \theta_{ij}^{n})_{l}^{\alpha_{l}^{1}}}$ and $(\Delta \Sigma_{ij}^{n})^{\alpha_{2}}=\prod \limits_{1 \leq u,v \leq d}{(\Delta \Sigma_{ij}^{n})_{uv}^{\alpha_{uv}^{2}}}$ where $(.)_{l}$ denotes the $l$-th component and $(.)_{uv}$ denotes the element in $u$-th row and $v$-th column. Finally, $D^{|\alpha|}f(x|\theta_{i}^{0},\Sigma_{i}^{0})=\dfrac{\partial^{|\alpha|}{f}}{\partial{\theta}^{\alpha_{1}}\partial{\Sigma}^{\alpha_{2}}} = \dfrac{\partial^{|\alpha|}{f}}{\prod \limits_{l=1}^{d}{\partial{\theta_{l}^{\alpha_{l}^{1}}}}\prod \limits_{1 \leq u,v \leq d}{\partial{\Sigma_{uv}^{\alpha_{uv}^{2}}}}}(x|\theta_{i}^{0},\Sigma_{i}^{0})$. 

From Lemma \ref{lemma:multivariatenormaldistribution}, we have the identity $\dfrac{\partial^{2}{f}}{\partial{\theta^{2}}}(x|\theta,\Sigma)=2\dfrac{\partial{f}}{\partial{\Sigma}}(x|\theta,\Sigma)$ for all $\theta \in \mathbb{R}^{d}$ and $\Sigma \in S_{d}^{++}$. Therefore, for any $\alpha=(\alpha_{1},\alpha_{2})$, we can check that
\begin{eqnarray}
\dfrac{\partial^{|\alpha|}{f}}{\partial{\theta}^{\alpha_{1}}\partial{\Sigma}^{\alpha_{2}}}=\dfrac{1}{2^{|\alpha_{2}|}}\dfrac{\partial^{|\beta|}{f}}{\theta^{\beta}}, \label{eqn:generaloverfittedGaussiangeneralmultivariatesecond}
\end{eqnarray}
where $\beta_{l}=\alpha_{l}^{1}+\mathop {\sum }\limits_{j=1}^{d}{\alpha_{lj}^{2}}+\mathop {\sum }\limits_{j=1}^{d}{\alpha_{jl}^{2}}$ for all $1 \leq l \leq d$, which means $|\beta|=|\alpha_{1}|+2|\alpha_{2}|$. This equality means that we can convert all the derivatives involving $\Sigma$ to the derivatives only respect to $\theta$. Therefore, we can rewrite \eqref{eqn:generaloverfittedGaussiangeneralmultivariateone} as follows:
\begin{eqnarray}
p_{G_{n}}(x)-p_{G_{0}}(x) & = & \mathop {\sum }\limits_{i=1}^{k_{0}}{\mathop {\sum }\limits_{j=1}^{s_{i}}{p_{ij}^{n}{\mathop {\sum }\limits_{|\beta| \geq 1}{\dfrac{(\Delta \theta_{ij}^{n})^{\alpha_{1}}(\Delta \Sigma_{ij}^{n})^{\alpha_{2}}}{2^{|\alpha_{2}|}\alpha_{1}!\alpha_{2}!}}\dfrac{\partial^{|\beta|}{f}}{\theta^{\beta}}(x|\theta_{i}^{0},\Sigma_{i}^{0})}}} \nonumber \\
& + & A_{1}(x) + R_{1}(x) \nonumber \\
& := & A_{1}(x)+B_{1}(x)+R_{1}(x), \label{eqn:generaloverfittedGaussiangeneralmultivariatethird}
\end{eqnarray}
where $\beta$ is defined as in equation \ref{eqn:generaloverfittedGaussiangeneralmultivariatesecond}. 

Now, we proceed to proving part (a) of the theorem.
From the hypothesis for $\overline{r}$, we have non-trivial solutions $(x_{i}^{*},a_{i}^{*},b_{i}^{*})_{i=1}^{k-k_{0}+1}$ for equation \eqref{eqn:generalovefittedGaussianzero} when $r=\overline{r}-1$.  We choose the sequence of probability measures $G_{n}=\mathop {\sum }\limits_{i=1}^{k}{p_{i}^{n}\delta_{(\theta_{i}^{n},\Sigma_{i}^{n})}}$ as $(\theta_{i}^{n})_{1}=(\theta_{1}^{0})_{1}+a_{i}^{*}/n, (\theta_{i}^{n})_{j}=(\theta_{1}^{0})_{j}$ for $2 \leq j \leq d$, $(\Sigma_{i}^{n})_{11}=(\Sigma_{1}^{0})_{11}+2b_{i}^{*}/n^{2}$, $(\Sigma_{i}^{n})_{uv}=(\Sigma_{1}^{0})_{uv}$ for $(u,v) \neq (1,1)$, $p_{i}^{n}=p_{1}^{0}(x_{i}^{*})^{2}/\mathop {\sum }\limits_{j=1}^{k-k_{0}+1}{(x_{j}^{*})^{2}}$ when $1 \leq i \leq k-k_{0}+1$, and $\theta_{i}^{n}=\theta_{i-k+k_{0}}^{0}$, $\Sigma_{i}^{n}=\Sigma_{i-k+k_{0}}^{0}$, $p_{i}^{n}=p_{i-k+k_{0}}^{0}$ when $k-k_{0}+2 \leq i \leq k$. As $n$ is sufficiently large, we still guarantee that $\Sigma_{i}^{n}$ are positive definite matrices as $1 \leq i \leq k-k_{0}+1$. We can check that $W_{1}^{r}(G_{n},G_{0})= \left(\mathop {\sum }\limits_{i=1}^{k-k_{0}+1}{p_{1i}^{n}\left(\dfrac{|a_{i}^{*}|}{n}+\dfrac{|b_{i}^{*}|}{n^{2}}\right)}\right)^{r}>0$ for all $r \geq 1$. Additionally, under this construction, $s_{1}=k-k_{0}+1$, $s_{i}=1$ for all $2 \leq i \leq k_{0}$, $(\Delta \theta_{1j}^{n})_{l}=(\Delta \Sigma_{1j}^{n})_{uv}=0$ for all $1 \leq j \leq s_{1}$, $2 \leq l \leq d$ and  $(u,v) \neq (1,1)$, $\Delta \theta_{ij}^{n}=\vec{0} \in \mathbb{R}^{d}$, $\Delta \Sigma_{ij}^{n}=\vec{0} \in \mathbb{R}^{d \times d}$ for all $k-k_{0}+2 \leq i \leq k$ and $1 \leq j \leq s_{i}$. Now, by choosing $N=\overline{r}$ in \ref{eqn:generaloverfittedGaussiangeneralsecond}, we obtain $A_{1}(x)=0$ and $\mathop {\sup }\limits_{x \in \mathbb{R}^{d}}{|R_{1}(x)|}/W_{r}^{r}(G_{n},G_{0}) \to 0$. Moreover, we can rewrite $B_{1}(x)$ in \eqref{eqn:generaloverfittedGaussiangeneralmultivariatethird} as follows
\begin{eqnarray}
B_{1}(x) &=&  \mathop {\sum }\limits_{i=1}^{k-k_{0}+1}{p_{1i}^{n}\mathop {\sum }\limits_{\gamma=1}^{\overline{r}-1}{\mathop {\sum }\limits_{\alpha_{1}^{1}, \alpha_{11}^{2}}{\dfrac{(\Delta \theta_{1i}^{n})^{\alpha_{1}^{1}}(\Delta \Sigma_{1i}^{n})^{\alpha_{11}^{2}}}{2^{\alpha_{11}^{2}}\alpha_{1}!\alpha_{11}^{2}!}}\dfrac{\partial^{\gamma}{f}}{\partial{\theta_{1}^{\gamma}}}(x|\theta_{1}^{0},\Sigma_{1}^{0})}} \nonumber \\
&+& \mathop {\sum }\limits_{i=1}^{k-k_{0}+1}{p_{1i}^{n}\mathop {\sum }\limits_{\gamma \geq \overline{r}}{\mathop {\sum }\limits_{\alpha_{1}^{1}, \alpha_{11}^{2}}{\dfrac{(\Delta \theta_{1i}^{n})^{\alpha_{1}^{1}}(\Delta \Sigma_{1i}^{n})^{\alpha_{11}^{2}}}{2^{\alpha_{11}^{2}}\alpha_{1}^{1}!\alpha_{11}^{2}!}}\dfrac{\partial^{\alpha}{f}}{\partial{\theta_{1}^{\gamma}}}(x|\theta_{1}^{0},\Sigma_{1}^{0})}} \nonumber \\
&:= & \mathop {\sum }\limits_{\gamma=1}^{\overline{r}-1}{B_{\gamma,n}\dfrac{\partial^{\gamma}{f}}{\partial{\theta_{1}}^{\gamma}}(x|\theta_{1}^{0},\Sigma_{1}^{0})}+\mathop {\sum }\limits_{\gamma \geq \overline{r}}{C_{\gamma,n}\dfrac{\partial^{\gamma}{f}}{\partial{\theta_{1}}^{\gamma}}(x|\theta_{1}^{0},\Sigma_{1}^{0})}. \nonumber
\end{eqnarray}
where $\gamma= \alpha_{1}^{1}+2\alpha_{11}^{2}$. From the formation of $G_{n}$, for each $1 \leq \gamma \leq \overline{r}-1$,
\begin{eqnarray}
B_{\gamma,n} =\dfrac{1}{Cn^{\gamma}} \mathop {\sum }\limits_{i=1}^{k-k_{0}+1}{(x_{i}^{*})^{2} \mathop {\sum }\limits_{\alpha_{1}^{1}+2\alpha_{11}^{2}=\alpha}{\dfrac{(a_{i}^{*})^{\alpha_{1}^{1}}(b_{i}^{*})^{\alpha_{11}^{2}}}{\alpha_{1}^{1}!\alpha_{11}^{2}!}}} = 0, \nonumber
\end{eqnarray}
where $C=\mathop {\sum }\limits_{i=1}^{k-k_{0}+1}{(x_{i}^{*})^{2}}$. As a consequence, $B_{\gamma,n}/W_{r}^{r}(G_{n},G_{0})=0$ for all $1 \leq \gamma \leq \overline{r}-1$. Similarly, for each $\gamma \geq \overline{r}$,
\begin{eqnarray}
C_{\gamma,n}/W_{r}^{r}(G_{n},G_{0})=An^{2r-\gamma}/\left(\mathop {\sum }\limits_{i=1}^{k-k_{0}+1}{p_{1i}^{0}(n|a_{i}^{*}|+|b_{i}^{*}|)}\right)^{r}\to 0, \nonumber
\end{eqnarray}
where $A=\mathop {\sum }\limits_{\substack{\alpha_{1}^{1}+2\alpha_{11}^{2}=\gamma \\ \alpha_{1}^{1}+\alpha_{11}^{2} \leq \overline{r}-1}}{\dfrac{(a_{i}^{*})^{\alpha_{1}^{1}}(b_{i}^{*})^{\alpha_{11}^{2}}}{\alpha_{1}^{1}!\alpha_{11}^{2}!}}$ and the last result is due to $r < \overline{r}$. From now, it is straightforward to extend this argument to address the Hellinger distance of
mixture densities in the same way as the proof for the case $d=1$.\\

\noindent
We now turn to part (b). It suffices to show that \eqref{eqn:generaloverfittedGaussiansecond} holds.
Assume by contrary that it does not hold. Follow the same argument as that of Theorem \ref{theorem-secondorder}, we can find a sequence $G_{n}=\mathop {\sum }\limits_{i=1}^{k_{0}}{\mathop {\sum }\limits_{j=1}^{s_{i}}{p_{ij}^{n}\delta_{(\theta_{ij}^{n},\Sigma_{ij}^{n})}}} \in \mathcal{O}_{k,c_{0}}(\Theta \times \Omega) \to G_{0}$ in $W_{\overline{r}}$ as $n \to \infty$ and $G_{n}$ have exactly $k^{*}$ support points where $k_{0} \leq k^{*} \leq k$. Additionally, $(p_{ij}^{n},\theta_{ij}^{n},\Sigma_{ij}^{n}) \to (p_{i}^{0},\theta_{i}^{0},\Sigma_{i}^{0})$ for all $1 \leq i \leq k_{0}$ and $1 \leq j \leq s_{i} \leq k-k_{0}+1$. Denote
\begin{eqnarray}
d(G_{n},G_{0})=\mathop {\sum }\limits_{i=1}^{k_{0}}{\mathop {\sum }\limits_{j=1}^{s_{i}}{p_{ij}^{n}(|\Delta \theta_{ij}^{n}|^{\overline{r}}+|\Delta \Sigma_{ij}^{n}|^{\overline{r}})}}+\mathop {\sum }\limits_{i=1}^{k_{0}}{|p_{i.}^{n}-p_{i}^{0}|}, \nonumber
\end{eqnarray}
As we point out in the proof of Theorem \ref{theorem-secondorder}, the assumption $(p_{G_{n}}(x)-p_{G_{0}}(x))/W_{\overline{r}}^{\overline{r}}(G_{n},G_{0}) \to 0$ for all $x \in \mathbb{R}$ leads to $(p_{G_{n}}(x)-p_{G_{0}}(x))/d(G_{n},G_{0}) \to 0$ for all $x \in \mathbb{R}$. Now, by combining this fact with \eqref{eqn:generaloverfittedGaussiangeneralsecond} and choosing $N=\overline{r}$, we obtain
\begin{eqnarray}
(A_{1}(x)+B_{1}(x)+R_{1}(x))/d(G_{n},G_{0}) \to 0. \label{eqn:generaloverfittedGaussiangeneralmultivariatefour}
\end{eqnarray}
Now, $A_{1}(x)/d(G_{n},G_{0})$, $B_{1}/d(G_{n},G_{0})$ are just the linear combination of elements of $\dfrac{\partial^{|\beta|}{f}}{\partial{\theta}^{\beta}}(x|\theta,\Sigma)$ where $\beta$ is defined in equation \eqref{eqn:generaloverfittedGaussiangeneralmultivariatesecond}, i.e  $\beta_{l}=\alpha_{l}^{1}+\mathop {\sum }\limits_{j=1}^{d}{\alpha_{lj}^{2}}+\mathop {\sum }\limits_{j=1}^{d}{\alpha_{jl}^{2}}$ for all $1 \leq l \leq d$, $|\beta|=|\alpha_{1}|+2|\alpha_{2}|$, and $|\alpha_{1}|+|\alpha_{2}| \leq \overline{r}$. Therefore, it implies that $0 \leq |\beta| \leq 2\overline{r}$, which is the range of all possible values of $|\beta|$. Denote $E_{\beta}(\theta,\Sigma)$ to be the corresponding coefficient of $\dfrac{\partial^{|\beta|}{f}}{\partial{\theta}^{\beta}}(x|\theta,\Sigma)$. Assume that $E_{\beta}(\theta_{i}^{0},\Sigma_{i}^{0}) \to 0$ for all $1 \leq i \leq k_{0}$ and $0 \leq |\beta| \leq 2\overline{r}$ as $n \to \infty$. Using the result from \eqref{eqn:generaloverfittedGaussiangeneralsecond}, the specific formula for $E_{\beta}(\theta_{i}^{0},\Sigma_{i}^{0})$ as $|\beta| \geq 1$ is
\begin{eqnarray}
E_{\beta}(\theta_{i}^{0},\Sigma_{i}^{0})=\left[\mathop {\sum }\limits_{j=1}^{s_{i}}{p_{ij}^{n}\mathop {\sum }\limits_{\alpha_{1},\alpha_{2}}{\dfrac{(\Delta \theta_{ij}^{n})^{\alpha_{1}}(\Delta \Sigma_{ij}^{n})^{\alpha_{2}}}{2^{|\alpha_{2}|}\alpha_{1}!\alpha_{2}!}}}\right]/d(G_{n},G_{0}). \nonumber
\end{eqnarray}
where $\alpha_{1},\alpha_{2}$ satisfies $\alpha_{l}^{1}+\mathop {\sum }\limits_{j=1}^{d}{\alpha_{lj}^{2}}+\mathop {\sum }\limits_{j=1}^{d}{\alpha_{jl}^{2}}=\beta_{l}$ for all $1 \leq l \leq d$.

By taking the summation of all $|E_{0}(\theta_{i}^{0},\Sigma_{i}^{0})|$, i.e $\beta=\vec{0}$, we get 
$\mathop {\sum }\limits_{i=1}^{k_{0}}{|p_{i.}^{n}-p_{i}^{0}|}/d(G_{n},G_{0}) \to 0  \ \ \text{as } \ n \to \infty$. As a consequence, we get
\begin{eqnarray}
\mathop {\sum }\limits_{i=1}^{k_{0}}{\mathop {\sum }\limits_{j=1}^{s_{i}}{p_{ij}^{n}(||\Delta \theta_{ij}^{n}||^{\overline{r}}+||\Delta \Sigma_{ij}^{n}||^{\overline{r}})}}/d(G_{n},G_{0}) \to 1 \ \ \text{as } \ n \to \infty. \nonumber
\end{eqnarray}
As $||.||$ and $||.||_{\overline{r}}$ are equivalent, the above result also implies that
\begin{eqnarray}
\mathop {\sum }\limits_{i=1}^{k_{0}}{\mathop {\sum }\limits_{j=1}^{s_{i}}{p_{ij}^{n}(||\Delta \theta_{ij}^{n}||_{\overline{r}}^{\overline{r}}+||\Delta \Sigma_{ij}^{n}||_{\overline{r}}^{\overline{r}})}}/d(G_{n},G_{0}) \not \to 0 \ \ \text{as } \ n \to \infty. \nonumber
\end{eqnarray}
Therefore, we can find an index $1 \leq i^{*} \leq d$ such that
\begin{eqnarray}
\mathop {\sum }\limits_{j=1}^{s_{i^{*}}}{p_{i^{*}j}^{n}(||\Delta \theta_{i^{*}j}^{n}||_{\overline{r}}^{\overline{r}}+||\Delta \Sigma_{i^{*}j}^{n}||_{\overline{r}}^{\overline{r}})}/d(G_{n},G_{0}) \not \to 0.
\end{eqnarray}
Without loss of generality, we assume $i^{*}=1$. There are two cases regarding the above result:
\paragraph{Case 1:} There exists $1 \leq u^{*} \leq d$ and  such that $U_{n}=\mathop {\sum }\limits_{j=1}^{s_{1}}{p_{1j}^{n}(|(\Delta \theta_{1j})_{u^{*}}|^{\overline{r}}+|(\Delta \Sigma_{1j}^{n})_{u^{*}u^{*}}|^{\overline{r}})}/d(G_{n},G_{0}) \not \to 0$. Without loss of generality, we assume $u^{*}=1$. With this result, for any $|\beta| \geq 1$, we obtain
\begin{eqnarray}
F_{\beta}(\theta_{1}^{0},\Sigma_{1}^{0})=\dfrac{E_{\beta}(\theta_{1}^{0},\Sigma_{1}^{0})}{U_{n}}=\dfrac{\mathop {\sum }\limits_{j=1}^{s_{i}}{p_{1j}^{n}\mathop {\sum }\limits_{\alpha_{1},\alpha_{2}}{\dfrac{(\Delta \theta_{1j}^{n})^{\alpha_{1}}(\Delta \Sigma_{1j}^{n})^{\alpha_{2}}}{2^{|\alpha_{2}|}\alpha_{1}!\alpha_{2}!}}}}{\mathop {\sum }\limits_{j=1}^{s_{1}}{p_{1j}(|(\Delta \theta_{1j})_{1}|^{\overline{r}}+|(\Delta \Sigma_{1j}^{n})_{11}|^{\overline{r}})}} \to 0. \nonumber
\end{eqnarray}
Now, we choose $\alpha_{l}^{1}=0$ for all $2 \leq l \leq d$ and $\alpha_{uv}^{2}=0$ for all $(u,v) \neq (1,1)$, then $|\beta|=\alpha_{1}^{1}+2\alpha_{11}^{2}$. Therefore,
\begin{eqnarray}
H_{|\beta|}(\theta_{1}^{0},\Sigma_{1}^{0})=\dfrac{\mathop {\sum }\limits_{j=1}^{s_{i}}{p_{1j}^{n}\mathop {\sum }\limits_{\alpha_{1}^{1},\alpha_{11}^{2}}{\dfrac{(\Delta \theta_{1j}^{n})_{1}^{\alpha_{1}^{1}}(\Delta \Sigma_{1j}^{n})_{11}^{\alpha_{11}^{2}}}{2^{\alpha_{11}^{2}}\alpha_{1}^{1}!\alpha_{11}^{2}!}}}}{\mathop {\sum }\limits_{j=1}^{s_{1}}{p_{1j}(|(\Delta \theta_{1j})_{1}|^{\overline{r}}+|(\Delta \Sigma_{1j}^{n})_{11}|^{\overline{r}})}} \to 0,
\end{eqnarray}
where $\alpha_{1}^{1}+2\alpha_{11}^{2}=|\beta|$ and $1 \leq |\beta| \leq 2\overline{r}$.\\\\
Denote $\overline{p}_{n}=\mathop {\max }\limits_{1 \leq j \leq s_{1}}{\left\{p_{1j}^{n}\right\}}$, $\overline{M}_{n}=\mathop {\max }{\left\{|(\Delta \theta_{11}^{n})_{1}|,\ldots,|(\Delta \theta_{1s_{1}}^{n})_{1})|,|(\Delta \Sigma_{11}^{n})_{11}|^{1/2},\ldots,|(\Delta \Sigma_{1s_{1}}^{n})_{11}|^{1/2}\right\}}$. Since $0 < p_{1j}^{n}/\overline{p}_{n} \leq 1$ for all $1 \leq j \leq s_{1}$, we define $\mathop {\lim }\limits_{n \to \infty}{p_{1j}^{n}/\overline{p}_{n}} = c_{j}^{2}$ for all $1 \leq j \leq s_{i^{*}}$. Similarly, define $\mathop {\lim }\limits_{n \to \infty}{(\Delta \theta_{1j}^{n})_{1}/\overline{M}_{n}}=a_{j}$ and $\mathop {\lim }\limits_{n \to \infty}{(\Delta \Sigma_{1j}^{n})_{11}/\overline{M}_{n}^{2}}=2b_{j}$ for all $1 \leq j \leq s_{1}$.Since $p_{1j}^{n} \geq c_{0}$ for all $1 \leq j \leq s_{1}$, all of $x_{j}^{2}$ differ from 0 and at least one of them equals to 1. Likewise, at least one element of $\left(a_{j},b_{j}\right)_{j=1}^{s_{1}}$ equal to -1 or 1. Now, for $1 \leq |\beta| \leq \overline{r}$, divide both the numerator and denominator of $H_{|\beta|}(\theta_{1}^{0},\Sigma_{1}^{0})$ by $\overline{M}_{n}^{|\beta|}$ and let $n \to \infty$, we obtain the following system of polynomial equations
\begin{eqnarray}
\mathop {\sum }\limits_{j=1}^{s_{1}}{\mathop {\sum }\limits_{\alpha_{1}^{1}+2\alpha_{11}^{2} = |\beta|}{\dfrac{c_{j}^{2}a_{j}^{\alpha_{1}^{1}}b_{j}^{\alpha_{11}^{2}}}{\alpha_{1}^{1}!\alpha_{11}^{2}!}}}=0 \ \ \text{for all } \ 1 \leq |\beta| \leq \overline{r}. \nonumber
\end{eqnarray}
As $2 \leq s_{1} \leq k-k_{0}+1$, the hardest scenario is when $s_{1}=k-k_{0}+1$. However, from the hypothesis, as $s_{1}=k-k_{0}+1$, the above system of polynomial equations does not have non-trivial solution, which is a contradiction. 

\paragraph{Case 2:} There exists $1 \leq u^{*} \neq v^{*} \leq d$ such that $V_{n}=\mathop {\sum }\limits_{j=1}^{s_{1}}{p_{1j}^{n}|(\Delta \Sigma_{1j}^{n})_{u^{*}v^{*}}|^{\overline{r}}}/d(G_{n},G_{0}) \not \to 0$. Without loss of generality, we assume $u^{*}=1, v^{*}=2$. With this result, for any $|\beta| \geq 1$, we obtain
\begin{eqnarray}
F_{\beta}'(\theta_{1}^{0},\Sigma_{1}^{0})=\dfrac{E_{\beta}(\theta_{i}^{0},\Sigma_{i}^{0})}{V_{n}}=\dfrac{\mathop {\sum }\limits_{j=1}^{s_{i}}{p_{1j}^{n}\mathop {\sum }\limits_{\alpha_{1},\alpha_{2}}{\dfrac{(\Delta \theta_{1j}^{n})^{\alpha_{1}}(\Delta \Sigma_{1j}^{n})^{\alpha_{2}}}{2^{|\alpha_{2}|}\alpha_{1}!\alpha_{2}!}}}}{\mathop {\sum }\limits_{j=1}^{s_{1}}{p_{1j}^{n}|(\Delta \Sigma_{1j}^{n})_{12}|^{\overline{r}}}} \to 0. \nonumber
\end{eqnarray}
By choosing $\alpha_{1}=\vec{0} \in \mathbb{N}^{d}$, $\alpha_{uv}^{2}=0$ for all $(u,v) \not \in \left\{(1,2),(2,1)\right\}$, then $|\beta|=\alpha_{12}^{2}+\alpha_{21}^{2}$. Therefore,
\begin{eqnarray}
H_{|\beta|}'(\theta_{1}^{0},\Sigma_{1}^{0})=\dfrac{\mathop {\sum }\limits_{j=1}^{s_{1}}{p_{1j}^{n}\mathop {\sum }\limits_{\alpha_{12}^{2},\alpha_{21}^{2}}{\dfrac{(\Delta \Sigma_{1j}^{n})_{12}^{\alpha_{12}^{2}+\alpha_{21}^{2}}}{2^{|\beta|}\alpha_{12}^{2}!\alpha_{21}^{2}!}}}}{\mathop {\sum }\limits_{j=1}^{s_{1}}{p_{1j}^{n}|(\Delta \Sigma_{1j}^{n})_{12}|^{\overline{r}}}} \to 0. \nonumber 
\end{eqnarray}
Denote $\overline{p}_{n}'=\mathop {\max }\limits_{1 \leq j \leq s_{1}}{\left\{p_{1j}^{n}\right\}}$, $\overline{M}_{n}'=\mathop {\max }\limits_{1 \leq j \leq s_{1}}{\left\{|(\Delta \Sigma_{1j}^{n})_{12}|\right\}}$. Then, we have $p_{1j}^{n}/\overline{p}_{n}' \to (c_{j}^{'})^{2} > 0$ and $(\Delta \Sigma_{1j}^{n})_{12}/\overline{M}_{n}'=d_{j}$ for all $1 \leq j \leq s_{1}$. Again, we have at least one of $d_{j}$ differs from 0. Now, by dividing both the numerator and denominator of $H_{2}'(\theta_{1}^{0},\Sigma_{1}^{0})$ by $(\overline{M}_{n}')^{2}$ and letting $n \to \infty$, we obtain $\mathop {\sum }\limits_{j=1}^{s_{1}}{(c_{j}')^{2}d_{j}^{2}}=0$. This equation implies $d_{j}=0$ for all $1 \leq j \leq d$, which is a contradiction. 

Therefore, at least one of the coefficients $E_{\beta}(\theta_{i}^{0},\Sigma_{i}^{0})$ does not converge to 0 as $n \to \infty$. Now, we denote $m_{n}$ to the maximum of the absolute values of $E_{\beta}(\theta_{i}^{0},\Sigma_{i}^{0})$ where $\beta$ is defined as in equation \eqref{eqn:generaloverfittedGaussiangeneralmultivariatesecond}, $1 \leq i \leq k_{0}$ and let $d_{n}=1/m_{n}$. As $m_{n} \not \to 0$ as $n \to \infty$, $d_{n}$ is uniformly bounded above for all $n$. As $d_{n}|E_{\beta}(\theta_{i}^{0},\sigma_{i}^{0})| \leq 1$, we denote $d_{n}E_{\beta}(\theta_{i}^{0},\Sigma_{i}^{0}) \to \tau_{i\beta}$ where at least one of $\tau_{i\beta}$ differs from 0. Combining these notations with \eqref{eqn:generaloverfittedGaussiangeneralmultivariatefour} we get that for all $x \in \mathbb{R}^{d}$,
\begin{eqnarray}
\dfrac{p_{G_{n}}(x)-p_{G_{0}}(x)}{d(G_{n},G_{0})} \to \mathop {\sum }\limits_{i=1}^{k_{0}}{\mathop {\sum }\limits_{\beta}{\tau_{i\beta}\dfrac{\partial^{|\beta|}{f}}{\partial{\theta}^{\beta}}(x|\theta_{i}^{0},\Sigma_{i}^{0})}} = 0. \nonumber 
\end{eqnarray}
Using the technique we have in the proof of part (a) of Theorem \ref{identifiability-multivariatecharacterization}, it is sufficient to demonstrate the above equation as $d=1$. However, from the result when $d=1$, we have already known that $\tau_{i\beta}=0$ for all $1 \leq i \leq k_{0}$, $0 \leq |\beta| \leq 2\overline{r}$, which is a contradiction. Therefore, the assertion of our theorem follows immediately.

\paragraph{PROOF OF PROPOSITION~\ref{theorem:Gaussianoverfittedbytwo} (Continue)}
The case $k-k_0 = 1$ was shown in Appendix I. Here we consider the 
case $k-k_{0}=2$. As in the argument of case when $k-k_{0}=1$, 
we can find $i^{*} \in \left\{1,2,\ldots,k_{0}+m\right\}$ where $0 \leq m \leq 2$ such that
\begin{eqnarray}
F_{\alpha}'(\theta_{i^{*}}^{0},v_{i^{*}}^{0}) & = & \dfrac{\mathop {\sum }\limits_{j=1}^{s_{i^{*}}}{p_{i^{*}j}^{n}(|\Delta \theta_{i^{*}j}^{n}|^{6}+|\Delta v_{i^{*}j}^{n}|^{6})}}{\mathop {\sum }\limits_{j=1}^{s_{i^{*}}}{p_{i^{*}j}^{n}|\Delta \theta_{i^{*}j}^{n}|^{6}}}F_{\alpha}(\theta_{i^{*}}^{0},v_{i^{*}}^{0}) \nonumber \\
& = & \dfrac{\mathop {\sum }\limits_{j=1}^{s_{i^{*}}}{p_{i^{*}j}^{n}\mathop {\sum }\limits_{n_{1},n_{2}}{\dfrac{(\Delta \theta_{i^{*}j}^{n})^{n_{1}}(\Delta v_{i^*j}^{n})^{n_{2}}}{n_{1}!n_{2}!}}}}{\mathop {\sum }\limits_{j=1}^{s_{i^{*}}}{p_{i^{*}j}^{n}|\Delta \theta_{i^{*}j}^{n}|^{6}}} \to 0, \label{eqn:overfittedGaussianbythird}
\end{eqnarray}
where $n_{1}+2n_{2}=\alpha$ and $1 \leq \alpha \leq 6$. As $i^{*} \in \left\{1,2,\ldots,k_{0}+m\right\}$, we have $i^{*} \in \left\{1,\ldots,k_{0}\right\}$ or $i^{*} \in \left\{k_{0}+1,\ldots,k_{0}+m\right\}$. Firstly, we assume that $i^{*} \in \left\{1,\ldots,k_{0}\right\}$. Without loss of generality, let $i^{*}=1$. 
Since $s_{1} \leq k-k_0+1 = 3$, there are two possibilities.

\paragraph{Case 1.} If $s_{1}\leq 2$, then since $\mathop {\sum }\limits_{j=1}^{s_{1}}{p_{1j}^{n}|\Delta \theta_{1j}^{n}|^{6}} \lesssim \mathop {\sum }\limits_{j=1}^{s_{1}}{p_{1j}^{n}|\Delta \theta_{1j}^{n}|^{4}}$, we also obtain
\begin{eqnarray}
\mathop {\sum }\limits_{j=1}^{s_{i^{*}}}{p_{i^{*}j}^{n}\mathop {\sum }\limits_{n_{1},n_{2}}{\dfrac{(\Delta \theta_{i^{*}j}^{n})^{n_{1}}(\Delta v_{i^*j}^{n})^{n_{2}}}{n_{1}!n_{2}!}}}/\mathop {\sum }\limits_{j=1}^{s_{i^{*}}}{p_{i^{*}j}^{n}|\Delta \theta_{i^{*}j}^{n}|^{4}} \to 0, \nonumber
\end{eqnarray}
which we easily get the contradiction by means of the argument of Case $k-k_{0}=1$.

\paragraph{Case 2.} If $s_{1}=3$, we assume WLOG
that $p_{11}^{n}|\Delta \theta_{11}^{n}| 
\leq p_{12}^{n}|\Delta \theta_{12}^{n}| \leq p_{13}^{n}|\Delta \theta_{13}^{n}|$ for all $n$. With the same argument as that of Case $k-k_{0}=1$, we can get $\Delta \theta_{11}^{n},\Delta \theta_{12}^{n}, \Delta \theta_{13}^{n} \neq 0$ for all $n$. Denote $a_{1} := p_{11}^{n}\Delta \theta_{11}^{n}/p_{13}^{n}\Delta_{13}^{n}  \in [-1,1]$, $a_{2} := p_{12}^{n}\Delta \theta_{12}^{n}/p_{13}^{n} \Delta_{13}^{n} \in [-1,1]$. By dividing both the numerator and denominator of $F_{1}'(\theta_{1}^{0},v_{1}^{0})$ by $p_{13}^{n}\Delta \theta_{13}^{n}$ and letting $n \to \infty$, we obtain $a_{1}+a_{2}=-1$. We have the following cases regarding $p_{11}^{n}/p_{13}^{n}, p_{12}^{n}/p_{13}^{n}$:
\noindent
\paragraph{Case 2.1:} If both $p_{11}^{n}/p_{13}^{n}, p_{12}^{n}/p_{13}^{n} \to \infty$ then $\Delta \theta_{11}^{n}/\Delta \theta_{13}^{n}$, $\Delta \theta_{12}^{n}/\Delta \theta_{13}^{n} \to 0$. Since $\Delta \theta_{11}^{n}, \Delta \theta_{12}^{n}, \Delta \theta_{13}^{n} \neq 0$, we denote $\Delta v_{1i}^{n}=h_{i}^{n}(\Delta \theta_{1i}^{n})^{2}$ for all $1 \leq i \leq 3$. By dividing the numerator and denominator of $F_{i}'(\theta_{1}^{0},v_{1}^{0})$ by $p_{13}^{n}(\Delta \theta_{13}^{n})^{i}$ for all $2 \leq i \leq 6$, we obtain
\begin{eqnarray}
K_{n,1} & = & \dfrac{1}{2}+h_{3}^{n}+\mathop {\sum }\limits_{i=1}^{2}{h_{i}^{n}\dfrac{p_{1i}^{n}(\Delta \theta_{1i}^{n})^{2}}{p_{13}^{n}(\Delta \theta_{13}^{n})^{2}}} \to 0, \nonumber \\
K_{n,2} & = & \dfrac{1}{3!}+h_{3}^{n}+\mathop {\sum }\limits_{i=1}^{2}{\left(\dfrac{1}{3!}+h_{i}^{n}\right)\dfrac{p_{1i}^{n}(\Delta \theta_{1i}^{n})^{3}}{p_{13}^{n}(\Delta \theta_{13}^{n})^{3}}} \to 0, \nonumber \\
K_{n,3} & = & \dfrac{1}{4!}+\dfrac{h_{3}^{n}}{2}+\dfrac{(h_{3}^{n})^{2}}{2}+\mathop {\sum }\limits_{i=1}^{2}{\left(\dfrac{1}{4!}+\dfrac{h_{i}^{n}}{2}+\dfrac{(h_{i}^{n})^{2}}{2}\right)\dfrac{p_{1i}^{n}(\Delta \theta_{1i}^{n})^{4}}{p_{13}^{n}(\Delta \theta_{13}^{n})^{4}}} \to 0, \nonumber \\
K_{n,4} & = & \dfrac{1}{5!}+\dfrac{h_{3}^{n}}{6}+\dfrac{(h_{3}^{n})^{2}}{2}+\mathop {\sum }\limits_{i=1}^{2}{\left(\dfrac{1}{5!}+\dfrac{h_{i}^{n}}{6}+\dfrac{(h_{i}^{n})^{2}}{2}\right)\dfrac{p_{1i}^{n}(\Delta \theta_{1i}^{n})^{5}}{p_{13}^{n}(\Delta \theta_{13}^{n})^{5}}} \to 0, \nonumber \\
K_{n,5} & = & \dfrac{1}{6!}+\dfrac{h_{3}^{n}}{4!}+\dfrac{(h_{3}^{n})^{2}}{4}+\dfrac{(h_{3}^{n})^{3}}{6}+\mathop {\sum }\limits_{i=1}^{2}{\left(\dfrac{1}{6!}+\dfrac{h_{i}^{n}}{4!}+\dfrac{(h_{i}^{n})^{2}}{4}+\dfrac{(h_{i}^{n})^{3}}{6}\right)\dfrac{p_{1i}^{n}(\Delta \theta_{1i}^{n})^{6}}{p_{13}^{n}(\Delta \theta_{13}^{n})^{6}}} \to 0. \nonumber
\end{eqnarray}
If $|h_{1}^{n}|,|h_{2}^{n}|$,$|h_{3}^{n}| \to \infty$ then $K_{n,3}>1/4!$ as $n$ is sufficiently large, which is a contradiction. Therefore, at least one of them is finite. If either $|h_{1}^{n}|$ or $h_{2}^{n}| \not \to \infty$, then we reduce to the case when $s_{1}=2$, which eventually leads to a contradiction. Therefore, $|h_{1}^{n}|,|h_{2}^{n}| \to \infty$ and $|h_{3}^{n}| \not \to \infty$. Now, $K_{n,3}$ implies that $(h_{i}^{n})^{2}\dfrac{p_{1i}^{n}(\Delta \theta_{1i}^{n})^{4}}{p_{13}^{n}(\Delta \theta_{13}^{n})^{4}} \not \to \infty$ for all $1 \leq i \leq 2$. As $p_{1i}^{n}/p_{13}^{n} \to \infty$ for all $1 \leq i \leq 2$, we obtain $h_{i}^{n}\dfrac{(\Delta \theta_{1i}^{n})^{2}}{(\Delta \theta_{13}^{n})^{2}} \to 0$. Combining these results with $K_{n,4}$ and $K_{n,5}$, we obtain $\dfrac{1}{5!}+\dfrac{h_{3}^{n}}{6}+\dfrac{(h_{3}^{n})^{2}}{2} \to 0$ and $\dfrac{1}{6!}+\dfrac{h_{3}^{n}}{4!}+\dfrac{(h_{3}^{n})^{2}}{4}+ \dfrac{(h_{3}^{n})^{3}}{6} \to 0$, which cannot happen. As a consequence, both $p_{11}^{n}/p_{13}^{n}$ and $p_{12}^{n}/p_{13}^{n} \to \infty$ cannot hold.

\noindent
\paragraph{Case 2.2:} Exactly one of $p_{11}^{n}/p_{13}^{n}, p_{12}^{n}/p_{13}^{n} \to \infty$. If  $p_{11}^{n}/p_{13}^{n} \to \infty$ and $p_{12}^{n}/p_{13}^{n} \not \to \infty$. It implies that $\Delta \theta_{11}^{n}/\Delta \theta_{13}^{n}\to 0$. Denote $p_{12}^{n}/p_{13}^{n} \to c$. If $c>0$ then as $p_{12}^{n}\Delta \theta_{12}^{n}/p_{13}^{n}\Delta \theta_{13}^{n} \to a_{2}$, $\Delta \theta_{12}^{n}/\Delta \theta_{13}^{n} \to a_{2}/c$. From the previous case 3.1, we know that at least one of $|h_{1}^{n}|,|h_{2}^{n}|,h_{3}^{n}|$ will not converge to $\infty$. If $|h_{1}^{n}| \not \to \infty$, then $K_{n,3}$ implies that 
\begin{eqnarray}
\dfrac{1}{4!}+\dfrac{h_{3}^{n}}{2}+\dfrac{(h_{3}^{n})^{2}}{2}+\left(\dfrac{h_{2}^{n}}{2}+\dfrac{(h_{2}^{n})^{2}}{2}\right)\dfrac{p_{12}^{n}(\Delta \theta_{12}^{n})^{4}}{p_{13}^{n}(\Delta \theta_{13}^{n})^{4}} \to 0, \nonumber
\end{eqnarray}
which means that at least one of $|h_{2}^{n}|, |h_{3}^{n}| \not \to \infty$. As $\left|\dfrac{p_{12}^{n}(\Delta \theta_{1i}^{n})^{j}}{p_{13}^{n}(\Delta\theta_{13}^{n})^{j}}\right|\not \to \infty$ for all $1 \leq j \leq 6$, we have both $|h_{2}^{n}|,|h_{3}^{n}| \not \to \infty$. Denote $h_{2}^{n} \to h_{2}$ and $h_{3}^{n} \to h_{3}$. Now, $K_{n,1},K_{n,2}$,$K_{n,3}$, and $K_{n,4}$ yield the following system of polynomial equations
\begin{eqnarray}
\dfrac{1}{2}+h_{3}+\left(\dfrac{1}{2}+h_{2}\right)\dfrac{a_{2}^{2}}{c}=0, \nonumber \\
\dfrac{1}{3!}+h_{3}+\left(\dfrac{1}{3!}+h_{2}\right)\dfrac{a_{2}^{3}}{c^2}=0, \nonumber \\
\dfrac{1}{4!}+\dfrac{h_{3}}{2}+\dfrac{h_{3}^{2}}{2}+\left(\dfrac{1}{4!}+\dfrac{h_{2}}{2}+\dfrac{h_{2}^{2}}{2}\right)\dfrac{a_{2}^4}{c^3}=0, \nonumber \\
\dfrac{1}{5!}+\dfrac{h_{3}}{6}+\dfrac{h_{3}^{2}}{2}+\left(\dfrac{1}{5!}+\dfrac{h_{3}}{6}+\dfrac{h_{3}^{2}}{2}\right)\dfrac{a_{2}^{5}}{c^{4}}=0. \nonumber
\end{eqnarray}
By converting the above equations into polynomial equations and using Groebner bases, we obtain that the bases contains an equation in terms of $c$ with all positive coefficient,which does not admit any solution since $c>0$. Therefore, the above system of polynomial equations does not admit any real solutions $(h_{2},h_{3},c,a_{2})$ where $c>0$. Therefore, the assumption $|h_{1}^{n}| \not \to \infty$ does not hold. As a consequence, $|h_{1}^{n}| \to \infty$. 

Now, if $|h_{2}^{n}| \not \to \infty$ then $K_{n,3}$ demonstrates that $|h_{3}^{n}| \not \to \infty$. Hence, $K_{n,1}$ yields $\left|h_{1}^{n}\dfrac{p_{11}^{n}(\Delta \theta_{11}^{n})^{2}}{p_{13}^{n}(\Delta \theta_{13}^{n})^{2}}\right| \not \to \infty$. As $\Delta \theta_{11}^{n}/\Delta \theta_{13}^{n} \to 0$ and $p_{11}^{n}/p_{13}^{n} \to \infty$, we achieve $h_{1}^{n}(\Delta \theta_{11}^{n})^{i}/p_{13}^{n}(\Delta \theta_{13}^{n})^{i} \to 0$ for all $3 \leq i \leq 6$, $(h_{1}^{n})^{2}p_{11}^{n}(\Delta \theta_{11}^{n})^{i}/p_{13}^{n}(\Delta \theta_{13}^{n})^{i} \to 0$ for all $4 \leq i \leq 6$, and $(h_{1}^{n})^{3}p_{11}^{n}(\Delta \theta_{11}^{n})^{6}/p_{13}^{n}(\Delta \theta_{13}^{n})^{6} \to 0$. With these results, by denoting $h_{2}^{n} \to h_{2}$ and $h_{3}^{n} \to h_{3}$, $K_{n,3},K_{n,4},K_{n,5},K_{n,6}$ yield the following system of polynomial equations
\begin{eqnarray}
\dfrac{1}{2}+h_{3}+\left(\dfrac{1}{2}+h_{2}\right)\dfrac{a_{2}^{2}}{c}=0, \nonumber \\
\dfrac{1}{3!}+h_{3}+\left(\dfrac{1}{3!}+h_{2}\right)\dfrac{a_{2}^{3}}{c^2}=0, \nonumber \\
\dfrac{1}{4!}+\dfrac{h_{3}}{2}+\dfrac{h_{3}^{2}}{2}+\left(\dfrac{1}{4!}+\dfrac{h_{2}}{2}+\dfrac{h_{2}^{2}}{2}\right)\dfrac{a_{2}^4}{c^3}=0, \nonumber \\
\dfrac{1}{5!}+\dfrac{h_{3}}{6}+\dfrac{h_{3}^{2}}{2}+\left(\dfrac{1}{5!}+\dfrac{h_{3}}{6}+\dfrac{h_{3}^{2}}{2}\right)\dfrac{a_{2}^{5}}{c^{4}}=0, \nonumber \\
\dfrac{1}{6!}+\dfrac{h_{3}}{4!}+\dfrac{h_{3}^{2}}{4}+\dfrac{h_{3}^{3}}{6}+\left(\dfrac{1}{6!}+\dfrac{h_{2}}{4!}+\dfrac{h_{2}^{2}}{4}+\dfrac{h_{2}^{3}}{6}\right)\dfrac{a_{2}^{6}}{c^{5}}=0.\nonumber
\end{eqnarray}
We can check again that Groebner bases contains a polynomial of $c$ with all positive coefficients. Therefore, the possibility that $h_{2}^{n}$ is finite does not hold. As a consequence, $|h_{2}^{n}| \to \infty$. However, as both $|h_{1}^{n}|, |h_{2}^{n}| \to \infty$, we get $|h_{3}^{n}| \to \infty$, which is a contradiction. Therefore, $c>0$ cannot happen. It implies that $p_{12}^{n}/p_{13}^{n} \to c=0$. 

If $a_{2} \neq 0$ then $\Delta \theta_{13}^{n}/\Delta \theta_{12}^{n} \to 0$. Since $p_{11}^{n}/p_{12}^{n}, p_{13}^{n}/p_{12}^{n} \to \infty$, $p_{11}^{n}\Delta \theta_{11}^{n}/p_{12}^{n}\Delta\theta_{12}^{n}, p_{13}^{n}\Delta \theta_{13}^{n}/p_{12}^{n}\Delta \theta_{12}^{n}$ are finite, with the same argument as that of Case 3.1, we get the contradiction. Thus, $a_{2}=0$. However, as $\left|\dfrac{p_{11}^{n}\Delta \theta_{11}^{n}}{p_{13}^{n}\Delta \theta_{13}^{n}}\right| \leq \left|\dfrac{p_{12}^{n}\Delta \theta_{12}^{n}}{p_{13}^{n}\Delta \theta_{13}^{n}}\right|$, it implies that $p_{11}^{n}\Delta \theta_{11}^{n}/p_{13}^{n}\Delta \theta_{13}^{n} \to 0$. It follows that $a_{1}+a_{2}=0$, which is a contradiction to the fact that $a_{1}+a_{2}=1$. Overall, the possibility that $p_{11}^{n}/p_{13}^{n} \to \infty$ and $p_{12}^{n}/p_{13}^{n} \not \to \infty$ cannot happen. 

As a consequence, $p_{11}^{n}/p_{13}^{n} \not \to \infty$ and $p_{12}^{n}/p_{13}^{n} \to \infty$. Using the same argument as before, eventually, we get to the case when $p_{11}^{n}/p_{13}^{n} \to 0$ and $a_{1}=0$. If $\Delta \theta_{11}^{n}/\Delta \theta_{13}^{n}$ is finite then $p_{11}^{n}(\Delta \theta_{11}^{n})^{j}/p_{13}^{n}(\Delta \theta_{13}^{n})^{j} \to 0$ for all $1 \leq j \leq 6$. As we also have $p_{12}^{n}(\Delta \theta_{12}^{n})^{j}/p_{13}^{n}(\Delta \theta_{13}^{n})^{j} \to 0$ for all $1 \leq j \leq 6$, $K_{n,1},K_{n,2},K_{n,3}, K_{n,4}$ demonstrate that $|h_{1}^{n}|,|h_{2}^{n}| \to \infty$. However, it also implies that $|h_{3}^{n}| \to \infty$, which is a contradiction. Therefore, $\left|\dfrac{\Delta \theta_{11}^{n}}{\Delta \theta_{13}^{n}}\right| \to \infty$. 

If $h_{2}^{n}$ is finite then at least one of $h_{1}^{n}$ and $h_{3}^{n}$ is finite. First, we assume that $h_{1}^{n}$ is finite. Now, if $p_{11}^{n} (\Delta \theta_{11} ^{n})^{2}/p_{13}^{n}(\Delta \theta_{13}^{n})^{2} \not \to 0$ then $p_{11}^{n}(\theta_{11}^{n})^{j}/p_{13}^{n}(\theta_{13}^{n})^{j}$ becomes infinite for all $j \geq 3$. Consider $K_{n,2}-K_{n,1}$, we achieve $\dfrac{1}{3!}+h_{1}^{n}\to 0$.Similarly, consider $K_{n,4}-K_{n,3}+\dfrac{1}{3}K_{n,2}$, we obtain $\dfrac{1}{5!}+\dfrac{h_{1}^{n}}{6}+\dfrac{(h_{1}^{n})^{2}}{2} \to 0$, which contradicts to $\dfrac{1}{3!}+h_{1}^{n}\to 0$. Therefore, $p_{11}^{n}(\Delta\theta_{11}^{n})^{2}/p_{13}^{n}(\Delta\theta_{13}^{n})^{2} \to 0$. From $K_{n,1}$, it shows that $h_{3}^{n}+\dfrac{1}{2} \to 0$. Combining this result with $K_{n,2}, K_{n,3}, K_{n,4}, K_{n,5}$, we obtain $p_{11}^{n}(\Delta \theta_{11}^{n})^{j}/p_{13}^{n}(\Delta \theta_{13}^{n})^{j}$ are finite for all $2 \leq j \leq 6$. However, as $\Delta \theta_{11}^{n}/\Delta \theta_{13}^{n}$ is infinite, we obtain $p_{11}^{n}(\theta_{11}^{n})^{2}/p_{13}^{n}(\theta_{13}^{n})^{3} \to 0$. Combining it with $K_{n,2}$,we obtain $h_{3}^{n}+\dfrac{1}{3!} \to 0$, which contradicts $h_{3}^{n}+1/2 \to 0$. As a consequence, $h_{1}^{n}$ is not finite, which also implies that $h_{3}^{n}$ is finite. 

However, it means that $p_{11}^{n}(\Delta \theta_{11}^{n})^{j}/p_{13}^{n}(\Delta \theta_{13}^{n})^{j} \to 0$ for all $2 \leq j \leq 6$. If $h_{1}^{n}\dfrac{p_{11}^{n}(\Delta \theta_{11}^{n})^{2}}{p_{13}^{n}(\Delta \theta_{13}^{n})^{2}} \not \to 0$ then $K_{n,2}$ cannot happen as $\Delta \theta_{11}^{n}/\Delta \theta_{13}^{n}$ is infinite. Hence, $h_{1}^{n}\dfrac{p_{11}^{n}(\Delta \theta_{11}^{n})^{2}}{p_{13}^{n}(\Delta \theta_{13}^{n})^{2}} \to 0$, which implies $h_{3}^{n}+1/2 \to 0$. From $K_{n,4}$, since $h_{1}^{n}$ is infinite, we achieve $(h_{1}^{n})^{2}\dfrac{p_{11}^{n}(\Delta \theta_{11}^{n})^{4}}{p_{13}^{n}(\Delta \theta_{13}^{n})^{4}}$ is finite. It also means that $h_{1}^{n}\dfrac{p_{11}^{n}(\Delta \theta_{11}^{n})^{3}}{p_{13}^{n}(\Delta \theta_{13}^{n})^{3}} \to 0$. Combining this result with $K_{n,2}$, we achieve $h_{3}^{n}+\dfrac{1}{3!} \to 0$, which contradicts $h_{3}^{n}+1/2 \to 0$. Thus, the possibility that $h_{2}^{n}$ is finite does not hold. Therefore, $|h_{2}^{n}| \to \infty$. Using the same line of argument as before, we also obtain $h_{1}^{n},h_{3}^{n}$ are infinite, which is a contradiction. As a consequence, case 3.2 cannot hold.

\noindent
\paragraph{Case 2.3:} At least one of $p_{11}^{n}/p_{13}^{n}$ and $p_{12}^{n}/p_{13}^{n} \to 0$ and they are both finite. As $a_{1}+a_{2}=-1$, it means that at least one of $a_{1},a_{2}$ is different from 0. Without loss of generality, we assume $a_{1} \neq 0$. It implies that $p_{12}^{n}\Delta \theta_{12}^{n}/ p_{11}^{n}\Delta \theta_{11}^{n} \to a_{2}/a_{1} \neq \infty$ and $p_{13}^{n}\Delta \theta_{13}^{n}/p_{11}^{n} \Delta \theta_{11}^{n} \to 1/a_{1} \neq \infty$. Since $p_{11}^{n}/p_{13}^{n}$ is finite, $p_{13}^{n}/p_{11}^{n}\not \to 0$. Additionally, if $a_{2}=0$ then $p_{13}^{n}\Delta\theta_{13}^{n}/p_{12}^{n}\Delta \theta_{12}^{n}\to \infty$ and $p_{11}^{n}\Delta\theta_{11}^{n}/p_{12}^{n}\Delta\theta_{12}^{n} \to \infty$, which is a contradiction to $p_{11}^{n}|\Delta \theta_{11}^{n}| \leq p_{12}^{n} |\Delta \theta_{12}^{n}|$. Therefore, $a_{2} \neq 0$. 

If $p_{12}^{n}/p_{11}^{n} \not \to \left\{0,\infty\right\}$ then by dividing the numerator and denominator of $F_{\alpha}'(\theta_{1}^{0},v_{1}^{0})$ by $p_{11}^{n}(\Delta \theta_{11}^{n})^{\alpha}$ for all $1 \leq \alpha \leq 6$ and letting $n \to \infty$, we achieve the scaling system of polynomial equations \eqref{eqn:generalovefittedGaussianzero} when $r=6$, which we have already known that it does not have any soltution. 

If $p_{12}^{n}/p_{11}^{n} \to \infty$ then we can argue in the same way as that of Case 3.2 by dividing both the numerator and denominator of $F_{\alpha}'(\theta_{1}^{0},v_{1}^{0})$ by $p_{11}^{n}(\Delta \theta_{11}^{n})^{\alpha}$ for all $1 \leq \alpha \leq 6$ to get the contradiction. 

If $p_{12}^{n}/p_{11}^{n} \to 0$ then it implies that $p_{11}^{n}/p_{12}^{n} \to \infty$ and $p_{13}^{n}/p_{12}^{n} \to \infty$. Now, we also have $p_{13}^{n}\Delta \theta_{13}^{n}/p_{12}^{n}\Delta \theta_{12}^{n}\to 1/a_{2} \neq \infty$ and $p_{11}^{n}\Delta \theta_{11}^{n}/p_{12}^{n}\Delta\theta_{12}^{n} \to a_{1}/a_{2} \neq \infty$. Therefore, we can argue in the same way as that of Case 3.1 by dividing both the numerator and denominator of $F_{\alpha}'(\theta_{1}^{0},v_{1}^{0})$ by $p_{12}^{n}(\Delta \theta_{11}^{n})^{\alpha}$ to get the contradiction. Therefore, case 3.3 cannot happen.

\noindent
\paragraph{Case 2.4:} Both $p_{11}^{n}/p_{13}^{n}, p_{12}^{n}/p_{13}^{n} \not \to \left\{0,\infty\right\}$. By diving both the numerator and denominator of $F_{\alpha}'(\theta_{1},v_{1})$ by $p_{13}^{n}(\Delta \theta_{13}^{n})^{\alpha}$ for all $1 \leq \alpha \leq 6$, we achieve the scaling system of polynomial equations \eqref{eqn:generalovefittedGaussianzero} when $r=6$, which does not admit any solution.\\

As a consequence, $i^{*} \not \in \left\{1,\ldots,k_{0}\right\}$. Therefore, $i^{*} \in \left\{k_{0}+1,\ldots,k_{0}+m\right\}$. However, since $m \leq 2$, with the observation that when $k_{0}+1 \leq i \leq k_{0}+m$, each support point $(\theta_{i}^{0},v_{i}^{0})$ only has at most 2 points converge to, we can use the same argument as that of Case 1 to get the contradiction. Overall, we get the conclusion of our theorem. 


\paragraph{PROOF OF THEOREM~\ref{theorem:nonconformantexactfittedskewnormal}.}

(a) As we have seen the proof of part (b) in Theorem \ref{theorem:exactfittedskewnormal}, 
condition $m_{j}^{0}m_{i}^{0}>0$ plays an important role to get the inequality in \eqref{eqn:exactfittedskewnormalgeneral}
to yield a contradiction. If this condition does not hold, then it is possible that $\mathop {\sum }\limits_{j=s_{1}}^{s_{2}-1}{p_{j}^{n}m_{j}^{0}(\Delta m_{j}^{n})^{2}}/\mathop {\sum }\limits_{j=s_{1}}^{s_{2}-1}{p_{j}^{n}(\Delta m_{j}^{n})^{2}} \to 0$. Therefore, we need a 
special treatment for this situation.  
For the simplicity of our argument later, 
we first consider the case $k^{*}=1$ to illustrate why 
$W_{\overline{s}}^{\overline{s}}$ may not be the best lower bound in general. 
All the notations in this proof are the same as those of part (b) of the proof of Theorem \ref{theorem:exactfittedskewnormal}. 
Going back to Equation \eqref{eqn:exactfittedskewnormalgeneralkey}, we divide our argument into two cases:

\paragraph{Case 1:} If cousin set $I_{s_{1}}$ is conformant, i.e, 
$m_{i}^{0}$ share the same sign for all $i \in I_{s_{1}}$. 
Then, we can proceed the proof in the same fashion as the part following Equation 
\eqref{eqn:exactfittedskewnormalgeneral} in part (b) of the proof of 
Theorem \ref{theorem:exactfittedskewnormal}. 
\paragraph{Case 2:} If cousin set $I_{s_{1}}$ is not conformant, 
from the assumption of part (a) of Theorem \ref{theorem:nonconformantexactfittedskewnormal} and 
$k^{*}=1$, we should have $|I_{s_{1}}|=1$. So, $s_{2}=s_{1}+2$.  
From Case 2 of part (b) of Theorem \ref{theorem:exactfittedskewnormal}, we have $\Delta p_{i}^{n}/d_{\text{new}}(p_{i^{*}}^{n},\theta_{i^{*}}^{n},v_{i^{*}}^{n},m_{i^{*}}^{n})$, $p_{i}^{n}\Delta v_{i}^{n}/d_{\text{new}}(p_{i^{*}}^{n},\theta_{i^{*}}^{n},v_{i^{*}}^{n},m_{i^{*}}^{n})$, $p_{i}^{n}(\Delta \theta_{i}^{n})^{2}/d_{\text{new}}(p_{i^{*}}^{n},\theta_{i^{*}}^{n},v_{i^{*}}^{n},m_{i^{*}}^{n})$,\\ $p_{i}^{n}(v_{i}^{n})^{2}/d_{\text{new}}(p_{i^{*}}^{n},\theta_{i^{*}}^{n},v_{i^{*}}^{n},m_{i^{*}}^{n}) \to 0$ for all $s_{1} \leq i \leq s_{2}-1$. Combining these results with the assumption that $\beta_{2i}^{n} \to 0$, we obtain 
\begin{eqnarray}
\mathop {\sum }\limits_{j=s_{1}}^{s_{2}-1}{p_{j}^{n}\Delta m_{j}^{n}}/d_{\text{new}}(p_{i^{*}}^{n},\theta_{i^{*}}^{n},v_{i^{*}}^{n},m_{i^{*}}^{n}) \to 0. \nonumber
\end{eqnarray}
Since  $U_{n}=\mathop {\sum }\limits_{j=s_{1}}^{s_{2}-1}{p_{j}^{n}(\Delta m_{j}^{n})^{2}}/d_{\text{new}}(p_{i^{*}}^{n},\theta_{i^{*}}^{n},v_{i^{*}}^{n},m_{i^{*}}^{n}) \not \to 0$, we get 
\begin{eqnarray}
Z_{n} :=\mathop {\sum }\limits_{j=s_{1}}^{s_{2}-1}{p_{j}^{n}\Delta m_{j}^{n}}/ \mathop {\sum }\limits_{j=s_{1}}^{s_{2}-1}{p_{j}^{n}(\Delta m_{j}^{n})^{2}} \to 0.\nonumber
\end{eqnarray} 
Without loss of generality, we assume $|\Delta m_{s_{1}+1}^{n}| \geq |\Delta m_{s_{1}}^{n}|$ for infinitely many $n$, which to avoid notational cluttering we also assume it holds for all $n$. Denote $\Delta m_{s_{1}}^{n}/\Delta m_{s_{1}+1}^{n} \to a$. Divide both the numerator and denominator of $Z_{n}$ by $\Delta m_{s_{1}}^{n}$ and let $n \to \infty$, we obtain $p_{s_{1}}^{0}+p_{s_{1}+1}^{0}a=0$. Similarly, from \eqref{eqn:exactfittedskewnormalgeneralkey}, by dividing both the numerator and denominator of $V_{n}/U_{n}$ for $(\Delta m_{s_{1}}^{n})^{2}$, we obtain $p_{s_{1}}^{0}+p_{s_{1}+1}^{0}a^{2}=0$. Therefore, we achieve a system of equations
\begin{eqnarray}
p_{s_{1}}^{0}+p_{s_{1}+1}^{0}a=0, \nonumber \\
p_{s_{1}}^{0}m_{s_{1}}^{0}+p_{s_{1}+1}^{0}m_{s_{1}+1}^{0}a^{2}=0. \nonumber
\end{eqnarray}
This is actually equation \eqref{eqn:noncannonicalexactfittedskewnorma} when $k^{*}=1$ and   $r=2$. Solving the first equation, we obtain $a=-p_{s_{1}}^{0}/p_{s_{1}+1}^{0}$. However, by substituting this result to the second equation, we get $p_{s_{1}}^{0}m_{s_{1}+1}^{0}+p_{s_{1}+1}^{0}m_{s_{1}}^{0}=0$. We have the following two small cases:
\paragraph{Case 2.1:} Assume we have $p_{s_{1}}^{0}m_{s_{1}}^{0}+p_{s_{1}+1}^{0}m_{s_{1}}^{0} \neq 0$, then it means the system of equation does not have any solution. Hence, in this case, the lower bound of $V(p_{G},p_{G_{0}})$ is still $W_{2}^{2}(G,G_{0})$.

\paragraph{Case 2.2:} Assume we have $p_{s_{1}}^{0}m_{s_{1}+1}^{0}+p_{s_{1}+1}^{0}m_{s_{1}}^{0}=0$ . We have two important steps:

\paragraph{Step 1- Construction to show that $V(p_{G},p_{G_{0}})$ cannot be lower bounded by $W_{1}^{r}$ as $r < \overline{s}=3$:}  We construct $G_{n}$ such that both $Z_{n}$ and $U_{n}/V_{n}$ can go to $0$. We choose $G_{n}=\mathop {\sum }\limits_{i=1}^{s_{i}}{p_{i}^{n}\delta_{(\theta_{i}^{n},v_{i}^{n},m_{i}^{n})}}$ such that $(p_{i}^{n},\theta_{i}^{n},v_{i}^{n})=(p_{i}^{0},\theta_{i}^{0},v_{i}^{0})$ for all $1 \leq i \leq k_{0}$, $m_{i}^{n}=m_{i}^{0}$ for all $i \not \in \left\{s_{1},s_{1}+1\right\}$. Choose $\Delta m_{s_{1}}^{n}=-p_{s_{1}+1}^{0}/(p_{s_{1}}^{0}n)$ and $\Delta m_{s_{1}+1}^{n}=1/n$, then we can check that $\mathop {\sum }\limits_{j=s_{1}}^{s_{1}+1}{p_{j}^{n}\Delta m_{j}^{n}}=\mathop {\sum }\limits_{j=s_{1}}^{s_{1}+1}{p_{j}^{n}m_{j}^{0}(\Delta m_{j}^{n})^{2}}=0$. Additionally, for any $1 \leq r <\overline{s}=3$, $W_{1}^{r}(G_{n},G_{0})=(p_{s_{1}}^{0}+p_{s_{1}+1}^{0})^{r}/n^{r}$.  By means of Taylor expansion up to third order, we can check that $\mathop {\sup }\limits_{x \in \mathbb{R}}{|p_{G_{n}}(x)-p_{G_{0}}(x)|}/W_{1}^{r}(G_{n},G_{0}) \to 0$ as $n \to \infty$ for all $x \in \mathbb{R}$. With this choice of $G_{n}$, we also have 
\begin{eqnarray}
V(p_{G_{n}},p_{G_{0}})/W_{1}^{r}(G_{n},G_{0}) \lesssim \int \limits_{(-\delta,\delta)}{|p_{G_{n}}(x)-p_{G_{0}}(x)|}dx/W_{1}^{r}(G_{n},G_{0}) \to 0, \nonumber
\end{eqnarray}
where $\delta$ is sufficiently large constant. Therefore, we achieve that for any $1 \leq r <3$
\begin{eqnarray}
\mathop {\lim }\limits_{\epsilon \to 0}{\mathop {\inf }\limits_{G \in \mathcal{O}_{k}(\Theta \times \Omega)}{\left\{\dfrac{V(p_{G},p_{G_{0}})}{W_{1}^{r}(G,G_{0})}: W_{1}(G,G_{0}) \leq \epsilon \right\}}} = 0. \nonumber
\end{eqnarray}

\paragraph{Step 2 - We show that $V(p_{G},p_{G_{0}}) \gtrsim W_{3}^{3}(G,G_{0})$:} In fact, it is sufficient to demonstrate that
\begin{eqnarray}
\mathop {\lim }\limits_{\epsilon \to 0}{\mathop {\inf }\limits_{G \in \mathcal{O}_{k}(\Theta \times \Omega)}{\left\{\dfrac{V(p_{G},p_{G_{0}})}{W_{3}^{3}(G,G_{0})}: W_{3}(G,G_{0}) \leq \epsilon \right\}}} > 0. \nonumber
\end{eqnarray}
Now, by assuming the contrary and carrying out the same argument as the proof of part (b) of Theorem \ref{theorem:exactfittedskewnormal} with Taylor expansion go up to third order, we can see that Case 1 and Case 3 of part (b) still applicable to the third order, i.e yield the contradiction, because they are not affected by the non-conformant conditions. Now, Case 2 will yield us the following results 
\begin{eqnarray}
\mathop {\sum }\limits_{j=s_{1}}^{s_{2}-1}{p_{j}^{n}\Delta m_{j}^{n}}/\mathop {\sum }\limits_{j=s_{1}}^{s_{2}-1}{p_{j}^{n}|\Delta m_{j}^{n}|^{3}} \to 0, \nonumber \\
\mathop {\sum }\limits_{j=s_{1}}^{s_{2}-1}{p_{j}^{n}m_{j}^{0}(\Delta m_{j}^{n})^{2}}/\mathop {\sum }\limits_{j=s_{1}}^{s_{2}-1}{p_{j}^{n}|\Delta m_{j}^{n}|^{3}} \to 0, \nonumber \\
\mathop {\sum }\limits_{j=s_{1}}^{s_{2}-1}{p_{j}^{n}(\Delta m_{j}^{n})^{3}}/\mathop {\sum }\limits_{j=s_{1}}^{s_{2}-1}{p_{j}^{n}|\Delta m_{j}^{n}|^{3}} \to 0, \nonumber \\
\mathop {\sum }\limits_{j=s_{1}}^{s_{2}-1}{p_{j}^{n}(m_{j}^{0})^{2}(\Delta m_{j}^{n})^{3}}/\mathop {\sum }\limits_{j=s_{1}}^{s_{2}-1}{p_{j}^{n}|\Delta m_{j}^{n}|^{3}} \to 0. \nonumber
\end{eqnarray}
Remind that $s_{2}=s_{1}+2$ and $\Delta m_{s_{1}}^{n}/\Delta m_{s_{1}+1}^{n} \to a$. By dividing both the numetor and denominator of first above result by $\Delta m_{s_{1}}^{n}$, second above result by $(\Delta m_{s_{1}}^{n})^{2}$, and third and fourth above result by $(\Delta m_{s_{1}}^{n})^{3}$, we obtain the following system of equations
\begin{eqnarray}
p_{s_{1}}^{0}+p_{s_{1}+1}^{0}a=0, \nonumber \\
p_{s_{1}}^{0}m_{s_{1}}^{0}+p_{s_{1}+1}^{0}m_{s_{1}+1}^{0}a^{2}=0, \nonumber \\
p_{s_{1}}^{0}+p_{s_{1}+1}^{0}a^{3}=0, \nonumber \\
p_{s_{1}}^{0}(m_{s_{1}}^{0})^{2}+p_{s_{1}+1}^{0}(m_{s_{1}+1}^{0})^{2}a^{3}=0. \nonumber
\end{eqnarray}
As $(p_{s_{1}}^{0},m_{s_{1}}^{0}) \neq (p_{s_{1}+1}^{0},-m_{s_{1}^{0}+1})$, the above system of equations does not admit any solution, which is a contradiction. Therefore, our assertion follows immediately.

\paragraph{General argument for $k^{*}$:} Now, for general case of $k^{*}$, we argue exactly the same way as that of Step 2 of Case 2.2. More specifically, by carrying out Taylor expansion up to $\overline{s}$-th order and using the same argument as the proof of part (b) of Theorem \ref{theorem:exactfittedskewnormal}, Case 1 and Case 3 under $W_{\overline{s}}^{\overline{s}}$ still yield the contradiction. As a consequence, we only need to deal with Case 2. In fact, it leads to the following results:
\begin{eqnarray}
\mathop {\sum }\limits_{j=s_{1}}^{s_{2}-1}{p_{j}^{n}(m_{j}^{0})^{u}(\Delta m_{j}^{n})^{v}}/ \mathop {\sum }\limits_{j=s_{1}}^{s_{2}+1}{p_{j}^{n}|\Delta m_{j}^{n}|^{\overline{s}}} \to 0, \label{eqn:partdgeneralskewnormal}
\end{eqnarray}
for any $v \leq \overline{s}$, $u \leq v$ are all odd numbers when $v$ is even or $0 \leq u \leq v$ are all even numbers when $v$ is odd. Notice that, now $s_{2}-s_{1} \leq k^{*}+1$. Without loss of generality, we assume $|\Delta m_{s_{1}}^{n}|=\mathop {\max }\limits_{s_{1} \leq j \leq s_{2}-1}{|\Delta m_{j}^{n}|}$. Denote $\Delta m_{l}^{n}/\Delta m_{s_{1}}^{n} \to x_{l}$ for all $s_{1}+1 \leq l \leq s_{2}-1$. and $x_{s_{1}}=1$. Then by dividing both the numerator and denominator of $\mathop {\sum }\limits_{j=s_{1}}^{s_{2}+1}{p_{j}^{n}(m_{j}^{0})^{u}(\Delta m_{j}^{n})^{v}}/ \mathop {\sum }\limits_{j=s_{1}}^{s_{2}+1}{p_{j}^{n}|\Delta m_{j}^{n}|^{\overline{s}}}$ by $(\Delta m_{s_{1}}^{n})^{v}$ and let $n \to \infty$, we achieve the following system of polynomial equations
\begin{eqnarray}
\mathop {\sum }\limits_{j=s_{1}}^{s_{2}-1}{p_{j}^{0}(m_{j}^{0})^{u}x_{j}^{v}}=0, \nonumber
\end{eqnarray}
for all $1 \leq v \leq \overline{s}$, $u \leq v$ are all odd numbers when $v$ is even or $0 \leq u \leq v$ are all even numbers when $v$ is odd. Since $s_{2}-s_{1} \leq k^{*}+1$, the hardest case will be when $s_{2}-s_{1}=k^{*}+1$. In this case, the above system of equations becomes system \eqref{eqn:noncannonicalexactfittedskewnorma}. From the hypothesis, we have already known that with that value of $\overline{s}$, the above system of equations does not have any highly non-trivial solution, which is a contradiction. Therefore, the assertion of our theorem follows immediately.\\\\
(b) Without loss of generality, we assume that $(p_{1}^{0},m_{1}^{0})=(p_{2}^{0},-m_{2}^{0})$. Now, we proceed to choose sequence $G_{n}$ as that of Step 1 of Case 2.2 where $s_{1}$ is replaced by $1$. Then we can check that $\mathop {\sum }\limits_{j=1}^{2}{p_{j}^{n}(m_{j}^{0})^{u}(\Delta m_{j}^{n})^{v}}=0$ for all odd number $u \leq v$ when $v$ is even number or for all even number $0 \leq u \leq v$ when $v$ is odd number. Therefore, for any $r \geq 1$, by carrying out Taylor expansion up to $[r]+1$-th order, we can check that $\mathop {\sup }\limits_{x \in \mathbb{R}}{|p_{G_{n}}(x)-p_{G_{0}}(x)|}/W_{1}^{r}(G_{n},G_{0}) \to 0$, thereby leading to $V(p_{G_{n}},p_{G_{0}})/W_{1}^{r}(G_{n},G_{0}) \to 0$. As a consequence, we obtain the conclusion of part (b) of our theorem.

\paragraph{Remark:} As we can see from the case $k^{*}=1$, $W_{3}^{3}$ is a lower bound of $V(p_{G},p_{G_{0}})$ under the condition (S.3), but it is not the best lower bound. More specifically, under the scenario of Case 2.1, $W_{2}^{2}$ is the best lower bound of $V(p_{G},p_{G_{0}})$(also $h(p_{G},p_{G_{0}})$) while under the scenario of Case 2.2, $W_{3}^{3}$ is the best lower bound of $V(p_{G},p_{G_{0}})$( also $h(p_{G},p_{G_{0}})$). It suggests the minimax optimal convergence rate $n^{-1/4}$ under $W_{2}$ distance in Case 2.1 or $n^{-1/6}$ under $W_{3}$ distance in Case 2.2. As $k^{*}$ is bigger, such as $k^{*}=2$, the minimax optimal convergence rate can be $n^{-1/8}$ under $W_{4}$ or $n^{-1/10}$ under $W_{5}$ or so on. These rates just reflect how broad convergence rate behaviors of skew-Gaussian are.

\paragraph{Supplementary arguments for the proof of Theorem \ref{theorem:exactfittedskewnormal}} 
Here, we give additional arguments and detailed calculations for the
proof of Theorem \ref{theorem:exactfittedskewnormal}, which are presented in Appendix I.
\paragraph{Detailed formulae of $A_{n,2}(x)$:}
\begin{eqnarray}
d(G_{n},G_{0})\alpha_{1ji}^{n} &=& \dfrac{2\Delta p_{j}^{n}}{\sigma_{j}^{0}}-\dfrac{p_{j}^{n}\Delta v_{j}^{n}}{(\sigma_{j}^{0})^{3}}-\dfrac{p_{j}^{n}(\Delta \theta_{j}^{n})^{2}}{(\sigma_{j}^{0})^{3}} + \dfrac{3p_{j}^{n}(\Delta v_{j}^{n})^{2}}{4(\sigma_{j}^{0})^{5}}, \nonumber \\
 d(G_{n},G_{0})\alpha_{2ji}^{n} &=& \dfrac{2p_{j}^{n}\Delta \theta_{j}^{n}}{(\sigma_{j}^{0})^{3}}-\dfrac{6p_{j}^{n}\Delta \theta_{j}^{n}\Delta v_{j}^{n}}{(\sigma_{j}^{0})^{5}}, \nonumber \\ 
 d(G_{n},G_{0})\alpha_{3ji}^{n} &=& \dfrac{p_{j}^{n}\Delta v_{j}^{n}}{(\sigma_{j}^{0})^{5}}+ \dfrac{p_{j}^{n}(\Delta \theta_{j}^{n})^{2}}{(\sigma_{j}^{0})^{5}}-\dfrac{3p_{j}^{n}(\Delta v_{j}^{n})^{2}}{2(\sigma_{j}^{0})^{7}}, \nonumber \\
 d(G_{n},G_{0})\alpha_{4ji}^{n} &=& \dfrac{2p_{j}^{n}\Delta \theta_{j}^{n}\Delta v_{j}^{n}}{(\sigma_{j}^{0})^{7}}, \nonumber \\
 d(G_{n},G_{0})\alpha_{5ji}^{n} &=& \dfrac{p_{j}^{n}(\Delta v_{j}^{n})^{2}}{4(\sigma_{j}^{0})^{9}}, \nonumber \\
d(G_{n},G_{0})\beta_{1i}^{n} &=& \mathop {\sum }\limits_{j=s_{i}}^{s_{i+1}-1}{-\dfrac{p_{j}^{n}m_{j}^{0}\Delta \theta_{j}^{n}}{\pi(\sigma_{j}^{0})^{2}}+\dfrac{2p_{j}^{n}m_{j}^{0}\Delta \theta_{j}^{n}\Delta v_{j}^{n}}{\pi(\sigma_{j}^{0})^{4}}}- \dfrac{2p_{j}^{n}\Delta \theta_{j}^{n}\Delta m_{j}^{n}}{\pi(\sigma_{j}^{0})^{2}}, \nonumber \\
d(G_{n},G_{0})\beta_{2i}^{n} &=& \mathop {\sum }\limits_{j=s_{i}}^{s_{i+1}-1}{ -\dfrac{p_{j}^{n}m_{j}^{0}\Delta v_{j}^{n}}{2\pi(\sigma_{j}^{0})^{4}}-\dfrac{p_{j}^{n}((m_{j}^{0})^{3}+2m_{j}^{0})(\Delta \theta_{j}^{n})^{2}}{2\pi(\sigma_{j}^{0})^{4}}} + \dfrac{p_{j}^{n}\Delta m_{j}^{n}}{\pi(\sigma_{j}^{0})^{2}} \nonumber \\
&+& \dfrac{5p_{j}^{n}m_{j}^{0}(\Delta v_{j}^{n})^{2}}{8\pi(\sigma_{i}^{0})^{6}}-\dfrac{p_{j}^{n}\Delta m_{j}^{n}\Delta v_{j}^{n}}{\pi(\sigma_{j}^{0})^{4}}, \nonumber \\
d(G_{n},G_{0})\beta_{3i}^{n}&=&\mathop {\sum }\limits_{j=s_{i}}^{s_{i+1}-1}{\dfrac{p_{j}^{n}(2(m_{j}^{0})^{2}+2)\Delta m_{j}^{n}\Delta \theta_{j}^{n}}{\pi(\sigma_{j}^{0})^{4}}} - \dfrac{p_{j}^{n}((m_{j}^{0})^{3}+2m_{j}^{0})\Delta \theta_{j}^{n}\Delta v_{j}^{n}}{2\pi(\sigma_{j}^{0})^{6}}, \nonumber \\
d(G_{n},G_{0})\beta_{4i}^{)} &=& \mathop {\sum }\limits_{j=s_{i}}^{s_{i+1}-1}{-\dfrac{p_{j}^{n}((m_{j}^{0})^{3}+2m_{j}^{0})(\Delta v_{j}^{n})^{2}}{8\pi(\sigma_{i}^{0})^{8}}-\dfrac{p_{j}^{n}m_{j}^{0}(\Delta m_{j}^{n})^{2}}{2\pi(\sigma_{j}^{0})^{4}}} \nonumber \\
&+& \dfrac{p_{j}^{n}((m_{j}^{0})^{2}+1)\Delta m_{j}^{n}\Delta v_{j}^{n}}{\pi(\sigma_{j}^{0})^{6}}. \nonumber
\end{eqnarray}

\paragraph{Detailed formulae of $A_{n,1}(x)$:}
\begin{eqnarray}
d(G_{n},G_{0})\gamma_{1j}^{n} &=& -\dfrac{p_{j}^{n}\Delta v_{j}^{n}}{2\sqrt{2\pi}(\sigma_{j}^{0})^{3}}-\dfrac{p_{j}^{n}(\Delta \theta_{j}^{n})^{2}}{2\sqrt{2\pi}(\sigma_{j}^{0})^{3}}+\dfrac{3p_{j}^{n}(\Delta v_{j}^{n})^{2}}{8\sqrt{2\pi}(\sigma_{j}^{0})^{5}} - \dfrac{2p_{j}^{n}\Delta \theta_{j}^{n}\Delta m_{j}^{n}}{\pi(\sigma_{j}^{0})^{2}}+\dfrac{\Delta p_{j}^{n}}{\sqrt{2\pi}\sigma_{j}^{0}}, \nonumber \\
d(G_{n},G_{0})\gamma_{2j}^{n} &=& \dfrac{p_{j}^{n}\Delta \theta_{j}^{n}}{\sqrt{2\pi}(\sigma_{j}^{0})^{3}}+ \dfrac{p_{j}^{n}\Delta m_{j}^{n}}{\pi(\sigma_{j}^{0})^{2}}-\dfrac{3p_{j}^{n}\Delta \theta_{j}^{n}\Delta v_{j}^{n}}{\sqrt{2\pi}(\sigma_{j}^{0})^{5}} - \dfrac{p_{j}^{n}\Delta v_{j}^{n}\Delta m_{j}^{n}}{\pi(\sigma_{j}^{0})^{4}}, \nonumber \\
d(G_{n},G_{0})\gamma_{3j}^{n} &=& \dfrac{p_{j}^{n}\Delta v_{j}^{n}}{2\sqrt{2\pi}(\sigma_{j}^{0})^{5}}+\dfrac{p_{j}^{n}(\Delta \theta_{j}^{n})^{2}}{2\sqrt{2\pi}(\sigma_{j}^{0})^{5}}-\dfrac{6p_{j}^{n}(\Delta v_{j}^{n})^{2}}{8\sqrt{2\pi}(\sigma_{j}^{0})^{7}} + \dfrac{2p_{j}^{n}\Delta \theta_{j}^{n}\Delta m_{j}^{n}}{\pi(\sigma_{j}^{0})^{4}}, \nonumber \\
d(G_{n},G_{0})\gamma_{4j}^{n} &=& \dfrac{p_{j}^{n}\Delta \theta_{j}^{n}\Delta v_{j}^{n}}{\sqrt{2\pi}(\sigma_{j}^{0})^{7}}+\dfrac{p_{j}^{n}\Delta v_{j}^{n}\Delta m_{j}^{n}}{\pi(\sigma_{j}^{0})^{6}}, \nonumber \\
d(G_{n},G_{0})\gamma_{5j}^{n} &=& \dfrac{p_{j}^{n}(\Delta v_{j}^{n})^{2}}{8\sqrt{2\pi}(\sigma_{j}^{0})^{9}}. \nonumber
\end{eqnarray}

\paragraph{Additional arguments for Step 1.1:} We divide this step into three further cases:

\paragraph{Case 1.1.1:} If $\theta_{i^{*}}^{n}=\theta_{i^{*}}^{0}$ for infinitely $n$, which without loss of generality, we can assume $\theta_{i^{*}}^{n}=\theta_{i^{*}}^{0}$ for all $n$, then as $C_{3}^{n},C_{5}^{n} \to 0$ as $n \to \infty$, we achieve $p_{i^{*}}^{n}\Delta v_{i^{*}}^{n}/d(p_{i^{*}}^{n},\theta_{i^{*}}^{n},v_{i^{*}}^{n},m_{i^{*}}^{n}) \to 0$ as $n \to \infty$. Combining this result with $C_{1}^{n}$, we get $\Delta p_{i^{*}}^{n}/d(p_{i^{*}}^{n},\theta_{i^{*}}^{n},v_{i^{*}}^{n},m_{i^{*}}^{n}) \to 0$ as $n  \to \infty$. With these results, $C_{2}^{n}$ yields that $p_{i^{*}}^{n}\Delta m_{i^{*}}^{n}/d(p_{i^{*}}^{n},\theta_{i^{*}}^{n},v_{i^{*}}^{n},m_{i^{*}}^{n}) \to 0$ as $n \to \infty$. As a consequence, by summing these terms up, we obtain 
\begin{eqnarray}
1 \lesssim \left(|\Delta p_{i^{*}}^{n}|+|\Delta \theta_{i^{*}}^{n}|+|\Delta v_{i^{*}}^{n}|+|\Delta m_{i^{*}}^{n}|\right) /d(p_{i^{*}}^{n},\theta_{i^{*}}^{n},v_{i^{*}}^{n},m_{i^{*}}^{n}) \to 0, \nonumber
\end{eqnarray}
which is a contradiction.

\paragraph{Case 1.1.2:} If $v_{i^{*}}^{n}=v_{i^{*}}^{0}$ for infinitely $n$, then we also can assume it holds for all $n$. From $C_{2}^{n}$ and $C_{3}^{n}$, we have $p_{i^{*}}^{n}(\Delta \theta_{i^{*}}^{n})^{2}/d(p_{i^{*}}^{n},\theta_{i^{*}}^{n},v_{i^{*}}^{n},m_{i^{*}}^{n}) \to 0$. Therefore, $p_{i^{*}}^{n}\Delta \theta_{i^{*}}^{n}\Delta m_{i^{*}}^{n}/d(p_{i^{*}}^{n},\theta_{i^{*}}^{n},v_{i^{*}}^{n},m_{i^{*}}^{n}) \to 0$. Combining these results with $C_{1}^{n}$, we get $\Delta  p_{i^{*}}^{n}/d(p_{i^{*}}^{n},\theta_{i^{*}}^{n},v_{i^{*}}^{n},m_{i^{*}}^{n}) \to 0$ as $n  \to \infty$. Additionally, by taking square of $C_{2}^{n}$, we obtain $p_{i^{*}}^{n}(\Delta m_{i^{*}}^{n})^{2}/d(p_{i^{*}}^{n},\theta_{i^{*}}^{n},v_{i^{*}}^{n},m_{i^{*}}^{n}) \to 0$ as $n \to \infty$. These results imply that $d(p_{i^{*}}^{n},\theta_{i^{*}}^{n},v_{i^{*}}^{n},m_{i^{*}}^{n})/d(p_{i^{*}}^{n},\theta_{i^{*}}^{n},v_{i^{*}}^{n},m_{i^{*}}^{n}) \to 0$ as $n \to \infty$, which is a contradiction.

\paragraph{Case 1.1.3:} If $m_{i^{*}}^{n}=m_{i^{*}}^{0}$ for infinitely $n$, then we can assume that it holds for all $n$. Combining $C_{2}^{n}$ and $C_{4}^{n}$, we obtain $p_{i^{*}}^{n}\Delta \theta_{i^{*}}^{n}/d(p_{i^{*}}^{n},\theta_{i^{*}}^{n},(v_{i^{*}}^{n})^{2},m_{i^{*}}^{n}) \to 0$ as $n \to \infty$. Combining this result with $C_{3}^{n}$ and $C_{1}^{n}$, we achieve $p_{i^{*}}^{n}\Delta v_{i^{*}}^{n}/d(p_{i^{*}}^{n},\theta_{i^{*}}^{n},v_{i^{*}}^{n},m_{i^{*}}^{n}) \to 0$ and $\Delta p_{i^{*}}^{n}/d(p_{i^{*}}^{n},\theta_{i^{*}}^{n},(v_{i^{*}}^{n})^{2},m_{i^{*}}^{n}) \to 0$ as $n  \to \infty$. This leads to a contradiction as well.

\end{document}